\documentclass[11pt]{article}

\usepackage{amsmath}
\renewcommand{\a}{\alpha}
\usepackage[hidelinks]{hyperref}

\renewcommand{\Re}{\mathrm{Re}\,}
\renewcommand{\Im}{\mathrm{Im}\,}

\makeatletter
\renewcommand\@dotsep{10000}
\makeatother

\usepackage{amssymb}
\usepackage{mathtools}
\usepackage[numbers]{natbib}
\usepackage[margin=1in]{geometry}
\usepackage[title]{appendix}

\usepackage{amsthm}
\newtheorem{thm}{Theorem}[section]
\newtheorem{lem}[thm]{Lemma}
\newtheorem{prop}[thm]{Proposition}
\newtheorem{cor}[thm]{Corollary}

\newtheorem{propalt}{Proposition}[thm]
\newenvironment{propC}[1]{
\renewcommand\thepropalt{#1}
\propalt
}{\endpropalt}

\theoremstyle{definition}

\usepackage{stucky}

\renewcommand{\theequation}{\arabic{section}.\arabic{equation}}
\usepackage{chngcntr}
\counterwithin*{equation}{section}

\allowdisplaybreaks

\let\LaTeXStandardTableOfContents\tableofcontents
\renewcommand{\tableofcontents}{%
\begingroup%
\renewcommand{\bfseries}{\relax}%
\LaTeXStandardTableOfContents%
\endgroup%
}%

\let\svthefootnote\thefootnote
\newcommand\freefootnote[1]{%
  \let\thefootnote\relax%
  \footnotetext{#1}%
  \let\thefootnote\svthefootnote%
}

\makeatletter
\renewcommand\tableofcontents{%
	\null\hfill\textbf{\Large\contentsname}\hfill\null\par
	\@mkboth{\MakeUppercase\contentsname}{\MakeUppercase\contentsname}%
	\@starttoc{toc}%
}

\begin{document}
\bibliographystyle{abbrv}



\author{Quanli Shen and Joshua Stucky}


\title{\scshape The fourth moment of quadratic Dirichlet $L$-functions II}




\date{}
\maketitle

\begin{abstract}
We prove an asymptotic formula with four main terms for the fourth moment of quadratic Dirichlet $L$-functions unconditionally. Our proof is based on the work of Li \cite{XiannanFourth}, Soundararajan \cite{SoundNonvanishing}, and Soundararajan-Young \cite{SoundYoungSecond}. Our proof requires several new ingredients. These include a modified large sieve estimate for quadratic characters where we consider a fourth moment, rather than a second, as well as observing cross cancellations between diagonal and off-diagonal terms, which involve somewhat delicate combinatorial arguments.
\end{abstract}

\freefootnote{2020 \textit{Mathematics Subject Classification}. 11M06, 11M50}
\freefootnote{Keywords: Moments of $L$-functions, quadratic Dirichlet $L$-functions}

\clearpage
\tableofcontents
\clearpage

\section{Introduction and Statement of Results}\label{sec:Intro}
The theory of moments of $L$-functions is one of the central problems in analytic number theory. It has attracted much research attention, both for its own sake and because of many fruitful applications. 
Here we study moments of the family of quadratic Dirichlet $L$-functions. This problem is closely related to Chowla's conjecture, which states that all quadratic Dirichlet $L$-functions do not vanish at the central point $s=\half$. A breakthrough was made by Soundararajan \cite{SoundNonvanishing}, who showed that the conjecture is true for $87.5\%$ of quadratic Dirichlet $L$-functions. This was accomplished by establishing asymptotic formulas for the first and second moments of quadratic Dirichlet $L$-functions. In the same paper, Soundararajan also obtained an asymptotic formula for the third moment of quadratic Dirichlet $L$-functions. The aim of this article is to establish unconditionally an asymptotic formula for the fourth moment of quadratic Dirichlet $L$-functions. This was previously known under the generalized Riemann hypothesis (GRH) by work of the first author \cite{ShenFourth}.

To better define our problem, we need some notation. For $m$ a fundamental discriminant, let $\chi_{m}(\cdot) = \fracp{m}{\cdot}$ be a real primitive Dirichlet character with modulus $\abs{m}$, where $\fracp{m}{\cdot}$ denotes the Kronecker symbol. The Dirichlet $L$-function associated to $\chi_{m}$ is
\[
L(s,\chi_{m}) := \sum_{n} \frac{\chi_{m}(n)}{n^s} = \prod_{p} \pth{1-\frac{\chi_{m}(p)}{p^s}}^{-1}
\]
for $\Re(s) > 1$, and this can be analytically continued to the entire complex plane. The completed $L$-function is given by
\begin{align}
\Lambda(s,\chi_{m}) := \fracp{\abs{m}}{\pi}^{\frac{s+\ida}{2}} \Gamma\pth{\frac{s+\ida}{2}}L(s,\chi_{m}) = \Lambda(1-s,\chi_{m}),
\label{eq:FunctionalEquation}
\end{align}
where $\ida = 0$ or 1 depending as $m > 0$ or $m < 0$.
We are interested in the family of quadratic Dirichlet $L$-functions,
\[
\mathscr{F}(X) = \set{L(s,\chi_m): \abs{m} \leq X,\ \text{$m$ a fundamental discriminant}}.
\]
The philosophy of Katz and Sarnak  \cite{K-S}  predicts that this family of $L$-functions has symplectic symmetry type. We will be concerned with the moments of this family at the central point $s=\frac{1}{2}$. These take the form
\begin{equation}
\sumflat_{\abs{m} \leq X} L\pth{\half,\chi_{m}}^k.
\label{equ:original4th}
\end{equation}
Here and throughout the paper, $\sumflat$ denotes a sum over fundamental discriminants. By using a random matrix model, Keating and Snaith \cite{KeatingSnaithRMT} conjectured that, for any positive real number $k$,
\begin{equation}
\sumflat_{\abs{m} \leq X} L\pth{\half,\chi_{m}}^k \sim C_k X(\log X)^\frac{k(k+1)}{2},
\label{conj-keating}
\end{equation}
where $C_k$ is an explicit constant.  Conrey {\it et al.} \cite{Conrey-Farmer-Keating-Rubinstein-Snaith} later established a more precise conjecture with lower-order main terms by using the ``recipe" method,
\begin{equation}
\sumflat_{\abs{m} \leq X} L\pth{\half,\chi_{m}}^k \sim XP_{\frac{k(k+1)}{2}}(\log X) ,
\label{conj-Conrey}
\end{equation}
where $P_{n}(x)$ is an explicit polynomial of degree $n$. Recently, Diaconu and Twiss \cite{Diaconu-Twiss} formulated a conjecture for further lower-order terms in the asymptotic formula in the function field setting.

Note that if $d$ is odd and squarefree, then $8d$ is a fundamental discriminant. 
It is also common for researchers to consider moments of the family of quadratic Dirichlet $L$-functions associated to these moduli, namely,
\begin{equation}\label{def:FourthMoment}
\sumstar_{\substack{0 < d \leq X\\ (d,2)=1}} L\pth{\half,\chi_{8d}}^k,
\end{equation}
where $\sumstar$ denotes the sum over squarefree integers. There is no essential difference between \eqref{equ:original4th} and \eqref{def:FourthMoment}, and the advantage of \eqref{def:FourthMoment} is that we may focus on the main method rather than on technical issues. One can also make conjectures for  \eqref{def:FourthMoment} 
similar to  \eqref{conj-keating} and \eqref{conj-Conrey}  by modifying the methods of 
Keating and Snaith \cite{KeatingSnaithRMT} (see Andrade and Keating   \cite[Conjecture 2]{AndradeKeating}) and Conrey {\it et al.} \cite{Conrey-Farmer-Keating-Rubinstein-Snaith}. 

For historical context, we note that Jutila \cite{Jutila1&2} established asymptotic formulas for the first and second moments of quadratic Dirichlet $L$-functions. Soundararajan \cite{SoundNonvanishing} improved Jutila's results and, as mentioned above, also proved an asymptotic formula for the third moment. The error terms for the first, second and third moments of quadratic Dirichlet $L$-functions have been subsequently improved, and we refer the reader to \cite{Diaconu-Goldfeld-Hoffstein, Diaconu-Whitehead, Goldfeld-Hoffstein, Sono, YoungActa, Young,  Zhangqiao} for more information on these problems. 

In the function field setting, Florea  \cite{Florea} obtained an asymptotic formula with lower-order main terms for the fourth moment. Assuming GRH, the first author  \cite{ShenFourth} proved an asymptotic formula with the leading main term for the fourth moment in the number field setting. This was done using the method of Soundararajan and Young \cite{SoundYoungSecond} for investigating the second moment of quadratic twists of modular $L$-functions. In that paper, Soundararajan and Young established an asymptotic formula conditionally under GRH. This dependence on GRH was recently removed by Li \cite{XiannanFourth}.

In this paper, we use the ideas of Li and prove an asymptotic formula for the  shifted  fourth moment of quadratic Dirichlet $L$-functions. This establishes, for the first time, an asymptotic formula for the fourth moment of a symplectic family of $L$-functions.
\begin{thm} \label{thm:Main} 
Let $\Phi: (0,\infty) \rightarrow \mathbb{R}$ be a
smooth, compactly supported function.
Let $\a_i\in \mathbb{C} $ and $\frac{1}{2\log X} \leq |\a_i| \leq \frac{1}{\log X} $ for $i=1,2,3,4$.  Assume 
\begin{equation*}
|\eta_1\a_1 +  \eta_2\a_2  +  \eta_3\a_3 +  \eta_4\a_4| \geq \frac{1}{50\log X}
\end{equation*}
for  $\eta_j=\pm 1, 0$  and $\eta_j$ not all zero.  Then we have 
\begin{align*}
&\sumstar_{(d,2)=1} L(\tfrac{1}{2} + \a_1,\chi_{8d})L(\tfrac{1}{2} + \a_2, \chi_{8d} )L(\tfrac{1}{2} + \a_3, \chi_{8d} )L(\tfrac{1}{2} + \a_4, \chi_{8d}) \Phi(\tfrac{d}{X})\nonumber \\
&=\frac{4X}{\pi^2} 
\sum_{\epsilon_1,\epsilon_2,\epsilon_3,\epsilon_4 = \pm 1}    (8X)^{-\sum_{i=1}^4\frac{1-\epsilon_i}{2}\a_i}
\prod_{1\leq i \leq 4}\lambda_i(\a_i)^{\frac{1-\epsilon_i}{2}}
\\
&\quad \times  \int\limits_{-\infty}^\infty  \Phi_{\{\epsilon_1,\epsilon_2,\epsilon_3,\epsilon_4\}}(x) \, dx  \prod_{1\leq i \leq j \leq 4} \zeta_2(1+\epsilon_i\a_i+\epsilon_j\a_j) \mathcal{H}_2(\tfrac{1}{2}+\epsilon_1\a_1 ,\tfrac{1}{2}+\epsilon_2\a_2 , \tfrac{1}{2}+\epsilon_3\a_3 , \tfrac{1}{2}+\epsilon_4\a_4 )\\
&\quad +O(X (\log X)^{6+ \varepsilon}),
\end{align*}
where $\lambda_i$ and $\Phi_{\{\epsilon_1,\epsilon_2,\epsilon_3,\epsilon_4\}}$ are defined in \eqref{def-lambda} and \eqref{tranf-phi}, and $\mathcal{H}_2$ is defined in \eqref{def:H}.
\end{thm}

\subsection{Background}

Conrey {\it et al.} \cite{Conrey-Farmer-Keating-Rubinstein-Snaith} have formulated conjectures for shifted versions of moments of various families of $L$-functions. They also show how to remove these shifts in order to obtain an asymptotic formula for the original moments. The advantage of using shifts is that, compared to the original moments, the main terms for shifted moments arise with a symmetric pattern, and when computing certain contour integrals, one only needs to deal with simple poles, rather than higher-order poles.  

The heuristic in  \cite{Conrey-Farmer-Keating-Rubinstein-Snaith} is that only the ``diagonal terms'' contribute to the main terms of moments of $L$-functions, while the other terms, including ``off-diagonal'' terms, somehow cancel out (our terminology will be clear later). The structure of the off-diagonal terms is usually much more complicated than that of the diagonal terms, and there is currently no clear explanation for why cancellation should occur among these terms. In our proof of Theorem \ref{thm:Main}, we see cancellations between off-diagonal terms and certain parts of the diagonal terms, and the remaining parts of the diagonal terms give the main term of the fourth moment. This gives evidence for the fundamental idea in \cite{Conrey-Farmer-Keating-Rubinstein-Snaith} described above in the case of the fourth moment of quadratic Dirichlet $L$-functions.

The method described above relies on approximate functional equations for $L$-functions (see Section \ref{sec:AFE} below) and has been widely used to study various moments of $L$-functions.  Another approach, known as multiple Dirichlet series, is also powerful in studying these moments and has yielded surprising results. For example, Diaconu and Whitehead \cite{Diaconu-Whitehead} proved the existence of an extra main term of the form $cX^{\frac{3}{4}}$ for the third moment. More recently, Diaconu, Paşol, and Popa \cite{Diaconu-4th} established an asymptotic formula for a weighted fourth moment in the functional field setting with a power saving error term. In our proof of Theorem \ref{thm:Main}, we use the method of approximate functional equations. Our main technical result leading to the error term in Theorem \ref{thm:Main} is Proposition \ref{prop:LargeSieve} below. From the proof of this result, it seems that one would need additional ideas to improve the error term.

By Theorem \ref{thm:Main} and the method of Conrey {\it et al.} \cite{Conrey-Farmer-Keating-Rubinstein-Snaith},  we can derive an asymptotic formula for the smoothed fourth moment of quadratic Dirichlet $L$-functions.

\begin{cor}
\label{thm:firstmain}
Let $\Phi$ and $\a_i$ be the same as in Theorem \ref{thm:Main}. Then 
\begin{align}
&\sumstar_{(d,2)=1} L(\tfrac{1}{2} + \a_1,\chi_{8d})L(\tfrac{1}{2} + \a_2, \chi_{8d} )L(\tfrac{1}{2} + \a_3, \chi_{8d} )L(\tfrac{1}{2} + \a_4, \chi_{8d}) \Phi(\tfrac{d}{X})\nonumber \\
&=
 X \int\limits_{-\infty}^\infty  {\Phi}(x)
P(\log \tfrac{8X}{\pi};x,\a_1,\a_2,\a_3,\a_4) \, dx +O(X (\log X)^{6+ \varepsilon}),
\label{equ-cor-1}
\end{align}
where 
\begin{align*}
&P(y;x,\a_1,\a_2,\a_3,\a_4)\\
&:= \frac{8}{3\pi^2}  \frac{1}{(2 \pi  i )^4} \oint \cdots \oint 
x^{\frac{1}{2}\sum_{i=1}^4(z_i-\a_i)}
e^{\frac{y}{2}\sum_{i=1}^4(z_i-\a_i)} 
\prod_{i=1}^4\frac{\Gamma(\frac{1/2-\a_i}{2})^{1/2}}{\Gamma(\frac{1/2+\a_i}{2})^{1/2}} 
\frac{\Gamma(\frac{1/2+z_i}{2})^{1/2}}{\Gamma(\frac{1/2-z_i}{2})^{1/2}}\\
& \quad \times \mathcal{H}_2(\tfrac{1}{2}+z_1,\tfrac{1}{2}+z_2,\tfrac{1}{2}+z_3,\tfrac{1}{2}+z_4) \prod_{1\leq i \leq j \leq 4} \zeta_2(1+z_i+z_j)\\
&\quad\times \frac{\Delta(z_1^2, \dots , z_4^2)^2 \prod_{1\leq i \leq 4} z_i}{ \prod_{1\leq i \leq 4} \prod_{1\leq j \leq 4} (z_i-\a_j)(z_i+\a_j) }
\, d z_1 \cdots d z_4 , 
\end{align*}
and
\begin{align*}
\Delta(z_1, \dots, z_4) := \prod_{1\leq i < j \leq 4} (z_j-z_i),
\end{align*}
and the contour of integration encloses the $\pm \a_j$.

As well, we have
\begin{align} 
\sumstar_{ (d,2)=1} L\pth{\half,\chi_{8d}}^4  \Phi(\tfrac{d}{X})= X 
Q_{10}\pth{\log \tfrac{8X}{\pi}} +O(X (\log X)^{6+ \varepsilon}).
\label{equ-cor-2}
\end{align}
Here $Q_{10}(y):= \int_{-\infty}^\infty  {\Phi}(x) P(y;x,0,0,0,0) \, dx$ is a polynomial of degree $10$ whose leading coefficient is
\begin{align*}
\frac{\hat{\Phi}(0)}{2^{18} \cdot 3^3\cdot  5^2\cdot 7 \cdot \pi^2} \prod_{(p,2)=1} \frac{(1-\frac{1}{p})^{10}}{1+\frac{1}{p}} \pth{\frac{(1+\frac{1}{\sqrt{p}})^{-4} + (1-\frac{1}{\sqrt{p}})^{-4}}{2} + \frac{1}{p}},
\end{align*}
and $\hat{\Phi}$ is the Fourier transform defined in \eqref{equ:fourier-trans}.
\end{cor}
The main term in \eqref{equ-cor-1}  is equivalent to the one in Conjecture 1.5.3 of \cite{Conrey-Farmer-Keating-Rubinstein-Snaith} in the case $k=4$ after applying Mellin inversion for $\Phi$ and converting fundamental discriminants $d$ into $8d$ with  $d$ squarefree in the argument of \cite{Conrey-Farmer-Keating-Rubinstein-Snaith}. The coefficients for the lower order terms of $Q(y)$ may be computed using an argument similar to that of Goulden, Huynh, Rishikesh and Rubinstein \cite{Goulden-etc}.  We reiterate that Theorem \ref{thm:Main} and Corollary \ref{thm:firstmain} are unconditional in that they do not depend on GRH. The proof of Theorem \ref{thm:Main} follows the method of Li \cite{XiannanFourth} recently used to establish unconditionally an asymptotic formula for the second moment of quadratic twists of modular $L$-functions. There, the key step is to establish the large sieve inequality
\begin{equation}\label{Li-prop}
\sumflat_{M <\abs{m} \leq 2M} \sumabs{\sum_{n} \frac{\lambda_f(n) \chi_m(n) }{n^{\frac{1}{2}+ i  t}}G\fracp{n}{N}}^2 \leq \coll^\frac{2}{3}(1+\abs{t})^2\pth{M + N\log(2+N/M)}.
\end{equation}
for some constant $\coll>0$. Here $\lambda_f(n)$ denote the Fourier coefficients of the modular form $f$. In our case, we establish the following estimate.

\begin{prop} \label{prop:LargeSieve}
For $M,N\geq 1$, there exists an absolute constant $\coll > 0$ such that
\begin{align}
\sumflat_{M < \abs{m} \leq 2M} \sumabs{\sum_{n} \frac{\chi_m(n)}{n^{\frac{1}{2}+ i  t}}G\fracp{n}{N}}^4 \leq \coll^\frac{2}{3}(1+\abs{t})^2\pth{M + N^2\log(2+N^2/M)}\log^6(2+MN).
\label{Large-sie}
\end{align}
Here $G$ is a smooth, compactly supported function in $[3/4, 2]$ defined in Lemma \ref{lem:ConvenientTestFandG}(ii).
\end{prop}

The critical ranges of \eqref{Li-prop} and \eqref{Large-sie} are $M \asymp N$ and $M \asymp N^2$, respectively. In this case, standard tools from harmonic analysis (Poisson summation in the variable $m$ or the functional equations for $L(s,f\otimes \chi)$ and $L(s,\chi)$ for the variable $n$) are of no use. The key innovation of Li is to ``inflate'' the sum over $m$ so that Poisson summation becomes useful again, and \eqref{Li-prop} is then established via a nested induction argument on $M$ and $N$. We discuss this idea more at the beginning of Section \ref{sec:Inflation}.

Let us remark on a few of the major differences between \eqref{Li-prop} and \eqref{Large-sie}. First, we note that the proof of \eqref{Li-prop} relies crucially on the properties of the coefficients $\lambda_f(n)$ in the following way. In order to apply Poisson summation to $m$ in either \eqref{Li-prop} or \eqref{Large-sie}, one essentially needs $m$ to run over the full dyadic range $M < m \leq 2M$, rather than just fundamental discriminants. To obtain an upper bound, one may simply ignore the condition that $m$ is a fundamental discriminant. Critically, the sum over $m$ in the upper bound now includes those $m$ which are squares of integers. For these $m$, the character $\chi_m$ is trivial and so contributes no oscillation to the sum over $n$. In the case of \eqref{Li-prop}, this is no issue because one still has oscillation from the coefficeints $\lambda_f(n)$. In the case of \eqref{Large-sie}, however, we are not so lucky. Indeed, if $m$ is a square, we have
\[
\sum_{n} \frac{\chi_m(n)}{n^{\frac{1}{2}+ i  t}}G\fracp{n}{N} = \sum_{(n,m)=1} \frac{1}{n^{\frac{1}{2}+ i  t}}G\fracp{n}{N} \asymp \frac{\phi(m)}{m} \sqrt{N}.
\]
Summing over the square values of $m$ in the interval $[M,2M)$, we see that these values of $m$ give a contribution that is
\[
\gg \sqrt{M}N^2,
\]
which is disastrously large (see the footnote on page 5 of \cite{XiannanFourth}). To circumvent this issue, we simply note that since the original sum over $m$ included no squares, we may subtract the contribution of these $m$ and still obtain an upper bound. Applying Poisson summation in $m$ in the full dyadic interval $[M,2M)$, we then proceed as in \cite{XiannanFourth}. On the dual side, in the variable $k$, say, one encounters the same situation as in Soundararajan's work \cite{SoundNonvanishing}, in that those $k$ which are perfect squares contribute meaningfully to the dual sum and cannot be ignored. Thankfully, after some careful manipulations with Dirichlet series, one finds that the contribution of square $m$ cancels with the contribution of square $k$ on the dual side up to a pair of sufficiently small error terms. In order to see this cancellation, we require a significantly more detailed evaluation of certain multiple Dirichlet series (compare Lemma \ref{lem:DirichletSeriesOffDiagonal} below with Lemma 2.5 of Li \cite{XiannanFourth} and Lemma 3.3 of Soundararajan and Young \cite{SoundYoungSecond}). 

The second major difference between \eqref{Li-prop} and \eqref{Large-sie} is the differing exponents 2 and 4. The most natural analogue of the expression in \eqref{Li-prop} for quadratic Dirichlet $L$-functions is
\begin{equation}\label{eq:analogue}
\sumflat_{M \leq \abs{m} < 2M} \sumabs{\sum_{n} \frac{\tau(n) \chi_m(n) }{n^{\frac{1}{2}+ i  t}}G\fracp{n}{N}}^2.
\end{equation}
The reason we consider the expression in \eqref{Large-sie} rather than \eqref{eq:analogue} can best be explained by considering the nature of the function $\tau(n)$. Classically, $\tau(n)$ is of size $\log N$ on average over $n$ of size $N$. However, one expects $\tau(n)$ to typically be much smaller; its average value is actually dominated by a few exceptional integers having many small prime factors. The advantage of working with \eqref{Large-sie} over \eqref{eq:analogue} can be seen in the simple estimates
\[
\sum_{n\sim N} \tau(n) \asymp N\log N \mand \sumpth{\sum_{n\sim \sqrt{N}} 1}^2 \asymp N.
\]
By working with \eqref{Large-sie} instead of \eqref{eq:analogue}, we avoid the contribution of those integers with an exceptionally large number of divisors. 

This phenomenon manifests itself via the nature of our harmonic-analytic manipulations. After opening the absolute values and applying Poisson summation in $m$, we express the 2 or 4 sums over $n$ as 2- or 4-fold contour integrals of certain Dirichlet series. In shifting contours to evaluate these integrals, we essentially lose a pole every time we shift contours. As will be seen, using \eqref{Large-sie} leads to a bound of size $M(\log N)^6$, whereas \eqref{eq:analogue} would only lead to a bound of size $M(\log N)^8$.

As in \cite{XiannanFourth}, the crucial role of Proposition \ref{prop:LargeSieve} is to shorten the length of the Dirichlet polynomial corresponding to the fourth moment of quadratic Dirichlet $L$-functions from $X^2$ to $X^2 U$, where $U = (\log X)^{-Q}$ for some large constant $Q>0$ (see \eqref{defU}). Previously, Soundararjan and Young \cite{SoundYoungSecond} used the correlation between twisted $L$-functions, which is established under GRH, to reduce the length of the Dirichlet polynomial, and similar ideas were later employed in \cite{Florea, NSW, ShenFourth}. 

To evaluate the diagonal terms and off-diagonal square terms of the fourth moment of quadratic Dirichlet $L$-functions, similar to the ideas in \cite{Conrey-Farmer-Keating-Rubinstein-Snaith}, we expect cancellations between different terms, especially those contributions coming from small shifts. 
However, in our consideration, the parameter $U$ prevents these cancellations. Morally speaking, we have to deal with the fact that $X-XU\neq 0$ unless $U =1$. Other works have avoided computing these cancellations (and so avoided the issue arising from $U$ if we use this idea) by introducing a smooth function $\mathcal{G}(s)$ in the approximate functional equation (see \cite{Bettin,  YoungActa, Young} and in particular Remark 2.2 of \cite{YoungActa}) that is $1$ when $s=0$ and vanishes when $s$ are combinations of shifts, thus cancelling poles arising from shifts.
Unfortunately, this function contributes many powers of $\log X$ when $s$ is away from $0$. Since our error term is sensitive to powers of $\log X$, we are unable to use this technical device. In Section \ref{sec:removeU}, we essentially change $U$ back to $1$ via some careful combinatorial manipulations with a suitable error of $O\pth{X(\log X)^{6+ \varepsilon}}$, after which we are able to observe the cancellation we need. 

We conclude this section by showing how Corollary \ref{thm:firstmain} follows from Theorem  \ref{thm:Main}.

\begin{proof}[Proof of Corollary \ref{thm:firstmain} by Theorem  \ref{thm:Main}]
We write the main term in  Theorem  \ref{thm:Main} as 
\begin{align}
g(\a_1,\a_2,\a_3,\a_4) &:= \sum_{\epsilon_1,\epsilon_2,\epsilon_3,\epsilon_4 = \pm 1}    (8X)^{-\sum_{i=1}^4\frac{1-\epsilon_i}{2}\a_i}
\prod_{1\leq i \leq 4}\lambda_i(\a_i)^{\frac{1-\epsilon_i}{2}} \frac{4X}{\pi^2}  
\int\limits_{-\infty}^\infty  \Phi_{\{\epsilon_1,\epsilon_2,\epsilon_3,\epsilon_4\}}(x)
\, dx  \nonumber\\
&\quad \times \prod_{1\leq i \leq j \leq 4} \zeta_2(1+\epsilon_i\a_i+\epsilon_j\a_j) \mathcal{H}_2(\tfrac{1}{2}+\epsilon_1\a_1 ,\tfrac{1}{2}+\epsilon_2\a_2 ,\tfrac{1}{2} +\epsilon_3\a_3 , \tfrac{1}{2}+\epsilon_4\a_4 ) .
\label{equ:sec1-01}
\end{align}
By  \eqref{def-lambda} and \eqref{tranf-phi}, we  have
\begin{align*}
& (8X)^{-\sum_{i=1}^4\frac{1-\epsilon_i}{2}\a_i}
\prod_{1\leq i \leq 4}\lambda_i(\a_i)^{\frac{1-\epsilon_i}{2}}  \Phi_{\{\epsilon_1,\epsilon_2,\epsilon_3,\epsilon_4\}}(x)\\
&= \prod_{i=1}^4 \left( \frac{8xX}{\pi}\right)^{\frac{-\a_i}{2}} \frac{\Gamma(\frac{1/2-\a_i}{2})^{1/2}}{\Gamma(\frac{1/2+\a_i}{2})^{1/2}}
\times \prod_{i=1}^4 \left( \frac{8xX}{\pi}\right)^{\frac{\epsilon_i\a_i}{2}} \frac{\Gamma(\frac{1/2+\epsilon_i\a_i}{2})^{1/2}}{\Gamma(\frac{1/2-\epsilon_i\a_i}{2})^{1/2}}  \Phi(x), 
\end{align*}
where we have used the fact that for $\epsilon_i =\pm 1$, 
\[
\frac{\Gamma(\frac{1/2+\a_i}{2})^{\epsilon_i/2}}{\Gamma(\frac{1/2-\a_i}{2})^{\epsilon_i/2}} =
\frac{\Gamma(\frac{1/2+\epsilon_i\a_i}{2})^{1/2}}{\Gamma(\frac{1/2-\epsilon_i\a_i}{2})^{1/2}}.
\]
Combining this with \eqref{equ:sec1-01}, we have
\begin{align*}
g(\a_1,\a_2,\a_3,\a_4) &= \frac{4X}{\pi^2}  \int\limits_{-\infty}^\infty  \Phi(x) \prod_{i=1}^4 \left( \frac{8xX}{\pi}\right)^{\frac{-\a_i}{2}} \frac{\Gamma(\frac{1/2-\a_i}{2})^{1/2}}{\Gamma(\frac{1/2+\a_i}{2})^{1/2}} \nonumber \\
&\quad \times \sum_{\epsilon_1,\epsilon_2,\epsilon_3,\epsilon_4 = \pm 1}  
F(\epsilon_1 \a_1, \epsilon_2 \a_2, \epsilon_3 \a_3, \epsilon_4 \a_4)\prod_{1\leq i \leq j \leq 4} f(\epsilon_i\a_i+\epsilon_j\a_j)\, dx,
\nonumber
\end{align*}
where 
\begin{align*}
F(y_1,y_2,y_3,y_4)&:= \prod_{i=1}^4 \left( \frac{8xX}{\pi}\right)^{\frac{y_i}{2}} \frac{\Gamma(\frac{1/2+y_i}{2})^{1/2}}{\Gamma(\frac{1/2-y_i}{2})^{1/2}} \mathcal{H}_2(\tfrac{1}{2}+y_1,\tfrac{1}{2}+ y_2,\tfrac{1}{2}+ y_3,\tfrac{1}{2}+ y_4) \\
&\quad \times \prod_{1\leq i\leq j\leq 4}\left( 1- \frac{1}{2^{1+y_i+y_j}}\right),\\
f(y) &:= \zeta(1+y).
\end{align*}
By Lemma 2.5.2 of \cite{Conrey-Farmer-Keating-Rubinstein-Snaith}, we know 
\begin{align*}
&g(\a_1,\a_2,\a_3,\a_4) \nonumber\\
&= \frac{4X}{\pi^2}  \int\limits_{-\infty}^\infty  \Phi(x) \prod_{i=1}^4 \left( \frac{8xX}{\pi}\right)^{\frac{-\a_i}{2}} \frac{\Gamma(\frac{1/2-\a_i}{2})^{1/2}}{\Gamma(\frac{1/2+\a_i}{2})^{1/2}} 
\frac{1}{(2 \pi  i )^4} \frac{2}{3} \oint \cdots \oint \prod_{i=1}^4 \left( \frac{8xX}{\pi}\right)^{\frac{z_i}{2}} \frac{\Gamma(\frac{1/2+z_i}{2})^{1/2}}{\Gamma(\frac{1/2-z_i}{2})^{1/2}}\\
&\quad \times \mathcal{H}_2(\tfrac{1}{2}+ z_1,\tfrac{1}{2}+z_2,\tfrac{1}{2}+z_3,\tfrac{1}{2}+z_4) \prod_{1\leq i \leq j \leq 4} \zeta_2(1+z_i+z_j) \nonumber\\
&\quad \times \frac{\Delta(z_1^2, \dots , z_4^2)^2 \prod_{1\leq i \leq 4} z_i}{ \prod_{1\leq i \leq 4} \prod_{1\leq j \leq 4} (z_i-\a_j)(z_i+\a_j) }
\, d z_1 \cdots d z_4 \, dx,
\label{equ:sec1-02}
\end{align*}
where the contours of integration enclose $\pm \a_i$. This gives the main term of  \eqref{equ-cor-1}. Note that in the above, $g(\a_1,\a_2,\a_3,\a_4)$ is analytic when $|\a_i|<\eta$, $i=1,2,3,4$ for sufficiently small $\eta>0$. Moreover, the left-hand side of Theorem \ref{thm:Main} is also analytic in this region. It follows that the error term in Theorem \ref{thm:Main} is an analytic function in the same region, which means that the error term is valid for $|\a_i|<\eta$ by the maximum modulus
principle. Therefore, letting  $\a_i \rightarrow 0$, $i=1,2,3,4$, we obtain 
\begin{align*}
&\sumstar_{(d,2)=1} L(\tfrac{1}{2} + \a_1,\chi_{8d})L(\tfrac{1}{2} + \a_2, \chi_{8d} )L(\tfrac{1}{2} + \a_3, \chi_{8d} )L(\tfrac{1}{2} + \a_4, \chi_{8d}) \Phi(\tfrac{d}{X})\nonumber \\
&=
g(0,0,0,0)
+O(X (\log X)^{6+ \varepsilon}).
\end{align*} 
This completes the proof of \eqref{equ-cor-2}. 
\end{proof}    

\subsection{Outline of the Paper}\label{sec:Outline}

We give here an overview of the proof of Theorem \ref{thm:Main}. The proof is divided into two parts: the proof of Proposition \ref{prop:LargeSieve} and its application in calculating the shifted moment given in Theorem \ref{thm:Main}. The various sections of the paper are as follows.

In Section \ref{sec:Prelims}, we collect a number of preliminary calculations we will need throughout the paper. Most of these are cited from other works and stated without proof. 

In Sections \ref{sec:LargeSieve-InitialSteps} -- \ref{sec:LargeSieve-NonSqare}, we prove Proposition \ref{prop:LargeSieve}. In Section \ref{sec:LargeSieve-InitialSteps}, we perform several initial steps in the proof, such as introduce some notation, state and prove our ``inflation lemma,'' establish the base cases of our induction argument, and perform the initial application of Poisson summation in the form of Lemma \ref{lem:Poisson}. This reduces the proof of  Proposition \ref{prop:LargeSieve} to proving Lemmas \ref{lem:LS-Diagonal}, \ref{lem:LS-SquareCancellation}, and \ref{lem:LS-NonSquare}. In Section \ref{sec:LargeSieve-Diagonal}, we evaluate the diagonal contribution from Poisson summation and prove Lemma \ref{lem:LS-Diagonal}. In Section \ref{sec:LargeSieve-OffDiagonal}, we show how the main terms of the contributions from square values of both the original variable $m$ and the dual variable $k$ cancel exactly. We are left with a pair of contour integrals, and we also handle these in this section. In Section \ref{sec:LargeSieve-NonSqare}, we complete the proof of Proposition \ref{prop:LargeSieve} by bounding the contribution of the non-square values of the dual variable $k$ and establishing Lemma \ref{lem:LS-NonSquare}. It is in the treatment of these off-diagonal non-square values of $k$ that we need to apply the induction hypothesis.

The second half of the paper consists of Sections \ref{sec:Asymptotic-Shortening} -- \ref{sec:secU-error}. In Section \ref{sec:Asymptotic-Shortening}, we apply the approximate functional equation to the shifted moment given in Theorem \ref{thm:Main} and shorten the resulting Dirichlet polynomials using Proposition \ref{prop:LargeSieve}. This involves some delicate combinatorial arguments. In Section \ref{sec:Asymptotic-Diagonal-4th}, we perform some standard reductions and apply Poisson summation in the form of Lemma \ref{lem:Poisson}.  In Section \ref{sec:Asymptotic-Diagonal}, we evaluate the diagonal contribution from Poisson summation, and then in Section \ref{sec:Asymptotic-OffDiagonal}, we evaluate the contribution of square values of the dual variable. In Section \ref{sec:Asymptotic-NonSquare}, we estimate the contribution of non-square values of the dual variable by applying Proposition \ref{prop:LargeSieve}. In Section \ref{sec:removeU}, we show how the variable $U$ may be removed and reduce our problem to proving Lemmas \ref{lem:secU-main} and \ref{lem:secU-error}. This again requires some delicate combinatorial manipulations. The proofs of these two lemmas are postponed to Sections \ref{sec:cross-cancel} and \ref{sec:secU-error}. Using these two lemmas, we conclude the proof of Theorem \ref{thm:Main}. In Section \ref{sec:cross-cancel}, we prove Lemma \ref{lem:secU-main} by showing that there is significant cancellation between the contributions from the diagonal terms and the off-diagonal, square terms. In Section \ref{sec:secU-error}, we prove Lemma \ref{lem:secU-error}, showing that our removal of the variable $U$ comes at the cost of an acceptable error. The proofs in this section are similar to those in Sections \ref{sec:LargeSieve-Diagonal} and \ref{sec:LargeSieve-OffDiagonal}.

We include an appendix at the end of the paper which gives the details of our calculations regarding various Dirichlet series. Specifically, in Appendix \ref{prfDiri}, we prove Lemmas \ref{lem:DirichletSeriesDiagonal}, \ref{lem:DirichletSeriesOffDiagonal}, \ref{lem:DirichletSeriesSuma}, and \ref{lem:iden}. Since we need to understand the Dirichlet series \eqref{def:D1} -- \eqref{def:Z*} is wide regions of their parameters, we need to be rather precise in our calculations, and so our proofs are somewhat lengthy. 

\subsection{Notation}

The Vinogradov-Landau symbols $\ll, \gg, O,o$ have their usual meanings. We use $m\sim M$ to denote the condition $M < m \leq 2M$ and $m\asymp M$ to denote $c_1 M \leq m \leq c_2 M$ for some constants $c_1 < c_2$, not necessarily the same at each occurrence. The symbols $u_i$ and $s$ will always denote complex numbers, $n_i$ is always a non-negative integer, and $t, t_i$ are always real numbers. The domains of other variables will be clear from context.

A bold letter represents a $2k$-tuple of the associated variable, e.g. $\boldu = (u_1,\ldots,u_{2k})$. Most often we will take $k=2$, though occasionally we will work with other values of $k$, such as in Appendix A. For a tuple $\boldu$ and variable $s$, we let
\[
\boldu + s = (u_1+s,\ldots,u_{2k}+s).
\]
The condition $\boldn \equiv a \mod{q}$ means that each coordinate of $\boldn$ satisfies the congruence. 

Throughout, we will specify various regions in which certain Dirichlet series converge absolutely. These Dirichlet series will depend on several complex parameters, such as $u_1,\ldots, u_4$. When a statement such as ``The function $\colh(\boldu)$ is analytic for $\Re(u_j)>\frac{1}{2}$'' appears, it is to be understood that the bound applies to each variable $u_j$. 

When dealing with certain combinatorial calculations, we will have reason to reference the subsets $A\subset \set{1,2,3,4}$. For such sets $A$, we use $\bar{A}$ to denote the complement of $A$ in $\set{1,2,3,4}$. Thus if $A = {1,2}$, then $\bar{A} = \set{3,4}$. We use the symbol $\delta$ exclusively to refer to a function that is either 1 or 0 depending as the argument is true or false, respectively. For example $\delta(m=\square)$ is 1 if $m$ is a square integer and 0 otherwise.

For a Dirichlet $L$-function $L(s,\chi)$ or the Riemann zeta function $\zeta$, we write $L_2(s,\chi)$ to mean the $L(s,\chi)$ with the local factor at $2$ removed.

Because several of our functions are lengthy to write if all arguments are written out in full, we use the notation $\underset{x=x_0}{\operatorname{Value}\,} f(x)$ to denote the value of $f(x)$ at $x=x_0$. For instance, if $G(\boldu) = G(u_1,u_2,u_3,u_4)$, then
\[
\underset{u_2=0}{\operatorname{Value}\,} G(\boldu) = G(u_1,0,u_3,u_4).
\]

\subsection*{Acknowledgements}
The authors would like to thank Xiannan Li and Nathan Ng for helpful comments and conversations.

\section{Preliminaries}\label{sec:Prelims}

In this section, we collect together some standard results from the literature. 

\subsection{Approximate Functional Equation}\label{sec:AFE}

First, we record the shifted version of the approximate functional equation for $L\pth{\half,\chi_{8d}}$ (see Proposition 2.1 of \cite{YoungActa}).
\begin{lem}
Assume $\alpha \in \mathbb{C}$ with $|\Re(\alpha)| \ll \frac{1}{\log X}$ and $|\Im(\alpha)| \ll (\log X)^2$. Then 
\begin{align*}
L(\tfrac{1}{2} + \a,\chi_{8d}) = \sum_{n=1}^\infty \frac{\chi_{8d} (n)}{n^{\frac{1}{2} + \a}} V_\a\left( \frac{n}{\sqrt{8d}}\right)
+
\colx_\a  \sum_{n=1}^\infty \frac{\chi_{8d} (n)}{n^{\frac{1}{2} - \a}} V_{-\a}\left( \frac{n}{\sqrt{8d}}\right),
\end{align*}
where 
\begin{align}
V_\a(x) &= \frac{1}{2\pi  i }  \int\limits_{(1)} \frac{1}{s} g_\a(s) x^{-s} \, ds, \label{VaDef}\\
g_\a(s) & = \pi^{-\frac{s}{2}}\frac{\Gamma(\frac{1/2+\a + s}{2})}{\Gamma(\frac{1/2+\a}{2})},\label{gaDef}\\
\colx_\a &= \left( \frac{8d}{\pi}\right)^{-\a} \frac{\Gamma(\frac{1/2-\a}{2})}{\Gamma(\frac{1/2+\a}{2})}.\label{colxDef}
\end{align}
Here $\int_{(c)}$ stands for $\int_{c-\infty i}^{c+\infty i}$.
\label{afe1}
\end{lem}

\subsection{Poisson Summation Formula for Quadratic Characters}
Define the Gauss-like sum
\begin{equation}\label{def:G_k}
G_k(n) = \pth{\frac{1- i }{2} + \fracp{-1}{n}\frac{1+ i }{2}} \sum_{a \mod{n}} \fracp{a}{n} e\fracp{ak}{n}.
\end{equation}
The next result we need is Lemma 2.3 of \cite{SoundNonvanishing}, which gives an explicit evaluation of the sum $G_k$.

\begin{lem}\label{lem:GkEval}
For $m,n$ relatively prime odd integers, $G_k(mn) = G_k(m) G_k(n)$, and for $p^\alpha  \mid \mid k$ (setting $\alpha = \infty$ for $k=0$), we have
\[
G_k(p^\beta) = \begin{cases}
0, & \text{if $\beta\leq \alpha$ is odd},\\
\phi(p^\beta) &\text{if $\beta\leq \alpha$ is even},\\
-p^\alpha &\text{if $\beta=\alpha+1$ is even},\\
\fracp{kp^{-\alpha}}{p} p^\alpha \sqrt{p} & \text{if $\beta = \alpha+1$ is odd}, \\
0 & \text{if $\beta \geq \alpha+2$}.
\end{cases}
\]
Here $\phi$ is the Euler's totient function.
\end{lem}
The function $G_k(n)$ appears when applying Poisson summation as in the following lemma.
\begin{lem}\label{lem:Poisson}
Let $F$ be a Schwartz class function over the real numbers and suppose that $n$ is an odd integer. Then
\begin{equation}\label{eq:PoissonAll}
\sum_{d} \fracp{d}{n} F\fracp{d}{Z} = \frac{Z}{n} \sum_{k\in \Z} G_k(n) \check{F}\fracp{kZ}{n},
\end{equation}
and also
\begin{equation}\label{eq:PoissonOdd}
\sum_{(d,2)=1} \fracp{d}{n} F\fracp{d}{Z} = \frac{Z}{2n}\left( \frac{2}{n}\right) \sum_{k\in \Z} (-1)^k G_k(n) \check{F}\fracp{kZ}{2n},
\end{equation}
where $G_k(n)$ is defined as in \eqref{def:G_k}, and the Fourier-type transform of $F$ is defined to be
\begin{align*}
\check{F}(y) =  \int\limits_{-\infty}^{\infty} (\cos(2\pi xy) + \sin(2\pi xy))F(x)\, dx.
\end{align*}
Further, for $F$ even and $y\neq 0$, 
\begin{align}
\check{F}(y) = \frac{2}{2\pi  i } \int\limits_{(\frac{1}{2})} \tilde{F}(1-s) \Gamma(s) \cos\fracp{\pi s}{2} (2\pi \abs{y})^{-s}\, ds,
\label{eq:Fhatfirst}
\end{align}
where $\tilde{F}$ is the Mellin transform of $F$ defined by 
\[
\tilde{F}(s) =  \int\limits_{0}^{\infty} F(t) t^{s-1} dt.
\]
If instead $F$ is supported on $[0,\infty)$, we have
\begin{equation}\label{eq:Fhatsecond}
\check{F}(y) = \frac{1}{2\pi  i } \int\limits_{(\frac{1}{2})} \tilde{F}(1-s) \Gamma(s) (\cos+\operatorname{sgn}(y)\sin)\fracp{\pi s}{2} (2\pi \abs{y})^{-s}\, ds.
\end{equation}
\end{lem}
The first assertions in \eqref{eq:PoissonAll} and \eqref{eq:PoissonOdd} are contained in the proof of Lemma 2.6 in \cite{SoundNonvanishing}. The expressions in \eqref{eq:Fhatfirst} and \eqref{eq:Fhatsecond} follow by Mellin inversion (see Section 3.3 of \cite{SoundYoungSecond}). For $F$ a Schwartz class function, we define the usual Fourier transform of $F$ via
\begin{equation}
\hat{F}(y) =  \int\limits_{-\infty}^{\infty} F(x)e(-xy) dx.
\label{equ:fourier-trans}
\end{equation}
Note that $\check{F}(x) = \frac{1+ i }{2} \hat{F}(x) + \frac{1-i}{2} \hat{F}(-x)$.

Applying this lemma, we obtain a dual sum in the variable $k$ (see \eqref{eq:PoissonAll}). This gives rise to the ``diagonal'' contribution corresponding to $k=0$,  an ``off-diagonal square'' contribution corresponding to $k=m^2$, $m\neq 0$, and an ``off-diagonal non-square'' contribution corresponding to the remaining $k$.

\subsection{Dirichlet Series}
After applying Lemmas \ref{lem:Poisson}, we will need to understand several multiple Dirichlet series.  For the diagonal contribution, we need to understand
\begin{equation}\label{def:D1}
\cold_1(\boldu) = \sum_{\substack{(n_1n_2n_3n_4,2)=1\\ n_1n_2n_3n_4=\square}} \frac{1}{n_1^{u_1}n_2^{u_2}n_3^{u_3}n_4^{u_4}} \prod_{p  \mid n_1n_2n_3n_4} \pth{1-\frac{1}{p}}
\end{equation}
and
\begin{equation}\label{def:D2}
\cold_2(\boldu) = \sum_{\substack{(n_1n_2n_3n_4,2)=1\\ n_1n_2n_3n_4=\square}} \frac{1}{n_1^{u_1}n_2^{u_2}n_3^{u_3}n_4^{u_4}} \prod_{p \mid n_1n_2n_3n_4} \pth{1-\frac{1}{p+1}},
\end{equation}
and for the off-diagonal terms, we need to understand
\begin{equation}\label{def:Z}
Z(\boldu,s;k_1,a) = \sum_{k_2 \geq 1} \frac{1}{k_2^{2s}}\sum_{(n_1n_2n_3n_4,2a)=1} \frac{1}{n_1^{u_1}n_2^{u_2}n_3^{u_3}n_4^{u_4}} \frac{G_{k_1k_2^2}(n_1n_2n_3n_4)}{n_1n_2n_3n_4}
\end{equation}
and 
\begin{equation}\label{def:Z*}
Z^*(\boldu,s;k_1,a) = \sum_{k_2 \geq 1} \frac{(-1)^{k_2}}{k_2^{2s}}\sum_{(n_1n_2n_3n_4,2a)=1} \frac{1}{n_1^{u_1}n_2^{u_2}n_3^{u_3}n_4^{u_4}} \frac{G_{k_1k_2^2}(n_1n_2n_3n_4)}{n_1n_2n_3n_4}.
\end{equation}
We give precise evaluations of these in the following lemmas, the proofs of which are postponed to Appendix \ref{prfDiri}.

\begin{lem}\label{lem:DirichletSeriesDiagonal}
Let 
\[
g_1(p) := 1 - \frac{1}{p},\quad 
\quad g_2(p) := 1 - \frac{1}{p+1},
\]
and
\begin{align}
C_p(\boldu) := \prod_{1\leq i < j \leq 4} \pth{1-\frac{1}{p^{u_i+u_j}}} \sum_{\substack{A\subset\set{1,2,3,4}\\ \abs{A} \equiv 0 \mod{2}}} \prod_{i\in A} \frac{1}{p^{u_i}}.
\label{def:Cpu}
\end{align}
With notation as above, for $\ell=1,2$ and  $\Re(u_i)>\frac{1}{2}$, we have
\begin{align*} \label{equ:Hiden}
\cold_\ell(\boldu) = \sumpth{\prod_{1\leq i\leq j\leq {4}} \zeta_2(u_i+u_j)}\colh_\ell(\boldu),
\end{align*}
where
\begin{equation}\label{def:H}
\colh_\ell(\boldu) = \prod_{p\neq 2}  \sumpth{(1-g_\ell(p)) \prod_{1\leq i \leq j \leq 4} \pth{1-\frac{1}{p^{u_i+u_j}}}+ g_\ell(p)C_p(\boldu)}
\end{equation}
The functions $\colh_\ell$ are analytic and uniformly bounded for $\Re(u_i)  \geq \frac{1}{4}+ \varepsilon$. In particular,
\begin{equation*}\label{eq:h1}
\colh_1\pth{\half,\half,\half,\half} = \prod_{p\neq 2} \pth{1-\frac{1}{p}}^{7}\pth{1+\frac{7}{p}-\frac{2}{p^2}+\frac{3}{p^3}-\frac{1}{p^4}}
\end{equation*}
and
\begin{align}
\colh_2\pth{\half,\half,\half,\half} = \prod_{p\neq 2} \frac{(1-\frac{1}{p})^{10}}{1+\frac{1}{p}} \pth{\frac{(1+\frac{1}{\sqrt{p}})^{-4} + (1-\frac{1}{\sqrt{p}})^{-4}}{2} + \frac{1}{p}}.
\label{equ:h2}
\end{align}
\end{lem}

\begin{lem}\label{lem:DirichletSeriesOffDiagonal}
Let $k_1$ be squarefree and define
\begin{equation}\label{def:m(k)}
\mathfrak{m} = \mathfrak{m}(k_1) = \begin{cases}
k_1 & \text{if $k_1 \equiv 1\mod{4}$}, \\
4k_1 & \text{if $k_1\equiv 2,3 \mod{4}$}.
\end{cases}
\end{equation}
For $a$ odd, we have 
\[
Z^*(\boldu,s;k_1,a) = (2^{1-2s}-1)Z(\boldu,s;k_1,a),
\]
and
\[
Z(\boldu,s;k_1,a) = \zeta(2s) \prod_{i=1}^4 L_2\pth{\half+u_i,\chi_\mathfrak{m}} Y(\boldu,s;k_1,a),
\]
where
\[
Y(\boldu,s;k_1,a) = \prod_{1\leq i\leq j\leq {4}} \frac{\zeta_2(u_i+u_j+2s)}{\zeta_2(u_i+u_j+1)} \frac{Z_2(\boldu,s;k_1,a)}{\colp(\boldu,s;k_1)}, 
\]
and
\[
\colp(\boldu,s;k_1) = \prod_{i=1}^4 L_2\pth{\half+u_i+2s,\chi_\mathfrak{m}}.
\]
In the following regions, the function $Z_2$ is analytic and satisfies the estimate $Z_2 \ll \tau(a)$:
\begin{enumerate}[label=\normalfont{(\roman*)}]
\item $\Re(s) \geq \frac{1}{2}$, $\Re(u_i) > 0$.
\item $\Re(s) \geq -\frac{1}{20}$, $\Re(u_i) \geq \frac{1}{3}$.
\item $\Re(s) \geq \frac{3}{5}$, $\Re(u_i) \geq -\frac{1}{20}$.
\end{enumerate}
For $\Re(u_i) > \frac{1}{2}$, we also have
\begin{align}
\underset{s=\half}{\text{\normalfont Res}}\ Z(\boldu,s;1,1) &=\prod_{i=1}^4 \zeta_2(u_i+\half) \prod_{p\neq 2}\sumpth{1-\frac{1}{p} + \frac{1}{p}\prod_{i=1}^4  \sumpth{1-\frac{1}{p^{\frac{1}{2}+u_i}}}},\label{eq:resZathalf} \\
Z(\boldu,0;1,1) &= -\frac{1}{2} \prod_{p\neq 2} \sumpth{\frac{1}{p} + \pth{1-\frac{1}{p}}B_p(\boldu;0)}.\label{eq:resZat0}
\end{align}
Here $B_{p,4}(\boldu;0)$ is defined in \eqref{eq:Bdef}.
\end{lem}

\begin{lem}\label{lem:DirichletSeriesSuma}
With notation as above,  let
\[
Z_3(\boldu,s) := \sum_{(a,2)=1} \frac{\mu(a)}{a^{2-2s}} Z_2(\boldu,s;1,a).
\]
Then
\begin{align}
Z_3(\boldu,s) &= \prod_{1\leq i\leq j\leq 4}\frac{\zeta_2(u_i+u_j+1)}{\zeta_2(u_i+u_j+2s)} \prod_{i=1}^4 \frac{\zeta_2(\frac{1}{2}+u_i+2s)}{\zeta_2(\frac{1}{2}+u_i)} \nonumber\\
&\quad\times\prod_{p\neq 2}\Bigg[\frac{1}{p} + \pth{1-\frac{1}{p}}B_p(\boldu+s;0) + \pth{1-\frac{1}{p^{2s}}}\frac{1}{p^{\frac{1}{2}-s}}B_p(\boldu+s;1) -\frac{1}{p^{2-2s}}\Bigg].
\label{eq:Z3}
\end{align}
Further, $Z_3$ is analytic and uniformly bounded in the region 
\begin{equation}\label{eq:Z3Region}
\Re(u_i) \geq \frac{1}{3}, \quad -\frac{1}{20} \leq \Re(s) \leq \frac{1}{3}.
\end{equation}
\end{lem}
Finally, we have the following relation between $Z_3$ and $\colh_2$. This lemma is one of the key identities we need in order to witness cancellation between the diagonal terms and off-diagonal, square terms.
\begin{lem}
\label{lem:iden}
Let $z_i \in \mathbb{C}$.
We have
\begin{enumerate}[label=\normalfont{(\roman*)}]

\item 
\begin{align*}
&\tfrac{1}{p} + \left(1 - \tfrac{1}{p} \right) B_p(\tfrac{1}{2}+z_1, \tfrac{1}{2}+z_2, \tfrac{1}{2}+z_3 , \tfrac{1}{2}+z_4;0)\nonumber\\
&\quad  + \tfrac{1}{p^{\frac{1}{2}-z_1}} 
\left(1 - \tfrac{1}{p^{2z_1}} \right)B_p(\tfrac{1}{2}+z_1, \tfrac{1}{2}+z_2, \tfrac{1}{2}+z_3 , \tfrac{1}{2}+z_4;1) -\tfrac{1}{p^{2-2z_1}}\nonumber\\
&=
\left(1-\tfrac{1}{p^{1-2z_1}} \right) \left(\tfrac{1}{p}+   B_p(\tfrac{1}{2}-z_1, \tfrac{1}{2}+z_2, \tfrac{1}{2}+z_3 , \tfrac{1}{2}+z_4;0) \right).
\end{align*}

\item If 
\begin{align*}
Z_4(\boldu, s) := \prod_{1 \leq i\leq j \leq 4} \zeta_2(1+u_i + u_j )^{-1}Z_3(\boldu, s),
\end{align*}
then
\begin{align*}
Z_4(\tfrac{1}{2},\tfrac{1}{2}+z_2-z_1,\tfrac{1}{2}+z_3-z_1,\tfrac{1}{2}+z_4-z_1,z_1)
=
\frac{8}{\pi^2} \mathcal{H}_2(\tfrac{1}{2}-z_1,\tfrac{1}{2}+z_2,\tfrac{1}{2}+z_3,\tfrac{1}{2}+z_4).
\end{align*}
\end{enumerate}

\end{lem}

\subsection{Functional Equations for Dirichlet Polynomials}
We require the following direct application of the functional equation for $L(s,\chi_m)$.

\begin{lem}\label{lem:FEDirect}
For $m$ a fundamental discriminant and $G$ any Schwartz class function,
\[
\sum_{n}\frac{\chi_{m}(n)}{n^{\frac{1}{2}+z}} G\fracp{n}{N} = \fracp{\pi}{\abs{m}}^{z} \sum_{n}\frac{\chi_{m}(n)}{n^{\frac{1}{2}-z}} \grave{G}_z\fracp{\pi n N}{\abs{m}},
\]
where
\[
\grave{G}_z(x) = \frac{1}{2\pi  i } \int\limits_{(2)} \frac{\Gamma\pth{\frac{1}{4}-\frac{z-s-\ida}{2}}}{\Gamma\pth{\frac{1}{4}+\frac{z-s+\ida}{2}}} x^{-s}\tilde{G}(-s) \,ds.
\]
\end{lem}

\begin{proof}
Let $c = \abs{\Re(z)} + 1$. By Mellin inversion,
\begin{align*}
\sum_{n}\frac{\chi_{m}(n)}{n^{\frac{1}{2}+z}} G\fracp{n}{N} = \frac{1}{2\pi  i } \int\limits_{(c)} L\pth{\half+z+s,\chi_m} \tilde{G}(s) N^s\, ds.
\end{align*}
Shifting contours to $\Re(s) = -c$ and applying the functional equation \eqref{eq:FunctionalEquation}, we have
\begin{align*}
\sum_{n}\frac{\chi_{m}(n)}{n^{\frac{1}{2}+z}} G\fracp{n}{N} &=\frac{1}{2\pi  i } \int\limits_{(-c)} \fracp{\abs{m}}{\pi}^{-s-z} \frac{\Gamma\fracp{1/2-s-z+\ida}{2}}{\Gamma\fracp{1/2+s+z+\ida}{2}}L\pth{\half-z-s,\chi_m} \tilde{G}(s) N^s\, ds \\
&=\frac{1}{2\pi  i } \int\limits_{(c)} \fracp{\abs{m}}{\pi}^{s-z} \frac{\Gamma\fracp{1/2+s-z+\ida}{2}}{\Gamma\fracp{1/2-s+z+\ida}{2}}L\pth{\half-z+s,\chi_m} \tilde{G}(-s) N^{-s}\, ds \\
&= \fracp{\pi}{\abs{m}}^z \sum_{n} \frac{\chi_m(n)}{n^{\frac{1}{2}-z}} \grave{G}_z\fracp{\pi n N}{\abs{m}}.
\end{align*}
\end{proof}

For
\[
g(s) = \frac{\Gamma\pth{\frac{1}{4}+\frac{s}{2}+\frac{\ida}{2}}}{\Gamma\pth{\frac{1}{4}-\frac{s}{2}+\frac{\ida}{2}}},
\]
and $s=\sigma+ it$ with $\sigma \geq -\frac{1}{5}$, say, Stirling's formula (see e.g. 5.A.4 of \cite{IwaniecKowalski}) implies that
\begin{equation}\label{eq:GammaRatioBound}
g(s) \ll (1+\abs{t})^\sigma.
\end{equation}
This gives the standard estimate
\begin{equation}\label{eq:GraveGStirlingBound}
\grave{G}_z(x) \ll \fracp{1+\abs{\Im(z)}}{1+x}^A
\end{equation}
for any $A > 0$ upon shifting contours to the right. We also note that $\grave{G}_z(x)$ is bounded uniformly for $\Re(z) \geq -\frac{1}{5}$, say, by shifting contours to $\Re(s) = 0$.

\subsection{Smooth Functions}
To simplify our arguments, we also use the following test functions.

\begin{lem}\label{lem:ConvenientTestFandG}
The following hold:
\begin{enumerate}[label={\normalfont(\roman*)}]
\item Let $c_0,c_1$ be any fixed positive real numbers. There is a function $F$ such that
\begin{enumerate}[label={\normalfont(\alph*)}]
\item $F$ is smooth, non-negative, even, and Schwartz class;
\item $F(x) \geq 1$ for all $x \in [-c_1, c_1]$;
\item $\hat{F}(x)$ is even and compactly supported on $[-c_0, c_0]$ (so $\check{F}(x)$ is also even and compactly supported on $[-c_0,c_0]$).
\end{enumerate}
\item There is a function $G$ such that
\begin{enumerate}[label={\normalfont(\alph*)}]
\item $G$ is smooth and non-negative;
\item  $G$ is compactly supported in $[3/4,2]$;
\item For any $J\geq 0$, the function
\[
G(x)  + G(x/2) + \cdots + G(x/2^J) = 1
\]
for all $x\in[1,3\cdot 2^{J-1}]$ and is supported on $[3/4,2^{J+1}]$. 
\end{enumerate}
\end{enumerate}
\end{lem}

For a proof, see Lemma 2.7 and the paragraph following this lemma in \cite{XiannanFourth}. As in that work, we fix, once and for all, a function $G$ with these properties, and we will specify a particular $F$ later.

\section{Proof of Proposition \ref{prop:LargeSieve}: Initial Steps}\label{sec:LargeSieve-InitialSteps}

Let $\coll_0$ be a sufficiently large constant satisfying that the number of primes in the interval $[\sqrt{\coll},\sqrt{2\coll}]$ exceeds $\frac{\sqrt{\coll}}{2\log\coll}$ for all $\coll\geq \coll_0$. We set
\[
S(M,N,t) = \sum_{\substack{M \leq \abs{m} < 2M\\ m\neq \square}} \sumabs{\sum_{n} \frac{\chi_m(n)}{n^{\frac{1}{2}+ i  t}} G\fracp{n}{N}}^4
\]
and
\[
S^\flat(M,N,t) = \sumflat_{M \leq \abs{m} < 2M} \sumabs{\sum_{n} \frac{\chi_m(n)}{n^{\frac{1}{2}+ i  t}}G\fracp{n}{N}}^4.
\]
Clearly $S^\flat(M,N,t)\leq S(M,N,t)$. We now state our inflation lemma alluded to in Section \ref{sec:Outline}.

\begin{lem}\label{lem:Inflation}
For any $\coll \geq \coll_0$, we have
\begin{align*}\label{eq:InflationSstar}
S^\flat&(M,N,t) \ll\frac{\log \coll}{\sqrt{\coll}} \sumpth{S(M\coll,N,t) + S(2M\coll,N,t) + \sum_{\sqrt{\coll} \leq p \leq \sqrt{2\coll}} \frac{S^\flat(M,N/p,t)}{p^2}},
\end{align*}
where the implied constants are absolute and in particular do not depend on $\coll$.
\end{lem}

\begin{proof}
For a prime $p\asymp \sqrt{\coll}$ and $a(n)$ arbitrary complex numbers, H\"{o}lder's inequality gives
\begin{align*}
\sumabs{\sum_{n} a(n) \fracp{m}{n}}^4 &= \sumabs{\sum_{\substack{n\\ p\nmid n}} a(n) \fracp{mp^2}{n} + \sum_{\substack{n\\ p \mid n}} a(n) \fracp{m}{n}}^4 \\
&\leq 8\sumabs{\sum_{n} a(n) \fracp{mp^2}{n}}^4 + 8\delta(p\nmid m) \sumabs{\sum_{n} a(np) \fracp{m}{n}}^4.
\end{align*}
For $a(n) = n^{-\frac{1}{2}- i  t} G\fracp{n}{N}$, we have
\[
\sumabs{\sum_{n} a(np) \fracp{m}{n}}^4 = \sumabs{\sum_{n} \frac{1}{(np)^{\frac{1}{2}+ i  t}} G\fracp{np}{N} \fracp{m}{n}}^4 = \frac{1}{p^2}\sumabs{\sum_{n} \frac{1}{n^{\frac{1}{2}+ i  t}} G\fracp{np}{N} \fracp{m}{n}}^4,
\]
and it follows that
\begin{align*}
&\sum_{\sqrt{\coll} \leq p \leq \sqrt{2\coll}} S^\flat(M,N,t) \\
&\ll \sum_{\sqrt{\coll} \leq p \leq \sqrt{2\coll}} \sumflat_{\substack{M\leq m < 2M}} \sumabs{\sum_{n} \frac{1}{n^{\frac{1}{2}+ i  t}}  G\fracp{n}{N} \fracp{mp^2}{n}}^4 +\sum_{\sqrt{\coll} \leq p \leq \sqrt{2\coll}} \frac{S^\flat(M,N/p,t)}{p^2}  \\
&\ll S(\coll M,N,t) + S(2\coll M,N,t) + \sum_{\sqrt{\coll} \leq p \leq \sqrt{2\coll}} \frac{S^\flat(M,N/p,t)}{p^2}.
\end{align*}
The last line follows from the following observations. If $m_1,m_2$ are fundamental discriminants and $p_1,p_2$ are odd primes, then $m_1p_1^2= m_2p_2^2$ if and only if $m_1=m_2$ and $p_1=p_2$, which holds since $m_1,m_2$ are either squarefree or 4 times a squarefree integer. This also implies that no $mp^2$ is a square. For $\coll\geq\coll_0$, the number of primes in the interval $[\sqrt{\coll},\sqrt{2\coll}]$ is $\geq \frac{\sqrt{\coll}}{2\log\coll}$, and the lemma follows.
\end{proof}
As already stated in the introduction, the main technical device used in our proof of Theorem \ref{thm:Main} is the large sieve inequality given in Proposition \ref{prop:LargeSieve}. For the reader's convenience, we restate this here in the notation introduced in this section.

\begin{propC}{\ref{prop:LargeSieve}}
For $M,N\geq 1$ and notation as above, there exists an absolute constant $\coll \geq \coll_0$ such that
\[
S^\flat(M,N,t) \leq \coll^\frac{2}{3}(1+\abs{t})^2\pth{M + N^2\log(2+N^2/M)}\log^6(2+MN).
\]
\end{propC}

We remark that we have made little effort to optimize the exponent of $1+\abs{t}$ above, as this is unimportant to our applications. With careful accounting, one can probably get $(1+\abs{t})^{1+\ep}$

As in \cite{XiannanFourth}, we prove Proposition \ref{prop:LargeSieve} via a nested induction argument on $M$ and $N$. The following lemma gives our base case.

\begin{lem}\label{lem:InductionBase}
For the $G$ fixed of Lemma \ref{lem:ConvenientTestFandG}, $N > 0$, and $m$ a fundamental discriminant,
\[
\sumabs{\sum_{n} \frac{\chi_m(n)}{n^{\frac{1}{2}+ i  t}} G\fracp{n}{N}} \ll \sqrt{N_0}+1,
\]
where
\[
N_0 =  \min\pth{N,\frac{m(1+\abs{t})}{N}} \leq \sqrt{m(1+\abs{t})}
\]
and the implied constant is absolute.
\end{lem}

\begin{proof}
Trivially
\[
\sum_{n} \frac{\chi_m(n)}{n^{\frac{1}{2}+ i  t}} G\fracp{n}{N} \ll \sqrt{N}.
\]
Let $N_1 = \frac{\abs{m}(1+\abs{t})}{N}$. Then by Lemma \ref{lem:FEDirect}, we have
\begin{align*}
\sumabs{\sum_{n} \frac{\chi_m(n)}{n^{\frac{1}{2}+ i  t}} G\fracp{n}{N}} &= \sumabs{\sum_{n} \frac{\chi_m(n)}{n^{\frac{1}{2}- i  t}} \grave{G}_{it}\fracp{\pi nN}{\abs{m}}} \\
&\ll \sum_{n\leq N_1} \frac{1}{n^\frac{1}{2}} + \sum_{n > N_1} \frac{1}{n^\frac{1}{2}} \fracp{(1+\abs{t})\abs{m}}{nN} \ll \sqrt{N_1}+1.
\end{align*}
by applying \eqref{eq:GraveGStirlingBound} and noting that $\grave{G}_{it}(x)$ is bounded uniformly for $t\in \R$ and $x>0$. 
\end{proof}

Lemma \ref{lem:InductionBase} implies that
\[
S^\flat(M,N,t) \ll M^2(1+\abs{t})
\]
where the implied constant $C'$ is absolute. Thus the base case $M\leq M_0$ is trivially true provided that $\coll^{\frac{2}{3}} \geq C' M_0$. The above also shows that the case $M\leq N$ holds, and so we may always assume $M > N$. For some fixed $M_1 \geq M_0$, our induction hypothesis is that for any $M \leq M_1$ that
\begin{equation}\label{eq:InductionHyp}
S^\flat(M,N,t) \leq \coll^\frac{2}{3}(1+\abs{t})^{2}\pth{M + N^2\log(2+N^2/M)}\log^6(2+MN)
\end{equation}
for all $N$ and $t$, and we now proceed to prove \eqref{eq:InductionHyp} for $M$ fixed with $M_1 < M \leq M_1 + 1$.

For $M$ fixed, we proceed by a nested induction argument on $N$. The base case $N \leq 2$ is trivial, and our second induction hypothesis is that \eqref{eq:InductionHyp} holds for all $N \leq N_1$ and any $\abs{t}$ for some $N_1 \geq 2$. We now fix some $N$ with $N_1 < N \leq N_1+1$. We want to prove \eqref{eq:InductionHyp} for our fixed $M$ and $N$.

Occasionally, we will need to assume $M$ and $N$ are large enough for certain expressions to be valid. By Lemma \ref{lem:InductionBase}, we may always do so by, at worst, suitably modifying the constant $\coll$.

The bound in \eqref{eq:InductionHyp} becomes ineffectual when $N$ is very large compared to $M$. The following lemma rectifies this situation and will be the form of the induction hypothesis we use most often.

\begin{lem}\label{lem:InductionAllN}
Suppose that \eqref{eq:InductionHyp} holds for all $M\leq M_1$ and all $N$ and $t$. Then we also have that there exists some absolute constant $C'$ such that
\[
S^\flat(M,N,t) \leq C' \coll^{\frac{2}{3}} (1+\abs{t})^{3} M\log^6(2+M)
\]
for all $M\leq M_1$ and all $N$ and $t$.
\end{lem}

\begin{proof}
If $N \leq \sqrt{M(1+\abs{t})}$, then the statement follows immediately from \eqref{eq:InductionHyp}. Suppose now that $N > \sqrt{M(1+\abs{t})}$. Lemma \ref{lem:FEDirect} gives
\[
S^\flat(M,N,t) = \sumflat_{M \leq \abs{m} < 2M} \sumabs{\sum_{n}\frac{\chi_{m}(n)}{n^{\frac{1}{2}- i  t}} \grave{G}_{it}\fracp{\pi n N}{\abs{m}}}^4.
\]
We have 
\begin{align}
\sum_{n}\frac{\chi_{m}(n)}{n^{\frac{1}{2}- i  t}} \grave{G}_{it}\fracp{\pi n N}{\abs{m}} &= \frac{1}{2\pi  i } \int\limits_{(1)} \fracp{\abs{m}}{\pi}^s \frac{\Gamma\pth{\frac{1}{4}-\frac{it-s-\ida}{2}}}{\Gamma\pth{\frac{1}{4}+\frac{it-s+\ida}{2}}} \sum_{n} \frac{\chi_m(n)}{n^{\frac{1}{2}- i  t+s}} N^{-s}\tilde{G}(-s)\, ds \nonumber\\
&= \sumd_{N_0} \frac{1}{2\pi  i } \int\limits_{(1)} \fracp{\abs{m}}{\pi}^s \frac{\Gamma\pth{\frac{1}{4}-\frac{it-s-\ida}{2}}}{\Gamma\pth{\frac{1}{4}+\frac{it-s+\ida}{2}}} \sum_{n} \frac{\chi_m(n)}{n^{\frac{1}{2}- i  t+s}} G\fracp{n}{N_0} N^{-s}\tilde{G}(-s)\, ds, \label{eq:InductionAllNExpansion}
\end{align}
where $\sumd_{N_0} $ denotes the sum over $N_0=2^j$, $i \geq 0$.
Let
\begin{equation}
V(x) = G(2x) + G(x) + G(x/2)
\label{V-def}
\end{equation}
for the same fixed $G$ so that $V(x) = 1$ on $[3/4,2]$. Then
\[
\sum_{n} \frac{\chi_m(n)}{n^{\frac{1}{2}- i  t+s}} G\fracp{n}{N_0} = \frac{1}{2\pi  i } \int\limits_{(1)} \sum_n \frac{\chi_m(n)}{n^{\frac{1}{2}- i  t+u}} V\fracp{n}{N_0} N_0^{u-s} \tilde{G}(u-s)\, du.
\]
Let $N_1 = \frac{M(1+\abs{t})}{N}$. When $N_0 \leq N_1$, we move the contour of integration in $s$ to $\Re(s) = -1$, and when $N_0 > N_1$, we move the contour to $\Re(s) = 4$. In both cases, we move the integration in $u$ to $\Re(u) = 0$. Using $\tilde{G}(s) \ll (1+\abs{s})^{-10}$, \eqref{eq:GammaRatioBound}  and Lemma 5.2 of \cite{XiannanFourth}, the quantity in \eqref{eq:InductionAllNExpansion} is $O(R(\leq N_1) + R(> N_1))$, where
\begin{align*}
R(\leq N_1) &= \sumd_{N_0\leq N_1} \int\limits_{(-1)} \fracp{NN_0}{\abs{m}(1+\abs{t})} \int\limits_{(0)} \sumabs{\sum_n \frac{\chi_m(n)}{n^{\frac{1}{2}- i  t+u}} V\fracp{n}{N_0}} \frac{1}{(1+\abs{s})^{10} (1+\abs{u-s})^{10}}\, du\, ds \\
&\ll \sumd_{N_0\leq N_1} \frac{N_0}{N_1} \int\limits_{(0)} \sumabs{\sum_n \frac{\chi_m(n)}{n^{\frac{1}{2}- i  t+u}} V\fracp{n}{N_0}} \frac{1}{(1+\abs{u})^9}\, du,
\end{align*}
and similarly
\[
R(>N_1) \ll \sumd_{N_0> N_1} \fracp{N_1}{N_0}^4 \int\limits_{(0)} \sumabs{\sum_n \frac{\chi_m(n)}{n^{\frac{1}{2}- i  t+u}} V\fracp{n}{N_0}} \frac{1}{(1+\abs{u})^9}\, du.
\]
By H\"{o}lder and the induction hypothesis \eqref{eq:InductionHyp}, the contribution of $R(\leq N_1)$ to $S^\flat(M,N,t)$ is
\begin{align*}
&\ll \sumd_{N_0\leq N_1} \frac{N_0}{N_1} \int\limits_{(0)} \sumflat_{M\leq \abs{m} < 2M} \sumabs{\sum_n \frac{\chi_m(n)}{n^{\frac{1}{2}- i  t+u}} V\fracp{n}{N_0}}^4 \frac{1}{(1+\abs{u})^9}\, du \\
&\leq \sumd_{N_0\leq N_1} \frac{N_0}{N_1} \int\limits_{(0)} \pth{\coll^\frac{2}{3}(1+\abs{t}+\abs{u})^{2}\pth{M + N_0^2\log(2+N_0^2/M)}\log^6(2+MN_0)} \frac{1}{(1+\abs{u})^9} \, du \\
&\ll \coll^\frac{2}{3}(1+\abs{t})^{2} \sumd_{N_0\leq N_1} \frac{N_0}{N_1} \pth{M + N_0^2\log(2+N_0^2/M)}\log^6(2+MN_0) \\
&\ll \coll^\frac{2}{3}(1+\abs{t})^{2} \pth{M + N_1^2\log(2+N_1^2/M)}\log^6(2+MN_1).
\end{align*}
A similar calculation shows that $R(>N_1)$ satisfies the same bound. Since $N > \sqrt{M(1+\abs{t})}$, we have $N_1 < \sqrt{M(1+\abs{t})}$, and the lemma follows.
\end{proof}

\subsection{Inflation and Initial Decomposition}\label{sec:Inflation}
Let
\[
\coll_2 = \coll(1+\abs{t})^{3} \max\pth{1,\frac{N^4}{M^2}},
\]
and
\begin{align}
X = \coll_2 M.
\label{def-X}
\end{align}
From Lemma \ref{lem:Inflation}, we will estimate $S(X,N,t)$ in place of $S(M,N,t)$. One should understand the choice of $\coll_2$ in the following way. Our goal is to make the sum over $m$ (of length $M$) in $S(M,N,t)$ just slightly longer than the sum over $n$, which we think of as having length $N^2$. We do this in order for Poisson summation  in $m$ to gain us something in terms of the length of the dual sum. If $M \geq N^2$, then we only need to ensure that $M$ is slightly larger than $N^2$, and hence we need only scale up $M$ by a small constant factor $\coll$, and we choose 1 in the max in this case. The factor of $(1+\abs{t})^3$ is included to ensure that certain other exponents of $1+\abs{t}$ remain under control when performing the induction step (see Lemma \ref{lem:UsefulEstimates} and the comments after Lemma \ref{lem:LS-NonSquare}). If, on the other hand, $M \leq N^2$, then we need to inflate the outer sum to be at least the length of the inner sum, and then scale the outer sum to be a bit longer. Hence we choose $N^4/M^2$ in the max in this case. The reason for using $N^4/M^2$ as opposed to $N^2/M$ can be seen in Lemma \ref{lem:UsefulEstimates} below. Our choice of $\coll_2$ ensures that we always have
\begin{equation*}
X \geq \coll N^2,
\end{equation*}
and so the outer sum of $S(X,N,t)$ is always longer than the inner sum by at least a large constant factor. This constant factor becomes a savings in the length of the dual sum after applying Poisson summation, and it is precisely this savings that we exploit in order to carry out our induction argument. This can be seen clearly in \eqref{eq:InductionRequirement}.

We now prove an analogue of Lemma 5.4 of \cite{XiannanFourth}, which gives some useful estimates involving the quantities above.

\begin{lem}\label{lem:UsefulEstimates}
With notation as above, we have
\[
\frac{N^4}{X} \leq \frac{M}{\coll}
\]
and
\[
\frac{\log\coll_2}{\sqrt{\coll_2}} X \ll \coll^{\frac{3}{5}} (1+\abs{t})^{\frac{5}{3}} \pth{M+N^2\log(2+N^2/M)}.
\]
\end{lem}

\begin{proof}
Suppose first that $\coll_2 = \coll(1+\abs{t})^{3}$ so that $N^2 \leq M$, whence
\[
\frac{N^4}{X} = \frac{N^4}{\coll(1+\abs{t})^{3}M} \leq \frac{M}{\coll}.
\]
In this case, we also have
\[
\frac{\log\coll_2}{\sqrt{\coll_2}} X =  M\sqrt{\coll}(1+\abs{t})^{\frac{3}{2}} \log(\coll(1+\abs{t})^{3}) \ll \coll^\frac{3}{5}M(1+\abs{t})^{\frac{5}{3}}.
\]
In the complimentary case, $\coll_2 = \coll(1+\abs{t})^{3} \dfrac{N^4}{M^2}$, so
\[
\frac{N^4}{X} = \frac{M}{\coll(1+\abs{t})^{3}} \leq \frac{M}{\coll}.
\]
As well,
\[
\frac{\log\coll_2}{\sqrt{\coll_2}} X =  N^2\sqrt{\coll}(1+\abs{t})^{\frac{3}{2}}\log \coll_2 \ll \coll^\frac{3}{5}N^2(1+\abs{t})^{\frac{5}{3}}\log(2+N^2/M).
\]
\end{proof}

Applying Lemma \ref{lem:Inflation} with $\coll_2$ as defined above, we have
\begin{align}
S^\flat(M,N,t) &\ll \frac{\log\coll_2}{\sqrt{\coll_2}}\sumpth{S(X,N,t) + S(2X,N,t) + \sum_{\sqrt{\coll_2} \leq p \leq \sqrt{2\coll_2}} \frac{S^\flat(M,N/p,t)}{p^2}}\nonumber \\
&\ll \frac{\log\coll_2}{\sqrt{\coll_2}}\bigg(S(X,N,t) + S(2X,N,t) \nonumber\\
&\quad\quad\quad\quad+ \frac{\coll^\frac{2}{3}(1+\abs{t})^{2}\pth{M + N^2\log(2+N^2/M)}\log^6(2+MN)}{\sqrt{\coll_2}\log \coll_2}\bigg),
\label{eq:InflatedSum}
\end{align}
where the implied constant is absolute, upon applying the induction hypothesis for the last term. The third term of the last line suffices, and we now examine $S(X,N,t) + S(2X,N,t)$. We fix $F$ to be the test function given in Lemma \ref{lem:ConvenientTestFandG} with
\begin{equation}\label{eq:Fbounds}
c_0 = \frac{1}{64} \mand c_1=4.
\end{equation}
By positivity and H\"{o}lder's inequality, we have
\begin{align}
S(X,N,t) + S(2X,N,t) &\leq \sum_{m\neq \square} F\fracp{m}{X} \sumabs{\sum_{n} \frac{\chi_m(n)}{n^{\frac{1}{2}+ i  t}} G\fracp{n}{N} }^4\nonumber\\
&\ll \sum_{k\geq 0} \frac{1}{2^k} \sum_{m\neq\square} F\fracp{m}{X} \sumabs{\sum_{(n,2)=1} \frac{\chi_m(n)}{n^{\frac{1}{2}+ i  t}} G\fracp{2^k n}{N} }^4.
\label{eq:MainPlusSquares}
\end{align}
We now look to estimate the sums
\[
\sings = \sum_{\substack{m\\ m\neq \square}} F\fracp{m}{X} \sumabs{\sum_{(n,2)=1} \frac{\chi_m(n)}{n^{\frac{1}{2}+ i  t}} G\fracp{n}{N}}^4.
\]
Opening the absolute values,
\[
\sings = \sum_{\substack{m\\ m\neq \square}} F\fracp{m}{X} \sum_{\substack{n_1,\ldots,n_4\\(n_i,2)=1}} \frac{\chi_m(n_1n_2n_3n_4)}{(n_1n_2)^{\frac{1}{2}+ i  t}(n_3n_4)^{\frac{1}{2}- i  t}} \prod_{i=1}^4 G\fracp{n_i}{N}.
\]
We wish to apply Poisson summation to $m$ in the form of Lemma \ref{lem:Poisson}. To do so, we must include the contribution of square $m$, and so we write
\[
\sings = \colm - \colt,
\]
where $\colm$ is the full sum and $\colt$ is the contribution of square $m$:
\begin{align}
\colt= \sum_{m} F\fracp{m^2}{X} \sum_{\substack{n_1,\ldots,n_4\\ (n_i,2m)=1}} \frac{1}{(n_1n_2)^{\frac{1}{2}+ i  t}(n_3n_4)^{\frac{1}{2}- i  t}} G\fracp{n_i}{N}.
\label{equ:def-S}
\end{align}
Poisson summation now gives
\[
\colm = \colm^{0} + \colm^{\square} + \colm^{\neq\square},
\]
where
\begin{align}
\colm^{0} &= \hat{F}(0)X \sum_{\substack{n_1n_2n_3n_4=\square\\ (n_i,2)=1}} \frac{1}{(n_1n_2)^{\frac{1}{2}+ i  t} (n_3n_4)^{\frac{1}{2}- i  t}} \prod_{p \mid n_1n_2n_3n_4}\pth{1-\frac{1}{p}} \prod_{i=1}^4 G\fracp{n_i}{N}, \label{def:PoissonDiag}
\\
\colm^{\square} &= X\sum_{k\geq 1} \sum_{\substack{n_1,\ldots,n_4 \\ (n_i,2)=1}} \frac{1}{(n_1n_2)^{\frac{1}{2}+ i  t} (n_3n_4)^{\frac{1}{2}- i  t}} \frac{G_{k^2}(n_1n_2n_3n_4)}{n_1n_2n_3n_4} \check{F}\fracp{k^2X}{n_1n_2n_3n_4} \prod_{i=1}^4 G\fracp{n_i}{N},\label{def:DualSquares}\\
\colm^{\neq \square} &= X\sum_{k\neq 0,\square} \sum_{\substack{n_1,\ldots,n_4 \\ (n_i,2)=1}} \frac{1}{(n_1n_2)^{\frac{1}{2}+ i  t} (n_3n_4)^{\frac{1}{2}- i  t}} \frac{G_k(n_1n_2n_3n_4)}{n_1n_2n_3n_4} \check{F}\fracp{kX}{n_1n_2n_3n_4} \prod_{i=1}^4 G\fracp{n_i}{N}.\label{def:NonSquare}
\end{align}
Note that \eqref{def:PoissonDiag} uses the fact that $G_0(n)= \phi(n)$ when $n=0$, and otherwise $G_0(n)= 0$. We have
\[
\sings = \colm^{0} + \colm^{\square} + \colm^{\neq \square}- \colt ,
\]
and we will prove the following lemmas concerning these quantities.

\begin{lem}\label{lem:LS-Diagonal}
With notation as above, we have
\[
\colm^{0} = \frac{\hat{F}(0)|\Tilde{G}(it)|^4}{2^{10}\cdot 6!} \colh_1\pth{\half,\half,\half,\half}\pth{1-\tfrac{12}{32}} X(\log N)^6 + O\pth{X(1+\abs{t})^{\frac{1}{5}}(\log N)^5}.
\]
\end{lem}

\begin{lem}\label{lem:LS-SquareCancellation}
With notation as above and
\[
\log^*v = \max(0,\log v),
\]
we have
\begin{align*}
\colm^{\square} - \colt &= -\frac{\hat{F}(0)|\Tilde{G}(it)|^4}{2^{10}\cdot 6!} \colh_1\pth{\half,\half,\half,\half} \cdot \frac{1}{64} \cdot X\left[(\log^*({N^4}/{X}))^6 - 4(\log^*({N^3}/{X}))^6 \right]\\
&\quad+ O\pth{X(\log X)^5(1+\abs{t})^{\frac{1}{5}}}.
\end{align*}
\end{lem}

\begin{lem}\label{lem:LS-NonSquare}
With notation as above, we have
\[
\colm^{\neq \square} \ll \coll^\frac{2}{3} (1+\abs{t})^{3+\frac{1}{5}} N^2 \log^6(2+M).
\]
\end{lem}

These estimates imply
\[
\sings \ll \pth{X(1+\abs{t})^\frac{1}{5} + \coll^\frac{2}{3} (1+\abs{t})^{3+\frac{1}{5}} N^2} \log^6(2+MN).
\]
Let us show that this suffices for Proposition \ref{prop:LargeSieve}. Applying this with $N/2^k$ in place of $N$ and summing over $k\geq 0$ in \eqref{eq:MainPlusSquares} implies that
\[
S(X,N,t) + S(2X,N,t) \ll  \pth{X(1+\abs{t})^\frac{1}{5} + \coll^\frac{2}{3} (1+\abs{t})^{3+\frac{1}{5}} N^2} \log^6(2+MN).
\]
 Combining this with \eqref{eq:InflatedSum} and applying Lemma \ref{lem:UsefulEstimates}, we have that there exist absolute constants $\colc_1$, $\colc_2$ such that
\begin{align*}
S^\flat(M,N,t) &\leq \colc_1\Bigg(\frac{\log\coll_2}{\sqrt{\coll_2}}\sumpth{X (1+\abs{t})^\frac{1}{5} + \coll^\frac{2}{3} (1+\abs{t})^{3+\frac{1}{5}} N^2} \\
&\quad\quad\quad\quad+ \frac{\coll^\frac{2}{3}(1+\abs{t})^{2}\pth{M + N^2\log(2+N^2/M)}}{\coll_2} \Bigg)\log^6(2+MN), \\
&\leq \colc_2 (\coll^{3/5-2/3} + \coll^{-1/3})\coll^\frac{2}{3} (1+\abs{t})^2 (M+N^2\log(2+N^2/M))\log^6(2+MN) \\
&\leq \coll^\frac{2}{3} (1+\abs{t})^{2}(M+N^2\log(2+N^2/M))\log^6(2+MN)
\end{align*}
upon choosing $\coll$ large enough compared to $\colc_2$, where the second line follows from $\coll_2 \geq \coll(1+\abs{t})^3$. 

In the following  Sections \ref{sec:LargeSieve-Diagonal}, \ref{sec:LargeSieve-OffDiagonal} and  \ref{sec:LargeSieve-NonSqare}, we will prove Lemmas \ref{lem:LS-Diagonal}, \ref{lem:LS-SquareCancellation} and \ref{lem:LS-NonSquare}, respectively.

\section{Diagonal Contribution $\colm^{0}$}\label{sec:LargeSieve-Diagonal}

To evaluate $\colm^{0}$ in \eqref{def:PoissonDiag}, we first expand the four factors of $G$ via Mellin inversion and write
\begin{align}
\colm^{0} &= \frac{\hat{F}(0)X}{(2\pi i)^4} \int\limits_{(2)}\int\limits_{(2)}\int\limits_{(2)}\int\limits_{(2)}  \sum_{\substack{n_1n_2n_3n_4=\square \\ (n_i,2)=1}} \frac{1}{\sqrt{n_1n_2n_3n_4} n_1^{u_1+ i  t}n_2^{u_2+ i  t}n_3^{u_3- i  t}n_4^{u_4- i  t}}\prod_{p \mid n_1n_2n_3n_4}\pth{1-\frac{1}{p}} \nonumber\\
&\quad \times\sumpth{\prod_{i=1}^4 \tilde{G}(u_i) N^{u_i}}  d\boldu \nonumber\\
&= \frac{\hat{F}(0)X}{(2\pi i)^4} \int\limits_{(2)}\int\limits_{(2)}\int\limits_{(2)}\int\limits_{(2)} \colg_t(\boldu) N^{u_1+u_2+u_3+u_4} \cold_1(\boldu+\half)\,  d\boldu, 
\label{eq:DiagonalIntegral}
\end{align}
say, where
\begin{equation}\label{eq:colgDef}
\colg_t(\boldu) = \tilde{G}(u_1- i  t)\tilde{G}(u_2- i  t)\tilde{G}(u_3+ i  t)\tilde{G}(u_4+ i  t)
\end{equation}
and $\cold_1(\boldu)$ is as in \eqref{def:D1}.  By Lemma \ref{lem:DirichletSeriesDiagonal},
\[
\cold_1(\boldu+\half) = \sumpth{\prod_{1\leq i\leq j\leq 4} \zeta_2(1+u_i+u_j)}\colh_1(\boldu+\half),
\]
where $\colh_1(\boldu)$ is analyitc and bounded uniformly for $\Re(u_i) > \frac{1}{4}+\varepsilon$. We then define
\[
\cole_t(\boldu) = \colg_t(\boldu) \cold_1(\boldu+\half) \prod_{1\leq i\leq j\leq 4}(u_i+u_j)
\]
and note that $\cole_t(\boldu)$ is analytic in the region $\Re(u_i) > -\frac{1}{4}+ \varepsilon$. Moreover, standard estimates for $\zeta$ (see, e.g., Theorem 3.5 and Equation 5.1.4 of \cite{Titchmarsh}) give
\[
\zeta_2(1+u) \ll \frac{1}{\abs{u}} + 
(1+\abs{\Im(u)})^{\max(0,-\frac{\Re(u)}{2})+\ep},
\]
uniformly for $-\frac{1}{2}\leq \Re(u) \leq 5$, say. In what follows, we fix $\ep = \frac{1}{250}$ (our exact choice here is unimportant) and we will always have $-2\ep \leq \Re(u_i)\leq 2$. In this region, the above estimate and the triangle inequality imply
\[
\zeta_2(1+u_i+u_j)(u_i+u_j) \ll 1+ \abs{u_i+u_j}(1+\abs{\Im(u_i)})^{3\ep}(1+\abs{\Im(u_j)})^{3\ep}.
\]
Thus
\begin{align*}
\cole_t(\boldu) \ll  \prod_{i=1}^4 \frac{(1+\abs{\Im(u_i)})^{12\ep}}{(1+\abs{u_i\pm it})^{10}} \prod_{1\leq i\leq j\leq k} \abs{u_i+u_j}^{\delta(\abs{u_i+u_j} > 1)},
\end{align*}
where in the $\pm$, one takes a $-$ if $i=1,2$ and a $+$ if $i=3,4$, and $\delta(\abs{u_i+u_j} > 1)=1$ if $\abs{u_i+u_j} > 1$ and 0 otherwise.\\

To illustrate the structure of the poles and subsequent residue calculations, we shift contours in \eqref{eq:DiagonalIntegral} once more and write $\colm^{0} = \hat{F}(0)X \coli$, where
\begin{align}
\coli= \frac{1}{(2\pi i)^4} \int\limits_{(\frac{1}{20})}\int\limits_{(\frac{1}{21})}\int\limits_{(\frac{1}{22})}\int\limits_{(\frac{1}{23})} \cole_t(\boldu) N^{u_1+u_2+u_3+u_4} \prod_{1\leq i\leq j\leq 4} (u_i+u_j)^{-1} \, d\boldu \nonumber.
\label{equ:M0-1}
\end{align}
Let
\begin{equation}\label{eq:Ldef}
L = \log N.
\end{equation}
We first take the line of integration in $u_1$ to $\Re(u_1) = - 2\ep$, passing a simple pole at $u_1 = 0$. This gives
\begin{align*}
\coli = \coli_1 + \coli_2,
\end{align*}
where
\begin{align*}
\coli_1 &= \frac{1}{16(2\pi i)^3} \int\limits_{(1)}\int\limits_{(1)}\int\limits_{(1)} \frac{\cole_t(0,u_2,u_3,u_4) N^{u_2+u_3+u_4}}{(u_2u_3u_4)^2 (u_2+u_3)(u_2+u_4)(u_3+u_4)} du_2\ du_3\ du_4 , \\
\coli_2 &= \frac{1}{(2\pi i)^4} \int\limits_{(1)}\int\limits_{(1)}\int\limits_{(1)}\int\limits_{(-2\ep)}  \cole_t(\boldu) N^{u_1+u_2+u_3+u_4} \prod_{1\leq i\leq j\leq 4} (u_i+u_j)^{-1} d\boldu.
\end{align*}

\subsection{Calculation of $\coli_1$}

To evaluate $\coli_1$, we first take the line of integration in $u_2$ to $\Re(u_2) = -2\ep$, passing a pole of order 2 at $u_2 = 0$, and so
\[
\coli_1 = \coli_{1,1} + \coli_{1,2},
\]
where
\begin{align*}
\coli_{1,1} &= \frac{1}{16(2\pi i)^2} \int\limits_{(1)}\int\limits_{(1)} \frac{f_1(u_3,u_4) N^{u_3+u_4}}{(u_3u_4)^2 (u_3+u_4)} du_3\ du_4, \\
\coli_{1,2} &=\frac{1}{16(2\pi i)^3} \int\limits_{(1)}\int\limits_{(1)}\int\limits_{(-2\ep)} \frac{\cole_t(0,u_2,u_3,u_4) N^{u_2+u_3+u_4}}{(u_2u_3u_4)^2 (u_2+u_3)(u_2+u_4)(u_3+u_4)} \, du_2\, du_3\, du_4,
\end{align*}
and
\[
f_1(u_3,u_4) := \res{u_2=0} \fracp{\cole_t(0,u_2,u_3,u_4)N^{u_2}}{u_2^2(u_2+u_3)(u_2+u_4)}
\]
We proceed similarly for $\coli_{1,1}$, taking the line of integration in $u_3$ to $\Re(u_3) = -2\ep$ and passing a pole at $u_3=0$. In the two resulting integrals, we do the same for the variable $u_4$. Ultimately we obtain
\[
\coli_1 = \colr^{\#}  + \coli_{1,2}+ \coli_{1,1,2} + O((1+\abs{t})^{\frac{1}{5}}),
\]
where
\begin{align*}
\colr^{\#} &= \res{\boldu=(0,0,0,0)} \sumpth{\cole_t(\boldu) N^{u_1+u_2+u_3+u_4} \prod_{1\leq i\leq j\leq 4} (u_i+u_j)^{-1}}
\end{align*}
and
\begin{align*}
\coli_{1,1,2} =\frac{1}{16(2\pi i)^2} \int\limits_{(1)}\int\limits_{(-2\ep)} \frac{f_1(u_3,u_4) N^{u_3+u_4}}{(u_3u_4)^2(u_3+u_4)}\, du_3\, du_4.
\end{align*}
For $\coli_{1,1,2}$, we take the line of integration in $u_4$ to $\Re(u_4) = 2\ep + \frac{1}{L}$, passing no poles in the process, and the new integral is $O((1+\abs{t})^{\frac{1}{5}}(\log N))$.

For $\coli_{1,2}$, we move the line of integration in $u_3$ to $\Re(u_3) = \frac{1}{L}$, passing a simple pole at $u_3 = -u_2$, so
\[
\coli_{1,2} = \coli_{1,2,1} + \coli_{1,2,2},
\]
where
\begin{align*}
\coli_{1,2,1} &= \frac{1}{16(2\pi i)^2} \int\limits_{(1)} \int\limits_{(-2\ep)} \frac{\cole_t(0,u_2,-u_2,u_4) N^{u_4}}{u_2^4 u_4^2 (u_2+u_4)(u_4-u_2)}\, du_2\, du_4, \\
\coli_{1,2,2} &= \frac{1}{16(2\pi i)^3} \int\limits_{(1)}\int\limits_{(\frac{1}{L})}\int\limits_{(-2\ep)} \frac{\cole_t(0,u_2,u_3,u_4) N^{u_2+u_3+u_4}}{(u_2u_3u_4)^2 (u_2+u_3)(u_2+u_4)(u_3+u_4)}\, du_2\, du_3\, du_4. \\
\end{align*}
In $\coli_{1,2,2}$, we take the line of integration in $u_4$ to $\Re(u_4) = 2\ep+\frac{1}{L}$, passing no poles in the process, and the new integral is $O((1+\abs{t})^{\frac{1}{5}}(\log N)^3)$. In $\coli_{1,2,1}$, we take the line of integration in $u_4$ to $\Re(u_4) = \frac{1}{L}$, passing a simple pole at $u_4 = -u_2$. The integral on the new line of integration is $O((1+\abs{t})^{\frac{1}{5}}(\log N)^2)$, and so
\[
\coli_{1,2}  = \coli_{1,2,1,1} + O((1+\abs{t})^{\frac{1}{5}}(\log N)^3),
\]
where
\begin{align*}
\coli_{1,2,1,1} &= -\frac{1}{32(2\pi i)} \int\limits_{(-2\ep)} \frac{\cole_t(0,u_2,-u_2,-u_2) N^{-u_2}}{u_2^7}\, du_2= \frac{1}{32(2\pi i)} \int\limits_{(2\ep)} \frac{\cole_t(0,-u_2,u_2,u_2) N^{u_2}}{u_2^7}du_2.
\end{align*}
Combining the above estimates, we have shown
\begin{equation*}\label{eq:coli1reduced}
\coli_1 = \colr^{\#} + \frac{1}{32(2\pi i)} \int\limits_{(2\ep)} \frac{\cole_t(0,-u,u,u) N^{u}}{u^7} du + O((1+\abs{t})^{\frac{1}{5}}(\log N)^3).
\end{equation*}

\subsection{Calculation of $\coli_2$}
We now calculate $\coli_{2}$. First, we move the line of integration in $u_2$ to $\Re(u_2) = \frac{1}{L}$, passing a simple pole at $u_2 = -u_1$, so 
\[
\coli_2 = \coli_{2,1} + \coli_{2,2},
\]
where
\begin{align*}
\coli_{2,1} &= \frac{-1}{16(2\pi i)^3} \int\limits_{(1)}\int\limits_{(1)}\int\limits_{(-2\ep)}   \frac{\cole_t(u_1,-u_1,u_3,u_4) N^{u_3+u_4}}{u_1^2u_3 u_4 (u_1+u_3) (u_1+u_4) (u_3+u_4)(u_3-u_1)(u_4-u_1)}\, du_1\, du_3\, du_4,\\
\coli_{2,2} &= \frac{1}{(2\pi i)^4} \int\limits_{(1)}\int\limits_{(1)}\int\limits_{(\frac{1}{L})}\int\limits_{(-2\ep)}  \cole_t(\boldu) N^{u_1+u_2+u_3+u_4} \prod_{1\leq i\leq j\leq 4} (u_i+u_j)^{-1}\, d\boldu.
\end{align*}
We proceed similarly in $\coli_{2,2}$ and move the line of integration in $u_3$ to $\Re(u_3) = \frac{1}{L}$, again passing only a simple pole at $u_3 = -u_1$, so
\[
\coli_{2,2} = \coli_{2,2,1} + \coli_{2,2,2},
\]
where
\begin{align}
\coli_{2,2,1} &= \frac{-1}{16(2\pi i)^3} \int\limits_{(1)}\int\limits_{(\frac{1}{L})}\int\limits_{(-2\ep)}   \frac{\cole_t(u_1,u_2,-u_1,u_4) N^{u_2+u_4}}{u_1^2u_2 u_4 (u_1+u_2) (u_1+u_4) (u_2+u_4)(u_2-u_1)(u_4-u_1)}\, du_1\, du_2\, du_4,\label{equ:I221}\\
\coli_{2,2,2} &= \frac{1}{(2\pi i)^4} \int\limits_{(1)}\int\limits_{(\frac{1}{L})}\int\limits_{(\frac{1}{L})}\int\limits_{(-2\ep)}  \cole_t(\boldu) N^{u_1+u_2+u_3+u_4} \prod_{1\leq i\leq j\leq 4} (u_i+u_j)^{-1}\, d\boldu.\nonumber
\end{align}
Once more for $\coli_{2,2,2}$, we take the line of integration in $u_4$ to $\Re(u_4) = \ep$, passing only a simple pole at $u_4 = -u_1$, so
\[
\coli_{2,2,2} = \coli_{2,2,2,1} + \coli_{2,2,2,2},
\]
where
\begin{align*}
\coli_{2,2,2,1} &= \frac{-1}{16(2\pi i)^3} \int\limits_{(\frac{1}{L})}\int\limits_{(\frac{1}{L})}\int\limits_{(-2\ep)}  \frac{\cole_t(u_1,u_2,u_3,-u_1) N^{u_2+u_3}}{u_1^2u_2u_3(u_1+u_2)(u_1+u_3)(u_2+u_3)(u_2-u_1)(u_3-u_1)}\, du_1\, du_2\, du_3,\\
&\ll(1+\abs{t})^{\frac{1}{5}}(\log N)^3,\\
\coli_{2,2,2,2} &= \frac{1}{(2\pi i)^4} \int\limits_{(\ep)}\int\limits_{(\frac{1}{L})}\int\limits_{(\frac{1}{L})}\int\limits_{(-2\ep)}  \cole_t(\boldu) N^{u_1+u_2+u_3+u_4} \prod_{1\leq i\leq j\leq 4} (u_i+u_j)^{-1}\, d\boldu\\
&\ll N^{-\ep/2} (1+\abs{t})^{\frac{1}{5}}.
\end{align*}
Hence,
\[
\coli_{2} = \coli_{2,1} + \coli_{2,2,1}  + O((1+\abs{t})^{\frac{1}{5}} (\log N)^3).
\]

We evaluate $\coli_{2,1}, \coli_{2,2,1}$ in reverse order. For $\coli_{2,2,1}$, we take the line of integration in $u_4$ to $\Re(u_4) = \frac{1}{L}$, passing a simple pole at $u_4 = -u_1$. The integral on the new lines of integration may be bounded as $\coli_{2,2,2,1}$ and so is $O((1+\abs{t})^{\frac{1}{5}}(\log N)^3)$. Thus
\begin{align}
\coli_{2,2,1} &= \frac{-1}{32(2\pi i)^2} \int\limits_{(\frac{1}{L})}\int\limits_{(-2\ep)}   \frac{\cole_t(u_1,u_2,-u_1,-u_1) N^{u_2-u_1}}{u_1^4u_2 (u_1+u_2) (u_2-u_1)^2}\, du_1\, du_2 + O((1+\abs{t})^{\frac{1}{5}}(\log N)^3)\nonumber \\
&= \frac{-1}{32(2\pi i)^2} \int\limits_{(\frac{1}{L})}\int\limits_{(2\ep)}   \frac{\cole_t(-u_1,u_2,u_1,u_1) N^{u_2+u_1}}{u_1^4u_2 (u_1+u_2)^2 (u_2-u_1)}\, du_1\, du_2 + O((1+\abs{t})^{\frac{1}{5}}(\log N)^3). \label{equ:I221+C}
\end{align}
We now take the line of integration in $u_2$ to $\Re(u_2) = -2\ep + \frac{1}{L}$, passing a simple pole at $u_2 = 0$. The integral on the new line of integration is $O((1+\abs{t})^{\frac{1}{5}}(\log N)^2)$, and so
\begin{equation}
\coli_{2,2,1} = \frac{1}{32(2\pi i)} \int\limits_{(2\ep)}   \frac{\cole_t(-u_1,0,u_1,u_1) N^{u_1}}{u_1^7}\, du_1+ O((1+\abs{t})^{\frac{1}{5}}(\log N)^3).
\label{equ:I221-A}
\end{equation}
We leave this integral as is for the moment and calculate $\coli_{2,1}$. Proceeding similarly as in the calculation above, our first goal is to take the lines of integration in $u_3$ and $u_4$ to $\Re(u_3) = \Re(u_4) = \frac{1}{L}$. We begin with $u_3$, and in moving the line of integration, we pass a simple pole at $u_3 = -u_1$, so
\[
\coli_{2,1} = \coli_{2,1,1} + \coli_{2,1,2},
\]
where
\begin{align*}
\coli_{2,1,1} &= \frac{-1}{32(2\pi i)^2} \int\limits_{(1)}\int\limits_{(-2\ep)}   \frac{\cole_t(u_1,-u_1,-u_1,u_4) N^{u_4-u_1}}{u_1^4 u_4 (u_1+u_4)(u_4-u_1) ^2}\, du_1\, du_4,\\
&= \frac{-1}{32(2\pi i)^2} \int\limits_{(1)}\int\limits_{(2\ep)}   \frac{\cole_t(-u_1,u_1,u_1,u_4) N^{u_4+u_1}}{u_1^4 u_4 (u_1+u_4)^2(u_4-u_1)}\, du_1\, du_4,\\
\coli_{2,1,2} &= \frac{-1}{16(2\pi i)^3} \int\limits_{(1)}\int\limits_{(\frac{1}{L})}\int\limits_{(-2\ep)}   \frac{\cole_t(u_1,-u_1,u_3,u_4) N^{u_3+u_4}}{u_1^2u_3 u_4 (u_1+u_3) (u_1+u_4)(u_3+u_4)(u_3-u_1)(u_4-u_1)}\, du_1\, du_3\, du_4.
\end{align*}
Note that $\coli_{2,1,2}$ is essentially the same as $\coli_{2,2,1}$ in \eqref{equ:I221}. Shifting contours in a similar manner as that case, by \eqref{equ:I221-A}, we find that
\[
\coli_{2,1,2} = \frac{1}{32(2\pi i)} \int\limits_{(2\ep)}   \frac{\cole_t(-u_1,u_1,0,u_1) N^{u_1}}{u_1^7}\, du_1+ O((1+\abs{t})^{\frac{1}{5}}(\log N)^3).
\]
Note the 0 in the arguments of $\cole_t$ is in the third argument, rather than the second.

The integral $\coli_{2,1,1}$ is quite similar to the expression for $\coli_{2,2,1}$  in \eqref{equ:I221+C}, except that the line of integration in the outer variable lies to the right of $2\ep$. Proceeding as in that case, we take the line of integration in $u_4$ to $\Re(u_4) = -2\ep + \frac{1}{L}$, passing simple poles at $u_4 = 0$ and $u_4 = u_1$. Just as for $\coli_{2,2,1}$, the integral on the new line of integration is $O((1+\abs{t})^{\frac{1}{5}}(\log N)^2)$. Thus
\[
\coli_{2,1,1} = \coli_{2,1,1,1} + \coli_{2,1,1,2} + O((1+\abs{t})^{\frac{1}{5}}(\log N)^2),
\]
where
\begin{align*}
\coli_{2,1,1,1} &= \frac{1}{32(2\pi i)} \int\limits_{(2\ep)}   \frac{\cole_t(-u_1,u_1,u_1,0) N^{u_1}}{u_1^7}\, du_1, \\
\coli_{2,1,1,2} &= \frac{-1}{128(2\pi i)} \int\limits_{(2\ep)}   \frac{\cole_t(-u_1,u_1,u_1,u_1) N^{2u_1}}{u_1^7}\, du_1.
\end{align*}
Collecting together the above calculations, we find that
\begin{align*}
\coli_2 &= \coli_{2,1,1,1} + \coli_{2,1,1,2} + \coli_{2,1,2} + \coli_{2,2,1}+O((1+\abs{t})^{\frac{1}{5}}(\log N)^3) \\ 
&= \frac{1}{32(2\pi i)} \int\limits_{(2\ep)} \frac{N^{u}\cola_t(u,N)}{u^7}\, du +O((1+\abs{t})^{\frac{1}{5}}(\log N)^3),
\end{align*}
where
\[
\cola_t(u,N) = \cole_t(-u,0,u,u) + \cole_t(-u,u,0,u) + \cole_t(-u,u,u,0) - \cole_t(-u,u,u,u) \frac{N^{u}}{4}.
\]

\subsection{Assembling Estimates}

Combining our estimates from the previous two sections, we find that
\[
\coli = \colr^{\#}  + \frac{1}{32(2\pi i)} \int\limits_{(2\ep)} \frac{N^{u}\cola_t^*(u,N)}{u^7}\, du +O((1+\abs{t})^{\frac{1}{5}}(\log N)^3),
\]
where
\begin{align*}
\cola_t^*(\boldu,N) &= \cola_t(\boldu,N)+\cole_t(0,-u,u,u) \\
&=\cole_t(-u,0,u,u) + \cole_t(-u,u,0,u) + \cole_t(-u,u,u,0) + \cole_t(0,-u,u,u) - \cole_t(-u,u,u,u) \frac{N^{u}}{4}.
\end{align*}
Shifting the contours in the integral to $\Re(u)= -\ep$, we pass a pole of order 7. The integral on the new line of integration is $O((1+\abs{t})^{\frac{1}{5}})$, and the residue at $u=0$ is
\[
\res{u=0} \fracp{N^{u}\cola_t^*(u,N)}{u^7} = -\frac{12}{6!}\cole_t(0,0,0,0)(\log N)^6+O((1+\abs{t})^{\frac{1}{5}}(\log N)^5).
\]
We also have
\[
\colr^{\#} = \frac{1}{6!} \cole_t(0,0,0,0)(\log N)^6+O((1+\abs{t})^{\frac{1}{5}}(\log N)^5).
\]
Since
\[
\cole_t(0,0,0,0) = \frac{1}{2^{10}}|\Tilde{G}(it)|^4 \colh_1\pth{\half,\half,\half,\half},
\]
we deduce that
\[
\coli = \frac{|\Tilde{G}(it)|^4}{2^{10}\cdot 6!} \pth{1-\frac{12}{32}} \colh_1\pth{\half,\half,\half,\half} (\log N)^6 + O((1+\abs{t})^{\frac{1}{5}}(\log N)^5).
\]
This completes the proof of Lemma \ref{lem:LS-Diagonal}. 

\section{Cancellation of Square Contributions $\colm^{\square} - {\colt}$}\label{sec:LargeSieve-OffDiagonal}

\subsection{Calculation of $\colt$}

We now calculate the contribution of square $m$ to the sum $\colm$ (see \eqref{equ:def-S}), which is
\[
\colt = \sum_{m} F\fracp{m^2}{X} \sum_{\substack{n_1,\ldots,n_4\\ (n_i,2m)=1}} \frac{1}{(n_1n_2)^{\frac{1}{2}+ i  t}(n_3n_4)^{\frac{1}{2}- i  t}} \prod_{i=1}^4 G\fracp{n_i}{N}.
\]
By M\"{o}bius inversion,
\[
\sum_{(n,2m)=1} \frac{1}{n^{\frac{1}{2}+ i  t}} G\fracp{n}{N} = \sum_{\substack{a \mid m\\ (a,2)=1}} \frac{\mu(a)}{a^{\frac{1}{2}+ i  t}}\sum_{ (n,2)=1} \frac{1}{n^{\frac{1}{2}+ i  t}} G\fracp{an}{N},
\]
and so
\begin{align*}
\colt &= \sum_{m} F\fracp{m^2}{X} \sum_{\substack{a_1,\ldots,a_4 \mid m\\ (a_i,2)=1}} \frac{\mu(a_1)\mu(a_2)\mu(a_3)\mu(a_4)}{(a_1a_2)^{\frac{1}{2}+ i  t}(a_3a_4)^{\frac{1}{2}- i  t}} \sum_{\substack{n_1,\ldots,n_4\\ (n_i,2)=1}} \frac{1}{(n_1n_2)^{\frac{1}{2}+ i  t}(n_3n_4)^{\frac{1}{2}- i  t}} \prod_{i=1}^4 G\fracp{a_in_i}{N}.
\end{align*}
Mellin inversion and a change of variables then gives
\begin{equation}\label{eq:TvarChanged}
\colt = \frac{1}{(2\pi i)^5} \int\limits_{(1)}\cdots\int\limits_{(1)} \colg_t(\boldu) N^{u_1+u_2+u_3+u_4}\tilde{F}(s) X^s D\pth{\boldu+\half, s} \prod_{i=1}^{4} \zeta_2(u_i+\half) \,  d\boldu\, ds,
\end{equation}
where $\colg_t(\boldu)$ is defined in \eqref{eq:colgDef}, and 
\[
D(\boldu,s) = \sum_{m \geq 1} \frac{1}{m^{2s}} \sum_{\substack{a_1,\ldots,a_4  \mid m\\ (a_i,2)=1}} \frac{\mu(a_1)\mu(a_2)\mu(a_3)\mu(a_4)}{a_1^{u_1}a_2^{u_2}a_3^{u_3}a_4^{u_4}} = \prod_{p} D_p(\boldu,s),
\]
and the local factors $D_p$ are defined as follows. If $p=2$, then
\[
D_2(\boldu,s) = \pth{1-\frac{1}{2^{2s}}}^{-1},
\]
and if $p\neq 2$, then
\begin{align*}
D_p(\boldu,s) &= \sum_{m\geq 0} \frac{1}{p^{2ms}} \summany_{\substack{a_1,\ldots,a_4\geq 0 \\ a_i\leq m}} \frac{\mu(p^{a_1})\mu(p^{a_2})\mu(p^{a_3})\mu(p^{a_4})}{p^{a_1u_1}p^{a_2u_2}p^{a_3u_3}p^{a_4u_4}} \\
&= \pth{1-\frac{1}{p^{2s}}}^{-1} \summany_{a_1,\ldots,a_4\geq 0 } \frac{\mu(p^{a_1})\mu(p^{a_2})\mu(p^{a_3})\mu(p^{a_4})}{p^{a_1u_1}p^{a_2u_2}p^{a_3u_3}p^{a_4u_4} p^{2s\max(a_i)}} \\
&=\pth{1-\frac{1}{p^{2s}}}^{-1} Z_{1,p}(\boldu,s),
\end{align*}
say. Here we also define $Z_{1,2}(\boldu,s) = 1$. The sum $Z_{1,p}(\boldu,s)$ contains 16 nonzero terms determined by the condition $\max(a_i) \leq 1$, and we have
\begin{align*}
Z_{1,p}(\boldu,s) &= 1+\frac{1}{p^{2s}}\sum_{\substack{A\subset \set{1,2,3,4}\\ A\neq \emptyset}} (-1)^{\abs{A}} \prod_{i\in A} \frac{1}{p^{u_i}} = 1- \frac{1}{p^{2s}} + \frac{1}{p^{2s}} \prod_{i=1}^4 \pth{1-\frac{1}{p^{u_i}}}.
\end{align*}
Note that if $\Re(u_i) \geq \frac{2}{3}$, $\Re(s) \geq -\frac{1}{10}$, then
\begin{align*}
Z_{1,p}(\boldu,s) &= 1 - \sum_{i=1}^4 \frac{1}{p^{u_i+2s}} + O\fracp{1}{p^{2\min(\Re(u_i)) + 2\Re(s)}} = \prod_{i=1}^4 \sumpth{1-\frac{1}{p^{u_i+2s}}} Z_{1,p}^*(\boldu,s),
\end{align*}
say, where $Z_{1,2}^*(\boldu,s) = 1$ and
\[
Z_{1,p}^*(\boldu,s) = 1+O\sumpth{\frac{1}{p^{1+\frac{2}{15}}}}.
\]
for $p\neq 2$. Writing 
\[
Z_1(\boldu,s) = \prod_{p} Z_{1,p}(\boldu,s), \quad Z_1^*(\boldu,s) = \prod_{p} Z_{1,p}^*(\boldu,s),
\]
we have, in this region,
\[
D_p(\boldu,s) = \zeta(2s) Z_1(\boldu,s) = \zeta(2s)\sumpth{\prod_{i=1}^{4} \zeta_2(u_i+2s)}^{-1} Z_1^*(\boldu,s).
\]
Combining the above calculations gives
\[
D(\boldu,s) = \zeta(2s)\sumpth{\prod_{i=1}^{4} \zeta_2(u_i+2s)}^{-1} Z_1^*(\boldu,s),
\]
where $Z_1^*(\boldu,s)$ is analytic and uniformly bounded in the region $\Re(u_i) \geq \frac{2}{3}$, $\Re(s) \geq -\frac{1}{10}$. Returning to \eqref{eq:TvarChanged}, we have
\begin{align*}
\colt &= \frac{1}{(2\pi i)^5} \int\limits_{(1)}\cdots\int\limits_{(1)} \colg_t(\boldu)N^{u_1+u_2+u_3+u_4}\tilde{F}(s) X^s \zeta(2s) Z_{1}\pth{\boldu+\half, s} \prod_{i=1}^{4} \zeta_2(u_i+\half) \, d\boldu\, ds\\
&=\frac{1}{(2\pi i)^5} \int\limits_{(1)}\cdots\int\limits_{(1)} \colg_t(\boldu)N^{u_1+u_2+u_3+u_4}\tilde{F}(s) X^s \zeta(2s) Z_{1}^*\pth{\boldu+\half, s} \sumpth{\prod_{i=1}^{4} \frac{\zeta_2(u_i+\half)}{\zeta_2(u_i+\half+2s)}}\, d\boldu\, ds
\end{align*}
For use later, we note that for $\Re(u_i) > 0$,
\begin{equation}\label{eq:Z0half}
Z_1(\boldu,\half) = \prod_{p\neq 2} \sumpth{1- \frac{1}{p} + \frac{1}{p} \prod_{i=1}^4 \pth{1-\frac{1}{p^{u_i}}}}.
\end{equation}
In the integral $\colt$, we take the line of integration in $s$ to $\Re(s) = \frac{1}{L}$, where now we put $L = \log X$ (note that we have changed the definition from that given in \eqref{eq:Ldef}). In doing so, we pass a simple pole at $s=\frac{1}{2}$. Thus
\[
\colt = \colr_1 + \colj_1,
\]
where
\begin{equation}\label{eq:ResidueNonDual}
\colr_1= \frac{\tilde{F}\fracp{1}{2}\sqrt{X}}{(2\pi i)^4} \int\limits_{(2)}\cdots\int\limits_{(2)} \colg_t(\boldu)N^{u_1+u_2+u_3+u_4} Z_{1}\pth{\boldu+\half, \half} \prod_{i=1}^{4} \zeta_2(u_i+\half)\, d\boldu
\end{equation}
and
\begin{align}
\colj_1 &= \frac{1}{(2\pi i)^5} \int\limits_{(\frac{1}{L})} \int\limits_{(2)}\cdots\int\limits_{(2)} \colg_t(\boldu)N^{u_1+u_2+u_3+u_4}\tilde{F}(s) X^s \zeta(2s) \nonumber\\
&\quad\times Z_{1}^*\pth{\boldu+\half, s} \sumpth{\prod_{i=1}^{4} \frac{\zeta_2(u_i+\half)}{\zeta_2(u_i+\half+2s)}}\, d\boldu\, ds,
\label{def:ErrorIntegralSquares}
\end{align}
and $\colg_t(\boldu)$ is defined in \eqref{eq:colgDef}.

\subsection{Calculation of $\colm^{\square}$}
Recall from \eqref{def:DualSquares} that
\[
\colm^{\square} = X\sum_{k\geq 1} \sum_{\substack{n_1,\ldots,n_4 \\ (n_i,2)=1}} \frac{1}{(n_1n_2)^{\frac{1}{2}+ i  t} (n_3n_4)^{\frac{1}{2}- i  t}} \frac{G_{k^2}(n_1n_2n_3n_4)}{n_1n_2n_3n_4} \check{F}\fracp{k^2X}{n_1n_2n_3n_4} \prod_{i=1}^4 G\fracp{n_i}{N}.
\]
By Mellin inversion, \eqref{eq:Fhatfirst}, \eqref{def:Z}, Lemma \ref{lem:DirichletSeriesOffDiagonal}, and  a change of variables, we have
\begin{align*}
    \colm^{\square} &= \frac{2X}{(2\pi i)^5} \int\limits_{(\frac{3}{4})} \int\limits_{(1)} \cdots \int\limits_{(1)} \\
    &\quad \times (2\pi X)^{-s} \tilde{F}(1-s) \Gamma(s) \cos\fracp{\pi s}{2} \colg_t(\boldu)N^{u_1+u_2+u_3+u_4} Z(\boldu+\half-s,s;1,1)\, d\boldu\, ds.
\end{align*}
We take the line of integration in $s$ to $\Re(s) = \frac{1}{12}$, passing only a simple pole at $s=\half$. Since
\[
2\Gamma(\tfrac{1}{2}) (2\pi)^{-\frac{1}{2}} \cos(\tfrac{\pi}{4}) = 1,
\]
we obtain
\[
\colm^{\square} = \colr_2 +\colj_2,
\]
where
\[
\colr_2 = \frac{\tilde{F}(\half) \sqrt{X}}{(2\pi i)^4} \int\limits_{(1)} \cdots \int\limits_{(1)} \colg_t(\boldu)N^{u_1+u_2+u_3+u_4} \sumpth{\res{s=\half}  Z(\boldu+\half-s,s;1,1)}\, d\boldu
\]
and

\begin{align}
\colj_2 &= \frac{X}{(2\pi i)^5} \int\limits_{(\frac{1}{12})} \int\limits_{(1)} \cdots \int\limits_{(1)}\colg_t(\boldu)N^{u_1+u_2+u_3+u_4}\tilde{F}(1-s) X^{-s} \Gamma(s) \cos\fracp{\pi s}{2} 2^{1-s} \pi^{-s}\nonumber\\
&\quad\times Z(\boldu+\half-s,s;1,1)\, d\boldu\, ds.
\label{def:ErrorIntegralDualSquares}
\end{align}
By \eqref{eq:resZathalf} and \eqref{eq:Z0half}, we have
\begin{align*}
\res{s=\half}  Z(\boldu+\half-s,s;1,1) &= \prod_{i=1}^4 \zeta_2(u_i+\half) \prod_{p\neq 2}\sumpth{1-\frac{1}{p} + \frac{1}{p}\prod_{i=1}^4  \sumpth{1-\frac{1}{p^{\frac{1}{2}+u_i}}}} \\
&= Z_1(\boldu+\half,\half)\prod_{i=1}^4 \zeta_2(u_i+\half),
\end{align*}
and so $\colr_1 = \colr_2$. Combining the above calculations with \eqref{eq:ResidueNonDual} and \eqref{def:ErrorIntegralSquares}, we find that
\[
\colm^{\square} - \colt= \colj_2 - \colj_1,
\]
where $\colj_1$ is given by \eqref{def:ErrorIntegralSquares} and $\colj_2$ is given by \eqref{def:ErrorIntegralDualSquares}. 

\subsection{Estimation of $\colj_1$}
To estimate $\colj_1$, we change variables $u_i\mapsto u_i + \frac{1}{2}$ and take the four lines of integration in $u_i$ to $\Re(u_i) = 1$, so
\begin{align*}
\colj_1 &= \frac{N^2}{(2\pi i)^5} \int\limits_{(\frac{1}{L})} \int\limits_{(1)} \cdots\int\limits_{(1)}
 \colg_t(\boldu+\half)N^{u_1+u_2+u_3+u_4}\tilde{F}(s) X^s \zeta(2s) \\
 &\quad \times Z_{1}^*\pth{\boldu+1, s} \sumpth{\prod_{i=1}^{4} \frac{\zeta_2(u_i+1)}{\zeta_2(u_i+1+2s)}} \, d\boldu\, ds.
\end{align*}
where $Z_1^*(1+\boldu,s)$ is analytic and uniformly bounded for $\Re(u_i) \geq -\frac{1}{3}$ and $\Re(s) \geq -\frac{1}{10}$ and recall that $L = \log X$ in this section. We now take the lines of integration in $u_i$ to $\Re(u_i) = \frac{1}{L}$, passing no poles in the process. We estimate the various factors inside the integral as follows:
\begin{itemize}
\item We have
\[
\colg_t(\boldu+\half) \ll \prod_{i=1}^4 \frac{1}{(1+\abs{u_i\pm it})^{10}}.
\]
with the $\pm$ convention as in Section \ref{sec:LargeSieve-Diagonal}. 
\item Integrating by parts twice, we have
\[
\tilde{F}(s) = \frac{1}{s(s+1)}  \int\limits_{0}^{\infty} F''(t) t^{s+1} dt.
\]
 Thus
\[
\tilde{F}(\sigma+ i  t) \ll \frac{1}{(\sigma+\abs{t})(1+\abs{t})} \ll \frac{\sigma^{-1}}{(1+\abs{t})^2}.
\]
\item By the standard convexity estimate for $\zeta$, we have
\[
\zeta(2s) \ll (1+\abs{\Im(s)})^{\frac{1}{2}}
\]
uniformly for $0 \leq \Re(s) \leq \frac{1}{2}$.
\item Since $\Re(s) = \Re(u_i) = \frac{1}{L}$ and $X > N^2$, we have 
\[
X^s \prod_{i=1}^4 N^{u_i} \ll 1.
\]
Using the zero free region for $\zeta$, we have
\[
\frac{1}{\zeta(u_i+1+2s)} \ll \log(2+\abs{\Im(u_i+s)})
\]
uniformly for $\Re(u_i) \geq 0$, $\Re(s) \geq 0$.
\item For $\Re(u_i) = \frac{1}{L}$, we have
\[
\zeta(1+u_i) \ll L.
\]
\end{itemize}
Putting $t_i = \Im(u_i)$ for $i=1,2,3,4$ and $t_5 = \Im(s)$, we combine all of the above estimates to see that
\begin{align*}
\colj_1 &\ll N^2 L^5 \int\limits_{-\infty}^\infty \cdots   \int\limits_{-\infty}^\infty\frac{1}{(1+\abs{t_5})^\frac{3}{2}} \prod_{i=1}^4 \frac{\log(2+\abs{t_i}+\abs{t_5})}{(1+\abs{t_i\pm it})^{10}}\, dt_1 \cdots dt_5 \\
&\ll N^2 L^5 \log^4(2+\abs{t}) \int\limits_{-\infty}^\infty \cdots   \int\limits_{-\infty}^\infty\frac{\log^4(2+\abs{t_5})}{(1+\abs{t_5})^\frac{3}{2}} \prod_{i=1}^4 \frac{\log(2+\abs{t_i})}{(1+\abs{t_i})^{10}}\, dt_1 \cdots dt_5 \\
&\ll N^2 L^5 \log^4(2+\abs{t}).
\end{align*}

Here we have used the inequality
\begin{equation}\label{eq:logSumProd}
\log(2+a_1+\cdots +a_r) \ll_r \prod_{i=1}^r \log(2+a_r)
\end{equation}
for $a_i \geq 0$ and $r\geq 1$. Therefore
\begin{equation*}\label{eq:J1Bound}
\colj_1 \ll N^2 (\log X)^5 \log^4(2+\abs{t}) \ll N^2 (1+\abs{t})^{\frac{1}{5}}(\log X)^5.
\end{equation*}

\subsection{Calculation of $\colj_2$}\label{sec:colj2}
Before evaluating the integral $\colj_2$ in \eqref{def:ErrorIntegralDualSquares}, let us make the following observation: the integrand has a simple pole at $s=0$ arising from the factor of $\Gamma(s)$. The residue $\colr$, say, of the integrand at this point is
\[
\colr = \frac{2\tilde{F}(1)X}{(2\pi i)^4} \int\limits_{(1)} \cdots \int\limits_{(1)}\colg_t(\boldu)N^{u_1+u_2+u_3+u_4} \sumpth{\res{s=0} Z(\boldu+\half-s,s;1,1)} d\boldu.
\]
By \eqref{eq:resZat0} and the proof of Lemma \ref{lem:DirichletSeriesDiagonal} (see \eqref{equ:D1-111++}), we have
\[
2\sumpth{\res{s=0} Z(\boldu+\half-s,s;1,1)} = -\prod_{p\neq 2} \sumpth{\frac{1}{p} + \pth{1-\frac{1}{p}}B_p(\boldu+\half;0)} = - \cold_1(\boldu+\half),
\]
and so by \eqref{eq:DiagonalIntegral},
\[
\colr = -\colm^{0}.
\]
Unfortunately, this perceived cancellation is somewhat ``illusory''. As our calculations below demonstrate, there are, in fact, additional poles of the integrand which contribute terms that do not cancel with $\colm^{0}$. Just as we did with \eqref{eq:DiagonalIntegral}, we will show that $\colj_2$ may be reduced to an integral in a single variable plus an acceptable error term. Calculating the leading order term of this integral then yields Lemma \ref{lem:LS-SquareCancellation}. 

Proceeding to details, we have
\begin{align*}
\colj_2 &= \frac{X}{(2\pi i)^5} \int\limits_{(\frac{1}{12})} \int\limits_{(1)} \cdots \int\limits_{(1)}\colg_t(\boldu)N^{u_1+u_2+u_3+u_4}\tilde{F}(1-s) X^{-s} \Gamma(s) \cos\fracp{\pi s}{2} 2^{1-s} \pi^{-s}\\
&\quad\times Z(\boldu+\half-s,s;1,1)\, d\boldu\, ds \\
&= \frac{X}{(2\pi i)^5} \int\limits_{(\eta+\frac{1}{L})} \int\limits_{(1)} \cdots \int\limits_{(1)}\colg_t(\boldu+s)\fracp{N^4}{X}^s N^{u_1+u_2+u_3+u_4}\tilde{F}(1-s)  \Gamma(s) \cos\fracp{\pi s}{2} 2^{1-s} \pi^{-s}\\
&\quad\times \zeta(2s) \prod_{i=1}^4 \frac{\zeta_2\pth{1+u_i}}{\zeta_2(1+u_i+2s)}\prod_{1\leq i\leq j\leq {4}} \frac{\zeta_2(1+u_i+u_j+2s)}{\zeta_2(2+u_i+u_j)} Z_2(\boldu+\half,s;1,1) d\boldu\, ds 
\end{align*}
after changing variables $u_i \mapsto u_i+s$, applying the identity for $Z$ given in Lemma \ref{lem:DirichletSeriesOffDiagonal}, and shifting the line of integration in $s$ to $\Re(s) = \eta + \frac{1}{L}$ for some small constant $\eta>0$ that we will choose later. Here, as before, $L = \log X$. We first isolate the relevant portions of the integrand of $\colj_2$ as follows. Let
\begin{align*}
\colg_t(\boldu;s) &= 2^{1-s} \pi^{-s} \tilde{F}(1-s) \cos\fracp{\pi s}{2} \zeta(2s) Z_2(\boldu+\half,s;1,1) \colg_t(\boldu+s)\prod_{1\leq i\leq j\leq 4} \frac{1}{\zeta_2\pth{2+u_i+u_j}} \\
&\quad \times s \Gamma(s) \sumpth{\prod_{i=1}^4 \frac{\zeta_2(1+u_i)u_i}{\zeta_2(1+u_i+2s)(u_i+2s)}} \sumpth{\prod_{1\leq i\leq j\leq 4} \zeta_2\pth{1+u_i+u_j+2s}(u_i+u_j+2s)}.
\end{align*}
In the region 
\[
- \frac{1}{20} \leq \Re(s) \leq \frac{1}{4}, \quad \Re(u_i) \geq -\frac{1}{10},
\]
the following estimates hold uniformly in the parameters $u_i$ and $s$:
\begin{itemize}
\item By Lemma \ref{lem:DirichletSeriesOffDiagonal}, the function $Z_2(\boldu+\half,s;1,1)$ is analytic and bounded.
\item The factors of $\zeta_2(2+u_i+u_j)^{-1}$ are bounded.
\item We have
\[
\colg_t(\boldu+s) \ll \prod_{i=1}^4 \frac{1}{(1+\abs{u_i+s\pm it})^{10}},
\]
with the $\pm$ convention as in Section \ref{sec:LargeSieve-Diagonal}.
\item By Stirling's formula and the convexity estimate for $\zeta$ (see Section \ref{sec:LargeSieve-Diagonal}), we have
\[
s\Gamma(s) \cos\fracp{\pi s}{2}\zeta(2s) \ll (1+\abs{\Im(s)})^2.
\]
and
\[
\zeta_2(1+u_i) \ll  \frac{1}{\abs{u_i}}+(1+ \abs{\Im(u_i)})^{\max(0, \frac{-\Re(u_i)}{2})+\ep}
\]
\item By repeated integration by parts, we have
\[
\tilde{F}(1-s) \ll \frac{1}{(1+\abs{\Im(s)})^{12}}.
\]
\item The other four factors of $\zeta^{-1}$ may be bounded using \eqref{eq:ZeroFreeRegion} and \eqref{eq:1/ZetaUpper}: $\zeta(\sigma+it)$ has no zeros in the region
\[
\sigma \geq 1 - \frac{c}{\log(2+\abs{t})},
\]
and in this region
\[
\frac{1}{\zeta_2(\sigma+it)} \ll \log(2+\abs{t}).
\]
Thus in the region
\[
- \frac{1}{20} \leq \Re(s) \leq \frac{1}{4}, \qquad \Re(u_i)\geq -2\Re(s), 
\]
we have the uniform estimate
\[
\prod_{i=1}^4 \frac{1}{\zeta_2(1+u_i+2s)(u_i+2s)} \ll \prod_{i=1}^4 \log(2+\abs{\Im(u_i+2s)}).
\]
\end{itemize}
With the above notation, we have
\[
\colj_2 = \frac{X}{(2\pi i)^5} \int\limits_{(\eta+\frac{1}{L})} \int\limits_{(1)} \cdots \int\limits_{(1)} \fracp{N^4}{X}^s N^{u_1+u_2+u_3+u_4} \colg_t(\boldu;s)\prod_{i=1}^4 \frac{u_i+2s}{u_i}\prod_{1\leq i\leq j\leq 4} \frac{1}{u_i+u_j+2s} d\boldu\, \frac{ds}{s}.
\]
 We begin by taking $u_1$ to $\Re(u_1) = -\eta$, passing a simple pole at $u_1=0$. Thus
\[
\colj_2 = \colj_{2,1} + \colj_{2,2},
\]
where
\begin{align*}
\colj_{2,1} &= \frac{X}{(2\pi i)^4} \int\limits_{(\eta+\frac{1}{L})} \int\limits_{(1)} \int\limits_{(1)} \int\limits_{(1)} \fracp{N^4}{X}^sN^{u_2+u_3+u_4} \frac{\colg_t(0,u_2,u_3,u_4;s)}{u_2u_3u_4} \\
&\quad\times\prod_{2\leq i\leq j\leq 4} \frac{1}{u_i+u_j+2s} du_2\, du_3\, du_4\, \frac{ds}{s}, \\
\colj_{2,2} &= \frac{X}{(2\pi i)^5} \int\limits_{(\eta+\frac{1}{L})} \int\limits_{(-\eta)}\int\limits_{(1)}\int\limits_{(1)}\int\limits_{(1)} \fracp{N^4}{X}^sN^{u_1+u_2+u_3+u_4} \colg_t(\boldu;s)\\
&\quad\times\prod_{i=1}^4 \frac{u_i+2s}{u_i}\prod_{1\leq i\leq j\leq 4} \frac{1}{u_i+u_j+2s}\, d\boldu\, \frac{ds}{s}.
\end{align*}
In both integrals above, we take the lines of integration in $u_2$ to $\Re(u_2) = -\eta$, passing a simple pole at $u_2 = 0$ in each. This gives a total of four integrals. In each of these four integrals, we then take the line of integration in $u_3$ to $\Re(u_3) = -\eta$, and then we do the same for the variable $u_4$. In shifting contours this way, we pass simple poles at $u_3 = u_4 = 0$, and we obtain 16 integrals in total. To list these concisely, recall the definition of the notation $\underset{x=x_0}{\operatorname{Value}\,}$ from the introduction. We have
\[
\colj_2 = \colr_0 + \sum_{i=1}^4 \colr_1^{(i)} + \sum_{1\leq i < j \leq 4} \colr_2^{(i,j)} + \sum_{i=1}^4 \colr_3^{(i)} + \colr_4,
\]
where
\begin{align*}
\colr_0 &= \frac{X}{64(2\pi i)} \int\limits_{(\eta+\frac{1}{L})}  \fracp{N^4}{X}^s \colg_t(0,0,0,0;s) \frac{ds}{s^7},\nonumber
\\
\colr_1^{(i)} &= \frac{X}{8(2\pi i)^2} \int\limits_{(\eta+\frac{1}{L})} \int\limits_{(-\eta)} \frac{N^{u_i}}{u_i(u_i+2s)^2(2u_i+2s)} \fracp{N^4}{X}^s\underset{\substack{u_k=0\\ k\neq i}}{\operatorname{Value}\,}  \colg_t(\boldu;s) \, du_i\, \frac{ds}{s^4},\nonumber\\
\colr_2^{(i,j)} &= \frac{X}{2(2\pi i)^3} \int\limits_{(\eta+\frac{1}{L})} \int\limits_{(-\eta)}\int\limits_{(-\eta)}  \frac{N^{u_i+u_j}}{u_iu_j(u_i+2s)(u_j+2s)(2u_i+2s)(2u_j+2s)(u_i+u_j+2s)} \nonumber\\
&\quad\times \fracp{N^4}{X}^s\underset{\substack{u_k=0,\\ k\neq i,j}}{\operatorname{Value}\,}\colg_t(\boldu;s)\, du_i\,du_j\, \frac{ds}{s^2},\nonumber\\
\colr_3^{(\ell)} &= \frac{X}{(2\pi i)^4} \int\limits_{(\eta+\frac{1}{L})} \int\limits_{(-\eta)} \int\limits_{(-\eta)} \int\limits_{(-\eta)} \frac{N^{u_1+u_2+u_3+u_4-u_\ell} u_\ell}{u_1u_2u_3u_4} \prod_{\substack{1\leq i\leq j\leq 4\\ i,j\neq \ell}} \frac{1}{u_i+u_j+2s} \nonumber\\
&\quad\times \fracp{N^4}{X}^s\underset{u_\ell=0}{\operatorname{Value}\,} \colg_t(\boldu;s)\, d\boldu^{(\ell)}\, \frac{ds}{s}, \nonumber
\end{align*}
where $d\boldu^{(\ell)}$ means $du_1\, \ldots du_4$ with $du_\ell$ removed, and
\begin{align*}
\colr_4& = \frac{X}{(2\pi i)^5} \int\limits_{(\eta+\frac{1}{L})} \int\limits_{(-\eta)}\int\limits_{(-\eta)}\int\limits_{(-\eta)}\int\limits_{(-\eta)} \fracp{N^4}{X}^s N^{u_1+u_2+u_3+u_4} \colg_t(\boldu;s)\\
&\quad \times \prod_{i=1}^4 \frac{u_i+2s}{u_i}\prod_{1\leq i\leq j\leq 4} \frac{1}{u_i+u_j+2s} d\boldu\, \frac{ds}{s}.
\end{align*}
Since $X > N^2$, we immediately have
\[
\colr_4,\colr_3^{(\ell)} \ll X^{1-\frac{\eta}{2}}(1+\abs{t})^{\frac{1}{5}}, \qquad \colr_2^{(i,j)} \ll X(\log X)^3(1+\abs{t})^{\frac{1}{5}}
\]
upon using the estimates for $\colg_t$ given above so long as $\eta$ is small enough. Therefore
\[
\colj_2 = \colr_0 + \sum_{i=1}^4 \colr_1^{(i)} + O\pth{X(\log X)^3(1+\abs{t})^{\frac{1}{5}}}.
\]
For each of the four integrals $\colr_1^{(i)}$, we take the line of integration in $u_i$ to $\Re(u_i) = -2\eta$, passing a simple pole at $u_i = -s$. The integral on the new line of integration is $O\pth{X(\log X)^2(1+\abs{t})^{\frac{1}{5}}}$, and thus
\[
\colr_1^{(i)} = \colp^{(i)} + O\pth{X(\log X)^2(1+\abs{t})^{\frac{1}{5}}},
\]
where
\[
\colp^{(1)} = -\frac{X}{16(2\pi i)^2} \int\limits_{(\eta+\frac{1}{L})} \fracp{N^3}{X}^s  \colg_t(-s,0,0,0;s) \frac{ds}{s^7}.
\]
and $\colp^{(i)}$ are defined similarly for $i=2,3,4$. 
Thus
\begin{equation}\label{eq:DualSquaresIntegralReduced}
\colj_2 = \colr_0 + \sum_{i=1}^4 \colp^{(i)} + O\pth{X(\log X)^3(1+\abs{t})^{\frac{1}{5}}}.
\end{equation} 
Note that
\begin{align*}
\colg_t(0,0,0,0;s) &= \frac{1}{16\zeta_2(2)^{10}}2^{1-s} \pi^{-s} \tilde{F}(1-s) \cos\fracp{\pi s}{2} \zeta(2s) Z_2(\half,\half,\half,\half,s;1,1) \colg_t(s,s,s,s)\\
&\quad \times s \Gamma(s)(\zeta_2\pth{1+2s}(2s))^6.
\end{align*}
\begin{align*}
\colg_t(-s,0,0,0;s) &= \frac{1}{16\zeta_2(2)^6}2^{1-s} \pi^{-s} \tilde{F}(1-s) \cos\fracp{\pi s}{2} \zeta(2s) Z_2(\half-s,\half,\half,\half,s;1,1) \colg_t(-s,0,0,0)  \\
&\quad \times s \Gamma(s) \frac{\zeta_2(1-s)(-s)(\zeta_2(1+s)s)^2(\zeta_2(1+2s)(2s))^3}{\zeta_2\pth{2-2s}\zeta_2(2-s)^3}
\end{align*}
and $\colg_t(0,-s,0,0;s)$, $\colg_t(0,0,-s,0;s)$, $\colg_t(0,0,0,-s;s)$ are analogous to $\colg_t(-s,0,0,0;s)$.  In particular, since we will have $-\frac{1}{10} \leq \Re(s) \leq \frac{1}{4}$, we do not need to use the zero free region for $\zeta$ when shifting contours in $s$ to the left.

We estimate the integrals in \eqref{eq:DualSquaresIntegralReduced} as follows. If $X > N^4$, we keep the lines of integration in $s$ on $\Re(s) = \eta+\frac{1}{L}$, and in this case $\colj_2 \ll X(\log X)^3(1+\abs{t})^{\frac{1}{5}}$. If $N^3 < X \leq N^4$, then $\colp^{(i)} \ll X(1+\abs{t})^{\frac{1}{5}}$ and 
\[
\colj_2 = \frac{X}{64(2\pi i)} \int\limits_{(\eta+\frac{1}{L})}X^{-s}N^{4s} \colg_t(0,0,0,0;s) \frac{ds}{s^7}+ O\pth{X(\log X)^3(1+\abs{t})^{\frac{1}{5}}}.
\]
We then take the line of integration in $s$ to $\Re(s) = -\frac{1}{10}$, passing the pole of order 7 at $s=0$. The integral on the new line of integration is $O(X(1+\abs{t})^{\frac{1}{5}})$, and so
\[
\colj_2 = \frac{\colg_t(0,0,0,0;0)X}{64\cdot 6!} \sumpth{\log\fracp{N^4}{X}}^6  + O\pth{X(\log X)^5(1+\abs{t})^{\frac{1}{5}}}.
\]
Lastly, if $N^2 < X \leq N^3$, we shift all lines of integration to $\Re(s) = -\frac{1}{10}$, and in this case
\[
\colj_2 = \frac{\colg_t(0,0,0,0;0)X}{64 \cdot 6!} \sumpth{\sumpth{\log\fracp{N^4}{X}}^6 - 4\sumpth{\log\fracp{N^3}{X}}^6} + O\pth{X(\log X)^5(1+\abs{t})^{\frac{1}{5}}}.
\]
Defining 
\begin{equation*}\label{eq:logstar}
\log^*v = \max(0,\log v),
\end{equation*}
we always have
\[
\colj_2 = \frac{\colg_t(0,0,0,0;0)X}{64 \cdot 6!} \sumpth{\sumpth{\log^*\fracp{N^4}{X}}^6 - 4\sumpth{\log^*\fracp{N^3}{X}}^6} + O\pth{X(\log X)^5(1+\abs{t})^{\frac{1}{5}}}.
\]
Now
\[
\colg_t(0,0,0,0;0) = -\frac{\tilde{F}(1)}{2^{10}\zeta_2(2)^{10}} Z_2(\half,\half,\half,\half,0;1,1) |\Tilde{G}(it)|^4,
\]
where, by \eqref{eq:Z2, p not dividing k1},
\[
Z_2(\half,\half,\half,\half,0;1,1) = \prod_{p\neq 2} Z_{2,p}(\half,\half,\half,\half,0;1,1)
\]
and by \eqref{eq:Chalf},
\begin{align*}
Z_{2,p}(\half,\half,\half,\half,0;1,1) &= \prod_{1\leq i\leq j\leq 4}\pth{1-\frac{1}{p^{2}}}^{-1} \sumpth{\frac{1}{p}\prod_{1\leq i \leq j \leq 4} \pth{1-\frac{1}{p}} + C_p(\half,\half,\half,\half) \sumpth{1-\frac{1}{p}}} \\
&= \pth{1-\frac{1}{p^2}}^{-10} \sumpth{\frac{1}{p}\pth{1-\frac{1}{p}}^{10} + \pth{1-\frac{1}{p}}^{7}\pth{1+\frac{6}{p}+\frac{1}{p^2}}}.
\end{align*}
Thus by \eqref{eq:Colh1Half}, we have
\[
Z_2(\half,\half,\half,\half,0;1,1) = \zeta_2(2)^{10}\colh_1\pth{\half,\half,\half,\half},
\]
and so
\[
\colg_t(0,0,0,0;0) = -\frac{\tilde{F}(1)}{2^{10}}|\Tilde{G}(it)|^4 \colh_1\pth{\half,\half,\half,\half}.
\]
This gives Lemma \ref{lem:LS-SquareCancellation}.

\section{Off-diagonal, Non-square Contribution $\colm^{\neq \square}$}\label{sec:LargeSieve-NonSqare}
Recall from \eqref{def:NonSquare} that
\[
\colm^{\neq \Box} = X\sum_{k\neq 0,\square} \sum_{\substack{n_1,\ldots,n_4 \\ (n_i,2)=1}} \frac{1}{(n_1n_2)^{\frac{1}{2}+ i  t} (n_3n_4)^{\frac{1}{2}- i  t}} \frac{G_k(n_1n_2n_3n_4)}{n_1n_2n_3n_4} \check{F}\fracp{kX}{n_1n_2n_3n_4} \prod_{i=1}^4 G\fracp{n_i}{N}.
\]
By \eqref{eq:Fbounds} and Lemma \ref{lem:ConvenientTestFandG}, we have that $\check{F}$ is supported on $[-\frac{1}{64},\frac{1}{64}]$ and $G$ is supported on $[\frac{3}{4},2]$. Thus the range of $k$ is
\begin{align}
\abs{k} \leq \frac{N^4}{4X} =: K,
\label{equ:defK}
\end{align}
say. For each $k\neq \square$, we write $k= k_1k_2^2$ with $k_1\neq 1$ squarefree. By Mellin inversion, \eqref{eq:Fhatfirst}, Lemma \ref{lem:DirichletSeriesOffDiagonal}, and a change of variables, we have
\begin{align*}
\colm^{\neq \square} &= \frac{2X}{(2\pi i)^5} \int\limits_{(\frac{3}{5})} \int\limits_{(1)} \cdots \int\limits_{(1)} (2\pi X)^{-s} \tilde{F}(1-s) \Gamma(s) \cos\fracp{\pi s}{2} \colg_t(\boldu)N^{u_1+u_2+u_3+u_4}\\
&\quad\times \zeta(2s)\sumstar_{\substack{\abs{k_1} \leq K\\ k_1\neq 0,1}} \frac{1}{\abs{k_1}^s} \sumpth{\prod_{i=1}^4 L\pth{1+u_i-s,\chi_\mathfrak{m}}} Y\pth{\boldu+\half-s,s;k_1,1} \, d\boldu\, ds,
\end{align*}
where recall that $\mathfrak{m} = \mathfrak{m}(k_1)$ is the fundamental discriminant defined in \eqref{def:m(k)}.
Writing out the multiple Dirichlet series
\[
Y(\boldu,s;k_1) = \sum_{r_1,\ldots,r_5} \frac{C_{k_1}(r_1,r_2,r_3,r_4,r_5)}{r_1^{u_1}r_2^{u_2}r_3^{u_3}r_4^{u_4} r_5^{2s}}, 
\]
we have
\begin{align*}
&\frac{2X}{(2\pi i)^4} \int\limits_{(1)}\cdots\int\limits_{(1)} \colg_t(\boldu)N^{u_1+u_2+u_3+u_4} \sumpth{\prod_{i=1}^4 L\pth{1+u_i-s,\chi_\mathfrak{m}}} Y\pth{\boldu +\half-s,s;k_1,1} d\boldu \\
&= \sum_{r_1,\ldots,r_5} \frac{C_{k_1}(r_1,r_2,r_3,r_4,r_5)}{r_1^{\frac{1}{2}+ i  t-s}r_2^{\frac{1}{2}+ i  t-s}r_3^{\frac{1}{2}- i  t-s}r_4^{\frac{1}{2}- i  t-s} r_5^{2s}} \sum_{n_1,\ldots,n_4} \frac{\chi_\mathfrak{m}(n_1n_2n_3n_4)}{n_1^{1+ i  t-s}n_2^{1+ i  t-s}n_3^{1- i  t-s}n_4^{1- i  t-s}} \prod_{i=1}^4 G\fracp{n_ir_i}{N} \\
&= \sumd_{R_1,\ldots, R_4} \sum_{r_1,\ldots,r_5} \frac{C_{k_1}(r_1,r_2,r_3,r_4,r_5)}{r_1^{\frac{1}{2}+ i  t-s}r_2^{\frac{1}{2}+ i  t-s}r_3^{\frac{1}{2}- i  t-s}r_4^{\frac{1}{2}- i  t-s} r_5^{2s}} \prod_{i=1}^4 G\fracp{r_i}{R_i} \\
&\quad\times \sum_{n_1,\ldots,n_4} \frac{\chi_\mathfrak{m}(n_1n_2n_3n_4)}{n_1^{1+ i  t+u_1-s}n_2^{1+ i  t-s}n_3^{1- i  t-s}n_4^{1- i  t-s}} \prod_{i=1}^4 G\fracp{n_ir_i}{N} V\fracp{n_iR_i}{N},
\end{align*}
where we have applied partitions of unity to the sums over $r_i$ and included redundant factors $V$, which we define as follows. The sum $\sumd_{R_1,\ldots,R_4}$ denotes sum over $R_i = 2^j$ for $j\geq 0$. Since $G(\frac{r_in_i}{N})$ restricts $\frac{3N}{4r_i} \leq n_i \leq \frac{2N}{r_i}$ and $G(\frac{r_i}{R_i})$ restricts $\frac{3}{4} R_i \leq r_i \leq 2R_i$, the only nonzero terms in sums over $n_i$ satisfy
\[
\frac{3N}{8R_i} \leq n_i \leq \frac{8N}{3R_i}. 
\]
Thus the above holds for any redundant factor $V$ which is identically 1 on $[\frac{3}{8},\frac{8}{3}]$. We set
\[
V_1(x) = G(x/2) + G(x) + G(2x) + G(4x)
\]
so that $V$ is identically 1 on $[\frac{1}{4},3]$. We now apply Mellin inversion once to obtain
\begin{align*}
\colm^{\neq \square} &= \sumd_{R_1,\ldots, R_4} \frac{2X}{(2\pi i)^5} \int\limits_{(\frac{3}{5})} \int\limits_{(1)} \cdots \int\limits_{(1)} (2\pi X)^{-s} \tilde{F}(1-s) \Gamma(s) \cos\fracp{\pi s}{2} \sumpth{\prod_{i=1}^4 \tilde{G}(u_i)}N^{u_1+u_2+u_3+u_4}\\
&\quad\times \zeta(2s)\sumstar_{\substack{\abs{k_1} \leq K\\ k_1\neq 0,1}} \frac{1}{\abs{k_1}^s} \sum_{r_1,\ldots,r_5} \frac{C_{k_1}(r_1,r_2,r_3,r_4,r_5)}{r_1^{\frac{1}{2}+ i  t+u_1-s}r_2^{\frac{1}{2}+ i  t+u_2-s}r_3^{\frac{1}{2}- i  t+u_3-s}r_4^{\frac{1}{2}- i  t+u_4-s} r_5^{2s}} \prod_{i=1}^4 G\fracp{r_i}{R_i} \\
&\quad\times \sum_{n_1,\ldots,n_4} \frac{\chi_\mathfrak{m}(n_1n_2n_3n_4)}{n_1^{1+ i  t+u_1-s}n_2^{1+ i  t+u_2-s}n_3^{1- i  t+u_3-s}n_4^{1- i  t+u_4-s}} \prod_{i=1}^4 V_1\fracp{n_iR_i}{N} \, d\boldu\, ds,
\end{align*}
Changing variables $u_i \mapsto u_i +s$ and recalling the definition of $V_1$, we see that the above is finite sum of terms of the form
\begin{align}
&\sumd_{R_1,\ldots, R_4} \frac{2X}{(2\pi i)^5} \int\limits_{(\frac{3}{5})} \int\limits_{(-\frac{1}{2})} \cdots \int\limits_{(-\frac{1}{2})} (2\pi X)^{-s} \tilde{F}(1-s) \Gamma(s) \cos\fracp{\pi s}{2} \sumpth{\prod_{i=1}^4 \tilde{G}(u_i+s)}N^{u_1+u_2+u_3+u_4+4s}\nonumber\\
&\times \zeta(2s)\sumstar_{\substack{\abs{k_1} \leq K\\ k_1\neq 0,1}} \frac{1}{\abs{k_1}^s} \sum_{r_1,\ldots,r_5} \frac{C_{k_1}(r_1,r_2,r_3,r_4,r_5)}{r_1^{\frac{1}{2}+ i  t+u_1}r_2^{\frac{1}{2}+ i  t+u_2}r_3^{\frac{1}{2}- i  t+u_3}r_4^{\frac{1}{2}- i  t+u_4} r_5^{2s}} \prod_{i=1}^4 G\fracp{r_i}{R_i}\nonumber \\
&\times \sum_{n_1,\ldots,n_4} \frac{\chi_\mathfrak{m}(n_1n_2n_3n_4)}{n_1^{1+ i  t+u_1}n_2^{1+ i  t+u_2}n_3^{1- i  t+u_3}n_4^{1- i  t+u_4}} \prod_{i=1}^4 G\fracp{n_i}{N_i} \, d\boldu\, ds,
\label{eq:OffDiagSeparated}
\end{align}
where $N_i \asymp N/R_i$. The variables $n_i$ and $r_i$ have been separated, and we now bound the sums over $r_i$.

\begin{lem}\label{lem:CsumBound}
Suppose $\Re(s) \geq \frac{3}{5}$ and write $u_i = -1/2+i\mu_i$, for real $\mu_i$. For any $k_1 \neq 0,1$, we have
\begin{align*}
&\sumabs{\sum_{r_1,\ldots,r_5} \frac{C_{k_1}(r_1,r_2,r_3,r_4,r_5)}{r_1^{\frac{1}{2}+ i  t+u_1}r_2^{\frac{1}{2}+ i  t+u_2}r_3^{\frac{1}{2}- i  t+u_3}r_4^{\frac{1}{2}- i  t+u_4} r_5^{2s}} \prod_{i=1}^4 G\fracp{r_i}{R_i} } \\
&\qquad \ll(1+\abs{t})^{\frac{1}{5}} \sumpth{\prod_{i=1}^4 (1+\abs{\mu_i})}\exp(-c_1 \sqrt{\log (R_1R_2R_3R_4)}),
\end{align*}
where the implied constant and $c_1$ are absolute.
\end{lem}

\begin{proof}
We may assume $\max(R_1,R_2,R_3,R_4) > e$, otherwise the estimate is trivial. We quote the standard zero-free region and lower bound for $\zeta(s)$. There is a positive constant $c$ such that for $s = \sigma+ i  t$, there are no zeros of $\zeta(s)$ in the region
\begin{equation}\label{eq:ZeroFreeRegion}
\sigma \geq 1 - \frac{c}{\log(2+\abs{t})}.
\end{equation}
See, for instance, Chapter 3 of \cite{Titchmarsh}. In this region, we have the standard estimate (see 3.11.8 of \cite{Titchmarsh})
\begin{equation}\label{eq:1/ZetaUpper}
\frac{1}{\zeta(s)} \ll \log(2+\abs{t}).
\end{equation}
By Mellin inversion, we have
\begin{align}
&\sum_{r_1,\ldots,r_5} \frac{C_{k_1}(r_1,r_2,r_3,r_4,r_5)}{r_1^{\frac{1}{2}+ i  t+u_1}r_2^{\frac{1}{2}+ i  t+u_2}r_3^{\frac{1}{2}- i  t+u_3}r_4^{\frac{1}{2}- i  t+u_4} r_5^{2s}} \prod_{i=1}^4 G\fracp{r_i}{R_i} \nonumber\\
&= \frac{1}{(2\pi i)^4} \int\limits_{(0)}\cdots \int\limits_{(0)} \sumpth{\prod_{i=1}^4 \tilde{G}(\omega_i)R^{\omega_i}}\nonumber\\
&\quad\times Y(\omega_1+i(\mu_1+t),\omega_2+i(\mu_2+t),\omega_3+i(\mu_3-t),\omega_4+i(\mu_4-t),s,k_1)\, d\boldomega. \label{eq:CsumMellin}
\end{align}
By Lemma \ref{lem:DirichletSeriesOffDiagonal},
\[
Y(\boldu,s;k_1) = \prod_{1\leq i\leq j\leq {4}} \sumpth{\frac{\zeta(u_i+u_j+2s)}{\zeta(u_i+u_j+1)}} \frac{Z_2(\boldu,s;k_1)}{\colp(\boldu,s;k_1)} =: Z_2^*(\boldu,s;k_1)\prod_{1\leq i\leq j\leq {4}} \zeta(u_i+u_j+1)^{-1}
\]
say, where
\[
\colp(\boldu,s;k_1) = \prod_{i=1}^4 L\pth{\half+u_i+2s,\chi_\mathfrak{m}}
\]
and the function $Z_2$ is analytic and uniformly bounded in the region $\Re(s) \geq \frac{3}{5}$, $\Re(u_i) \geq -\frac{1}{20}$. This implies that $Z_2^*$ is analytic and uniformly bounded in the same region. 

Using \eqref{eq:logSumProd}, we see that in the zero free region for $\zeta$, 
\begin{align}
&Y(\omega_1+i(\mu_1+t),\omega_2+i(\mu_2+t),\omega_3+i(\mu_3-t),\omega_4+i(\mu_4-t),s,k_1)\nonumber\\
&\ll \prod_{1\leq i\leq j \leq 4} \log(2+\abs{\omega_i+\omega_j + \mu_i+\mu_j} + \abs{t})\nonumber \\
&\ll \log^{10}(2+\abs{t}) \sumpth{\prod_{i=1}^4 \log(2+\abs{\omega_i}) \log(2+\abs{\mu_i})}^4.
\label{eq:YZFbound}
\end{align}
Here the implied constant is absolute, and we may assume $c$ is large enough to force $\Re(\omega_i) \geq -\frac{1}{20}$ up to a suitable height in the zero-free region. 

Let $T = \exp(\sqrt{\log (R_1R_2R_3R_4)})$. If $\abs{\Im(\omega_i)} > T$ for any $i$, the bound $\tilde{G}(\omega_i) \ll (1+\abs{\omega_i})^{-A}$ implies that this range contributes to \eqref{eq:CsumMellin} a quantity of order
\[
\ll \log^{10}(2+\abs{t}) \sumpth{\prod_{i=1}^4 \log(2+\abs{\mu_i}) }^4 \exp(-\sqrt{\log (R_1R_2R_3R_4)}).
\]

Now suppose that $\abs{\Im(\omega_i)} \leq T$ for all $i$.  If $\max(\abs{t},\abs{\mu_1},\abs{\mu_2},\abs{\mu_3},\abs{\mu_4}) \leq T$, then we move the contours of integration to $\Re(\omega_i) = -\frac{c'}{\log T} \geq -\frac{1}{20}$ for some $1\geq c' > 0$ so that \eqref{eq:YZFbound} is satisfied (note $\log T > 1$ since $\max(R_i) > e$). The contribution of this is 
\[
\ll \log^{10}(2+\abs{t}) \sumpth{\prod_{i=1}^4 \log(2+\abs{\mu_i})}^4 \exp\pth{-\frac{c'}{2}\sqrt{\log (R_1R_2R_3R_4)}},
\]
where we have used $\tilde{G}(\omega_i) \ll (1+\abs{\omega_i})^{-A}$ to estimate the contribution of the horizontal segments.

Finally, if $\max(\abs{t},\abs{\mu_1},\abs{\mu_2},\abs{\mu_3},\abs{\mu_4}) > T$, we do not shift contours. The contribution of this case to  \eqref{eq:CsumMellin} is
\begin{align*}
&\ll \log^6(2+\abs{t}) \sumpth{\prod_{i=1}^4 \log(2+\abs{\mu_i})}^4 \\
&\ll (1+\abs{t})^\frac{1}{5}\sumpth{\prod_{i=1}^4 (1+\abs{\mu_i})} \exp\pth{-\frac{1}{6}\sqrt{\log (R_1R_2R_3R_4)}}.
\end{align*}
\end{proof}
Using $\Gamma(s) \cos(\frac{\pi s}{2}) \ll \abs{s}^{\Re(s)-1/2}$ and Lemma \ref{lem:CsumBound}, the quantity in \eqref{eq:OffDiagSeparated} is bounded by
\begin{align}
&X^\frac{2}{5}N^\frac{2}{5} (1+\abs{t})^\frac{1}{5} \sumd_{R_1,\ldots,R_4} \exp\pth{-c_1\sqrt{\log(R_1R_2R_3R_4)}}  \int\limits_{-\infty}^\infty \cdots   \int\limits_{-\infty}^\infty \sumpth{\prod_{i=1}^4 \frac{1}{1+\abs{\mu_i}}}^{10}\nonumber \\
&\times \sumstar_{\substack{\abs{k_1} \leq K\\ k_1\neq 0,1}} \frac{1}{\abs{k_1}^\frac{3}{5}} \prod_{i=1}^4\sumabs{\sum_{n_i} \frac{\chi_\mathfrak{m}(n_i)}{n_i^{\frac{1}{2}+i\mu_i \pm it}} G\fracp{n_i}{N_i}}\, d\boldmu.\label{eq:AlmostDone}
\end{align}
We split the sum in $k_1$ into dyadic intervals of the form $\colk \leq \abs{k_1} < 2\colk$. Since $N_i \leq 2N$ and 
\begin{equation}\label{eq:InductionRequirement}
4\colk \leq 4K = \frac{N^4}{X} \leq \frac{M}{\coll} < M-1
\end{equation}
by \eqref{equ:defK} and Lemma \ref{lem:UsefulEstimates}, we may apply the induction hypothesis in the form of Lemma \ref{lem:InductionAllN} to see that
\begin{align*}
\sumstar_{\colk \leq \abs{k_1} < 2\colk} \sumabs{\sum_{n} \frac{\chi_{\mathfrak{m}}(n)}{n^{\frac{1}{2}+ix+ i  t}} G\fracp{n}{\coln}}^4 &\leq \sumflat_{\colk \leq \abs{m} < 8\colk} \sumabs{\sum_{n} \frac{\chi_m(n)}{n^{\frac{1}{2}+ix+ i  t}} G\fracp{n}{\coln}}^4 \\
&\ll \coll^{\frac{2}{3}} (1+\abs{t+x})^{3} \colk\log^6(2+\colk) \\
&\ll \coll^\frac{2}{3} (1+\abs{t})^{3}(1+\abs{x})^{3} \colk\log^6(2+\colk). 
\end{align*} 
By H\"{o}lder's inequality and dyadic summation over $\colk$, we have
\begin{align*}
\sumstar_{\substack{\abs{k_1} \leq K\\ k_1\neq 0,1}} \frac{1}{\abs{k_1}^\frac{3}{5}} \prod_{i=1}^4\sumabs{\sum_{n_i} \frac{\chi_\mathfrak{m}(n_i)}{n_i^{\frac{1}{2}+i\mu_i \pm it}} G\fracp{n_i}{N_i}} &\ll \sumd_{\colk\leq K} \frac{1}{\colk^\frac{3}{5}} \prod_{i=1}^4 \sumpth{\sumstar_{\colk\leq \abs{k_1} < 2\colk} \sumabs{\sum_{n_i} \frac{\chi_m(n_i)}{n_i^{\frac{1}{2}+i\mu_i \pm it}} G\fracp{n_i}{N_i}}^4}^\frac{1}{4} \\
&\ll \coll^\frac{2}{3} (1+\abs{t})^{3} K^\frac{2}{5}\log^6(2+K) \prod_{i=1}^4 (1+\abs{\mu_i})^2.
\end{align*}
Substituting this into \eqref{eq:AlmostDone} and recalling that $K = \frac{N^4}{4X}$ and  Lemma \ref{lem:UsefulEstimates},  we obtain a bound of
\[
\ll \coll^\frac{2}{3} (1+\abs{t})^{3+\frac{1}{5}} N^2 \log^6(2+M).
\]
Here we have estimated
\begin{align*}
\sumd_{R_1,\ldots,R_4} \exp(-c_1\sqrt{\log(R_1R_2R_3R_4)}) &= \sum_{\ell_1,\ldots,\ell_4\geq 0} \exp(-c_1\sqrt{\log 2}\sqrt{\ell_1+\ell_2+\ell_3+\ell_4}) \\
&= \sum_{h\geq 0} \binom{h+3}{3} \exp(-c_1\sqrt{\log 2}\sqrt{h}) \ll 1.
\end{align*}
This concludes the proof of Lemma \ref{lem:LS-NonSquare}.

\section{Proof of Theorem \ref{thm:Main}: Initial Steps}\label{sec:Asymptotic-Shortening}
In the rest of the paper, we prove Theorem \ref{thm:Main}. In this section, we shorten the length of the Dirichlet polynomial of the fourth power of quadratic Dirichlet $L$-functions from $X^2$ to $X^2/(\log X)^C$ for some large $C>0$. The key idea is to use  techniques of Soundararajan-Young (see Page 1110 of \cite{SoundYoungSecond}) and Petrow  (see Page 1585 of \cite{Petrow}), combined with the large sieve inequality given in Proposition \ref{prop:LargeSieve}.

Define
\begin{align}
U:= {(\log X)^{-Q}}
\label{defU}
\end{align}
for some large constant $Q\geq 1$ to be chosen later. We remark that our use of the symbol $X$ here is unrelated to the previous use of this symbol in \eqref{def-X}.  Recalling Lemma \ref{afe1}, we write, for $ i=1,2,3,4$,
\begin{align}
M^+(\a_i)&: = \sum_{n=1}^\infty \frac{\chi_{8d} (n)}{n^{\frac{1}{2} + \a_i}} V_{\a_i}\left( \frac{n}{\sqrt{8d}}\right),\nonumber\\
M^-(\a_i)& := \colx_{\a_i}  \sum_{n=1}^\infty \frac{\chi_{8d} (n)}{n^{\frac{1}{2} - \a_i}} V_{-\a_i}\left( \frac{n}{\sqrt{8d}}\right),\nonumber\\
N^+(\a_i)&: = \sum_{n=1}^\infty \frac{\chi_{8d} (n)}{n^{\frac{1}{2} + \a_i}} V_{\a_i}\left( \frac{n}{\sqrt{8dU}}\right),\nonumber\\
N^-(\a_i)& := U^{-\a_i}\colx_{\a_i}  \sum_{n=1}^\infty \frac{\chi_{8d} (n)}{n^{\frac{1}{2} - \a_i}} V_{-\a_i}\left( \frac{n}{\sqrt{8dU}}\right),
\label{defMN}
\end{align}
where $\colx_{\a_i}$ is given by \eqref{colxDef} and we define
\begin{align}
E^+(\a_i) &:=  M^+(\a_i) - N^+(\a_i),\nonumber\\
E^-(\a_i) &:=  M^-(\a_i) - N^-(\a_i). \label{defR}
\end{align}
For convenience, we also write 
\begin{align}
\label{def-lambda}
\lambda_i(\a_i) := \left(\frac{1}{\pi}\right)^{-\a_i} \frac{\Gamma(\frac{1/2-\a_i}{2})}{\Gamma(\frac{1/2+\a_i}{2})} = \colx_{\a_i} \cdot   (8d)^{\a_i}.
\end{align}
We abuse notation and also write $\lambda_i:= \lambda_i(\a_i)$, where we note that $\lambda_i(- \a_i) = \lambda_i^{-1}$. With the above notation, we have 
\begin{align}
&\sumstar_{(d,2)=1} L(\tfrac{1}{2} + \a_1,\chi_{8d})L(\tfrac{1}{2} + \a_2, \chi_{8d} )L(\tfrac{1}{2} + \a_3, \chi_{8d} )L(\tfrac{1}{2} + \a_4, \chi_{8d}) \Phi(\tfrac{d}{X})  \nonumber\\
&= \sumstar_{(d,2)=1} \Phi\left(\frac{d}{X} \right) \prod_{i=1}^4 \left( M^+(\a_i ) + M^-(\a_i) \right)  \nonumber\\
&= \sumstar_{(d,2)=1} \Phi\left(\frac{d}{X} \right) \prod_{i=1}^4 \left( E^+(\a_i ) +N^+(\a_i ) + E^-(\a_i) + N^-(\a_i) \right) \nonumber \\
&=\sumstar_{(d,2)=1} \Phi\left(\frac{d}{X} \right) \prod_{i=1}^4 \left( E^+(\a_i) +E^-(\a_i) \right) \nonumber\\
&\quad+\sumstar_{(d,2)=1}\Phi\left(\frac{d}{X} \right) \sum_{\substack{A \subsetneqq \{1,2,3,4\} }} \prod_{i \in A}  \left( E^+(\a_i) +E^-(\a_i) \right) \prod_{i \in \bar{A}} \left( N^+(\a_i) +N^-(\a_i) \right) \nonumber \\
&= \sumstar_{(d,2)=1}\Phi\left(\frac{d}{X} \right) \prod_{i=1}^4 \left( E^+(\a_i) +E^-(\a_i) \right) \nonumber\\
&\quad+\sumstar_{(d,2)=1} \Phi\left(\frac{d}{X} \right) \sum_{\substack{A \subsetneqq \{1,2,3,4\} }} \prod_{i \in A}  \left( M^+(\a_i) -N^+(\a_i) +M^-(\a_i)-N^-(\a_i) \right) \nonumber\\
&\quad\times \prod_{i \in \bar{A}} \left( N^+(\a_i) +N^-(\a_i) \right).
\label{Combin}
\end{align}

\begin{lem}
\label{lem:4therror}
With notation as above, we have 
\begin{align*}
\sumstar_{(d,2)=1} \Phi\left(\frac{d}{X} \right) \prod_{i=1}^4 \left( E^+(\a_i) +E^-(\a_i) \right)\ll X (\log X)^{6+ \varepsilon}.
\end{align*}
\end{lem} 
\begin{proof}
The product above can be expanded into a sum of 16 terms. Of these terms, we evaluate only the one corresponding to
$E^+(\a_1)E^+(\a_2)E^-(\a_3)E^-(\a_4) $, as the other terms may be evaluated similarly. Note that by \eqref{defMN} and \eqref{defR},  
\begin{align}
E^+(\a_i)& = \frac{1}{2\pi  i }  \int\limits_{(1)}g_{\a_i}(u_i) \sum_{n_i=1}^\infty
\frac{\chi_{8d}(n_i)}{n_i^{\frac{1}{2} + \a_i +u_i}}  \frac{{(8d)}^{\frac{u_i}{2}} - (8dU)^{\frac{u_i}{2}}}{u_i} \, du_i\nonumber\\
&=  \sumd_{R_i} \frac{1}{2\pi  i }  \int\limits_{(1)}g_{\a_i}(u_i) \sum_{n_i=1}^\infty
\frac{\chi_{8d}(n_i)}{n_i^{\frac{1}{2} + \a_i +u_i}}  \frac{{(8d)}^{\frac{u_i}{2}} - (8dU)^{\frac{u_i}{2}}}{u_i} G \left(\frac{n_i}{R_i} \right)\, du_i\nonumber\\
&=  \sumd_{R_i} \frac{1}{(2\pi  i )^2} \int\limits_{(1)} \int\limits_{(1)}g_{\a_i}(u_i) \sum_{n_i=1}^\infty
\frac{\chi_{8d}(n_i)}{n_i^{\frac{1}{2} + \omega_i}} V\left(\frac{n_i}{R_i} \right)R_i^{\omega_i-u_i-\a_i} \tilde{G}(\omega_i-u_i-\a_i)\nonumber\\
&\quad \times \frac{{(8d)}^{\frac{u_i}{2}} - (8dU)^{\frac{u_i}{2}}}{u_i}  \, d\omega_i \, du_i
,\nonumber\\
E^-(\a_i) &=\lambda_i(\a_i) \frac{1}{2\pi  i }  \int\limits_{(1)}g_{-\a_i}(u_i) \sum_{n_i=1}^\infty
\frac{\chi_{8d}(n_i)}{n_i^{\frac{1}{2} - \a_i +u_i}}  \frac{{(8d)}^{\frac{u_i}{2} -\a_i} - (8dU)^{\frac{u_i}{2}-\a_i}}{u_i} \, du_i\nonumber\\
&=
\sumd_{R_i} \lambda_i(\a_i) \frac{1}{(2\pi i)^2}  \int\limits_{(1)}\int\limits_{(1)}g_{-\a_i}(u_i) \sum_{n_i=1}^\infty
\frac{\chi_{8d}(n_i)}{n_i^{\frac{1}{2} +\omega_i}} 
V\left(\frac{n_i}{R_i} \right)R_i^{\omega_i-u_i+\a_i} \tilde{G}(\omega_i-u_i+\a_i)\nonumber\\
&\quad \times 
\frac{{(8d)}^{\frac{u_i}{2} -\a_i} - (8dU)^{\frac{u_i}{2}-\a_i}}{u_i} \, d \omega_i \, du_i,
\label{equ:addVG}
\end{align}
where $g_{\a}$ is given by \eqref{gaDef},  $G$ is the smooth function given in Lemma \ref{lem:ConvenientTestFandG}, and $V$ is the same as in \eqref{V-def}.

Note that  $\frac{{(8d)}^{\frac{u_i}{2} -\a_i} - (8dU)^{\frac{u_i}{2}-\a_i}}{u_i}$ in $E^-(\a_i)$ has a pole at $u_i=0$. To fix this,  let 
\begin{align*}
E^-(\a_i) = E^-_{\text{I}}(\a_i) + E^-_{\text{II}}(\a_i),
\end{align*}
where 
\begin{align*}
E^-_{\text{I}}(\a_i)& = 
\sumd_{R_i} \lambda_i(\a_i) \frac{1}{(2\pi i)^2}  \int\limits_{(1)}\int\limits_{(1)}g_{-\a_i}(u_i) \sum_{n_i=1}^\infty
\frac{\chi_{8d}(n_i)}{n_i^{\frac{1}{2} +\omega_i}} 
V\left(\frac{n_i}{R_i} \right)R_i^{\omega_i-u_i+\a_i} \tilde{G}(\omega_i-u_i+\a_i)\\
&\quad \times 
\frac{{(8d)}^{\frac{u_i}{2} -\a_i} -  (8dU)^{\frac{u_i}{2}} (8d)^{-\a_i} }{u_i} \, d\omega_i \, du_i,\\
E^-_{\text{II}}(\a_i)& = 
\sumd_{R_i} \lambda_i(\a_i) \frac{1}{(2\pi i)^2}  \int\limits_{(1)}\int\limits_{(1)}g_{-\a_i}(u_i) \sum_{n_i=1}^\infty
\frac{\chi_{8d}(n_i)}{n_i^{\frac{1}{2} +\omega_i}} 
V\left(\frac{n_i}{R_i} \right)R_i^{\omega_i-u_i+\a_i} \tilde{G}(\omega_i-u_i+\a_i)\\
&\quad \times 
\frac{ (8dU)^{\frac{u_i}{2}} (8d)^{-\a_i} - (8dU)^{\frac{u_i}{2}-\a_i}}{u_i} \, d\omega_i \, du_i.
\end{align*}
Note that $\frac{{(8d)}^{\frac{u_i}{2} -\a_i} -  (8dU)^{\frac{u_i}{2}} (8d)^{-\a_i} }{u_i}$ in $E^-_{\text{I}}(\a_i)$ is entire, and for $E^-_{\text{II}}(\a_i)$, we also have 
\begin{align}
(8d)^{-\a_i} - (8dU)^{-\a_i} = -(8d)^{-\a_i}(U^{-\a_i} -1) = -(8d)^{-\a_i} (e^{-\a_i \log U } -1) \ll \frac{\log \log X}{\log X}.
\label{equ:upperbdwithpole}
\end{align}
We evaluate only the following case, as the other cases are similar. We have
\begin{align}
&\sumstar_{(d,2)=1} \Phi\left(\frac{d}{X} \right)E^+(\a_1)E^+(\a_2)E^-_{\text{I}}(\a_3)E^-_{\text{II}}(\a_4) \nonumber\\
&=
\sumd_{R_1,R_2,R_3,R_4}
\sumstar_{(d,2)=1} \Phi\left(\frac{d}{X} \right) \frac{\lambda_3(\a_3)\lambda_4(\a_4)}{(2\pi i)^4}  \int\limits_{(1)}  \cdots  \int\limits_{(1)}g_{\a_1}(u_1)g_{\a_2}(u_2)g_{-\a_3}(u_3) g_{-\a_4}(u_4) \nonumber\\
&\quad\times  \frac{1}{(2\pi i)^4}  \int\limits_{(1)} \cdots \int\limits_{(1)}
\sum_{n_1=1}^\infty
\frac{\chi_{8d}(n_1)}{n_1^{\frac{1}{2} + \omega_1}} V\left(\frac{n_1}{R_1} \right)R_1^{\omega_1-u_1-\a_1}
\sum_{n_2=1}^\infty
\frac{\chi_{8d}(n_2)}{n_2^{\frac{1}{2} + \omega_2}} V\left(\frac{n_2}{R_2} \right)R_2^{\omega_2-u_2-\a_2} \nonumber\\
&\quad\times 
\sum_{n_3=1}^\infty
\frac{\chi_{8d}(n_3)}{n_3^{\frac{1}{2} + \omega_3}} V\left(\frac{n_3}{R_3} \right)R_3^{\omega_3-u_3+\a_3}
\sum_{n_4=1}^\infty
\frac{\chi_{8d}(n_4)}{n_4^{\frac{1}{2} + \omega_4}} V\left(\frac{n_4}{R_4} \right)R_4^{\omega_4-u_4+\a_4}\nonumber\\
&\quad\times   \tilde{G}(\omega_1-u_1-\a_1) \tilde{G}(\omega_2-u_2-\a_2)  \tilde{G}(\omega_3-u_3+\a_3) \tilde{G}(\omega_3-u_4+\a_4)\nonumber\\
&\quad \times 
\frac{(8d)^{\frac{u_1}{2}} - (8dU)^{\frac{u_1}{2}}}{u_1} \frac{(8d)^{\frac{u_2}{2}} - (8dU)^{\frac{u_2}{2}}}{u_2}\nonumber\\
&\quad\times\frac{{(8d)}^{\frac{u_3}{2} -\a_3} -  (8dU)^{\frac{u_3}{2}} (8d)^{-\a_3} }{u_3}  \frac{ (8dU)^{\frac{u_4}{2}} (8d)^{-\a_4} - (8dU)^{\frac{u_4}{2}-\a_4}}{u_4} \, d\boldomega \, d \boldu.
\label{errorR}
\end{align}
We first move the lines of integration in $\omega_i$ to $\Re(\omega_1) = \Re(\omega_2) =\Re(\omega_3)=\Re(\omega_4) = 0$. We then split each sum  over $R_i$, $i=1,2,3,4$ into three pieces depending as $R_i \leq \sqrt{XU}$, $\sqrt{XU} <R_i \leq \sqrt{X}$ and $R_i>\sqrt{X}$. For $i=1,2,3$, we move the line of integration in $u_i$ to
\[
\Re(u_i) = \begin{cases}
-1 & \text{if $R_i \leq \sqrt{XU}$}, \\
0 & \text{if $\sqrt{XU} <R_i \leq \sqrt{X}$}, \\
4 & \text{if $R_i>\sqrt{X}$},
\end{cases}
\]
and when $i=4$, we take $u_4$ to  $\Re(u_4) = \frac{10}{\log X}$ in each range, since $ \frac{ (8dU)^{\frac{u_4}{2}} (8d)^{-\a_4} - (8dU)^{\frac{u_4}{2}-\a_4}}{u_4}$ has a pole at $s=0$.

Shifting contours in this way, we encounter no poles. Note that there are 81 different pieces after the decomposition applied to  \eqref{errorR}. We only evaluate one of them here, which is $R_1 \leq \sqrt{XU}$, $\sqrt{XU}<R_2 \leq \sqrt{X}$, $R_3 > \sqrt{X}$ and $R_4 \leq \sqrt{XU}$, and the other cases can be estimated in a similar way. Write $L = \log X$. We see that 
\begin{align}
&
\sumd_{R_1\leq \sqrt{XU}}\sumd_{\sqrt{XU} <R_2 \leq \sqrt{X}}\sumd_{R_3 > \sqrt{X}}\sumd_{R_4\leq \sqrt{XU}}
\sumstar_{(d,2)=1} \Phi\left(\frac{d}{X} \right) \frac{\lambda_3(\a_3) \lambda_4(\a_4)}{(2\pi i)^4}  \int\limits_{(\frac{10}{L})}  \int\limits_{(4)}  \int\limits_{(0)}  \int\limits_{(-1)} \nonumber\\
&\quad\times g_{\a_1}(u_1)g_{\a_2}(u_2)g_{-\a_3}(u_3) g_{-\a_4}(u_4)  \nonumber\\
&\quad\times  \frac{1}{(2\pi i)^4}  \int\limits_{(0)} \cdots  \int\limits_{(0)}
\sum_{n_1=1}^\infty
\frac{\chi_{8d}(n_1)}{n_1^{\frac{1}{2} + \omega_1}} V\left(\frac{n_1}{R_1} \right)R_1^{\omega_1-u_1-\a_1}
\sum_{n_2=1}^\infty
\frac{\chi_{8d}(n_2)}{n_2^{\frac{1}{2} + \omega_2}} V\left(\frac{n_2}{R_2} \right)R_2^{\omega_2-u_2-\a_2}  \nonumber\\
&\quad\times 
\sum_{n_3=1}^\infty
\frac{\chi_{8d}(n_3)}{n_3^{\frac{1}{2} + \omega_3}} V\left(\frac{n_3}{R_3} \right)R_3^{\omega_3-u_3+\a_3}
\sum_{n_4=1}^\infty
\frac{\chi_{8d}(n_4)}{n_4^{\frac{1}{2} + \omega_4}} V\left(\frac{n_4}{R_4} \right)R_4^{\omega_4-u_4+\a_4}  \nonumber\\
&\quad\times  \tilde{G}(\omega_1-u_1-\a_1) \tilde{G}(\omega_2-u_2-\a_2)  \tilde{G}(\omega_3-u_3+\a_3) \tilde{G}(\omega_3-u_4+\a_4) \nonumber\\
&\quad \times 
\frac{(8d)^{\frac{u_1}{2}} - (8dU)^{\frac{u_1}{2}}}{u_1} \frac{(8d)^{\frac{u_2}{2}} - (8dU)^{\frac{u_2}{2}}}{u_2}r \nonumber \\
&\quad\times\frac{{(8d)}^{\frac{u_3}{2} -\a_3} -  (8dU)^{\frac{u_3}{2}} (8d)^{-\a_3} }{u_3}  \frac{ (8dU)^{\frac{u_4}{2}} (8d)^{-\a_4} - (8dU)^{\frac{u_4}{2}-\a_4}}{u_4} \, d\boldomega \, d \boldu \nonumber\\
&\ll  (XU)^{-\frac{1}{2}} (\log\log X) X^2 \frac{\log \log X}{\log X} \nonumber\\
&\quad \times
\sumd_{R_1\leq \sqrt{XU}} R_1^{1-\a_1} \sumd_{\sqrt{XU} <R_2 \leq \sqrt{X}} R_2^{-\a_2} \sumd_{ R_3 > \sqrt{X}} R_3^{-4+\a_3} \sumd_{R_4 \leq \sqrt{XU}}R_4^{-\frac{10}{L}+\a_4} \nonumber \\
&\quad\times   \int\limits_{(\frac{10}{L})}  \int\limits_{(4)}  \int\limits_{(0)}  \int\limits_{(-1)}\left| g_{\a_1}(u_1)g_{\a_2}(u_2)g_{-\a_3}(u_3) g_{-\a_4}(u_4)\right| \nonumber \\
&\quad\times   \sumstar_{(d,2)=1} \Phi\left(\frac{d}{X} \right)
\sumabs{\sum_{n_1=1}^\infty
\frac{\chi_{8d}(n_1)}{n_1^{\frac{1}{2} + \omega_1 }} V\left(\frac{n_1}{R_1} \right )} 
\sumabs{\sum_{n_2=1}^\infty
\frac{\chi_{8d}(n_2)}{n_2^{\frac{1}{2} + \omega_2 }} V\left(\frac{n_2}{R_2} \right)}  \nonumber \\
&\quad\times 
\sumabs{\sum_{n_3=1}^\infty
\frac{\chi_{8d}(n_3)}{n_3^{\frac{1}{2} + \omega_3 }} V\left(\frac{n_3}{R_3} \right)}
\sumabs{\sum_{n_4=1}^\infty
\frac{\chi_{8d}(n_4)}{n_4^{\frac{1}{2} + \omega_4 }} V\left(\frac{n_4}{R_4} \right)} \nonumber\\
&\quad\times \left| \tilde{G}(\omega_1 -u_1-\a_1)
\tilde{G}(\omega_2  -u_2-\a_2)
\tilde{G}(\omega_3   -u_3+\a_3)
\tilde{G}(\omega_4   -u_4+\a_4) \right| \frac{1}{|u_4 |}
\, d \boldomega \, d \boldu.
\label{eqsplit}
\end{align}
To establish the above inequality, we have used \eqref{equ:upperbdwithpole}
and the fact that $\frac{(8d)^{it} - (8dU)^{it}}{it}\ll \log (\frac{8d}{8dU} ) \ll \log \log X$.
By Cauchy-Schwarz inequality, \eqref{V-def}, Proposition \ref{prop:LargeSieve} and Lemma \ref{lem:InductionAllN}, we have
\begin{align}
\sumstar_{(d,2)=1}  \Phi\left(\frac{d}{X} \right) \prod_{i=1}^4
\sumabs{\sum_{n_i=1}^\infty
\frac{\chi_{8d}(n_i)}{n_i^{\frac{1}{2} + \omega_i }} V\left(\frac{n_i}{R_i} \right )} 
&\ll\prod_{i=1}^4 \sumpth{\sumstar_{(d,2)=1} \Phi\left(\frac{d}{X} \right)\sumabs{\sum_{n_i=1}^\infty
\frac{\chi_{8d}(n_i)}{n_i^{\frac{1}{2} + \omega_i }} V\left(\frac{n_i}{R_i} \right ) } ^4}^{\frac{1}{4}}\nonumber\\
&\ll  X \log^6 (2+X) \prod_{1 \leq i \leq 4} (1+|\omega_i|)^{3}  .
\label{error-4th-case}
\end{align}
This combined with the bound 
\[
\int_{(\frac{10}{L})} \frac{1}{|u_4|}  \frac{1}{(1+|u_4|)^{10}} \, d u_4 \ll \log \log X ,
\]
imply that  \eqref{eqsplit} is bounded by 
\begin{align*}
&\ll (XU)^{-\frac{1}{2}} (\log\log X)^2 X^2 \frac{\log \log X}{\log X} \\
&\quad \times 
\sumd_{R_1\leq \sqrt{XU}} R_1^{1-\a_1} \sumd_{\sqrt{XU} <R_2 \leq \sqrt{X}} R_2^{-\a_2} \sumd_{ R_3 > \sqrt{X}} R_3^{-4+\a_3} \sumd_{R_4\leq \sqrt{XU}}R_4^{-\frac{10}{L}+\a_4}
\cdot X \log^6 (2+X) \\
&\ll  (XU)^{-\frac{1}{2}} (\log\log X)^2 X^2 \frac{\log \log X}{\log X} \cdot  (XU)^{\frac{1}{2}}(\log \log X) X^{-2} \log  X \cdot
X(\log X)^{6+ \varepsilon} \\
&\ll X (\log X)^{6+ \varepsilon},
\end{align*}
as claimed. Here we have used the fact $\sum_{\sqrt{XU} <R_2 \leq \sqrt{X}} \ll \log \log X$ and $\sum_{R_4\leq \sqrt{XU}}R_4^{-\frac{10}{L}+\a_4} \ll \log X$. We remark that since  $\sum_{R_4> \sqrt{X}}R_4^{-\frac{10}{L}+\a_4} \ll \log X$, the argument above also works for the piece $\sum_{R_4> \sqrt{X}}$.
\end{proof}

By \eqref{Combin} and Lemma 7.1, in order to prove Theorem \ref{thm:Main}, it suffices to derive an asymptotic formula for 
\begin{align*}
\sumstar_{(d,2)=1} \Phi\left(\frac{d}{X} \right) \sum_{A \subsetneqq \{1,2,3,4\}} &\prod_{i \in A}  \left( M^+(\a_i) -N^+(\a_i) +M^-(\a_i)-N^-(\a_i) \right) \\
&\quad\times\prod_{i \in \bar{A}} \left( N^+(\a_i) +N^-(\a_i) \right).
\end{align*}
We first put the above expression into a more manageable form. When $A = \emptyset$, then the first product above is $1$. If $A \neq \emptyset$, we write
\begin{align*}
\prod_{i \in A}  \left( M^+(\a_i) -N^+(\a_i) +M^-(\a_i)-N^-(\a_i) \right)
& =
\sum_{\substack{D_i=M \text{\,or\,}N \\ \Delta_i = \pm\\  i \in A}} \prod_{i\in A} (-1)^{h(D_i)} D_i^{\Delta_i}(\alpha_i),\\
\prod_{i \in \bar{A}}   \left( N^+(\a_i) +N^-(\a_i) \right)
&= \sum_{\substack{\Delta_i = \pm\\ i \in  \bar{A}}} \prod_{i \in \bar{A}} N^{\Delta_i} (\alpha_i),
\end{align*}
where 
\[
h(D) = \begin{cases}
0 & \text{if $D=M$}, \\
1 & \text{if $D=N$}.
\end{cases}
\]
We remark that $h(D)$ only depends on the symbols $M,N$ and, in particular, is independent of $d$. Thus
\begin{align}
&\sum_{\substack{A \subsetneqq \{1,2,3,4\} \\ A \neq \emptyset}} \prod_{i \in A}  \left( M^+(\a_i) -N^+(\a_i) +M^-(\a_i)-N^-(\a_i) \right) \prod_{i \in \bar{A}} \left( N^+(\a_i) +N^-(\a_i) \right)\nonumber\\
&=\sum_{\substack{A \subsetneqq \{1,2,3,4\} \\ A \neq \emptyset}}
\sum_{\substack{D_i=M \text{\,or\,}N\\i \in A }} \sum_{\substack{\Delta_i = \pm\\i=1,2,3,4}}
\prod_{i\in A} (-1)^{h(D_i)} D_i^{\Delta_i}(\alpha_i) \prod_{i \in \bar{A}} N^{\Delta_i} (\alpha_i).
\label{general-1}
\end{align}
We would like to write the expression above in terms of only one of the symbols $+,-$ in the exponents of $D_i$ and $N$. To do so, we note that by \eqref{defMN}, we have $D_i^-(\a_i) = U^{-\a_i\delta(D_i=N)} \colx_{\a_i}  D_i^+(-\a_i)$, and so
\[
D_i^{\Delta_i}(\alpha_i) = \pth{U^{-\a_i\delta(D_i=N)} \colx_{\a_i}}^{\frac{1-\epsilon_i}{2}}  D^+(\epsilon_i\a_i), 
\]
where
\[
\epsilon_i = \begin{cases}
1 & \text{if $\Delta_i=+$}, \\
-1 & \text{if $\Delta_i = -$}.
\end{cases}
\]
Inserting this into \eqref{general-1}, the sum over $\Delta_i\in\set{-,+}$ is changed to a sum over $\epsilon_i\in\set{-1,1}$ and we have
\begin{align*}
&\sum_{\substack{\Delta_i = \pm\\i=1,2,3,4}}
\prod_{i\in A} (-1)^{h(D_i)} D_i^{\Delta_i}(\alpha_i) \prod_{i \in \bar{A}} N^{\Delta_i} (\alpha_i) \\
&= \sum_{\substack{\epsilon_i = \pm 1\\i=1,2,3,4}} 
\prod_{i\in A} (-1)^{h(D_i)} \pth{U^{-\a_i\delta(D_i=N)}\colx_{\alpha_i}}^{\frac{1-\epsilon_i}{2}}D_i^+(\epsilon_i \alpha_i) \prod_{i \in \bar{A}} \pth{U^{-\a_i}\colx_{\alpha_i}}^{\frac{1-\epsilon_i}{2}}N^+(\epsilon_i\alpha_i).
\end{align*}
From \eqref{def-lambda}, we have the identity
\[
\colx_{\a_i} = (8X)^{-\alpha_i}\lambda_i(\alpha_i) \fracp{d}{X}^{-\alpha_i}.
\]
Letting
\begin{align}
\Phi_{\{\epsilon_1,\epsilon_2,\epsilon_3,\epsilon_4\}} (x) := x^{-\sum_{i=1}^4 \frac{1-\epsilon_i}{2}\a_i} \Phi(x),
\label{tranf-phi}
\end{align}
the previous three displays imply
\begin{align*}
&\sumstar_{(d,2)=1} \sum_{\substack{\Delta_i = \pm\\i=1,2,3,4}}
\prod_{i\in A} (-1)^{h(D_i)} D_i^{\Delta_i}(\alpha_i) \prod_{i \in \bar{A}} N^{\Delta_i} (\alpha_i) \Phi\left(\frac{d}{X} \right)\nonumber\\
&= \sum_{\substack{\epsilon_i = \pm 1\\i=1,2,3,4}}
(8X)^{-\sum_{i=1}^4\frac{1-\epsilon_i}{2}\a_i}
\sumstar_{(d,2)=1} \Phi_{\{ \epsilon_1,\epsilon_2,\epsilon_3,\epsilon_4\}}\left(\frac{d}{X} \right) \\
&\quad \times \prod_{i\in A} (-1)^{h(D_i)} U^{-\frac{1-\epsilon_i}{2}\a_i\delta(D_i=N)} \lambda_i^{\frac{1-\epsilon_i}{2}}   D_i^+(\epsilon_i\alpha_i) \prod_{i \in \bar{A}} U^{-\frac{1-\epsilon_i}{2}\a_i} \lambda_i^{\frac{1-\epsilon_i}{2}} N^+ (\epsilon_i\alpha_i) \nonumber \\
&= \sum_{\substack{\epsilon_i = \pm 1\\i=1,2,3,4}} 
\prod_{i=1}^4
\pth{(8X)^{-\a_i}\lambda_i}^{\frac{1-\epsilon_i}{2}} \prod_{i\in A} (-1)^{h(D_i)} \prod_{i \in A} U^{-\frac{1-\epsilon_i}{2}\a_i\delta(D_i=N)} \prod_{i \in \bar{A}} U^{-\frac{1-\epsilon_i}{2}\a_i} \nonumber\\
&\quad\times 
\sumstar_{(d,2)=1} 
\prod_{i\in A} D_i^+(\epsilon_i\alpha_i) \prod_{i \in \bar{A}}   N^+ (\epsilon_i\alpha_i) \Phi_{\{\epsilon_1,\epsilon_2,\epsilon_3,\epsilon_4\}}\left(\frac{d}{X} \right).
\end{align*}
Combining this with \eqref{general-1}, we find that
\begin{align}
&\sumstar_{(d,2)=1} \Phi\left(\frac{d}{X} \right) \sum_{\substack{A \subsetneqq \{1,2,3,4\}\\ A\neq \emptyset}}   \prod_{i \in A}  \left( M^+(\a_i) -N^+(\a_i) +M^-(\a_i)-N^-(\a_i) \right)\nonumber\\
&\quad\times\prod_{i \in \bar{A}} \left( N^+(\a_i) +N^-(\a_i) \right)\nonumber\\
&=  \sum_{\substack{A \subsetneqq \{1,2,3,4\}\\ A\neq \emptyset}} \sum_{\substack{D_i=M \text{\,or\,}N\\i \in A }} \sum_{\substack{\epsilon_i = \pm 1\\i=1,2,3,4}} \prod_{i=1}^4
\pth{(8X)^{-\a_i}\lambda_i}^{\frac{1-\epsilon_i}{2}}  \prod_{i\in A} (-1)^{h(D_i)}  \prod_{i \in A}{U^{-\a_i\delta(D_i=N)}}^{\frac{1-\epsilon_i}{2}} \prod_{i \in \bar{A}} U^{-\frac{1-\epsilon_i}{2}\a_i}\nonumber \\
&\quad \times\sumstar_{(d,2)=1} 
\prod_{i\in A}  D_i^+(\epsilon_i\alpha_i) \prod_{i \in \bar{A}}   N^+ (\epsilon_i\alpha_i) \Phi_{\{\epsilon_1,\epsilon_2,\epsilon_3,\epsilon_4\}}\left(\frac{d}{X} \right).
\label{equ:4thsec1ran2}
\end{align}
Our goal is to evaluate the inner sums over $d$ for fixed choices of $A$,  $D_i$, and $\epsilon_i$. We note that the expression above is a sum of 1024 terms indexed by the possible choices of these variables.

We also note that the various choices of the $\epsilon_i$ are not particularly significant. Indeed, since the $\a_i$ are any variables satisfying the conditions in Theorem \ref{thm:Main}, it suffices to consider the simplest case $\epsilon_i=1$ for each $i=1,2,3,4$, and our arguments will apply to the other possible choices of $\epsilon_i$ as well. Further, we consider only a specific case of the choice of $A$ and $D_i$. Of the 1024 expressions above, we will only evaluate the one corresponding to the choices
\begin{align*}
\epsilon_i &= 1,\ i=1,2,3,4,\\
A &= \{1,2,3\},\\
D_1 & = D_2 = D_3 = M.
\end{align*}
This expression is both notationally simplest and representative of the other expressions, and our arguments do not change substantially when considering any of the other expressions, and we remark later about the minor changes that occur. Examining \eqref{equ:4thsec1ran2}, the expression we need to consider is
\begin{align}
\sumstar_{(d,2)=1}   M^+(\a_1)  M^+(\a_2)  M^+(\a_3 )  N^+(\a_4)\Phi\left(\tfrac{d}{X} \right).
\label{splitting1}
\end{align}
We evaluate this sum in the following three sections.

\section{Separation of the Sum}\label{sec:Asymptotic-Diagonal-4th}
Recall definitions of $M^+$ and $N^+$ in  \eqref{defMN}. Using M\"{o}bius inversion to express the squarefree condition on $d$, the expression \eqref{splitting1} is 
\begin{align}
&\sumstar_{(d,2)=1}   M^+(\a_1)  M^+(\a_2)  M^+(\a_3 )  N^+(\a_4)\Phi\left(\tfrac{d}{X} \right)\nonumber\\
&=\sumstar_{(d,2)=1}  \sum_{n_1,n_2,n_3,n_4=1}^\infty \frac{\chi_{8d} (n_1n_2n_3n_4)}{n_1^{\frac{1}{2} + \a_1} n_2^{\frac{1}{2} + \a_2} n_3^{\frac{1}{2} + \a_3} n_4^{\frac{1}{2} + \a_4}} \nonumber\\
&\quad \times V_{\a_1}\left( \frac{n_1}{\sqrt{8d}}\right) V_{\a_2}\left( \frac{n_2}{\sqrt{8d}}\right) V_{\a_3}\left( \frac{n_3}{\sqrt{8d}}\right)V_{\a_4}\left( \frac{n_4}{\sqrt{8dU}}\right)\Phi\left(\frac{d}{X} \right)\nonumber\\
&=\sum_{(a,2)=1}  \mu(a) \sum_{(d,2)=1}  \sum_{(n_1n_2n_3n_4,2 a)=1} \frac{\chi_{8 d} (n_1n_2n_3n_4)}{n_1^{\frac{1}{2} + \a_1} n_2^{\frac{1}{2} + \a_2} n_3^{\frac{1}{2} + \a_3} n_4^{\frac{1}{2} + \a_4}}\nonumber\\
&\quad\times  V_{\a_1}\left( \frac{n_1}{\sqrt{8a^2 d}}\right) V_{\a_2}\left( \frac{n_2}{\sqrt{8a^2d}}\right) V_{\a_3}\left( \frac{n_3}{\sqrt{8a^2 d}}\right)V_{\a_4}\left( \frac{n_4}{\sqrt{8a^2dU}}\right)\Phi\left(\frac{a^2 d}{X} \right)\nonumber\\
&=:\mathcal{S} + \mathcal{S^*},
\label{asy1}
\end{align}
where $\mathcal{S}$ denotes the sum over $a$ above with $a \leq Y$ and $\mathcal{S}^*$ denotes the sum with $a>Y$, where $Y (\leq X)$ is a parameter chosen later.  An upper bound for $\mathcal{S}^*$ is given in the following lemma.
\begin{lem}
\label{lem4:SStar}
For some constant $A>0$, we have 
\[
\mathcal{S}^* \ll X Y^{-1} (\log X)^A .
\]
\end{lem}
\begin{proof}
In $\mathcal{S}^*$, the range of $a$ can be restricted to $Y<a\ll\sqrt{X}$ since $\Phi(x)$ is compactly supported. Let $d= \ell b^2$ where $\ell$ is squarefree and $b>0$. By grouping terms according to $c=ab^2$, it follows that 
\begin{align*}
\mathcal{S}^* &=    \sum_{\substack{(a,2)=1\\Y<a\ll\sqrt{X}}}  \mu(a) \sum_{(d,2)=1}  \sum_{(n_1n_2n_3n_4,2 a)=1} \frac{\chi_{8 d} (n_1n_2n_3n_4)}{n_1^{\frac{1}{2} + \a_1} n_2^{\frac{1}{2} + \a_2} n_3^{\frac{1}{2} + \a_3} n_4^{\frac{1}{2} + \a_4}}\nonumber\\
&\quad\times  V_{\a_1}\left( \frac{n_1}{\sqrt{8a^2 d}}\right) V_{\a_2}\left( \frac{n_2}{\sqrt{8a^2d}}\right) V_{\a_3}\left( \frac{n_3}{\sqrt{8a^2 d}}\right)V_{\a_4}\left( \frac{n_4}{\sqrt{8a^2 dU}}\right)\Phi\left(\frac{a^2 d}{X} \right) \\
&= \sum_{(c,2)=1}\sum_{\substack{(a,2)=1\\a|c\\Y<a\ll\sqrt{X}}}   \mu(a) \sumstar_{(\ell,2)=1}  \sum_{(n_1n_2n_3n_4,2 c)=1} \frac{\chi_{8 \ell} (n_1n_2n_3n_4)}{n_1^{\frac{1}{2} + \a_1} n_2^{\frac{1}{2} + \a_2} n_3^{\frac{1}{2} + \a_3} n_4^{\frac{1}{2} + \a_4}}\nonumber\\
&\quad\times  V_{\a_1}\left( \frac{n_1}{\sqrt{8\ell c^2 }}\right) V_{\a_2}\left( \frac{n_2}{\sqrt{8\ell c^2}}\right) V_{\a_3}\left( \frac{n_3}{\sqrt{8\ell c^2}}\right)V_{\a_4}\left( \frac{n_4}{\sqrt{8\ell c^2 U}}\right)\Phi\left(\frac{\ell c^2}{X} \right)
\end{align*}
Recalling  $V_\alpha$ in \eqref{VaDef}, and using a similar argument as in \eqref{equ:addVG}, we see that
\begin{align*}
\mathcal{S}^* 
&= \frac{1}{(2\pi i)^4}  \int\limits_{(1)} \cdots \int\limits_{(1)} 
\prod_{i=1}^4\frac{1}{u_i}
g_{\a_i}(u_i) 
\nonumber\\
&\quad\times  
\sum_{(c,2)=1}\sum_{\substack{(a,2)=1\\a|c\\Y<a\ll\sqrt{X}}} \mu(a)  \sumstar_{(\ell,2)=1}  \sum_{(n_1n_2n_3n_4,2 c)=1} \frac{\chi_{8\ell } (n_1n_2n_3n_4)}{n_1^{\frac{1}{2} + \a_1+u_1} n_2^{\frac{1}{2} + \a_2+u_2} n_3^{\frac{1}{2} + \a_3+u_3} n_4^{\frac{1}{2} + \a_4+u_4}}
\\
&\quad\times  U ^{\frac{u_4}{2}} (8\ell c^2 )^{\frac{u_1}{2} + \frac{u_2}{2} + \frac{u_3}{2}+ \frac{u_4}{2}}\Phi\left(\frac{\ell c^2}{X} \right)\, d\boldu\\
&= \sumd_{R_1,R_2,R_3,R_4} \frac{1}{(2\pi i)^8}  \int\limits_{(0)} \cdots \int\limits_{(0)}
\int\limits_{(\frac{2}{L})} \cdots \int\limits_{(\frac{2}{L})}
\prod_{i=1}^4\frac{1}{u_i}
g_{\a_i}(u_i) 
\sum_{\substack{(c,2)=1\\ c\leq \sqrt{X}}}\sum_{\substack{(a,2)=1\\a|c\\Y<a\ll\sqrt{X}}} \mu(a)  \\
&\quad \times \sumstar_{(\ell,2)=1}  
\prod_{i=1}^4\prod_{p|2c} \left(1-\frac{\chi_{8\ell}(p)}{p^{\frac{1}{2}+\a_i+u_i}}\right)
\sum_{n_i=1}^\infty \frac{\chi_{8\ell}(n_i)}{n_i^{\frac{1}{2}  +  \omega_i}} V\left( \frac{n_i}{R_i}\right)\tilde{G}(\omega_i-u_i-\a_i)R_i^{\omega_i-u_i-\a_i}
\\
&\quad\times U ^{\frac{u_4}{2}}  (8\ell c^2 )^{\frac{u_1}{2} + \frac{u_2}{2} + \frac{u_3}{2} + \frac{u_4}{2}}\Phi\left(\frac{\ell c^2}{X} \right)\, d\boldu \,d\boldomega,
\end{align*}
where we have moved the lines of the integration to $\Re(\omega_i) =0$ and $\Re(u_i)=\frac{2}{L}$ without encountering any poles. We have
\begin{align}
\mathcal{S}^* 
&\ll \sumd_{R_1,R_2,R_3,R_4} \frac{1}{(2\pi i)^8}  \int\limits_{(0)} \cdots \int\limits_{(0)}
\int\limits_{(\frac{2}{L})} \cdots \int\limits_{(\frac{2}{L})}
\prod_{i=1}^4 \sumabs{\frac{1}{u_i}
g_{\a_i}(u_i)}
\sum_{\substack{(c,2)=1\\ c\leq \sqrt{X}}} \sum_{\substack{(a,2)=1\\a|c\\Y<a\ll\sqrt{X}}} \tau(c)^4  \nonumber\\
&\quad \times  \prod_{i=1}^4 R_j ^{-\frac{2}{L}-\a_i}|\tilde{G}(\omega_i-u_i-\a_i) |\sumstar_{(\ell,2)=1}  
\prod_{i=1}^4
\sumabs{\sum_{n_i=1}^\infty \frac{\chi_{8\ell}(n_i)}{n_i^{\frac{1}{2}  +  \omega_i}}  V\left( \frac{n_i}{R_i}\right)}\Phi\left(\frac{\ell }{X/c^2} \right) d\boldu\, d\boldomega.
\label{S>Y-1}
\end{align}
Replacing $X$ by $X/c^2$ in \eqref{error-4th-case}, we obtain
\begin{align*}
&\sumstar_{(\ell,2)=1}    \prod_{i=1}^4
\sumabs{\sum_{n_i=1}^\infty
\frac{\chi_{8\ell}(n_i)}{n_i^{\frac{1}{2} + \omega_i}} V\left(\frac{n_i}{R_i} \right )} \Phi\left(\frac{\ell}{X/c^2} \right)\\
&\ll \frac{X}{c^2} \log^6 (2+X) \prod_{1 \leq i \leq 4} (1+|\omega_i|)^{3}.
\end{align*}
Substituting this into \eqref{S>Y-1}, we find that 
\begin{align*}
\mathcal{S}^* 
\ll 
X (\log X)^{10} \sum_{\substack{(c,2)=1\\ c\leq \sqrt{X}}}\sum_{\substack{(a,2)=1\\a|c\\Y<a\ll\sqrt{X}}} \frac{ \tau(c)^4  }{c^2}
\ll X (\log X)^{10}\sum_{Y< c\leq \sqrt{X}}   \frac{ \tau(c)^5  }{c^2} \ll X Y^{-1} (\log X)^A
\end{align*}
for some $A>0$.
\end{proof}

We now evaluate $\mathcal{S}$ in \eqref{asy1}. Let
\begin{align*}
H(x)&:=H(x;n_1, n_2, n_3, n_4) \\
&:= \Phi(a^2 x)V_{\a_1}\left(\frac{n_1}{\sqrt{8a^2xX} }\right)V_{\a_2}\left(\frac{n_2}{\sqrt{8a^2xX} }\right)V_{\a_3}\left(\frac{n_3}{\sqrt{8a^2xX} }\right)V_{\a_4}\left(\frac{n_4}{\sqrt{8a^2 xX U }}\right).
\end{align*}
Then
\[
\cols = \sum_{\substack{(a,2)=1\\ a \leq Y}}  \mu(a) \sum_{(d,2)=1}  \sum_{(n_1n_2n_3n_4,2 a)=1} \frac{\chi_{8 d} (n_1n_2n_3n_4)}{n_1^{\frac{1}{2} + \a_1} n_2^{\frac{1}{2} + \a_2} n_3^{\frac{1}{2} + \a_3} n_4^{\frac{1}{2} + \a_4}}H\left(\frac{d}{X};n_1, n_2, n_3, n_4\right).
\]
By the Possion summation in Lemma \ref{lem:Poisson}, we derive 
\begin{align*}
\cols &=
\frac{X}{2}\sum_{\substack{(a,2)=1\\ a \leq Y}}  \mu(a)  \sum_{(n_1n_2n_3n_4,2 a)=1} \frac{1}{n_1^{\frac{1}{2} + \a_1} n_2^{\frac{1}{2} + \a_2} n_3^{\frac{1}{2} + \a_3} n_4^{\frac{1}{2} + \a_4}}\\
&\quad \times \sum_{k \in \mathbb{Z}} (-1)^k \frac{G_k(n_1 n_2 n_3 n_4)}{n_1 n_2 n_3 n_4} \check{H}\left(\frac{kX}{2n_1 n_2 n_3 n_4}\right).
\end{align*}
Write 
\begin{align}
\cols = \cols^{0} + \cols^{\Box} + \cols^{\neq \Box}, 
\label{splitS}
\end{align}
where 
\begin{align} 
\cols^{0}&: = \frac{X}{2} \hat{H}(0) \sum_{\substack{(a,2)=1\\ a \leq Y}}  \mu(a)  \sum_{\substack{(n_1n_2n_3n_4,2 a)=1 \\ n_1n_2n_3n_4 = \Box}} \frac{1}{n_1^{\frac{1}{2} + \a_1} n_2^{\frac{1}{2} + \a_2} n_3^{\frac{1}{2} + \a_3} n_4^{\frac{1}{2} + \a_4}}  \prod_{p|n_1 n_2 n_3 n_4} \left(1- \frac{1}{p} \right) ,\label{OR-11}
\\
\cols^{\Box}&:= \frac{X}{2}\sum_{\substack{(a,2)=1\\ a \leq Y}}  \mu(a)  \sum_{(n_1n_2n_3n_4,2 a)=1} \frac{1}{n_1^{\frac{1}{2} + \a_1} n_2^{\frac{1}{2} + \a_2} n_3^{\frac{1}{2} + \a_3} n_4^{\frac{1}{2} + \a_4}} \nonumber\\
&\quad \times \sum_{k =\Box} (-1)^k \frac{G_k(n_1 n_2 n_3 n_4)}{n_1 n_2 n_3 n_4} \check{H}\left(\frac{kX}{2n_1 n_2 n_3 n_4}\right),\label{O1}\\
\cols^{\neq \Box}&:= 
\frac{X}{2}\sum_{\substack{(a,2)=1\\ a \leq Y}}  \mu(a)  \sum_{(n_1n_2n_3n_4,2 a)=1} \frac{1}{n_1^{\frac{1}{2} + \a_1} n_2^{\frac{1}{2} + \a_2} n_3^{\frac{1}{2} + \a_3} n_4^{\frac{1}{2} + \a_4}} \nonumber\\
&\quad \times \sum_{\substack{k \in \mathbb{Z} \\ k \neq 0, \Box}}  (-1)^k \frac{G_k(n_1 n_2 n_3 n_4)}{n_1 n_2 n_3 n_4} \check{H}\left(\frac{kX}{2n_1 n_2 n_3 n_4}\right)
.\label{R1}
\end{align}
 As we have mentioned previously, both $\cols^0$ (the diagonal terms) and $\cols^{\square}$ (the off-diagonal, square terms) contribute main terms, while $\cols^{\neq\square}$ is an error term. We evaluate or estimate each of the expressions above in the following sections.
\section{Evaluation of Diagonal and Off-diagonal, Square Terms }
\subsection{Diagonal Contribution $\cols^0$}\label{sec:Asymptotic-Diagonal}
By \eqref{OR-11}, we have
\begin{align}
\cols^{0}& = \frac{X}{2} \sum_{\substack{(a,2)=1 \\ a\leq Y}}  \mu(a)   \sum_{\substack{(n_1n_2n_3n_4,2 a)=1 \\ n_1n_2n_3n_4 = \Box}} \frac{1}{n_1^{\frac{1}{2} + \a_1} n_2^{\frac{1}{2} + \a_2} n_3^{\frac{1}{2} + \a_3} n_4^{\frac{1}{2} + \a_4}}   \prod_{p|n_1 n_2 n_3 n_4} \left(1- \frac{1}{p} \right)\nonumber\\
&\quad\times 
\int\limits_{-\infty}^\infty  \Phi(a^2 x)V_{\a_1}\left(\frac{n_1}{\sqrt{8a^2xX} }\right)V_{\a_2}\left(\frac{n_2}{\sqrt{8a^2xX} }\right)V_{\a_3}\left(\frac{n_3}{\sqrt{8a^2xX} }\right)V_{\a_4}\left(\frac{n_4}{\sqrt{8a^2xX U}}\right)\, dx\nonumber \\
&=\frac{X}{2} \sum_{\substack{(a,2)=1 \\ a\leq Y}} \frac{ \mu(a)}{a^2}   \sum_{\substack{(n_1n_2n_3n_4,2 a)=1 \\ n_1n_2n_3n_4 = \Box}} \frac{1}{n_1^{\frac{1}{2} + \a_1} n_2^{\frac{1}{2} + \a_2} n_3^{\frac{1}{2} + \a_3} n_4^{\frac{1}{2} + \a_4}}    \prod_{p|n_1 n_2 n_3 n_4} \left(1- \frac{1}{p} \right)\nonumber\\
&\quad\times 
\int\limits_{-\infty}^\infty  \Phi(x)V_{\a_1}\left(\frac{n_1}{\sqrt{8xX} }\right)V_{\a_2}\left(\frac{n_2}{\sqrt{8xX} }\right)V_{\a_3}\left(\frac{n_3}{\sqrt{8xX} }\right)V_{\a_4}\left(\frac{n_4}{\sqrt{8xXU }}\right)\, dx.
\label{D1}
\end{align}
Since $V_{\a_j}(\xi)$ decays rapidly as $\xi \rightarrow + \infty$ and $\xi \ll 1$ when $\xi$ near $0$, it follows that 
\begin{align*}
&\frac{X}{2} \sum_{\substack{(a,2)=1 \\ a> Y}} \frac{ \mu(a)}{a^2}   \sum_{\substack{(n_1n_2n_3n_4,2 a)=1 \\ n_1n_2n_3n_4 = \Box}} \frac{1}{n_1^{\frac{1}{2} + \a_1} n_2^{\frac{1}{2} + \a_2} n_3^{\frac{1}{2} + \a_3} n_4^{\frac{1}{2} + \a_4}}    \prod_{p|n_1 n_2 n_3 n_4} \left(1- \frac{1}{p} \right)\\
&\quad\times 
\int\limits_{-\infty}^\infty  \Phi(x)V_{\a_1}\left(\frac{n_1}{\sqrt{8xX} }\right)V_{\a_2}\left(\frac{n_2}{\sqrt{8xX} }\right)V_{\a_3}\left(\frac{n_3}{\sqrt{8xX} }\right)V_{\a_4}\left(\frac{n_4}{\sqrt{8xXU }}\right)\, dx \\
&\ll  \frac{X}{Y} \sum_{n_1 n_2 n_3 n_4 = \Box} \frac{1}{n_1^{\frac{1}{2}+\a_1} n_2^{\frac{1}{2}+\a_2} n_3^{\frac{1}{2}+\a_3} n_4^{\frac{1}{2}+\a_4}  } \\
&\quad \times 
\left(1+ \frac{n_1}{\sqrt{X}} \right)^{-100}\left(1+ \frac{n_2}{\sqrt{X}} \right)^{-100}
\left(1+ \frac{n_3}{\sqrt{X}} \right)^{-100} \left(1+ \frac{n_4}{\sqrt{XU}} \right)^{-100}\\
&\ll X  Y^{-1} (\log X)^{A} 
\end{align*}
for some $A>0$. Here the last inequality follows by splitting the sums over $n_i$ depending as $n_i\leq X$. The contribution from the ranges with $n_i > X$ for at least one $i$ are clearly negligible, and the remaining contribution is bounded by
\[
\sum_{\substack{n_i\leq X\\ n_1 n_2 n_3 n_4 = \Box}} \frac{1}{\sqrt{n_1n_2n_3n_4}} \leq \sum_{n\leq X^4} \frac{\tau_4(n^2)}{n} \ll (\log X)^A.
\]
Thus we may extend the sum over $a$ in \eqref{D1} to all positive odd integers with an error of at most $X  Y^{-1} (\log X)^{A}$. Since
\[
\sum_{(a,2n_1n_2n_3n_4)=1} \frac{ \mu(a)}{a^2} 
=
\frac{8}{\pi^2} \prod_{p|n_1 n_2 n_3 n_4} \left(1 - \frac{1}{p^2} \right)^{-1},
\]
 we find that
\begin{align*}
\cols^{0}
&=\frac{4X}{\pi^2}   \sum_{\substack{(n_1n_2n_3n_4,2 )=1 \\ n_1n_2n_3n_4 = \Box}} \frac{1}{n_1^{\frac{1}{2} + \a_1} n_2^{\frac{1}{2} + \a_2} n_3^{\frac{1}{2} + \a_3} n_4^{\frac{1}{2} + \a_4}}    \prod_{p|n_1 n_2 n_3 n_4} \frac{p}{p+1} \\
&\quad\times 
\int\limits_{-\infty}^\infty  \Phi(x)V_{\a_1}\left(\frac{n_1}{\sqrt{8xX} }\right)V_{\a_2}\left(\frac{n_2}{\sqrt{8xX} }\right)V_{\a_3}\left(\frac{n_3}{\sqrt{8xX} }\right)V_{\a_4}\left(\frac{n_4}{\sqrt{8xXU }}\right)\, dx \\
&\quad + O\left(X  Y^{-1} (\log X)^{A} \right).
\end{align*}
By \eqref{afe1}, we then have
\begin{align}   
\cols^{0}
&=\frac{4X}{\pi^2} \frac{1}{(2\pi i)^4}  \int\limits_{(1)} \cdots \int\limits_{(1)} 
\prod_{i=1}^4\frac{1}{u_i}  g_{\a_i}(u_i) \mathcal{D}_2(\tfrac{1}{2}+\a_1+u_1,\tfrac{1}{2}+\a_2+u_2,\tfrac{1}{2}+ \a_3+u_3, \tfrac{1}{2}+\a_4+u_4) \nonumber\\
&\quad\times 
(8 X)^{\sum_{i=1}^4\frac{u_i}{2}} U^{\frac{u_4}{2}}  \int\limits_{-\infty}^\infty  \Phi(x)
x^{\sum_{i=1}^4\frac{u_i}{2}} \, dx\, d\boldu + O\left(X Y^{-1} (\log X)^{A} \right),
    \label{D2}
  \end{align}
where $\mathcal{D}_2$ is defined in \eqref{def:D2}. 
It follows from \eqref{D2} and Lemma \ref{lem:DirichletSeriesDiagonal} that 
\begin{lem} \label{final-diag}
We have
\begin{align} 
\cols^{0}
&=\frac{4X}{\pi^2} \frac{1}{(2\pi i)^4}  \int\limits_{(\frac{1}{10})} \cdots \int\limits_{(\frac{1}{10})} 
\prod_{i=1}^4\frac{1}{u_i}  g_{\a_i}(u_i)(8 X)^{\sum_{i=1}^4\frac{u_i}{2}} U^{\frac{u_4}{2}}  \int\limits_{-\infty}^\infty  \Phi(x)
x^{\sum_{i=1}^4\frac{u_i}{2} } \, dx \nonumber\\
&\quad\times 
\prod_{1 \leq i\leq j \leq 4} \zeta_2(1+\a_i+\a_j+u_i+u_j)  \mathcal{H}_2(\tfrac{1}{2}+\a_1+u_1,\tfrac{1}{2}+\a_2+u_2,\tfrac{1}{2}+ \a_3+u_3, \tfrac{1}{2}+\a_4+u_4) \, d\boldu\nonumber \\
& \quad+ O\left(X  Y^{-1}  (\log X)^{A}\right).
\label{Dia-nongen}
\end{align}
The function $\mathcal{H}_2(\tfrac{1}{2}+\a_1+u_1,\tfrac{1}{2}+\a_2+u_2,\tfrac{1}{2}+ \a_3+u_3, \tfrac{1}{2}+\a_4+u_4)$ is analytic and uniformly bounded for $\Re(u_i)>-\frac{1}{4} + \varepsilon$.
\end{lem}

\subsection{Off-diagonal, Square Contribution $\cols^{\Box}$}\label{sec:Asymptotic-OffDiagonal}
Recall the expression for $\cols^{\Box}$ given in \eqref{O1}. By \eqref{eq:Fhatsecond}, 
\begin{align*}
&\check{H}\left(\frac{kX}{2n_1 n_2 n_3 n_4}\right)\\
&= \frac{1}{2\pi  i }  \int\limits_{(\varepsilon)}  \int\limits_0^\infty H(x; n_1, n_2, n_3,n_4)  x^{-s} \, dx  \left( \frac{n_1 n_2 n_3 n_4}{\pi X |k|}\right)^s
\Gamma(s) (\cos + \operatorname{sgn}(k)\sin)\left(\frac{\pi s}{2} \right)\, ds\\
&= \frac{1}{2\pi  i }  \int\limits_{(\varepsilon)}  \int\limits_0^\infty  \Phi(a^2 x)V_{\a_1}\left(\frac{n_1}{\sqrt{8a^2xX} }\right)V_{\a_2}\left(\frac{n_2}{\sqrt{8a^2xX} }\right)V_{\a_3}\left(\frac{n_3}{\sqrt{8a^2xX} }\right)V_{\a_4}\left(\frac{n_4}{\sqrt{8a^2 xXU}}\right)\\
&\quad\times  x^{-s} \, dx  \left( \frac{n_1 n_2 n_3 n_4}{\pi X |k|}\right)^s
\Gamma(s) (\cos + \operatorname{sgn}(k)\sin)\left(\frac{\pi s}{2} \right)\, ds\\
&= \frac{1}{a^2}\frac{1}{(2\pi i)^5}  \int\limits_{(\varepsilon)}  \int\limits_{(2\varepsilon)}  \cdots \int\limits_{(2\varepsilon)}   \int\limits_0^\infty  \Phi(x)
\prod_{i=1}^4 \frac{1}{u_i} g_{\a_i}(u_i)(8xX)^{\sum_{i=1}^4\frac{u_i}{2}} U^{\frac{u_4}{2}} x^{-s} \, dx \\
&\quad\times  \frac{1}{n_1^{u_1}n_2^{u_2}n_3^{u_3}n_4^{u_4}} \left( \frac{n_1 n_2 n_3 n_4 a^2}{\pi X |k|}\right)^s
\Gamma(s) (\cos + \operatorname{sgn}(k)\sin)\left(\frac{\pi s}{2} \right)\, d\boldu\, ds.
\end{align*}
By the  change of variables $u_i \mapsto u_i +s$, we see that 
\begin{align}
\check{H}\left(\frac{kX}{2n_1 n_2 n_3 n_4}\right)
&= \frac{1}{a^2}\frac{1}{(2\pi i)^5}  \int\limits_{(\varepsilon)} \cdots  \int\limits_{(\varepsilon)}  \int\limits_{(\varepsilon)}  \int\limits_0^\infty  \Phi(x)
\prod_{i=1}^4 \frac{1}{u_i+s} g_{\a_i}(u_i+s)  (8xX)^{\sum_{i=1}^4\frac{u_i+s}{2}} U^{\frac{u_4+s}{2}}  x^{-s} \, dx  \nonumber\\
&\quad\times \frac{1}{n_1^{u_1}n_2^{u_2}n_3^{u_3}n_4^{u_4}}\left( \frac{a^2}{\pi X |k|}\right)^s
\Gamma(s) (\cos + \operatorname{sgn}(k)\sin)\left(\frac{\pi s}{2} \right)\, d\boldu\, ds.
\label{exI}
\end{align}
Inserting this into \eqref{O1},
by \eqref{def:Z*}, we obtain
\begin{align}
\cols^{\Box}&= \frac{X}{2}\sum_{\substack{(a,2)=1 \\ a\leq Y}} \frac{ \mu(a)}{a^2}   \frac{1}{(2\pi i)^5}  \int\limits_{(\frac{1}{2}+\varepsilon)} \int\limits_{(2\varepsilon)}  \cdots \int\limits_{(2\varepsilon)}   \int\limits_0^\infty  \Phi(x)
\prod_{i=1}^4 \frac{1}{u_i+s} g_{\a_i}(u_i+s) \nonumber\\
&\quad\times (8xX)^{\sum_{i=1}^4\frac{u_i+s}{2} } U^{\frac{u_4+s}{2}}  x^{-s} \, dx  \left( \frac{a^2}{\pi X}\right)^s\Gamma(s) (\cos + \sin)\left(\frac{\pi s}{2} \right) \nonumber\\
&\quad\times  Z^*(\tfrac{1}{2}+u_1+\a_1,\tfrac{1}{2}+u_2+\a_2,\tfrac{1}{2}+u_3+\a_3,\tfrac{1}{2}+u_4+\a_4, s;1,a) \, d\boldu \, ds.
\label{ODDiri+1}
\end{align}
By \eqref{ODDiri+1} and Lemma \ref{lem:DirichletSeriesOffDiagonal}, we  derive
\begin{align*}
\cols^{\Box}&=\frac{X}{2} \sum_{\substack{(a,2)=1 \\ a\leq Y}} \frac{ \mu(a)}{a^2}   \frac{1}{(2\pi i)^5} \int\limits_{(\frac{1}{2}+\varepsilon)}   \int\limits_{(2\varepsilon)} \cdots \int\limits_{(2\varepsilon)}   \int\limits_0^\infty \Phi(x)
\prod_{i=1}^4 \frac{1}{u_i+s} g_{\a_i}(u_i+s)  (8xX)^{\sum_{i=1}^4\frac{u_i+s}{2}}  \nonumber\\
&\quad \times U^{\frac{u_4+s}{2}}x^{-s} \, dx  \left( \frac{a^2}{\pi X }\right)^s  \Gamma(s) (\cos + \sin)\left(\frac{\pi s}{2} \right)\left( 2^{1-2s} -1 \right) \zeta(2s)\nonumber \\
&\quad\times  \prod_{1\leq i \leq j \leq 4}
\frac{\zeta_2(1+2s+ \a_i +u_i + \a_j+u_j)}{\zeta_2(2+u_i+u_j+\a_i +\a_j)}  \prod_{1 \leq i \leq 4} \frac{\zeta_2(1+ \a_i+u_i)}{\zeta_2(1+2s+\a_i+u_i )} \nonumber\\ 
&\quad\times Z_2(\tfrac{1}{2}+u_1+\a_1,\tfrac{1}{2}+u_2+\a_2,\tfrac{1}{2}+u_3+\a_3,\tfrac{1}{2}+u_4+\a_4, s;1,a) \, d\boldu \, ds.
\end{align*}
We move the lines of the integration above to $\Re(s) =\frac{10}{L}$ and $\Re(u_i) =  \frac{5}{L}$. By (ii) of Lemma \ref{lem:DirichletSeriesOffDiagonal}, we do not enounter any poles when shifting contours (the possible pole arising from $\zeta(2s)$ at $s=0$ is cancelled by the zero of $2^{1-2s}-1$). As in the previous section, by Lemma \ref{lem:DirichletSeriesOffDiagonal}, we may extend the sum over $a$ to all positive odd integers at the cost  of
\[
\ll  X(\log X)^A \sum_{a>Y} \frac{\tau(a)}{a^{2-\frac{20}{\log X}}}
\ll XY^{-1}(\log X)^A.
\]
Therefore, 
\begin{align*} 
\cols^{\Box}&=\frac{X}{2} \sum_{(a,2)=1 } \frac{ \mu(a)}{a^2}  
\frac{1}{(2\pi i)^5} \int\limits_{(\frac{10}{L})}   \int\limits_{(\frac{5}{L})} \cdots \int\limits_{(\frac{5}{L})}   \int\limits_0^\infty  \Phi(x)
\prod_{i=1}^4 \frac{1}{u_i+s} g_{\a_i}(u_i+s)\nonumber\\
&\quad\times    (8xX)^{\sum_{i=1}^4\frac{u_i+s}{2}} U^{\frac{u_4+s}{2}}  x^{-s} \, dx  \left( \frac{a^2}{\pi X }\right)^s \Gamma(s) (\cos + \sin)\left(\frac{\pi s}{2} \right)\left( 2^{1-2s} -1 \right) \zeta(2s) \nonumber\\
&\quad\times  \prod_{1\leq i \leq j \leq 4}\frac{\zeta_2(1+2s+ \a_i +u_i + \a_j+u_j)}{\zeta_2(2+u_i+u_j+\a_i +\a_j)}   \prod_{1 \leq i \leq 4} \frac{\zeta_2(1+ \a_i+u_i)}{\zeta_2(1+2s+\a_i+u_i )} \nonumber\\
&\quad\times Z_2(\tfrac{1}{2}+u_1+\a_1,\tfrac{1}{2}+u_2+\a_2,\tfrac{1}{2}+u_3+\a_3,\tfrac{1}{2}+u_4+\a_4, s;1,a) \, d\boldu \, ds +O \left( XY^{-1}(\log X)^A\right).
\end{align*}
By Lemma \ref{lem:DirichletSeriesSuma}, the above is 
\begin{align}   
\cols^{\Box}&=\frac{X}{2}  
\frac{1}{(2\pi i)^5} \int\limits_{(\frac{10}{L})} \int\limits_{(\frac{5}{L})}  \cdots \int\limits_{(\frac{5}{L})}  X^{\sum_{i=1}^4\frac{u_i+s}{2}} U^{\frac{u_4+s}{2}}   X^{-s}\int\limits_0^\infty  \Phi(x) x^{-s} (8x)^{\sum_{i=1}^4\frac{u_i+s}{2}}\, dx \nonumber\\
&\quad\times    \prod_{i=1}^4 \frac{1}{u_i+s} g_{\a_i}(u_i+s)
\pi^{-s } (\cos + \sin)\left(\frac{\pi s}{2} \right)\left( 2^{1-2s} -1 \right) \nonumber\\
&\quad\times  \zeta(2s) \Gamma(s)   \prod_{1\leq i \leq j \leq 4}\frac{\zeta_2(1+2s+ \a_i +u_i + \a_j+u_j)}{\zeta_2(2+u_i+u_j+\a_i +\a_j)}  \prod_{1 \leq i \leq 4} \frac{\zeta_2(1+ \a_i+u_i)}{\zeta_2(1+2s+\a_i+u_i )} \nonumber\\ 
&\quad\times Z_3(\tfrac{1}{2}+\a_1+u_1,\tfrac{1}{2}+\a_2+u_2,\tfrac{1}{2}+ \a_3+u_3,\tfrac{1}{2}+\a_4+u_4,s) \, d\boldu \, ds +O \left( XY^{-1}(\log X)^A\right).
\label{ODDiri4-1}
\end{align}
\begin{lem}\label{lem:afeG}
Let $s \in \mathbb{C}$. Then 
\[
(2^{1-2s} -1) (\cos + \sin) \left(\frac{\pi s}{2}  \right) \pi ^{-s} \Gamma(s) = 2 \left(\frac{8}{\pi}\right)^{-s} \frac{\Gamma(\frac{1}{4}-\frac{s}{2})}{\Gamma(\frac{1}{4}+ \frac{s}{2})} \frac{\zeta_2(1-2s)}{\zeta(2s)}.
\] 
\end{lem}
\begin{proof}
See Lemma 6.1 of \cite{Young}.
\end{proof}
Using Lemma \ref{lem:afeG} in \eqref{ODDiri4-1} and by  Lemma \ref{lem:DirichletSeriesSuma}, it follows that 
\begin{lem} \label{lem4:off}
We have 
\begin{align}
\cols^{\Box}&=
\frac{X}{(2\pi i)^5}  \int\limits_{(\frac{1}{10})}  \int\limits_{(\frac{1}{10})} \cdots \int\limits_{(\frac{1}{10})} X^{\sum_{i=1}^4\frac{u_i+s}{2}} U^{\frac{u_4+s}{2}}   X^{-s} \int\limits_0^\infty  \Phi(x) x^{-s} (8x)^{\sum_{i=1}^4\frac{u_i+s}{2}}\, dx\nonumber\\
&\quad\times
\prod_{i=1}^4 \frac{1}{u_i+s} g_{\a_i}(u_i+s) 
\left(\frac{8}{\pi}\right)^{-s} \frac{\Gamma(\frac{1}{4}-\frac{s}{2})}{\Gamma(\frac{1}{4}+ \frac{s}{2})} \nonumber\\
&\quad\times  \zeta_2(1-2s) \prod_{1\leq i \leq j \leq 4}\frac{\zeta_2(1+2s+ \a_i +u_i + \a_j+u_j)}{\zeta_2(2+u_i+u_j+\a_i +\a_j)}  \prod_{1 \leq i \leq 4} \frac{\zeta_2(1+ \a_i+u_i)}{\zeta_2(1+2s+\a_i+u_i)}\nonumber\\ 
&\quad\times  Z_3(\tfrac{1}{2}+\a_1+u_1,\tfrac{1}{2}+\a_2+u_2,\tfrac{1}{2}+ \a_3+u_3,\tfrac{1}{2}+\a_4+u_4,s)  \, d\boldu \, ds+O \left( XY^{-1}(\log X)^A\right).
\label{ODDiri4}
\end{align}
The functions  $Z_3(\tfrac{1}{2}+\a_1+u_1,\tfrac{1}{2}+\a_2+u_2,\tfrac{1}{2}+ \a_3+u_3,\tfrac{1}{2}+\a_4+u_4,s)$ is analytic and uniformly bounded in the region 
\begin{equation*}
\Re(u_i) \geq -\frac{1}{8}, \quad -\frac{1}{20} \leq \Re(s) \leq \frac{1}{3}.
\end{equation*}
\end{lem}

\section{Off-diagonal, Non-square Estimation $\cols^{\neq\Box}$}\label{sec:Asymptotic-NonSquare}
Recalling  \eqref{R1}, and combining it with \eqref{exI}, we obtain
\begin{align*}
&\cols^{\neq \Box}\\
&= \sum_{\substack{(a,2)=1 \\ a\leq Y}}  \frac{1}{(2\pi i)^5}  \int\limits_{(\frac{1}{2}+2\varepsilon)} \cdots  \int\limits_{(\frac{1}{2}+ 2\varepsilon)} \int\limits_{(1+\varepsilon)}   \mathcal{C}(\boldu;s;a) \sum_{\substack{k \in \mathbb{Z} \\ k \neq 0, \Box}} (-1)^k \frac{1}{|k|^s} (\cos + \operatorname{sgn}(k)\sin)\left(\frac{\pi s}{2} \right)\nonumber \\
&\quad\times
  \sum_{(n_1n_2n_3n_4,2 a)=1} \frac{1}{n_1^{\frac{1}{2} + \a_1+u_1} n_2^{\frac{1}{2} + \a_2+u_2} n_3^{\frac{1}{2} + \a_3+u_3} n_4^{\frac{1}{2} + \a_4+u_4}}\frac{G_k(n_1 n_2 n_3 n_4)}{n_1 n_2 n_3 n_4}  
\, ds \, d\boldu  \nonumber\\ 
&= \sum_{\substack{(a,2)=1 \\ a\leq Y}}  \frac{1}{(2\pi i)^5}  \int\limits_{(\frac{1}{2}+ 2\varepsilon)} \cdots  \int\limits_{(\frac{1}{2}+ 2\varepsilon)} \int\limits_{(1+\varepsilon)}   \mathcal{C}(\boldu;s;a) \sumstar_{ k_1 \neq 1}  \sum_{k_2=1}^\infty (-1)^{k_1k_2^2}\frac{1}{|k_1k_2^2|^s} (\cos + \operatorname{sgn}(k_1)\sin)\left(\frac{\pi s}{2} \right)
  \nonumber  \\
&\quad \times \sum_{(n_1n_2n_3n_4,2 a)=1} \frac{1}{n_1^{\frac{1}{2} + \a_1+u_1} n_2^{\frac{1}{2} + \a_2+u_2} n_3^{\frac{1}{2} + \a_3+u_3} n_4^{\frac{1}{2} + \a_4+u_4}} \frac{G_{k_1k_2^2}(n_1 n_2 n_3 n_4)}{n_1 n_2 n_3 n_4}
\, ds \, d\boldu
\end{align*}
where 
\begin{align*}
    \mathcal{C}(\boldu;s;a)&:= 
    \mu(a)  \frac{X}{2}  \frac{1}{a^2}
    \int\limits_0^\infty  \Phi(x)(8xX)^{\sum_{i=1}^4\frac{u_i+s}{2} }x^{-s} \, dx  \cdot 
     U^{\frac{u_4+s}{2}} X^{-s}\\
     &\quad \times \prod_{i=1}^4\frac{1}{u_i+s}   g_{\a_i}(u_i+s) \left( \frac{a^2}{\pi  }\right)^s
\Gamma(s).
\end{align*}
In the following, we assume $k_1$ is odd and positive, and other cases can be done similarly. We have
\begin{align*}
&\cols^{\neq \Box}(k_1 \geq 1, \, \text{odd})\\
&= 
\sum_{\substack{(a,2)=1 \\ a\leq Y}}  \frac{1}{(2\pi i)^5}  \int\limits_{(\frac{1}{2}+ 2\varepsilon)} \cdots  \int\limits_{(\frac{1}{2}+ 2\varepsilon)} \int\limits_{(1+\varepsilon)}   \mathcal{C}(\boldu;s;a) \sumstar_{\substack{\\ k_1 \geq 1 \\ k_1\,\text{odd} }}\frac{1}{|k_1|^s}(\cos + \operatorname{sgn}(k_1)\sin)\left(\frac{\pi s}{2} \right)\\
&\quad\times  \sum_{k_2=1}^\infty (-1)^{k_2} \frac{1}{k_2^{2s}}    \sum_{(n_1n_2n_3n_4,2 a)=1} \frac{1}{n_1^{\frac{1}{2} + \a_1+u_1} n_2^{\frac{1}{2} + \a_2+u_2} n_3^{\frac{1}{2} + \a_3+u_3} n_4^{\frac{1}{2} + \a_4+u_4}} 
\frac{G_{k_1k_2^2}(n_1 n_2 n_3 n_4)}{n_1 n_2 n_3 n_4}
\, ds \, d\boldu.
\end{align*}
By Lemma \ref{lem:DirichletSeriesOffDiagonal}, it follows that 
\begin{align*}
&\cols^{\neq \Box}(k_1 \geq 1, \, \text{odd})\\
&= 
\sum_{\substack{(a,2)=1 \\ a\leq Y}}  \frac{1}{(2\pi i)^5}  \int\limits_{(\frac{1}{2}+ 2\varepsilon)} \cdots  \int\limits_{(\frac{1}{2}+ 2\varepsilon)} \int\limits_{(1+\varepsilon)}   \mathcal{C}(\boldu;s;a) \sumstar_{\substack{\\ k_1 \geq 1 \\ k_1\,\text{odd} }}\frac{1}{|k_1|^s}(\cos + \operatorname{sgn}(k_1)\sin)\left(\frac{\pi s}{2} \right)\\
&\quad\times  \left( 2^{1-2s} -1 \right) \zeta(2s)  \prod_{1\leq i \leq j \leq 4}\frac{\zeta_2(1+2s+ \a_i +u_i + \a_j+u_j)}{\zeta_2(2+u_i+u_j+\a_i +\a_j)}  \prod_{1 \leq i \leq 4} \frac{L_2(1+ \a_i+u_i,\chi_{\mathfrak{m}})}{L_2(1+2s+\a_i+u_i,\chi_{\mathfrak{m}} )} \\ &\quad\times Z_2(\tfrac{1}{2}+ \a_1+u_1,\tfrac{1}{2}+ \a_2+u_2,\tfrac{1}{2}+ \a_3+u_3,\tfrac{1}{2}+\a_4+u_4,s;k_1,a)
\, ds \, d\boldu.
\end{align*}
As we did in \eqref{equ:addVG}, we can derive 
\begin{align*}
&\cols^{\neq \Box}(k_1 \geq 1, \, \text{odd})\\
&= \sumd_{R_1,R_2,R_3,R_4} \frac{1}{(2\pi i)^4}
\int\limits_{(\varepsilon)}  \cdots \int\limits_{(\varepsilon)}
  \frac{1}{(2\pi i)^5}    \int\limits_{(2\varepsilon)}  \cdots \int\limits_{(2\varepsilon)}  \int\limits_{(1+\varepsilon)} \sum_{\substack{(a,2)=1 \\ a\leq Y}}  \mathcal{C}(\boldu;s;a)\\
&\quad\times
\sumstar_{\substack{ k_1 \geq 1 \\ k_1\,\text{\,odd} }}
\frac{1}{|k_1|^s}(\cos + \operatorname{sgn}(k_1)\sin)\left(\frac{\pi s}{2} \right)\\
&\quad\times  \left( 2^{1-2s} -1 \right) \zeta(2s)  \prod_{1\leq i \leq j \leq 4}\frac{\zeta_2(1+2s+ \a_i +u_i + \a_j+u_j)}{\zeta_2(2+u_i+u_j+\a_i +\a_j)}   \prod_{1 \leq i  \leq 4} {L_2(1+2s+\a_i+u_i,\chi_{\mathfrak{m}} )^{-1}} \\
&\quad\times
\prod_{i=1}^4 \left(1- \frac{\chi_{\mathfrak{m}}(2)}{2^{1+\alpha_i+u_i}} \right)
\prod_{1\leq i \leq 4} \sum_{n_i=1}^\infty  \frac{\chi_{\mathfrak{m}}(n_i)}{n_i^{1+\omega_i}} V \left(\frac{n_i}{R_i} \right) \tilde{G}(\omega_i-u_i-\a_i) R_i^{\omega_i-u_i-\a_i}\\
&\quad\times 
Z_2(\tfrac{1}{2}+\a_1+u_1,\tfrac{1}{2}+\a_2+u_2,\tfrac{1}{2}+\a_3+u_3,\tfrac{1}{2}+\a_4+u_4,s;k_1,a)  
\, ds\, d\boldu\, d\boldomega.
\end{align*}
Move the contours of the integrations to $\Re(u_i)=-\frac{1}{2} + \frac{10}{L}$ and $\Re(s) = 1+ \frac{1}{L}$ and $\Re(\omega_i) = -\frac{1}{2}$. By Lemma \ref{lem:DirichletSeriesOffDiagonal}(i), the above is
\begin{align}
&\ll X U^{\frac{1}{4}} \sumd_{R_1,R_2,R_3,R_4}  \int\limits_{-\infty}^\infty   \int\limits_{-\infty}^\infty   \int\limits_{-\infty}^\infty   \int\limits_{-\infty}^\infty  \prod_{i=1}^4 \frac{1}{(1+|t_i|)^{10}} \nonumber\\
&\quad\times \sum_{a\leq Y} \tau(a) \sumstar_{k_1\geq 1} \frac{1}{k_1^{1+ \frac{1}{L}}}  
\prod_{1\leq i \leq 4}\left| \sum_{n_i=1}^\infty  \frac{\chi_{\mathfrak{m}}(n_i)}{n_i^{\frac{1}{2}+ i  t_i}} V \left(\frac{n_i}{R_i} \right)\right| R_i^{-\frac{1}{L}}\, dt_1\, dt_2\, dt_3\, dt_4 \nonumber  \\
&=X U^{\frac{1}{4}}  \sum_{a\leq Y} \tau(a)  \sumstar_{k_1\geq 1} \frac{1}{k_1^{1+ \frac{1}{L}}} 
\left( \sumd_{R}  \int\limits_{-\infty}^\infty
\left| \sum_{n=1}^\infty  \frac{\chi_{\mathfrak{m}}(n)}{n^{\frac{1}{2}+ i  t}} V \left(\frac{n}{R} \right)\right| R^{-\frac{1}{L}} \frac{1}{(1+|t|)^{10}}\, dt \right)^4\nonumber \\
&\ll X U^{\frac{1}{4}} (\log X)^A \sumd_{R} \sum_{a\leq Y} \tau(a)   
\int\limits_{-\infty}^\infty
\sumstar_{k_1\geq 1} \frac{1}{k_1^{1+ \frac{1}{L}}} 
\left| \sum_{n=1}^\infty  \frac{\chi_{\mathfrak{m}}(n)}{n^{\frac{1}{2}+ i  t}} V \left(\frac{n}{R} \right)\right|^4 R^{-\frac{1}{L}} \frac{1}{(1+|t|)^{10}}\, dt .
\label{non-non-1}
\end{align}
In the last equality, we have invoked the  Cauchy–Schwarz inequality several times. 
By  \eqref{V-def}, Proposition \ref{prop:LargeSieve} and Lemma \ref{lem:InductionAllN},
\begin{align*}
\sumstar_{k_1\geq 1} \frac{1}{k_1^{1+ \frac{1}{L}}}    \left| \sum_{n=1}^\infty  \frac{\chi_{\mathfrak{m}}(n)}{n^{\frac{1}{2}+ i  t}} V \left(\frac{n}{R} \right)\right|^4
&\ll 
\sum_{r=0}^\infty \frac{1}{2^{(1+ \frac{1}{L})r}}  \sumstar_{2^r \leq k_1 < 2^{r+1}}   \left| \sum_{n=1}^\infty  \frac{\chi_{\mathfrak{m}}(n)}{n^{\frac{1}{2}+ i  t}} V \left(\frac{n}{R} \right)\right|^4\\
&  \ll \sum_{r=0}^\infty \frac{1}{2^{(1+ \frac{1}{L})r}} 
(1+\abs{t})^{3} 2^r \log^6(2+2^r )\log^6(2+\abs{t}) \\
&\ll (\log X)^A (1+\abs{t})^{3.5},
\end{align*}
which implies that \eqref{non-non-1} is 
\begin{align*}
& \ll 
X U^{\frac{1}{4}} (\log X)^A \sum_{a\leq Y} \tau(a)    
\sumd_{R} R^{-\frac{1}{L}}  
\ll X U^{\frac{1}{4}} Y (\log X) ^A.
\end{align*}
Now we have proved  
\begin{lem}\label{remainE}
For some $A>0$,
\[
\cols^{\neq \Box} \ll X U^{\frac{1}{4}} Y (\log X)^A.
\]
\end{lem}

\section{Removing the Parameter $U$ and Completing the Proof of Theorem \ref{thm:Main}}\label{sec:removeU}

Recall from \eqref{asy1} and \eqref{splitS} that
\begin{align*}
&\sumstar_{(d,2)=1}   M^+(\a_1)  M^+(\a_2)  M^+(\a_3 )  N^+(\a_4)\Phi\left(\tfrac{d}{X} \right)\\
&=  \mathcal{S} + \mathcal{S}^*\\
&=\mathcal{S}^0  + \mathcal{S}^{\Box} + \mathcal{S}^{\neq \Box} + \mathcal{S}^*. 
\end{align*}
Combining this with Lemmas \ref{lem4:SStar}, \ref{final-diag}, \ref{lem4:off} and \ref{remainE}, we have
\begin{align}
&\sumstar_{(d,2)=1}   M^+(\a_1)  M^+(\a_2)  M^+(\a_3 )  N^+(\a_4)\Phi\left(\tfrac{d}{X} \right) \nonumber\\
& = \frac{1}{(2\pi i)^4}  \int\limits_{(\frac{1}{10})} \cdots \int\limits_{(\frac{1}{10})} \colw_1(\boldsymbol{u} ; \boldsymbol{\a})U^{\frac{u_4}{2}} \, d\boldu
+ \frac{1}{(2\pi i)^5}  \int\limits_{(\frac{1}{10})}  \int\limits_{(\frac{1}{10})} \cdots \int\limits_{(\frac{1}{10})}  \colw_2(s;\boldsymbol{u};\boldsymbol{\a})U^{\frac{u_4+s}{2}}  \, d\boldu \, ds \nonumber
\\
&\quad  + O\left(X  Y^{-1}  (\log X)^{A}\right) + O(X U^{\frac{1}{4}} Y (\log X)^A),
\label{eq:MMMN}
\end{align}
where
\begin{align*}
&\colw_1(\boldsymbol{u};\boldsymbol{\a})U^{\frac{u_4}{2}}\\
&:= \frac{4X}{\pi^2}(8 X)^{\sum_{i=1}^4\frac{u_i}{2}} U^{\frac{u_4}{2}} \prod_{i=1}^4\frac{g_{\a_i}(u_i)}{u_i} \int\limits_{-\infty}^\infty  \Phi(x)
x^{\sum_{i=1}^4\frac{u_i}{2} } \, dx \prod_{1 \leq i\leq j \leq 4} \zeta_2(1+\a_i+\a_j+u_i+u_j) \\
&\quad\times \mathcal{H}_2(\tfrac{1}{2}+\a_1+u_1,\tfrac{1}{2}+\a_2+u_2,\tfrac{1}{2}+ \a_3+u_3, \tfrac{1}{2}+\a_4+u_4)
\end{align*}
is the integrand in \eqref{Dia-nongen} and
\begin{align*}
\colw_2(s;\boldsymbol{u};\boldsymbol{\a})U^{\frac{u_4+s}{2}} &:= X^{1+\sum_{i=1}^4\frac{u_i+s}{2}} U^{\frac{u_4+s}{2}}   X^{-s} \int\limits_0^\infty  \Phi(x) x^{-s} (8x)^{\sum_{i=1}^4\frac{u_i+s}{2}}\, dx\\
&\quad\times
\prod_{i=1}^4 \frac{1}{u_i+s} g_{\a_i}(u_i+s) 
\left(\frac{8}{\pi}\right)^{-s} \frac{\Gamma(\frac{1}{4}-\frac{s}{2})}{\Gamma(\frac{1}{4}+ \frac{s}{2})} \zeta_2(1-2s)\\
&\quad\times   \prod_{1\leq i \leq j \leq 4}\frac{\zeta_2(1+2s+ \a_i +u_i + \a_j+u_j)}{\zeta_2(2+u_i+u_j+\a_i +\a_j)}  \prod_{1 \leq i \leq 4} \frac{\zeta_2(1+ \a_i+u_i)}{\zeta_2(1+2s+\a_i+u_i)}\\ 
&\quad\times  Z_3(\tfrac{1}{2}+\a_1+u_1,\tfrac{1}{2}+\a_2+u_2,\tfrac{1}{2}+ \a_3+u_3,\tfrac{1}{2}+\a_4+u_4,s)
\end{align*}
is the integrand in \eqref{ODDiri4}. Letting $Y= (\log X)^A$ with $A\geq 10$ and  $Q=100A$ in the definition of $U$ \eqref{defU}, we see that the error term above is $O(X)$.

We now remark on what changes when considering the other expressions in \eqref{equ:4thsec1ran2}. By \eqref{defMN}, we see that if $N^+(\a_4)$ is replaced by $M^+(\a_4)$ on the left of \eqref{eq:MMMN}, then the factors of $U$ in the two integrals on the right of \eqref{eq:MMMN} are replaced by 1. Conversely, repeating the argument from Section \ref{sec:Asymptotic-Diagonal-4th} to Section \ref{sec:removeU} shows that if we replace $M^+(\a_i)$ by $N^+(\a_i)$ on the left of \eqref{eq:MMMN}, we need only add the factors $U^{\frac{u_i}{2}}$ and $U^{\frac{u_i+s}{2}}$, respectively, to the integrals on right of \eqref{eq:MMMN}. Hence, for a fixed nonempty $A \subsetneqq \{1,2,3,4\} $ and fixed choices of $D_i$ for each $i \in A$, we have
\begin{align}
& 
\sumstar_{(d,2)=1} 
\prod_{i\in A}  D_i^+(\alpha_i) \prod_{i \in \bar{A}}   N^+ (\alpha_i) \Phi\left(\frac{d}{X} \right) \nonumber\\
& = \frac{1}{(2\pi i)^4}  \int\limits_{(\frac{1}{10})} \cdots \int\limits_{(\frac{1}{10})} \colw_1(\boldsymbol{u} ; \boldsymbol{\a})U^{\sum_{\substack{ i \in A \\ D_i=N}}\frac{u_i}{2} + \sum_{i \in \bar{A}}  \frac{u_i}{2}} \, d\boldu \nonumber\\
&\quad + \frac{1}{(2\pi i)^5}  \int\limits_{(\frac{1}{10})}  \int\limits_{(\frac{1}{10})} \cdots \int\limits_{(\frac{1}{10})}  \colw_2(s;\boldsymbol{u};\boldsymbol{\a})U^{\sum_{\substack{ i \in A \\ D_i=N}}\frac{u_i+s}{2} + \sum_{i \in \bar{A}}  \frac{u_i+s}{2}} \, d\boldu \, ds
+ O(X).
\label{equ:secU-0}
\end{align}
That is to say, the various configurations of $M$ and $N$ which are indexed by $A$ and $D_i$ simply translate to different configurations of exponents for the parameter $U$ in the integrals above. In particular, the exponents of $U$ are the only portion \eqref{equ:secU-0} that depend on the subset $A$, and we also note that the equality above remains valid when $A = \emptyset$. By \eqref{equ:secU-0} and \eqref{equ:4thsec1ran2}, we obtain 
\begin{align}
&\sumstar_{(d,2)=1} \Phi\left(\frac{d}{X} \right) \sum_{\substack{A \subsetneqq \{1,2,3,4\}\\ A\neq\emptyset}}  \prod_{i \in A}  \left( M^+(\a_i) -N^+(\a_i) +M^-(\a_i)-N^-(\a_i) \right) \nonumber\\
&\quad\times\prod_{i \in \bar{A}} \left( N^+(\a_i) +N^-(\a_i) \right) \nonumber\\
&=  \sum_{\substack{\epsilon_i = \pm 1\\i=1,2,3,4}} \prod_{i=1}^4
(8X)^{-\frac{1-\epsilon_i}{2}\a_i} 
\lambda_i^{\frac{1-\epsilon_i}{2}} \sum_{\substack{A \subsetneqq \{1,2,3,4\}\\ A\neq\emptyset}}  
\sum_{\substack{D_i=M \text{\,or\,}N\\i \in A }} 
\prod_{i\in A} (-1)^{h(D_i)}\prod_{i \in A} U^{-\frac{1-\epsilon_i}{2}\a_i\delta(D_i=N)} \prod_{i \in \bar{A}} U^{-\frac{1-\epsilon_i}{2}\a_i}  \nonumber\\
&\quad \times \left\{
\frac{1}{(2\pi i)^4}  \int\limits_{(\frac{1}{10})} \cdots \int\limits_{(\frac{1}{10})} \colw_1(\boldsymbol{u};\boldsymbol{\a};\boldsymbol{\epsilon})U^{\sum_{\substack{ i \in A \\ D_i=N}}\frac{u_i}{2} + \sum_{i \in \bar{A}}  \frac{u_i}{2}} \, d\boldu \right.\nonumber\\
&\quad + \left.\frac{1}{(2\pi i)^5}  \int\limits_{(\frac{1}{10})}  \int\limits_{(\frac{1}{10})} \cdots \int\limits_{(\frac{1}{10})} \colw_2(s;\boldsymbol{u};\boldsymbol{\a};\boldsymbol{\epsilon})U^{\sum_{\substack{ i \in A \\ D_i=N}}\frac{u_i+s}{2} + \sum_{i \in \bar{A}}  \frac{u_i+s}{2}} \, d\boldu \, ds \right\}
+ O(X),
\label{equ:sec11-11}
\end{align}
where $ \colw_1(\boldsymbol{u};\boldsymbol{\a};\boldsymbol{\epsilon})$ is equal to   $ \colw_1(\boldsymbol{u} ; \boldsymbol{\a})$ with $\a_i$ and $\Phi(x)$ replaced by $\epsilon_i \alpha_i$ and $\Phi_{\{\epsilon_1,\epsilon_2, \epsilon_3,\epsilon_4\}}(x)$, respectively, and  $\colw_2(s;\boldsymbol{u};\boldsymbol{\a};\boldsymbol{\epsilon})$  is defined in the same way. The case when $A = \emptyset$ is given by
\begin{align}
&\sumstar_{(d,2)=1} \Phi\left(\frac{d}{X} \right) \prod_{i=1}^4 \left( N^+(\a_i) +N^-(\a_i) \right)\nonumber\\
&=  \sum_{\substack{\epsilon_i = \pm 1\\i=1,2,3,4}} \prod_{i=1}^4
(8X)^{-\frac{1-\epsilon_i}{2}\a_i} 
\lambda_i^{\frac{1-\epsilon_i}{2}} \prod_{i=1}^4 U^{-\frac{1-\epsilon_i}{2}\a_i} \times \left\{
\frac{1}{(2\pi i)^4}  \int\limits_{(\frac{1}{10})} \cdots \int\limits_{(\frac{1}{10})} \colw_1(\boldsymbol{u};\boldsymbol{\a};\boldsymbol{\epsilon})U^{\sum_{i=1}^4  \frac{u_i}{2}} \, d\boldu \right.\nonumber\\
&\quad + \left.\frac{1}{(2\pi i)^5}  \int\limits_{(\frac{1}{10})}  \int\limits_{(\frac{1}{10})} \cdots \int\limits_{(\frac{1}{10})} \colw_2(s;\boldsymbol{u};\boldsymbol{\a};\boldsymbol{\epsilon})U^{\sum_{i=1}^4  \frac{u_i+s}{2}} \, d\boldu \, ds \right\}
+ O(X).
\label{equ:sec11-22}
\end{align}
We now combine the expressions in the previous two displays via several combinatorial identities. Note that when $A\neq\emptyset$, we have
\begin{align}
& \sum_{\substack{D_i=M \text{\,or\,}N\\i \in A }} 
\prod_{i\in A} (-1)^{h(D_i)} \prod_{i \in A }
{U^{-\a_i\delta(D_i=N)}}^{\frac{1-\epsilon_i}{2}} \prod_{i \in \bar{A}}
U^{-\frac{1-\epsilon_i}{2}\a_i}  \cdot U^{\sum_{\substack{ i \in A \\ D_i=N}}\frac{u_i}{2} + \sum_{i \in \bar{A}}  \frac{u_i}{2}} \nonumber\\
&=  \sum_{\substack{D_i=M \text{\,or\,}N\\i \in A }} 
\prod_{\substack{i \in A \\ D_i =M}}  1 \cdot  \prod_{\substack{i \in A \\ D_i =N}} (-1) U^{-\frac{1-\epsilon_i}{2}\a_i+\frac{u_i}{2}} 
\prod_{i \in \bar{A}  } U^{-\frac{1-\epsilon_i}{2}\a_i+\frac{u_i}{2}}\nonumber\\
&= \prod_{i \in A} \pth{1-U^{-\frac{1-\epsilon_i}{2}\a_i+\frac{u_i}{2}}} \prod_{i \in \bar{A}} U^{-\frac{1-\epsilon_i}{2}\a_i+\frac{u_i}{2}},
\label{equ:secU-1}
\end{align}
Using the identity
\begin{align*}
\prod_{i=1}^4 (1-x_i) + \sum_{\substack{A \subsetneqq \{1,2,3,4\}\\ A\neq\emptyset}} \prod_{i\in A} (1-x_i) \prod_{i \in \bar{A}  } x_i + \prod_{i=1}^4 x_i =\prod_{i=1}^4 (1-x_i+x_i) =1,
\end{align*}
with $x_i =U^{-\frac{1-\epsilon_i}{2}\a_i+\frac{u_i}{2}}$ along with \eqref{equ:secU-1}, we have
\begin{align*}
&\sum_{\substack{A \subsetneqq \{1,2,3,4\}\\ A\neq \emptyset}}  
\sum_{\substack{D_i=M \text{\,or\,}N\\i \in A }} 
\prod_{i\in A} (-1)^{h(D_i)} \prod_{i \in A }
U^{-\frac{1-\epsilon_i}{2}\a_i\delta(D_i=N)} \prod_{i \in \bar{A}}
U^{-\frac{1-\epsilon_i}{2}\a_i}  \cdot U^{\sum_{\substack{ i \in A \\ D_i=N}}\frac{u_i}{2} + \sum_{i \in \bar{A}}  \frac{u_i}{2}} \\
&\quad+ \prod_{i=1}^4 U^{-\frac{1-\epsilon_i}{2}\a_i+\frac{u_i}{2}} \\
&= 1 -  \prod_{i=1}^4 \pth{1-U^{-\frac{1-\epsilon_i}{2}\a_i+\frac{u_i}{2}}}.
\end{align*}
 The above equality, together with \eqref{equ:sec11-11} and \eqref{equ:sec11-22}, implies
\begin{align}
&\sumstar_{(d,2)=1} \Phi\left(\frac{d}{X} \right) \sum_{A \subsetneqq \{1,2,3,4\}}  \prod_{i \in A}  \left( M^+(\a_i) -N^+(\a_i) +M^-(\a_i)-N^-(\a_i) \right) \nonumber\\
&\quad\times\prod_{i \in \bar{A}} \left( N^+(\a_i) +N^-(\a_i) \right)\nonumber \\
&
= \sum_{\substack{\epsilon_i = \pm 1\\i=1,2,3,4}} \prod_{i=1}^4
(8X)^{-\sum_{i=1}^4\frac{1-\epsilon_i}{2}\a_i} 
\lambda_i^{\frac{1-\epsilon_i}{2}}  
(\mathcal{Z}_1 +\mathcal{Z}_2  - \mathcal{Z}_3- \mathcal{Z}_4) +O(X),
\label{equ:Z1Z2Z3Z4}
\end{align}
where 
\begin{align}
\mathcal{Z}_1 &:=  \frac{1}{(2\pi i)^4}  \int\limits_{(\frac{1}{10})} \cdots \int\limits_{(\frac{1}{10})} \colw_1(\boldsymbol{u};\boldsymbol{\a};\boldsymbol{\epsilon})\, d\boldu,\nonumber\\
\mathcal{Z}_2 &:= \frac{1}{(2\pi i)^5}  \int\limits_{(\frac{1}{10})}  \int\limits_{(\frac{1}{10})} \cdots \int\limits_{(\frac{1}{10})} \colw_2(s;\boldsymbol{u};\boldsymbol{\a};\boldsymbol{\epsilon})\, d\boldu \, ds,\nonumber \\
\mathcal{Z}_3 &:=  \frac{1}{(2\pi i)^4}  \int\limits_{(\frac{1}{10})} \cdots \int\limits_{(\frac{1}{10})} \colw_1(s,\boldsymbol{u};\boldsymbol{\a};\boldsymbol{\epsilon}) \prod_{i=1}^4 \pth{1-U^{-\frac{1-\epsilon_i}{2}\a_i+\frac{u_i}{2}}}  \, d\boldu,\nonumber\\
\mathcal{Z}_4 &:= \frac{1}{(2\pi i)^5}  \int\limits_{(\frac{1}{10})}  \int\limits_{(\frac{1}{10})} \cdots \int\limits_{(\frac{1}{10})} \colw_2(s;\boldsymbol{u};\boldsymbol{\a};\boldsymbol{\epsilon}) \prod_{i=1}^4 \pth{1-U^{-\frac{1-\epsilon_i}{2}\a_i+\frac{u_i+s}{2}}}  \, d\boldu\, ds.
\label{equ-defcolz} 
\end{align}

In the next two sections, we will prove the following two lemmas regarding the quantities above.

\begin{lem}
\label{lem:secU-main}
We have
\begin{align*}
& \sum_{\substack{\epsilon_i = \pm 1\\i=1,2,3,4}} \prod_{i=1}^4
(8X)^{-\frac{1-\epsilon_i}{2}\a_i} 
\lambda_i^{\frac{1-\epsilon_i}{2}} \cdot (     \mathcal{Z}_1+ \mathcal{Z}_2) \\
&= \sum_{\substack{\epsilon_i = \pm 1\\i=1,2,3,4}}  \prod_{i=1}^4
(8X)^{-\frac{1-\epsilon_i}{2}\a_i}
\lambda_i^{\frac{1-\epsilon_i}{2}} \cdot \frac{4X}{\pi^2}  
\int\limits_{-\infty}^\infty  \Phi_{\{\epsilon_1,\epsilon_2, \epsilon_3,\epsilon_4\}}(x)\, dx \\
&\quad \times 
\prod_{1 \leq i\leq j \leq 4} \zeta_2(1+\epsilon_i\a_i+\epsilon_j\a_j)
\mathcal{H}_2(\tfrac{1}{2}+\epsilon_1\a_1,\tfrac{1}{2}+\epsilon_2\a_2,\tfrac{1}{2}+\epsilon_3 \a_3, \tfrac{1}{2}+\epsilon_4\a_4) 
\\
&\quad 
+O(X (\log X)^6) .
\end{align*}
\end{lem}

\begin{lem}
\label{lem:secU-error}
We have 
\begin{align*}
\mathcal{Z}_3, \mathcal{Z}_4\ll X (\log X)^{6+ \varepsilon}. 
\end{align*}
\end{lem}

\begin{proof}[Proof of Theorem \ref{thm:Main}] Combining the above two lemmas with  \eqref{Combin},   \eqref{equ:Z1Z2Z3Z4} and Lemma \ref{lem:4therror}, we obtain
\begin{align*}
&  \sumstar_{(d,2)=1} L(\tfrac{1}{2} + \a_1,\chi_{8d})L(\tfrac{1}{2} + \a_2, \chi_{8d} )L(\tfrac{1}{2} + \a_3, \chi_{8d} )L(\tfrac{1}{2} + \a_4, \chi_{8d}) \Phi(\tfrac{d}{X})\\
&=\sum_{\substack{\epsilon_i = \pm 1\\i=1,2,3,4}} 
\prod_{i=1}^4 (8X)^{-\frac{1-\epsilon_i}{2}\a_i}  
\lambda_i^{\frac{1-\epsilon_i}{2}} \cdot \frac{4X}{\pi^2}  
\int\limits_{-\infty}^\infty  \Phi_{\{\epsilon_1,\epsilon_2, \epsilon_3,\epsilon_4\}}(x)\, dx \\
&\quad \times 
\prod_{1 \leq i\leq j \leq 4} \zeta_2(1+\epsilon_i\a_i+\epsilon_j\a_j) 
\mathcal{H}_2(\tfrac{1}{2}+\epsilon_1\a_1,\tfrac{1}{2}+\epsilon_2\a_2,\tfrac{1}{2}+\epsilon_3 \a_3, \tfrac{1}{2}+\epsilon_4\a_4) \, d\boldu\\
&\quad +O(X (\log X)^{6+ \varepsilon}),
\end{align*}
which completes the proof of Theorem \ref{thm:Main}.
\end{proof}

\section{Cross-cancellation between $\colz_1$ and $\colz_2$} \label{sec:cross-cancel}
We prove Lemma \ref{lem:secU-main} in this section. By \eqref{equ-defcolz}, 
\begin{align*}
\mathcal{Z}_1 &= \frac{1}{(2\pi i)^4}  \int\limits_{(\frac{1}{10})} \cdots \int\limits_{(\frac{1}{10})} 
\colw_1(\boldsymbol{u};\boldsymbol{\a};\boldsymbol{\epsilon})   \, d\boldu,
\end{align*}
where 
\begin{align}
\colw_1(\boldsymbol{u};\boldsymbol{\a};\boldsymbol{\epsilon}) &= \colf_1(\boldsymbol{u}; \boldsymbol{\a};\boldsymbol{\epsilon})
\prod_{i=1}^4\frac{1}{u_i}     
\prod_{1 \leq i\leq j \leq 4}
\zeta_2(1+\epsilon_i\a_i+\epsilon_j\a_j+u_i+u_j),\nonumber\\
\colf_1(\boldsymbol{u}; 
\boldsymbol{\a};\boldsymbol{\epsilon}) 
&:=   \frac{4X}{\pi^2} \prod_{i=1}^4g_{\epsilon_i\a_i}(u_i)(8 X)^{\sum_{i=1}^4\frac{u_i}{2}} \int\limits_{-\infty}^\infty  \Phi_{\{\epsilon_1, \epsilon_2,\epsilon_3,\epsilon_4\}}(x)
x^{\sum_{i=1}^4\frac{u_i}{2} } \, dx\nonumber\\
&\quad \times \mathcal{H}_2(\tfrac{1}{2}+\epsilon_1\a_1+u_1,\tfrac{1}{2}+\epsilon_2\a_2+u_2,\tfrac{1}{2}+ \epsilon_3\a_3+u_3, \tfrac{1}{2}+\epsilon_4\a_4+u_4).
\label{def-colf1}
\end{align}
Changing  variables $u_i \mapsto u_i -s$ in $\colz_2$, we obtain
\begin{align*}
\colz_2 &=  \frac{1}{(2 \pi i)^5} \int\limits_{(\frac{1}{10})}  \int\limits_{(\frac{1}{10}+ \frac{5}{L})} \cdots \int\limits_{(\frac{1}{10}+ \frac{5}{L})} 
\colw_2(s;\boldsymbol{u}-s;\boldsymbol{\a};\boldsymbol{\epsilon}) 
\, d\boldu \, ds,
\end{align*}
where
\begin{align*}
&\colw_2(s;\boldsymbol{u}-s;\boldsymbol{\a};\boldsymbol{\epsilon})\\
& =\colf_2(s;\boldsymbol{u}; \boldsymbol{\a};\boldsymbol{\epsilon})\prod_{i=1}^4 \frac{1}{u_i}  
\zeta_2(1-2s) \prod_{1\leq i \leq j \leq 4} \zeta_2(1+ \epsilon_i\a_i  + \epsilon_j\a_j +u_i + u_j)  \prod_{1 \leq i \leq 4} \frac{\zeta_2(1+ \epsilon_i\a_i+u_i-s)}{\zeta_2(1+\epsilon_i\a_i+u_i+s)},
\end{align*}
and
\begin{align*}
&\colf_2(s;\boldsymbol{u}; \boldsymbol{\a};\boldsymbol{\epsilon}) \\
&:=X^{1+\sum_{i=1}^4\frac{u_i}{2}}    X^{-s}
\prod_{i=1}^4 g_{\epsilon_i\a_i}(u_i)\left(\frac{8}{\pi}\right)^{-s}\frac{\Gamma(\frac{1}{4}-\frac{s}{2})}{\Gamma(\frac{1}{4}+ \frac{s}{2})}  \int\limits_0^\infty  \Phi_{\{\epsilon_1, \epsilon_2,\epsilon_3,\epsilon_4\}}(x) x^{-s} (8x)^{\sum_{i=1}^4\frac{u_i}{2}}\, dx\\
&\quad \times 
Z_4(\tfrac{1}{2}+\epsilon_1\a_1+u_1-s,\tfrac{1}{2}+\epsilon_2\a_2+u_2-s,\tfrac{1}{2}+ \epsilon_3\a_3+u_3-s,\tfrac{1}{2}+\epsilon_4\a_4+u_4-s,s) .
\end{align*}
Here $Z_4$ is defined in Lemma \ref{lem:iden}. We begin with $\colz_1$ by taking the line of the integration in $u_1$ to $\Re(u_1) = -\frac{1}{10}+ \frac{10}{L}$, passing simple poles at $u_1 = 0,-\epsilon_1 \a_1$, so
\begin{align}
 \colz_1 = I_1 + I_2 + I_3,
\label{equ:Z1-first}
\end{align}
where 
\begin{align*}
I_1 &= \frac{1}{(2\pi i)^3}  \int\limits_{(\frac{1}{10})} \int\limits_{(\frac{1}{10})} \int\limits_{(\frac{1}{10})}  \res{u_1=0} \colf_1(\boldsymbol{u}; \boldsymbol{\a};\boldsymbol{\epsilon})
\prod_{i=1}^4\frac{1}{u_i}     
\prod_{1 \leq i\leq j \leq 4} \zeta_2(1+\epsilon_i\a_i+\epsilon_j\a_j+u_i+u_j)  \, du_2\,du_3\, du_4,\\
I_2 &= \frac{1}{(2\pi i)^3}  \int\limits_{(\frac{1}{10})} \int\limits_{(\frac{1}{10})} \int\limits_{(\frac{1}{10})}   \res{u_1=-\epsilon_1\a_1} 
\colf_1(\boldsymbol{u}; \boldsymbol{\a};\boldsymbol{\epsilon})
\prod_{i=1}^4\frac{1}{u_i}     
\prod_{1 \leq i\leq j \leq 4} \zeta_2(1+\epsilon_i\a_i+\epsilon_j\a_j+u_i+u_j)  \, du_2\,du_3\, du_4,\\
I_3 &= \frac{1}{(2\pi i)^4}  \int\limits_{(\frac{1}{10})} \cdots \int\limits_{(-\frac{1}{10}+ \frac{10}{L})}  \colw_1(\boldsymbol{u};\boldsymbol{\a};\boldsymbol{\epsilon}) \, d\boldu.
\end{align*}
We are not able to bound $I_3$ trivially, as this would give a contribution $\asymp X^{1+\frac{1}{10}}\asymp X^{\frac{11}{10}}$. As we will see though, this integral cancels exactly with an integral arising from $\colz_2$. For $\colz_2$, we shift contours in a similar way to $\colz_1$ and take the line of the integration in $u_1$ to $\Re(u_1) = -\frac{1}{10}+\frac{10}{L} $, passing simple poles  at $u_1 =0, -\epsilon_1 \a_1, s-\epsilon_1 \a_1$. By Lemma \ref{lem:DirichletSeriesSuma}, the integration on $\Re(u_1) = -\frac{1}{10}+\frac{10}{L} $ is $O\pth{X (\log X)^6}$, and so 
\begin{align}
\colz_2 = J_1 + J_2 + J_3 + O(X (\log X)^6),
\label{equ:Z2-first}
\end{align}
where 
\begin{align*}
J_1 &=  \frac{1}{(2 \pi i)^4} \int\limits_{(\frac{1}{10})}  \int\limits_{(\frac{1}{10}+ \frac{5}{L})} \int\limits_{(\frac{1}{10}+ \frac{5}{L})}  \int\limits_{(\frac{1}{10}+ \frac{5}{L})} 
\res{u_1=0}
\prod_{i=1}^4 \frac{1}{u_i}  
\zeta_2(1-2s) \prod_{1\leq i \leq j \leq 4} \zeta_2(1+ \epsilon_i\a_i  + \epsilon_j\a_j +u_i + u_j)  \\
&\quad \times \prod_{1 \leq i \leq 4} \frac{\zeta_2(1+ \epsilon_i\a_i+u_i-s)}{\zeta_2(1+s+\epsilon_i\a_i+u_i )}\colf_2(s;\boldsymbol{u}; \boldsymbol{\a};\boldsymbol{\epsilon}) \, du_2\,du_3\,du_4 \, ds,\\
J_2 &=  \frac{1}{(2 \pi i)^4} \int\limits_{(\frac{1}{10})}  \int\limits_{(\frac{1}{10}+ \frac{5}{L})} \int\limits_{(\frac{1}{10}+ \frac{5}{L})} \int\limits_{(\frac{1}{10}+ \frac{5}{L})} 
\res{u_1=-\epsilon_1 \a_1}
\prod_{i=1}^4 \frac{1}{u_i}  
\zeta_2(1-2s) \prod_{1\leq i \leq j \leq 4} \zeta_2(1+ \epsilon_i\a_i  + \epsilon_j\a_j +u_i + u_j)  \\
&\quad \times \prod_{1 \leq i \leq 4} \frac{\zeta_2(1+ \epsilon_i\a_i+u_i-s)}{\zeta_2(1+s+\epsilon_i\a_i+u_i )} \colf_2(s;\boldsymbol{u}; \boldsymbol{\a};\boldsymbol{\epsilon})\, du_2\,du_3\,du_4 \, ds,\\
J_3 &=  \frac{1}{(2 \pi i)^4} \int\limits_{(\frac{1}{10})}  \int\limits_{(\frac{1}{10}+ \frac{5}{L})} \int\limits_{(\frac{1}{10}+ \frac{5}{L})} \int\limits_{(\frac{1}{10}+ \frac{5}{L})} 
\res{u_1=s-\epsilon_1 \a_1}
\colw_2(s;\boldsymbol{u}-s;\boldsymbol{\a};\boldsymbol{\epsilon}) \, du_2\,du_3\,du_4 \, ds.
\end{align*}

\begin{lem} 
For $r =1,2,3,4$, we have the following identities: 
\begin{enumerate}[label=\normalfont{(\roman*)}]
\item 
$ \displaystyle 
\sum_{\epsilon_r = \pm 1}
(8X)^{-\frac{1-\epsilon_r}{2}\a_r} 
\lambda_r^{\frac{1-\epsilon_r}{2}}\cdot \epsilon_r \cdot  \underset{u_r=-\epsilon_r\a_r}{\operatorname{Value}\,}  
\colf_1(\boldsymbol{u}; \boldsymbol{\a};\boldsymbol{\epsilon}) 
= 0.
$
\item 
$ \displaystyle 
\displaystyle 
\sum_{\epsilon_r = \pm 1}
(8X)^{-\frac{1-\epsilon_r}{2}\a_r} 
\lambda_r^{\frac{1-\epsilon_r}{2}}\cdot  
\epsilon_r \cdot \underset{u_r=-\epsilon_r\a_r}{\operatorname{Value}\,}
\colf_2(s;\boldsymbol{u}; \boldsymbol{\a};\boldsymbol{\epsilon}) 
= 0.
$
\item 
Write 
\begin{align*}
t_d^r (\epsilon_r) &: =  (8X)^{-\frac{1-\epsilon_r}{2}\a_r} 
\lambda_r^{\frac{1-\epsilon_r}{2}}\cdot 
\underset{u_r=-s-\epsilon_r \a_r}{\operatorname{Value}\,} 
\colf_1(\boldsymbol{u}; \boldsymbol{\a};\boldsymbol{\epsilon}),\\
t_o^r (\epsilon_r) &: =  (8X)^{-\frac{1-\epsilon_r}{2}\a_r} 
\lambda_r^{\frac{1-\epsilon_r}{2}}\cdot 
\underset{u_r=s-\epsilon_r \a_r}{\operatorname{Value}\,} 
\colf_2(s;\boldsymbol{u}; \boldsymbol{\a};\boldsymbol{\epsilon}).
\end{align*}
Then 
\begin{align*}
t_o^r (\epsilon_r) = 2 t_d^r (-\epsilon_r).
\end{align*}
\end{enumerate}
\label{lem-canceli-iii}
\end{lem}

\begin{proof}
It suffices to prove the case $r=1$ due to the symmetry of $u_1,u_2,u_3,u_4$. 
We begin with the following identities. First, from \eqref{gaDef} and \eqref{def-lambda}, we have
\begin{equation}
g_{\a_r} (z-\a_r) = \pi^{-z}  \lambda_r(\alpha_r) \frac{\Gamma(\frac{1/2+z}{2})}{\Gamma(\frac{1/2-z}{2})} g_{-\a_r} (-z+\a_r).
\label{g-iden-1}
\end{equation}
In particular, taking $z=0, r=1$ gives
\[
g_{\a_1} (-\a_1)=  \lambda_1 g_{-\a_1} (\a_1).
\]
Next, the definition \eqref{tranf-phi} implies that
\[
\Phi_{\{1,\epsilon_2,\epsilon_3,\epsilon_4\} }(x)
x^{\frac{- \a_1}{2}} = \Phi_{\{-1,\epsilon_2,\epsilon_3,\epsilon_4\} }(x) x^{\frac{\a_1}{2}}.
\]
Combining the above two displays with \eqref{def-colf1}, we find 
\begin{align*}
&\underset{u_1=-\a_1}{\operatorname{Value}\,}\mathcal{F}_1(\boldsymbol{u};\boldsymbol{\a};1,\epsilon_2,\epsilon_3,\epsilon_4) \\
&= \frac{4X}{\pi^2} \lambda_1
g_{-\a_1}(\a_1)(8X)^{-\frac{\a_1}{2}} \prod_{i=2}^4
g_{\epsilon_i\a_i}(u_i)(8X)^{\frac{u_i}{2}} \int\limits_{-\infty}^\infty  \Phi_{\{-1,\epsilon_2,\epsilon_3,\epsilon_4\} }(x)
x^{\frac{\a_1}{2} + \frac{u_2}{2} + \frac{u_3}{2}+ \frac{u_4}{2} } \, dx\\
&\quad \times \mathcal{H}_2(\tfrac{1}{2}, \tfrac{1}{2}+ \epsilon_2 \a_2 + u_2 , \tfrac{1}{2}+ \epsilon_3 \a_3 + u_3,\tfrac{1}{2}+ \epsilon_4 \a_4 + u_4) \\
&= \lambda_1(8X)^{-\a_1}\ \underset{u_1=\a_1}{\operatorname{Value}\,}\mathcal{F}_1(\boldsymbol{u};\boldsymbol{\a};-1,\epsilon_2,\epsilon_3,\epsilon_4),
\end{align*}
which proves (i). A similar argument using the definition of $\colf_2$ gives (ii).


To prove (iii), note first that 
\begin{align*}
\underset{u_1=s-\epsilon_1 \a_1}{\operatorname{Value}\,} 
\colf_2(s;\boldsymbol{u}; \boldsymbol{\a};\boldsymbol{\epsilon})
& = X^{1+\sum_{i=2}^4\frac{u_i}{2}} X^{\frac{-s-\epsilon_1 \a_1}{2}}   
g_{\epsilon_1\a_1}(s-\epsilon_1 \a_1)
\left(\frac{8}{\pi}\right)^{-s}\frac{\Gamma(\frac{1}{4}-\frac{s}{2})}{\Gamma(\frac{1}{4}+ \frac{s}{2})}\prod_{i=2}^4 g_{\epsilon_i\a_i}(u_i)   \\
&\quad \times \int\limits_0^\infty  \Phi_{\{\epsilon_1, \epsilon_2,\epsilon_3,\epsilon_4\}}(x) x^{-s} (8x)^{\sum_{i=2}^4\frac{u_i}{2}} (8x)^{\frac{s}{2}-\frac{\epsilon_1 \a_1}{2}}\, dx\\
&\quad \times 
Z_4(\tfrac{1}{2},\tfrac{1}{2}+\epsilon_2\a_2+u_2-s,\tfrac{1}{2}+ \epsilon_3\a_3+u_3-s,\tfrac{1}{2}+\epsilon_4\a_4+u_4-s,s) 
\end{align*}
Then \eqref{g-iden-1} with $\a_r = \epsilon_1 \a_1$ and $z = s$ gives
\[
g_{\epsilon_1\a_1}(s-\epsilon_1 \a_1) = \pi^{-s} \lambda_1(\epsilon_1\a_1) \frac{\Gamma(\frac{1}{4}+ \frac{s}{2})}{\Gamma(\frac{1}{4}-\frac{s}{2})} g_{-\epsilon_1\a_1}(-s+\epsilon_1 \a_1).
\]
Note also that $\lambda_1(\epsilon_1\a_1) = \lambda_1^{\ep_1}$. By Lemma \ref{lem:iden}(ii) with $z_1 = s$ and $z_i = \epsilon_i \a_i + u_i $ for $i=2,3,4$, we have 
\begin{align*}
 \underset{u_1=s-\epsilon_1 \a_1}{\operatorname{Value}\,} \colf_2(s;\boldsymbol{u}; \boldsymbol{\a};\boldsymbol{\epsilon})
& = \lambda_1^{\epsilon_1}\frac{8}{\pi^2 } X^{1+\sum_{i=2}^4\frac{u_i}{2}+\frac{-s-\epsilon_1 \a_1}{2}}   
g_{-\epsilon_1\a_1}(-s+\epsilon_1 \a_1)
\prod_{i=2}^4 g_{\epsilon_i\a_i}(u_i)   \\
&\quad \times \int\limits_0^\infty \Phi_{\{\epsilon_1, \epsilon_2,\epsilon_3,\epsilon_4\}}(x)  (8x)^{\frac{-s-\epsilon_1\a_1}{2}+\sum_{i=2}^4\frac{u_i}{2}} \, dx \\
&\quad \times 
\colh_2(\tfrac{1}{2}-s, \tfrac{1}{2}+ \epsilon_2 \a_2 + u_2,\tfrac{1}{2}+ \epsilon_3 \a_3 + u_3, \tfrac{1}{2}+ \epsilon_4 \a_4 + u_4).
\end{align*}
On the other hand, by \eqref{def-colf1}, we have
\begin{align*}
&\underset{u_1=-s-\epsilon_1 \a_1}{\operatorname{Value}\,} \colf_1(\boldsymbol{u}; \boldsymbol{\a};\boldsymbol{\epsilon})\\
&=   \frac{4}{\pi^2}  X^{1+\sum_{i=2}^4\frac{u_i}{2}+\frac{-s-\epsilon_1 \a_1}{2}} 
g_{\epsilon_1\a_1}(-s-\epsilon_1\a_1)
\prod_{i=2}^4g_{\epsilon_i\a_i}(u_i) \int\limits_{-\infty}^\infty  \Phi_{\{\epsilon_1, \epsilon_2,\epsilon_3,\epsilon_4\}}(x)
(8x)^{\frac{-s-\epsilon_1\a_1}{2} + \sum_{i=2}^4\frac{u_i}{2} } \, dx\\
&\quad \times \mathcal{H}_2(\tfrac{1}{2}-s,\tfrac{1}{2}+\epsilon_2\a_2+u_2,\tfrac{1}{2}+ \epsilon_3\a_3+u_3, \tfrac{1}{2}+\epsilon_4\a_4+u_4).
\end{align*}
Since $\Phi(x)$ is supported on $(0,\infty)$, the integrals in $x$ in the above two displays are equal. Also, $\Phi_{\{\epsilon_1, \epsilon_2,\epsilon_3,\epsilon_4\}} =\Phi_{\{-\epsilon_1, \epsilon_2,\epsilon_3,\epsilon_4\}} x^{\epsilon_1\a_1}$, and so
\begin{align}
&
(8X)^{-\frac{1-\epsilon_1}{2}\a_1} 
\lambda_1^{\frac{1-\epsilon_1}{2}}\cdot 
\underset{u_1=s-\epsilon_1 \a_1}{\operatorname{Value}\,} 
\colf_2(s;\boldsymbol{u}; \boldsymbol{\a};\boldsymbol{\epsilon})\nonumber\\
&=
8^{-\frac{\a_1}{2}}
X^{-\frac{\a_1}{2}} 
\lambda_1^{\frac{1+\epsilon_1}{2}} \frac{8}{\pi^2 } X^{1+\sum_{i=2}^4\frac{u_i}{2}+\frac{-s}{2}}   
g_{-\epsilon_1\a_1}(-s+\epsilon_1 \a_1)
\prod_{i=2}^4 g_{\epsilon_i\a_i}(u_i) \nonumber\\
&\quad \times 
 \int\limits_0^\infty  \Phi_{\{\epsilon_1, \epsilon_2,\epsilon_3,\epsilon_4\}}(x) 8^{-\frac{s}{2}+\sum_{i=2}^4\frac{u_i}{2}} x^{\frac{-s-\epsilon_1\a_1}{2}+\sum_{i=2}^4\frac{u_i}{2}} \, dx \nonumber\\
&\quad \times 
\colh_2(\tfrac{1}{2}-s, \tfrac{1}{2}+ \epsilon_2 \a_2 + u_2,\tfrac{1}{2}+ \epsilon_3 \a_3 + u_3, \tfrac{1}{2}+ \epsilon_4 \a_4 + u_4),
\label{equ:res-s-epsion-iii}
\end{align}
and 
\begin{align}
& 
(8X)^{-\frac{1-\epsilon_1}{2}\a_1} 
\lambda_1^{\frac{1-\epsilon_1}{2}}\cdot  \underset{u_1=-s- \epsilon_1 \a_1}{\operatorname{Value}\,}     \colf_1(\boldsymbol{u}; \boldsymbol{\a};\boldsymbol{\epsilon})\nonumber
\\
&= 
8^{-\frac{\a_1}{2}}
X^{-\frac{\a_1}{2}} 
\lambda_1^{\frac{1-\epsilon_1}{2}} \frac{4}{\pi^2}  X^{1+\sum_{i=2}^4\frac{u_i}{2}+\frac{-s}{2}} 
g_{\epsilon_1\a_1}(-s-\epsilon_1\a_1) 
\prod_{i=2}^4g_{\epsilon_i\a_i}(u_i) 
\nonumber\\
& \quad \times   
\int\limits_{-\infty}^\infty  \Phi_{\{-\epsilon_1, \epsilon_2,\epsilon_3,\epsilon_4\}}(x) 8^{-\frac{s}{2}+\sum_{i=2}^4\frac{u_i}{2}}
x^{\frac{-s+\epsilon_1\a_1}{2} + \sum_{i=2}^4\frac{u_i}{2} } \, dx\nonumber\\
&\quad \times \mathcal{H}_2(\tfrac{1}{2}-s,\tfrac{1}{2}+\epsilon_2\a_2+u_2,\tfrac{1}{2}+ \epsilon_3\a_3+u_3, \tfrac{1}{2}+\epsilon_4\a_4+u_4).
\label{equ:res-s-epsion-iii-1}
\end{align}
The assertion (iii) with $r=1$ now follows from \eqref{equ:res-s-epsion-iii} and \eqref{equ:res-s-epsion-iii-1}.
\end{proof}

By Lemma \ref{lem-canceli-iii}(i) with $r=1$, we have 
\begin{align*}
& \sum_{\epsilon_1 = \pm 1} (8X)^{-\frac{1-\epsilon_1}{2}\a_1} 
\lambda_1^{\frac{1-\epsilon_1}{2}} \cdot    \res{u_1=-\epsilon_1\a_1} 
\colf_1(\boldsymbol{u}; \boldsymbol{\a};\boldsymbol{\epsilon})
\prod_{i=1}^4\frac{1}{u_i}     
\prod_{1 \leq i\leq j \leq 4} \zeta_2(1+\epsilon_i\a_i+\epsilon_j\a_j+u_i+u_j)\\
&=-\frac{1}{4}\prod_{i=2}^4\frac{1}{u_i} \frac{1}{\a_1}
\prod_{2\leq i\leq j \leq 4} \zeta_2(1+\epsilon_i\a_i+\epsilon_j\a_j+u_i+u_j)
\prod_{2 \leq j \leq 4} \zeta_2(1+\epsilon_j\a_j+u_j) \\
& \quad\times  \sum_{\epsilon_1 = \pm 1} (8X)^{-\frac{1-\epsilon_1}{2}\a_1} 
\lambda_1^{\frac{1-\epsilon_1}{2}}
\frac{1}{\epsilon_1}
\underset{u_1=-\epsilon_1\a_1}{\operatorname{Value}\,}  
\colf_1(\boldsymbol{u}; \boldsymbol{\a};\boldsymbol{\epsilon})\\
&=0,
\end{align*}
which implies that
\begin{equation}
\sum_{\epsilon_1 = \pm 1} (8X)^{-\frac{1-\epsilon_1}{2}\a_1} 
\lambda_1^{\frac{1-\epsilon_1}{2}} \cdot I_2 = 0, 
\label{equ:I2J2=0}
\end{equation}
The key observation above is that $\epsilon_1$ only occurs as the coefficient of $u_1$, and so when taking residues at $u_1 = -\epsilon_1\a_1$, the expressions in the integrals no longer depend on $\epsilon_1$. By a similar argument, Lemma \ref{lem-canceli-iii}(ii) with $r=1$ implies that
\begin{equation}
\sum_{\epsilon_1 = \pm 1} (8X)^{-\frac{1-\epsilon_1}{2}\a_1} 
\lambda_1^{\frac{1-\epsilon_1}{2}}  \cdot J_2 = 0.
\label{equ:I2J2=0-ii}
\end{equation}

We now introduce some more notation and write
\begin{align*}
&
h_{o}^1(\epsilon_1)\\
&:=
\res{u_1=s-\epsilon_1 \a_1}
\prod_{i=1}^4 \frac{1}{u_i}  
\zeta_2(1-2s) \prod_{1\leq i \leq j \leq 4} \zeta_2(1+ \epsilon_i\a_i  + \epsilon_j\a_j +u_i + u_j) 
\prod_{1 \leq i \leq 4} \frac{\zeta_2(1+ \epsilon_i\a_i+u_i-s)}{\zeta_2(1+s+\epsilon_i\a_i+u_i )} \\
&=\frac{1}{2(s-\epsilon_1\a_1)}   \prod_{i=2}^4 \frac{1}{u_i}  
\zeta_2(1-2s) 
\prod_{2\leq i \leq j \leq 4} \zeta_2(1+ \epsilon_i\a_i  + \epsilon_j\a_j +u_i + u_j)  
\prod_{2 \leq i \leq 4} \zeta_2(1+ \epsilon_i\a_i+u_i-s),
\end{align*}
and write 
\begin{equation*}
h_{d}^1(\epsilon_1)
:=
\underset{u_1 =-s-\epsilon_1 \a_1}{\operatorname{Value}\,}  
\prod_{i=1}^4\frac{1}{u_i}     
\prod_{1 \leq i\leq j \leq 4} \zeta_2(1+\epsilon_i\a_i+\epsilon_j\a_j+u_i+u_j).
\end{equation*}
One may check that
\begin{align}
h_{o}^1(\epsilon_1) =  - \frac{1}{2} h_{d}^1(-\epsilon_1),
\label{equ:hohd}
\end{align}
and by symmetry,
\begin{align}
h_{o}^r(\epsilon_r) =  - \frac{1}{2} h_{d}^r(-\epsilon_r).
\label{equ:hohd-2}
\end{align}
where $r=1,2,3,4$, and $h_{o}^r(\epsilon_r), h_{d}^r(\epsilon_r)$ are defined analogously to $h_{o}^1(\epsilon_1), h_{d}^1(\epsilon_1) $.
By \eqref{equ:hohd} and  Lemma \ref{lem-canceli-iii}(iii) with $r=1$, we have  
\begin{align*}
&   \sum_{\epsilon_1 = \pm 1} (8X)^{-\frac{1-\epsilon_1}{2}\a_1} 
\lambda_1^{\frac{1-\epsilon_1}{2}}  \left( 
\underset{u_1 =-s-\epsilon_1\a_1}{\operatorname{Value}\,}
\colw_1(\boldsymbol{u};\boldsymbol{\a};\boldsymbol{\epsilon}) + \res{u_1=s-\epsilon_1 \a_1} 
\colw_2(s;\boldsymbol{u}-s;\boldsymbol{\a};\boldsymbol{\epsilon})\right) \\
& =  \sum_{\epsilon_1 = \pm 1}  (  h_{d}^1(\epsilon_1) t_{d}^1(\epsilon_1)
+ h_{o}^1(\epsilon_1)   t_{o}^1(\epsilon_1)  )\\
&= \sum_{\epsilon_1 = \pm 1}  (  h_{d}^1(\epsilon_1) t_{d}^1(\epsilon_1)
- h_{d}^1(-\epsilon_1)   t_{d}^1(-\epsilon_1)  )\\
&=0.
\end{align*}
By the change of variable $u_1 \mapsto  -s- \epsilon_1\a_1$ in $I_3$, combined with the above formula, we find that 
\[
\sum_{\epsilon_1 = \pm 1} (8X)^{-\frac{1-\epsilon_1}{2}\a_1} 
\lambda_1^{\frac{1-\epsilon_1}{2}}(I_3+ J_3) = 0.
\]
Together with \eqref{equ:I2J2=0}, \eqref{equ:I2J2=0-ii}, \eqref{equ:Z1-first} and \eqref{equ:Z2-first}, we deduce that  
\begin{align}
&    \sum_{\substack{\epsilon_i = \pm 1\\i=1,2,3,4}} \prod_{i=1}^4
(8X)^{-\frac{1-\epsilon_i}{2}\a_i} 
\lambda_i^{\frac{1-\epsilon_i}{2}} \cdot  ( \colz_1 +   \colz_2)\nonumber\\
&=
\sum_{\substack{\epsilon_i = \pm 1\\i=1,2,3,4}} \prod_{i=1}^4
(8X)^{-\frac{1-\epsilon_i}{2}\a_i} 
\lambda_i^{\frac{1-\epsilon_i}{2}} \cdot (I_1 + J_1) +O(X (\log X)^6).
\label{equ:cross-second-1}
\end{align}

At this point, we are ready to continue shifting contours. In $I_1$, we take the line of the integration in $u_2$ to $\Re(u_2) = -\frac{1}{10}+ \frac{10}{L}$, passing simple poles at $u_2 =0, -\epsilon_2\a_2, -\epsilon_1 \a_1 -\epsilon_2 \a_2 $. Let $I_{1,1}, I_{1,2}, I_{1,3}$ denote the contribution of each residue, respectively, and let $I_{1,4}$ denote the integration on the line $\Re(u_2) = -\frac{1}{10}+ \frac{10}{L}$. 

We proceed similarly for $J_1$ and take the line of the integration in $u_2$ to  $\Re(u_2) = -\frac{1}{10}+ \frac{10}{L}$, passing simple poles at $u_2 = 0,-\epsilon_2\a_2, -\epsilon_1 \a_1 -\epsilon_2 \a_2, s-\epsilon_2 \a_2$, the contributions of which we denoted by $J_{1,1}, J_{1,2}, J_{1,3}, J_{1,4}$, respectively. One may check that the integral on $\Re(u_2) = -\frac{1}{10}+ \frac{10}{L}$ is $O\pth{X^{1- \frac{1}{21}}}$.

Arguing in a similar manner as above using Lemma \ref{lem-canceli-iii}(i) and (ii) with $r=2, u_1 =0$, we find that
\begin{align}
\sum_{\epsilon_2 = \pm 1} (8X)^{-\frac{1-\epsilon_2}{2}\a_2} 
\lambda_2^{\frac{1-\epsilon_2}{2}} \cdot  I_{1,2} &= 0,\nonumber\\
\sum_{\epsilon_2 = \pm 1} (8X)^{-\frac{1-\epsilon_2}{2}\a_2} 
\lambda_2^{\frac{1-\epsilon_2}{2}} \cdot   J_{1,2} &= 0,
\label{equ:cross-similar-1}
\end{align}
Moreover, Lemma \ref{lem-canceli-iii}(iii) and \eqref{equ:hohd-2} with $r=2, u_1 =0$ gives
\begin{equation}
\sum_{\epsilon_2 = \pm 1} (8X)^{-\frac{1-\epsilon_2}{2}\a_2} 
\lambda_2^{\frac{1-\epsilon_2}{2}} \cdot  (I_{1,4} +J_{1,4}) = 0.
\label{equ:cross-similar-2}
\end{equation}

Next, we show that the contributions from $I_{1,3}$ and $J_{1,3}$ vanish. To do, we need the following lemma.


\begin{lem}
\label{lem:cross-second-general}
Let $r,r' = 1,2,3,4$ and $r \neq r'$.
Write 
\begin{align*}
f^{r,r'}_1(\epsilon_r, \epsilon_{r'}) &:=(8X)^{-\frac{1-\epsilon_r}{2}\a_r -\frac{1-\epsilon_{r'}}{2}\a_{r'}} 
\lambda_r^{\frac{1-\epsilon_r}{2}} \lambda_{r'}^{\frac{1-\epsilon_{r'}}{2}}\cdot  \underset{u_{r'}=-\epsilon_r\a_r - \epsilon_{r'}\a_{r'}}{\operatorname{Value}\,} 
\underset{u_r = 0}{\operatorname{Value}\,} 
\colf_1(\boldsymbol{u}; \boldsymbol{\a};\boldsymbol{\epsilon}), \\
f^{r,r'}_2(\epsilon_r, \epsilon_{r'}) &:=(8X)^{-\frac{1-\epsilon_r}{2}\a_r -\frac{1-\epsilon_{r'}}{2}\a_{r'}} 
\lambda_r^{\frac{1-\epsilon_r}{2}} \lambda_{r'}^{\frac{1-\epsilon_{r'}}{2}}\cdot  \underset{u_{r'}=-\epsilon_1\a_1 - \epsilon_{r'}\a_{r'}}{\operatorname{Value}\,}  
\underset{u_r = 0}{\operatorname{Value}\,} 
\colf_2(s;\boldsymbol{u}; \boldsymbol{\a};\boldsymbol{\epsilon}).
\end{align*}
Then 
\begin{enumerate}[label=\normalfont{(\roman*)}]
\item 
$ 
f^{r,r'}_1(\epsilon_r, \epsilon_{r'}) =  f^{r,r'}_1(-\epsilon_r, -\epsilon_{r'}).
$
\item 
$ 
f^{r,r'}_2(\epsilon_r, \epsilon_{r'}) =  f^{r,r'}_2(-\epsilon_r, -\epsilon_{r'}).
$
\end{enumerate}
\end{lem}
\begin{proof}
We prove the case $r=1,r'=2$ and other cases follow by the symmetry. For brevity, write $f_1(\epsilon_1,\epsilon_2) := f_1^{1,2}(\epsilon_1,\epsilon_2)$, $ f_2 := f_2^{1,2}(\epsilon_1,\epsilon_2) $.
We have
\begin{align*}
f_1(\epsilon_1, \epsilon_2) 
& =
(8X)^{-\frac{\a_1}{2} -\frac{\a_2}{2}} 
\lambda_1^{\frac{1-\epsilon_1}{2}} \lambda_2^{\frac{1-\epsilon_2}{2}} 
\frac{4X}{\pi^2} g_{\epsilon_2\a_2}(-\epsilon_1\a_1 - \epsilon_2\a_2)
\prod_{i=3}^4 g_{\epsilon_i\a_i}(u_i)
(8 X)^{\sum_{i=3}^4\frac{u_i}{2}} \\
&\quad \times \int\limits_{-\infty}^\infty  \Phi_{\{\epsilon_1, \epsilon_2,\epsilon_3,\epsilon_4\}}(x)
x^{\sum_{i=3}^4\frac{u_i}{2} } x^{\frac{-\epsilon_1\a_1 - \epsilon_2\a_2}{2} } \, dx\\
&\quad \times \mathcal{H}_2(\tfrac{1}{2}+\epsilon_1\a_1,\tfrac{1}{2}-\epsilon_1\a_1,\tfrac{1}{2}+ \epsilon_3\a_3+u_3, \tfrac{1}{2}+\epsilon_4\a_4+u_4).
\end{align*}
This implies that
\begin{align*}
f_1(-\epsilon_1, -\epsilon_2) 
& =
(8X)^{-\frac{\a_1}{2} -\frac{\a_2}{2}} 
\lambda_1^{\frac{1+\epsilon_1}{2}} \lambda_2^{\frac{1+\epsilon_2}{2}}   
\frac{4X}{\pi^2} g_{-\epsilon_2\a_2}(\epsilon_1\a_1 +\epsilon_2\a_2)
\prod_{i=3}^4 g_{\epsilon_i\a_i}(u_i)
(8 X)^{\sum_{i=3}^4\frac{u_i}{2}} \\
&\quad \times \int\limits_{-\infty}^\infty  \Phi_{\{-\epsilon_1, -\epsilon_2,\epsilon_3,\epsilon_4\}}(x)
x^{\sum_{i=3}^4\frac{u_i}{2} } x^{\frac{\epsilon_1\a_1 + \epsilon_2\a_2}{2} } \, dx\\
&\quad \times \mathcal{H}_2(\tfrac{1}{2}-\epsilon_1\a_1,\tfrac{1}{2}+\epsilon_1\a_1,\tfrac{1}{2}+ \epsilon_3\a_3+u_3, \tfrac{1}{2}+\epsilon_4\a_4+u_4).
\end{align*}
By   \eqref{g-iden-1}  with  $\a_r = -\epsilon_2\a_2, z= \epsilon_1\a_1$, we get 
\begin{align*}
g_{-\epsilon_2\a_2}(\epsilon_1\a_1 +\epsilon_2\a_2)
&= 
\pi^{-\epsilon_1\a_1}  \lambda_2(-\epsilon_2\a_2) \frac{\Gamma(\frac{1/2+\epsilon_1\a_1}{2})}{\Gamma(\frac{1/2-\epsilon_1\a_1}{2})} g_{\epsilon_2\a_2} (-\epsilon_1\a_1-\epsilon_2\a_2)\\
&= \lambda_1^{-\epsilon_1}\lambda_2^{-\epsilon_2}  g_{\epsilon_2\a_2} (-\epsilon_1\a_1-\epsilon_2\a_2).
\end{align*}
Note that $\Phi_{\{-\epsilon_1, -\epsilon_2,\epsilon_3,\epsilon_4\}} =\Phi_{\{\epsilon_1, \epsilon_2,\epsilon_3,\epsilon_4\}} x^{-\epsilon_1\a_1-\epsilon_2\a_2}$ and that the variables of $\mathcal{H}_2$ are symmetric. From this, it follows that $f_1(-\epsilon_1, -\epsilon_2)=f_1(\epsilon_1, \epsilon_2)$. For $f_2$, we have
\begin{align*}
f_2(\epsilon_1, \epsilon_2) &=(8X)^{-\frac{\a_1}{2} -\frac{\a_2}{2}} 
\lambda_1^{\frac{1-\epsilon_1}{2}} \lambda_2^{\frac{1-\epsilon_2}{2}}  \\
&\quad \times X^{1+\sum_{i=3}^4\frac{u_i}{2}}     X^{-s}
g_{\epsilon_2\a_2}(-\epsilon_1\a_1 - \epsilon_2\a_2) \prod_{i=3}^4 g_{\epsilon_i\a_i}(u_i)\left(\frac{8}{\pi}\right)^{-s}\frac{\Gamma(\frac{1}{4}-\frac{s}{2})}{\Gamma(\frac{1}{4}+ \frac{s}{2})}  \\
&\quad \times \int\limits_0^\infty  \Phi_{\{\epsilon_1, \epsilon_2,\epsilon_3,\epsilon_4\}}(x) x^{-s} (8x)^{\sum_{i=3}^4\frac{u_i}{2}}x^{\frac{-\epsilon_1\a_1 - \epsilon_2\a_2}{2}}\, dx\\
&\quad \times 
Z_4(\tfrac{1}{2}+\epsilon_1\a_1-s,\tfrac{1}{2}-\epsilon_1\a_1-s,\tfrac{1}{2}+ \epsilon_3\a_3+u_3-s,\tfrac{1}{2}+\epsilon_4\a_4+u_4-s,s)
\end{align*}
The identity (ii) now follows by a similar argument as in the proof of (i).  
\end{proof}

Note that 
\begin{align}
&\res{u_2=-\epsilon_1\a_1-\epsilon_2\a_2}\ \res{u_1=0} 
\prod_{i=1}^4\frac{1}{u_i}     
\prod_{1 \leq i\leq j \leq 4} \zeta_2(1+\epsilon_i\a_i+\epsilon_j\a_j+u_i+u_j) \nonumber\\
&= -\frac{1}{2}\frac{1}{(\epsilon_1\a_1+\epsilon_2\a_2)u_3u_4} 
\prod_{3 \leq j \leq 4 }\zeta_2(1+\epsilon_i\a_i+\epsilon_j\a_j+u_i+u_j)
\nonumber\\
&\quad \times 
\zeta_2(1+2\epsilon_1\a_1) \zeta_2(1-2\epsilon_1\a_1)\prod_{3 \leq j \leq 4 } \zeta_2(1+\epsilon_1\a_1+\epsilon_j\a_j+u_j)\zeta_2(1-\epsilon_1\a_1+\epsilon_j\a_j+u_j).
\label{equ:cancel-u2u1}
\end{align}
Combining this with Lemma \ref{lem:cross-second-general}(i) with $r=1, r'=2$, it follows that 
\begin{align*}
&\sum_{\epsilon_1, \epsilon_2 = \pm 1}f_1 (\epsilon_1, \epsilon_2) \res{u_2=-\epsilon_1\a_1-\epsilon_2\a_2}\ \res{u_1=0} 
\prod_{i=1}^4\frac{1}{u_i}     
\prod_{1 \leq i\leq j \leq 4} \zeta_2(1+\epsilon_i\a_i+\epsilon_j\a_j+u_i+u_j) = 0.
\end{align*}
In fact,  note that the zeta factors in \eqref{equ:cancel-u2u1} do not change with $\epsilon_1$ replaced by $-\epsilon_1$, and $\epsilon_2$ does not appear in these zeta factors, so summands above with $(\epsilon_1, \epsilon_2) = (1,1), (-1,-1)$ cancel out, so are $(\epsilon_1, \epsilon_2) = (1,-1), (-1,1)$. In other words, we have obtained  
\begin{align}
&\sum_{\epsilon_1, \epsilon_2 = \pm 1} 
(8X)^{-\frac{1-\epsilon_1}{2}\a_1 -\frac{1-\epsilon_2}{2}\a_2} 
\lambda_1^{\frac{1-\epsilon_1}{2}} \lambda_2^{\frac{1-\epsilon_2}{2}} \cdot I_{1,3} = 0.
\label{equ:cross-similar-21}
\end{align}
Similarly, by Lemma \ref{lem:cross-second-general}(ii) with $r=1,r'=2$, we have 
\begin{align*}
& \sum_{\epsilon_1, \epsilon_2 = \pm 1}f_2 (\epsilon_1, \epsilon_2)\res{u_2=-\epsilon_1\a_1-\epsilon_2\a_2} \res{u_1=0} \prod_{i=1}^4 \frac{1}{u_i}  
\zeta_2(1-2s) \prod_{1\leq i \leq j \leq 4} \zeta(1+ \epsilon_i\a_i  + \epsilon_j\a_j +u_i + u_j)  \\
&\quad \times \prod_{1 \leq i \leq 4} \frac{\zeta_2(1+ \epsilon_i\a_i+u_i-s)}{\zeta_2(1+s+\epsilon_i\a_i+u_i )}\\
& = -\frac{1}{2}\sum_{\epsilon_1, \epsilon_2 = \pm 1}f_2 (\epsilon_1, \epsilon_2) \frac{1}{(\epsilon_1\a_1+ \epsilon_2\a_2)u_3u_4}  
\zeta_2(1-2s) \zeta_2(1+2\epsilon_1\a_1) \zeta_2(1-2\epsilon_1\a_1)\\
&\quad \times 
\prod_{3 \leq j \leq 4 }\zeta_2(1+\epsilon_1\a_1+\epsilon_j\a_j+u_j) \zeta_2(1-\epsilon_1\a_1+\epsilon_j\a_j+u_j) 
\\
&\quad \times \prod_{3 \leq i\leq j \leq 4} \zeta_2(1+\epsilon_i\a_i+\epsilon_j\a_j+u_i+u_j) \\
&\quad \times
\frac{\zeta_2(1+ \epsilon_1\a_1-s)}{\zeta_2(1+s+\epsilon_1\a_1 )}
\frac{\zeta_2(1 -\epsilon_1\a_1-s)}{\zeta_2(1+s-\epsilon_1\a_1 )}
\prod_{3 \leq i \leq 4} 
\frac{\zeta_2(1+ \epsilon_i\a_i+u_i-s)}{\zeta_2(1+s+\epsilon_i\a_i+u_i )}\\
&=
0,
\end{align*}
which means 
\begin{align}
&\sum_{\epsilon_1, \epsilon_2 = \pm 1}
(8X)^{-\frac{1-\epsilon_1}{2}\a_1 -\frac{1-\epsilon_2}{2}\a_2} 
\lambda_1^{\frac{1-\epsilon_1}{2}} \lambda_2^{\frac{1-\epsilon_2}{2}} \cdot J_{1,3} = 0.
\label{equ:cross-similar-22}
\end{align}
By the discussion from \eqref{equ:cross-second-1} to \eqref{equ:cross-similar-22}, we obtain  
\begin{align*}
&    \sum_{\substack{\epsilon_i = \pm 1\\i=1,2,3,4}} \prod_{i=1}^4
(8X)^{-\sum_{i=1}^4\frac{1-\epsilon_i}{2}\a_i} 
\lambda_i^{\frac{1-\epsilon_i}{2}} \cdot  ( \colz_1 +   \colz_2)\\
&=
\sum_{\substack{\epsilon_i = \pm 1\\i=1,2,3,4}} \prod_{i=1}^4
(8X)^{-\sum_{i=1}^4\frac{1-\epsilon_i}{2}\a_i} 
\lambda_i^{\frac{1-\epsilon_i}{2}} \cdot (I_{1,1} + J_{1,1}) +O(X (\log X)^6).
\end{align*}

We move the line of the integration in $I_{1,1}$ to $\Re(u_3) = -\frac{1}{10}+ \frac{10}{L}$, encountering poles at $u_3 = 0, -\epsilon_3\a_3, -\epsilon_1\a_1-\epsilon_3\a_3, -\epsilon_2\a_2-\epsilon_3\a_3$. We write these as  $I_{1,1,1},I_{1,1,2},I_{1,1,3}, I_{1,1,4}$, repectively, and write the integral on $\Re(u_3) = -\frac{1}{10}+ \frac{10}{L}$ as $I_{1,1,5}$.  We move the line of the integration in $J_{1,1}$ to $\Re(u_3) = -\frac{1}{10}+ \frac{10}{L}$, encountering poles at $u_3 = 0, -\epsilon_3\a_3, -\epsilon_1\a_1-\epsilon_3\a_3, -\epsilon_2\a_2-\epsilon_3\a_3, s- \epsilon_3\a_3$, denoted by $J_{1,1,1},J_{1,1,2},J_{1,1,3},J_{1,1,4},J_{1,1,5}$, respectively, and  the  integral on $\Re(u_3) = -\frac{1}{10}+ \frac{10}{L}$ is $\ll X^{1- \frac{1}{11}}$. 

By a similar argument as \eqref{equ:cross-similar-1} and \eqref{equ:cross-similar-2}, using Lemma \ref{lem-canceli-iii}(i) and (ii) with $r=3, u_1 =u_2 =0$, and using Lemma \ref{lem-canceli-iii}(iii) and \eqref{equ:hohd-2} with $r= 3, u_1=u_2 =0$, we find that
\begin{align*}
\sum_{\epsilon_3 = \pm 1} (8X)^{-\frac{1-\epsilon_3}{2}\a_3} 
\lambda_3^{\frac{1-\epsilon_3}{2}} \cdot  I_{1,1,2} &= 0,\\
\sum_{\epsilon_3 = \pm 1} (8X)^{-\frac{1-\epsilon_3}{2}\a_3} 
\lambda_3^{\frac{1-\epsilon_3}{2}} \cdot   J_{1,1,2} &= 0,  \\
\sum_{\epsilon_3 = \pm 1} (8X)^{-\frac{1-\epsilon_3}{2}\a_3} 
\lambda_3^{\frac{1-\epsilon_3}{2}} \cdot   (I_{1,1,5} + J_{1,1,5}) &= 0.
\end{align*}
One may deduce a version of \eqref{equ:cancel-u2u1} by calculating residues at $u_3 = -\epsilon_1\a_1 - \epsilon_3\a_3, u_1 =0$, and we let $u_2 =0 $.
By this and Lemma \ref{lem:cross-second-general} with $r =1 , r'= 3, u_2 = 0$,  proceeding similarly as for \eqref{equ:cross-similar-21} and \eqref{equ:cross-similar-22}, we obtain 
\begin{align*}
\sum_{\epsilon_1, \epsilon_3 = \pm 1} 
(8X)^{-\frac{1-\epsilon_1}{2}\a_1 -\frac{1-\epsilon_3}{2}\a_3} 
\lambda_1^{\frac{1-\epsilon_1}{2}} \lambda_3^{\frac{1-\epsilon_3}{2}} \cdot I_{1,1,3} &= 0,
\end{align*}
and similarly,
\begin{align*}
\sum_{\epsilon_1, \epsilon_3 = \pm 1} 
(8X)^{-\frac{1-\epsilon_1}{2}\a_1 -\frac{1-\epsilon_3}{2}\a_3} 
\lambda_1^{\frac{1-\epsilon_1}{2}} \lambda_3^{\frac{1-\epsilon_3}{2}} \cdot J_{1,1,3} &= 0.
\end{align*}
By the symmetry of the variables $\epsilon_i$, we also have 
\begin{align*}
\sum_{\epsilon_2, \epsilon_3 = \pm 1} 
(8X)^{-\frac{1-\epsilon_2}{2}\a_2 -\frac{1-\epsilon_3}{2}\a_3} 
\lambda_2^{\frac{1-\epsilon_2}{2}} \lambda_3^{\frac{1-\epsilon_3}{2}} \cdot I_{1,1,4} &= 0, \\
\sum_{\epsilon_2, \epsilon_3 = \pm 1} 
(8X)^{-\frac{1-\epsilon_2}{2}\a_2 -\frac{1-\epsilon_3}{2}\a_3} 
\lambda_2^{\frac{1-\epsilon_2}{2}} \lambda_3^{\frac{1-\epsilon_3}{2}} \cdot J_{1,1,4} &= 0.
\end{align*}
In summary, the above argument yields
\begin{align*}
&    \sum_{\substack{\epsilon_i = \pm 1\\i=1,2,3,4}} \prod_{i=1}^4
(8X)^{-\frac{1-\epsilon_i}{2}\a_i} 
\lambda_i^{\frac{1-\epsilon_i}{2}} \cdot  ( \colz_1 +   \colz_2)\\
&=
\sum_{\substack{\epsilon_i = \pm 1\\i=1,2,3,4}} \prod_{i=1}^4
(8X)^{-\frac{1-\epsilon_i}{2}\a_i} 
\lambda_i^{\frac{1-\epsilon_i}{2}} \cdot (I_{1,1,1} + J_{1,1,1}) +O(X (\log X)^3).
\end{align*}

Finally, we move the line of the integration in $I_{1,1,1}$ to $\Re(u_4) = -\frac{1}{10}+ \frac{10}{L}$, passing poles at $u_4 = 0, -\epsilon_4\a_4, -\epsilon_1- \epsilon_4\a_4, -\epsilon_2\a_2- \epsilon_4\a_4, -\epsilon_3\a_3- \epsilon_4\a_4$. In $J_{1,1,1}$, we move the line of the integration in $u_4$ to $\Re(u_4) = -\frac{1}{10}+ \frac{10}{L}$, passing poles at $u_4 = 0, -\epsilon_4\a_4, -\epsilon_1\a_1- \epsilon_4\a_4, -\epsilon_2\a_2- \epsilon_4\a_4, -\epsilon_3\a_3- \epsilon_4\a_4, s- \epsilon_4\a_4$. The integral on the new line is $\ll X^{1- \frac{3}{21}}$. 

By an argument similar to the one above, the contributions of the poles all cancel out except those at $u_4=0$. 
The pole at $u_4=0$ for $J_{1,1,1}$ contributes  $\ll X^{1- \frac{1}{11}}$, and for $I_{1,1,1}$, the pole at $u_4=0$ contributes
\begin{align*}
\res{u_4 =0}\ \res{u_3 =0}\ \res{u_2 =0}\ \res{u_1 =0}\colf_1(\boldsymbol{u}; \boldsymbol{\a};\boldsymbol{\epsilon})
\prod_{i=1}^4\frac{1}{u_i}     
\prod_{1 \leq i\leq j \leq 4} \zeta_2(1+\epsilon_i\a_i+\epsilon_j\a_j+u_i+u_j).
\end{align*}
This completes the proof of Lemma \ref{lem:secU-main}.

\section{Upper Bounds for  $\colz_3, \colz_4$}\label{sec:secU-error}
In this section, we prove Lemma \ref{lem:secU-error}. We be
\subsection{Upper Bound for  $\mathcal{Z}_3$}
Recall from \eqref{equ-defcolz} that 
\begin{align*}
\mathcal{Z}_3 &= \frac{4X}{\pi^2} \frac{1}{(2\pi i)^4}  \int\limits_{(\frac{1}{10})} \cdots \int\limits_{(\frac{1}{10})} 
(8 X)^{\sum_{i=1}^4\frac{u_i}{2}}\prod_{i=1}^4\frac{g_{\epsilon_i\a_i}(u_i)}{u_i}  \prod_{i=1}^4 \pth{U^{-\frac{1-\epsilon_i}{2}\a_i+\frac{u_i}{2}}-1}  \nonumber\\
&\quad\times 
\int\limits_{-\infty}^\infty  \Phi_{\{\epsilon_1, \epsilon_2,\epsilon_3,\epsilon_4\}}(x)
x^{\sum_{i=1}^4\frac{u_i}{2} } \, dx
\prod_{1 \leq i\leq j \leq 4} \zeta_2(1+\epsilon_i\a_i+\epsilon_j\a_j+u_i+u_j)\\
&\quad \times \mathcal{H}_2(\tfrac{1}{2}+\epsilon_1\a_1+u_1,\tfrac{1}{2}+\epsilon_2\a_2+u_2,\tfrac{1}{2}+ \epsilon_3\a_3+u_3, \tfrac{1}{2}+\epsilon_4\a_4+u_4) \, d\boldu
\end{align*}
As in Section \ref{sec:Asymptotic-Shortening}, we only estimate the case $\epsilon_1= \epsilon_2 =1, \epsilon_3=\epsilon_4 =-1$, as other cases are similar. In this case, we have
\begin{align*}
\mathcal{Z}_3 &= \frac{4X}{\pi^2} \frac{1}{(2\pi i)^4}  \int\limits_{(\frac{1}{10})} \cdots \int\limits_{(\frac{1}{10})} 
(8 X)^{\sum_{i=1}^4\frac{u_i}{2}}g_{\a_1}(u_1)g_{\a_2}(u_2)g_{\a_3}(-u_3)g_{\a_4}(-u_4)\\
&\quad\times \int\limits_{-\infty}^\infty  \Phi_{\{1,1,-1,-1\}}(x)
x^{\sum_{i=1}^4\frac{u_i}{2} } \, dx  \frac{U^{\frac{u_1}{2}}-1}{u_1}\frac{U^{\frac{u_2}{2}}-1}{u_2}\frac{U^{\frac{u_3}{2}-\a_3}-1}{u_3}\frac{U^{\frac{u_4}{2}-\a_4}-1}{u_4} \\
&\quad\times 
\prod_{1 \leq i\leq j \leq 4} \zeta_2(1+\epsilon_i\a_i+\epsilon_j\a_j+u_i+u_j)\\
&\quad \times \mathcal{H}_2(\tfrac{1}{2}+\a_1+u_1,\tfrac{1}{2}+\a_2+u_2,\tfrac{1}{2}-\a_3+u_3, \tfrac{1}{2}-\a_4+u_4) \, d\boldu
\end{align*}
Here we have preserved the $\epsilon_i$ notation in the product of zeta factors for brevity, but it is to be understood that we have made the specific choices defined above for these variables. As in the proof of Lemma \ref{lem:4therror}, we write
\[
\frac{U^{\frac{u_3}{2}-\a_3}-1}{u_3} = U^{-\a_3} \fracp{U^{\frac{u_3}{2}}-1}{u_3} + \frac{U^{-\a_3}-1}{u_3}
\]
and similarly for the factor involving $u_4$ and $\a_4$. The first term above has no pole at $u_3 = 0$, and the second factor is $O(\abs{u_3}^{-1}{(\log X)^{-1}}\log\log X)$. Rewriting in this way yields four integrals. Of these, we evaluate only 
\begin{align*}
\mathcal{Z}_3^* &:= \frac{4X}{\pi^2} \frac{1}{(2\pi i)^4}  \int\limits_{(\frac{1}{10})} \cdots \int\limits_{(\frac{1}{10})} 
(8 X)^{\sum_{i=1}^4\frac{u_i}{2}}g_{\a_1}(u_1)g_{\a_2}(u_2)g_{\a_3}(-u_3)g_{\a_4}(-u_4) \\
&\quad\times \int\limits_{-\infty}^\infty  \Phi_{\{1,1,-1,-1\}}(x)
x^{\sum_{i=1}^4\frac{u_i}{2} } \, dx \frac{U^{\frac{u_1}{2}}-1}{u_1}\frac{U^{\frac{u_2}{2}}-1}{u_2}\frac{U^{-\a_3}-1}{u_3}\frac{U^{-\a_4}-1}{u_4} \\
&\quad\times 
\prod_{1 \leq i\leq j \leq 4} \zeta_2(1+\epsilon_i\a_i+\epsilon_j\a_j+u_i+u_j)\\
&\quad \times \mathcal{H}_2(\tfrac{1}{2}+\a_1+u_1,\tfrac{1}{2}+\a_2+u_2,\tfrac{1}{2}-\a_3+u_3, \tfrac{1}{2}-\a_4+u_4) \, d\boldu,
\end{align*}
as the other three integrals are slightly easier. We evaluate this integral in a similar manner as the integral $\coli$ in Section \ref{sec:LargeSieve-Diagonal}. Since our arguments are very similar, we will be somewhat brief. As in the previous section, we let
\begin{align*}
\colf(\boldu,\boldalpha) &= \frac{4X}{\pi^2} g_{\a_1}(u_1)g_{\a_2}(u_2)g_{\a_3}(-u_3)g_{\a_4}(-u_4)\int\limits_{-\infty}^\infty  \Phi_{\{1,1,-1,-1\}}(x)
x^{\sum_{i=1}^4\frac{u_i}{2} } \, dx \\
&\quad\times \mathcal{H}_2(\tfrac{1}{2}+\a_1+u_1,\tfrac{1}{2}+\a_2+u_2,\tfrac{1}{2}-\a_3+u_3, \tfrac{1}{2}-\a_4+u_4)\\
&\quad \times \prod_{1 \leq i\leq j \leq 4} \zeta_2(1+\epsilon_i\a_i+\epsilon_j\a_j+u_i+u_j)(\epsilon_i\a_i+\epsilon_j\a_j+u_i+u_j)
\end{align*}
and note that $\colf(\boldu,\boldalpha)$ is analytic in the region $\Re (u_i) \geq -\frac{1}{4}+ \varepsilon$. Thus
\begin{align*}
\mathcal{Z}_3^* &= \frac{1}{(2\pi i)^4}  \int\limits_{(\frac{1}{10})} \cdots \int\limits_{(\frac{1}{10})} 
(8 X)^{\sum_{i=1}^4\frac{u_i}{2}} \colf(\boldu,\boldalpha) \prod_{1 \leq i\leq j \leq 4} \frac{1}{\epsilon_i\a_i+\epsilon_j\a_j+u_i+u_j} \\
&\quad\times \frac{U^{\frac{u_1}{2}}-1}{u_1}\frac{U^{\frac{u_2}{2}}-1}{u_2}\frac{U^{-\a_3}-1}{u_3}\frac{U^{-\a_4}-1}{u_4} d\boldu.
\end{align*}
We begin by taking the line of integration in $u_1$ to $\Re (u_1) = - \frac{1}{20}$. In doing so, we pass only a simple pole at $u=-\a_1$, and so
\[
\mathcal{Z}_3^* = I_1 + I_2,
\]
where
\begin{align*}
I_1 &= \frac{(8 X)^{-\frac{\a_1}{2}}}{8(2\pi i)^3}  \int\limits_{(\frac{1}{10})} \int\limits_{(\frac{1}{10})} \int\limits_{(\frac{1}{10})} 
(8 X)^{\sum_{i=2}^4\frac{u_i}{2}} \colf(-\a_1,u_2,u_3,u_4,\boldalpha) \\
&\quad\times \frac{1}{((u_2+\a_2)(u_3-\a_3)(u_4-\a_4))^2(u_2+u_3+\a_2-\a_3)(u_2+u_4+\a_2-\a_4)(u_3+u_4-\a_3-\a_4)} \\
&\quad\times \frac{U^{\frac{-\a_1}{2}}-1}{-\a_1}\frac{U^{\frac{u_2}{2}}-1}{u_2}\frac{U^{-\a_3}-1}{u_3}\frac{U^{-\a_4}-1}{u_4}\, du_2\, du_3\, du_4, \\
I_2 &= \frac{1}{(2\pi i)^4}  \int\limits_{(\frac{1}{10})} \cdots \int\limits_{(-\frac{1}{20})} 
(8 X)^{\sum_{i=1}^4\frac{u_i}{2}} \colf(\boldu,\boldalpha) \prod_{1 \leq i\leq j \leq 4} \frac{1}{\epsilon_i\a_i+\epsilon_j\a_j+u_i+u_j} \\
&\quad\times \frac{U^{\frac{u_1}{2}}-1}{u_1}\frac{U^{\frac{u_2}{2}}-1}{u_2}\frac{U^{-\a_3}-1}{u_3}\frac{U^{-\a_4}-1}{u_4} d\boldu.
\end{align*}
These should be compared with the expressions $I_1$ and $I_2$ in Section \ref{sec:LargeSieve-Diagonal}. Of these, we estimate only $I_1$, as $I_2$ can then be evaluated in a similar way as $I_1$. Let $\colg(u_2,u_3,u_4,\boldalpha)$ denote the integrand of $I_1$. We take the line of integration in $u_2$ to $\Re( u_2)=-\frac{1}{20}$. In doing so, we pass a pole of order 2 at $u_2 = -\a_2$, so
\[
I_1 = I_{1,1} + I_{1,2},
\]
where
\begin{align*}
I_{1,1} &= \frac{(8 X)^{-\frac{\a_1}{2}}}{8(2\pi i)^2}  \int\limits_{(\frac{1}{10})} \int\limits_{(\frac{1}{10})} \res{u_2=-\a_2} \colg(u_2,u_3,u_4,\boldalpha) du_3\, du_4, \\
I_{1,2} &= \frac{(8 X)^{-\frac{\a_1}{2}}}{8(2\pi i)^3}  \int\limits_{(\frac{1}{10})} \int\limits_{(\frac{1}{10})} \int\limits_{(-\frac{1}{20})} 
\colg(u_2,u_3,u_4,\boldalpha)\, du_2\, du_3\, du_4, \\
\end{align*}
For $I_{1,1}$, the pole of order 2 gives a sum of integrals involving powers of $\log X$, $\log U$, and the derivatives of $\colf$. To keep our calculations brief, we examine only the representative term in the residue, which is
\begin{align*}
I_{1,1}^* &=\frac{(8 X)^{-\frac{\a_1}{2}-\frac{\a_2}{2}}\log X}{16(2\pi i)^3}  \int\limits_{(\frac{1}{10})} \int\limits_{(\frac{1}{10})} \frac{(8 X)^{\frac{u_3}{2}+\frac{u_4}{2}} \colf(-\a_1,-\a_2,u_3,u_4,\boldalpha)}{((u_3-\a_3)(u_4-\a_4))^3(u_3+u_4-\a_3-\a_4)} \\
&\quad\times \frac{U^{\frac{-\a_1}{2}}-1}{-\a_1}\frac{U^{\frac{-\a_2}{2}}-1}{-\a_2}\frac{U^{-\a_3}-1}{u_3}\frac{U^{-\a_4}-1}{u_4}\, du_3\, du_4.
\end{align*}
We take the line of integration in $u_3$ to $\Re (u_3) = -\frac{1}{20}$. In doing so, we pass a pole of order 3 at $u_3 = \a_3$ and a simple pole at $u_3 = 0$. Abusing notation and letting $\colg(u_3,u_4;\boldalpha)$ denote the integrand of $I_{1,1}^*$, we have
\[
I_{1,1}^* = I_{1,1,1} + I_{1,1,2} + I_{1,1,3},
\]
where
\begin{align*}
I_{1,1,1} &= \frac{(8 X)^{-\frac{\a_1}{2}-\frac{\a_2}{2}}\log X}{16(2\pi i)}  \int\limits_{(\frac{1}{10})} \res{u_3=\a_3}\colg(u_3,u_4,\boldalpha) du_4, \\
I_{1,1,2} &= \frac{(8 X)^{-\frac{\a_1}{2}-\frac{\a_2}{2}}\log X}{16(2\pi i)^3}  \int\limits_{(\frac{1}{10})}  \frac{(8 X)^{\frac{u_4}{2}} \colf(-\a_1,-\a_2,0,u_4,\boldalpha)}{((u_4-\a_4)^3(u_4-\a_3-\a_4)} \\
&\quad\times \frac{U^{\frac{-\a_1}{2}}-1}{-\a_1}\frac{U^{\frac{-\a_2}{2}}-1}{-\a_2}\frac{U^{-\a_3}-1}{(-\a_3)^3}\frac{U^{-\a_4}-1}{u_4}\, du_4 \\
I_{1,1,3} &= \frac{(8 X)^{-\frac{\a_1}{2}-\frac{\a_2}{2}}\log X}{16(2\pi i)^3}  \int\limits_{(\frac{1}{10})} \int\limits_{(-\frac{1}{20})}  \colg(u_3,u_4,\boldalpha)\, du_3\, du_4
\end{align*}
For $I_{1,1,3}$ take the line of integratio in $u_4$ to $\Re (u_4) = \frac{1}{20} + \frac{5}{L}$, passing no poles in the process. Using the convexity estimate for $\zeta$ and Stirling's formula for $\Gamma$ to estimate $\colf$, we find that $I_{1,1,3} = O(X(\log X)^2(\log\log X)^2)$. Here we have used that
\begin{equation}\label{eq:Uzbound}
\frac{U^v-1}{u} \ll \abs{\frac{v\log\log X}{u}}
\end{equation}
when $\abs{v} \asymp \frac{1}{\log X}$.

For $I_{1,1,2}$, we take the line of integration in $u_4$ to $\Re (u_4) = -\frac{1}{20}$. In doing so, we pass a pole of order 3 at $u_4 = \a_4$ and simple poles at $u_4 = \a_3+\a_4$ and $u_4=0$. The integral on the new line is $O(X)$. Abusing notation once more and letting $\colg(u_4,\boldalpha)$ denote the integrand of $I_{1,1,2}$, we have
\[
I_{1,1,2} = \sumpth{\res{u_4 = \a_4} + \res{u_4 = \a_4+\a_3} + \res{u_4 = 0} }\colg(u_4,\boldalpha) + O(X).
\]
A straightforward residue calculation and the estimate \eqref{eq:Uzbound} shows that the residues contribute $O(X(\log X)^6(\log\log X)^4)$. 

For $I_{1,1,1}$, as before, we consider only the representative term from the residue, which is
\begin{align*}
I_{1,1,1}^* &=C\frac{(8 X)^{-\frac{\a_1}{2}-\frac{\a_2}{2}-\frac{\a_3}{2}}(\log X)^3}{(2\pi i)^3}  \int\limits_{(\frac{1}{10})} \frac{(8 X)^{\frac{u_4}{2}} \colf(-\a_1,-\a_2, \a_3,u_4,\boldalpha)}{(u_4-\a_4)^4} \\
&\quad\times \frac{U^{\frac{-\a_1}{2}}-1}{-\a_1}\frac{U^{\frac{-\a_2}{2}}-1}{-\a_2}\frac{U^{-\a_3}-1}{\a_3}\frac{U^{-\a_4}-1}{u_4}\, du_4
\end{align*}
for an appropriate constant $C$. Taking the line of integration in $u_4$ to $\Re (u_4) = -\frac{1}{20}$, we pass a pole of order 4 at $u_4 = \a_4$  and a simple pole at $u_4=0$. The integral on the new line of integration is $O(X)$, and the residues contribute $O(X(\log X)^6(\log\log X)^4)$. Combining our estimates above, we have
\[
I_1 = I_{1,2} + O(X(\log X)^6(\log\log X)^4).
\]
To estimate $I_{1,2}$ we proceed as in Section \ref{sec:LargeSieve-Diagonal}. We take the line of integration in $u_3$ to $\Re(u_3) = \frac{5}{L}$. In doing so, we pass only a simple pole at $u_3 = -u_2 -\a_2 + \a_3$, so
\[
I_{1,2} = I_{1,2,1} + I_{1,2,2},
\]
where
\begin{align*}
I_{1,2,1} &= \frac{(8 X)^{-\frac{\a_1}{2}}}{8(2\pi i)^2}  \int\limits_{(\frac{1}{10})} \int\limits_{(-\frac{1}{20})} 
(8 X)^{\frac{ -\a_2+\a_3}{2}+\frac{u_4}{2}} \colf(-\a_1,u_2,-u_2-\a_2+\a_3,u_4,\boldalpha) \\
&\quad\times \frac{1}{(u_2+\a_2)^4(u_4-\a_4)^2(u_2+u_4+\a_2-\a_4)(u_4-u_2-\a_2-\a_4)} \\
&\quad\times \frac{U^{\frac{-\a_1}{2}}-1}{-\a_1}\frac{U^{\frac{u_2}{2}}-1}{u_2}\frac{U^{-\a_3}-1}{-u_2-\a_2+\a_3}\frac{U^{-\a_4}-1}{u_4}\, du_2\, du_4, \\
I_{1,2,2} &= \frac{(8 X)^{-\frac{\a_1}{2}}}{8(2\pi i)^3}  \int\limits_{(\frac{1}{10})} \int\limits_{(\frac{5}{L})} \int\limits_{(-\frac{1}{20})} 
\colg(u_2,u_3,u_4,\boldalpha)\, du_2\, du_3\, du_4.
\end{align*}
In $I_{1,2,2}$, we take the line of integration in $u_4$ to $\Re (u_4) = \frac{1}{20} + \frac{5}{L}$, passing no poles in the process. The integral on the new line of integration is then $O(X(\log X)^4(\log\log X)^3)$. For $I_{1,2,1}$, we take the line of integration in $u_4$ to $\Re (u_4) = \frac{5}{L}$, passing a simple pole at $u_4 = -u_2 - \a_2 + \a_4$. The integral on the new line of integration is $O(X(\log X)^5(\log\log X)^3)$, and so
\begin{align*}
I_{1,2,1} &= -\frac{(8 X)^{-\frac{\a_1}{2}-\a_2-\frac{\a_3}{2}+\frac{\a_4}{2}}}{32(2\pi i)}  \int\limits_{(-\frac{1}{20})} 
(8 X)^{\frac{-u_2}{2}} \colf(-\a_1,u_2,-u_2-\a_2+\a_3,-u_2-\a_2+\a_4,\boldalpha) \\
&\quad\times \frac{1}{(u_2+\a_2)^7} \frac{U^{\frac{-\a_1}{2}}-1}{-\a_1}\frac{U^{\frac{u_2}{2}}-1}{u_2}\frac{U^{-\a_3}-1}{-u_2-\a_2+\a_3}\frac{U^{-\a_4}-1}{-u_2-\a_2+\a_4}\, du_2 \\
&\quad+ O(X(\log X)^5(\log\log X)^3)
\end{align*}
Finally, we take the line of integration in $u_2$ to $\Re(u_2) = \frac{1}{20}$. In doing so, we pass a pole of order 7 at $u_2=-\alpha_2$ as well as simple poles at $u_2 = 0,-\alpha_2+\alpha_3,-\alpha_2+\alpha_4$. The integral on the new line of integration is $O(X)$, and the contribution of the residues is $O(X(\log X)^6(\log\log X)^4)$.  We remark that there may be an occasion to bound the factor $(U^{-\frac{u_1}{2}} -1)(U^{\frac{u_1}{2}} -1)$ with $\Re(u_1) = -\frac{1}{20}$. In this case, we move the line to $\Re(u_1) = \eta>0$. Here $\eta$ is a small but fixed constant, say $\eta = \frac{1}{10Q}$, where  $Q$ is defined in \eqref{defU}. Then this factor is bounded by $(\log X)^{0.2}$, which is niligible.  Combining our calculations, we find that
\[
I_{1,2} \ll X(\log X)^6(\log\log X)^4,
\]
and therefore
\[
I_1 \ll X(\log X)^6(\log\log X)^4.
\]

\subsection {Upper Bound for $\mathcal{Z}_4$}
Recall from \eqref{equ-defcolz} that
\begin{align*}
\colz_4 &= \frac{X}{(2\pi i)^5}  \int\limits_{(\frac{1}{10})}  \int\limits_{(\frac{1}{10})} \cdots \int\limits_{(\frac{1}{10})}X^{\sum_{i=1}^4\frac{u_i+s}{2}}   X^{-s} \prod_{i=1}^4 \pth{1-U^{-\frac{1-\epsilon_i}{2}\a_i+\frac{u_i+s}{2}}} \nonumber\\
&\quad\times
\prod_{i=1}^4 \frac{1}{u_i+s} g_{\epsilon_i\a_i}(u_i+s) 
\left(\frac{8}{\pi}\right)^{-s} \frac{\Gamma(\frac{1}{4}-\frac{s}{2})}{\Gamma(\frac{1}{4}+ \frac{s}{2})} \int\limits_0^\infty  \Phi_{\{\epsilon_1, \epsilon_2,\epsilon_3,\epsilon_4\}}(x) x^{-s} (8x)^{\sum_{i=1}^4\frac{u_i+s}{2}}\, dx\nonumber\\
&\quad\times  \zeta_2(1-2s)  \prod_{1 \leq i \leq 4} \frac{\zeta_2(1+ \epsilon_i\a_i+u_i)}{\zeta_2(1+2s+\epsilon_i\a_i+u_i )}\prod_{1\leq i \leq j \leq 4}\frac{\zeta_2(1+2s+ \epsilon_i\a_i +u_i + \epsilon_j\a_j+u_j)}{\zeta_2(2+u_i+u_j+\epsilon_i\a_i +\epsilon_j\a_j)} \nonumber \\ 
&\quad\times  Z_3(\tfrac{1}{2}+\epsilon_1\a_1+u_1,\tfrac{1}{2}+\epsilon_2\a_2+u_2,\tfrac{1}{2}+ \epsilon_3\a_3+u_3,\tfrac{1}{2}+\epsilon_4\a_4+u_4,s)  \, d\boldu \, ds.
\end{align*}
As in the previous section, we consider only the case $\epsilon_1= \epsilon_2 =1, \epsilon_3=\epsilon_4 =-1$ and also define
\begin{align*}
\colf(\boldu,s,\boldalpha) &= \prod_{i=1}^4 \frac{1}{u_i+s} g_{\epsilon_i\a_i}(u_i+s) 
\left(\frac{8}{\pi}\right)^{-s} \frac{\Gamma(\frac{1}{4}-\frac{s}{2})}{\Gamma(\frac{1}{4}+ \frac{s}{2})} \int\limits_0^\infty  \Phi_{\{\epsilon_1, \epsilon_2,\epsilon_3,\epsilon_4\}}(x) x^{-s} (8x)^{\sum_{i=1}^4\frac{u_i+s}{2}}\, dx\nonumber\\
&\quad\times  \zeta_2(1-2s)  \prod_{1 \leq i \leq 4} \frac{\zeta_2(1+ \epsilon_i\a_i+u_i)(\epsilon_i\a_i+u_i)}{\zeta_2(1+2s+\epsilon_i\a_i+u_i )(2s+\epsilon_i\a_i+u_i)} \\
&\quad\times\prod_{1\leq i \leq j \leq 4}\frac{\zeta_2(1+2s+ \epsilon_i\a_i +u_i + \epsilon_j\a_j+u_j)(2s+ \epsilon_i\a_i +u_i + \epsilon_j\a_j+u_j)}{\zeta_2(2+u_i+u_j+\epsilon_i\a_i +\epsilon_j\a_j)} \nonumber \\ 
&\quad\times  Z_3(\tfrac{1}{2}+\epsilon_1\a_1+u_1,\tfrac{1}{2}+\epsilon_2\a_2+u_2,\tfrac{1}{2}+ \epsilon_3\a_3+u_3,\tfrac{1}{2}+\epsilon_4\a_4+u_4,s)  \, d\boldu \, ds.
\end{align*}
Thus we need to estimate
\begin{align*}
&\frac{X}{(2\pi i)^5}  \int\limits_{(\eta+\frac{5}{L})}  \int\limits_{(\frac{1}{10})} \cdots \int\limits_{(\frac{1}{10})}X^{\sum_{i=1}^4\frac{u_i+s}{2}} \colf(\boldu,s,\boldalpha) X^{-s} \prod_{i=1}^4 \frac{U^{-\frac{1-\epsilon_i}{2}\a_i+\frac{u_i+s}{2}}-1}{u_i+s} \nonumber\\
&\quad\times \prod_{1\leq i \leq j \leq 4}\frac{1}{u_i +u_j+ \epsilon_i\a_i+ \epsilon_j\a_j+2s}  \prod_{i=1}^4 \frac{\epsilon_j\a_j+u_j+2s}{\epsilon_i\a_i+u_i}\, d\boldu \, \frac{ds}{s}.
\end{align*}
after shifting the line of integration in $s$ to $\Re (s) = \eta+\frac{5}{L}$ for some small but fixed $\eta > 0$ that we will choose later. The estimation of this integral is now very similar to Section \ref{sec:LargeSieve-OffDiagonal}. We apply the same trick of separating the factors
\[
\frac{U^{-\frac{1-\epsilon_i}{2}\a_i+\frac{u_i+s}{2}}-1}{u_i+s}
\]
as we did in the previous section. We then shift contours in a similar manner as in Section \ref{sec:LargeSieve-OffDiagonal}. In doing so, we have more poles that we pass due to the presence of the shifts $\a_i$. However, these result in residues that are larger by, at worse, powers of $\log\log X$, which are negligible for us. We leave the full details of the calculation to the interested reader. This concludes the proof of Lemma \ref{lem:secU-error}.

\appendix
\renewcommand{\theequation}{\thesection.\arabic{equation}}

\section{Appendix: Dirichlet series} \label{prfDiri}

In this appendix, we prove Lemmas \ref{lem:DirichletSeriesDiagonal}, \ref{lem:DirichletSeriesOffDiagonal}, \ref{lem:DirichletSeriesSuma}, and \ref{lem:iden}. As with most calculations involving multiple Dirichlet series, these calculations are somewhat long and technical. We first state some preliminaries that we will need for our calculations. 

Let
\begin{equation}\label{eq:Bdef}
B_{p,k}(\boldu;a) = \sum_{\substack{\boldn,\ n_i\geq 0\\  n_1+\cdots+n_{k} \equiv a \mod{2}}} \frac{1}{p^{n_1u_1+\cdots +n_{k}u_{k}}}.
\end{equation}
In the definition of $B_{p,k}(\boldu;a)$, we are mainly interested in the case $k=4$, and we write $B_p(\boldu;a) := B_{p,4}(\boldu;a)$ and $B_p(\boldu):= B_{p,4}(\boldu;0)$. Since some of our calculations work in slightly more generality, we include the more general case for completeness.

The following simple estimate will be used when determining where our Dirichlet series are absolutely convergent. Its simple proof is left as an exercise to the reader.
\begin{lem}\label{lem:EulerFactorBound}
Let $x_1,\ldots, x_n$ be complex numbers and let $X = \max(\abs{x_i})$. Suppose that $X \leq \theta$ for some $\theta < 1$. Then for $\epsilon = \pm 1$, we have
\[
\sumpth{1+\epsilon\sum_{i=1}^n x_i + O_n\pth{X^2}}\prod_{i=1}^{n} (1-x_i)^{\epsilon} = 1+O_n\pth{X^2}.
\]
\end{lem}
Next, we require a combinatorial evaluation of the sum $B_{p,k}$.

\begin{lem}\label{lem:BsumEval}
For any $k\geq 1$, $\boldu$, and $\theta > 0$ such that $\Re(u_i) \geq \theta$, we have
\[
B_{p,2k}(\boldu;0) =\prod_{1\leq i \leq j \leq 2k}  \pth{1-\frac{1}{p^{u_i+u_j}}}^{-1} C_{p,2k}(\boldu),
\]
where $C_{p,2k}$ satisfies
\begin{align}
C_{p,2k}(\boldu) = \prod_{1\leq i < j \leq 2k} \pth{1-\frac{1}{p^{u_i+u_j}}} \sum_{\substack{A\subset\set{1,\ldots,2k}\\ \abs{A} \equiv 0 \mod{2}}} \prod_{i\in A} \frac{1}{p^{u_i}} = 1 + O\fracp{1}{p^{4\theta}}.
\label{def:Cpu-2}
\end{align}
\end{lem}

\begin{proof}
Since $\Re(u_i) \geq \theta > 0$, the sum $B_{p,2k}(\boldu;0)$ converges absolutely, so we may freely rearrange its terms. For a set $A \subseteq \set{1,\ldots, 2k}$, let $\bar{A} = \set{1,\ldots,2k} \setminus A$ and denote $\boldn_A = (n_i)_{i\in A}$.

Note that the condition $n_1 + \cdots + n_{2k} \equiv 0\mod{2}$ is satisfied if and only if an even number of the $n_i$ are odd. Thus we may write
\[
B_{p,2k}(\boldu;0) =  \sum_{\substack{A\subset\set{1,\ldots,2k}\\ \abs{A} \equiv 0 \mod{2}}} \sum_{\substack{\boldn,\ n_i\geq 0\\ \boldn_A \equiv 1 \mod{2}\\ \boldn_{\bar{A}} \equiv 0 \mod{2}}} \frac{1}{p^{n_1u_1 + \cdots + n_{2k}u_{2k}}}.
\]
Since
\begin{align*}
\sum_{\substack{n\geq 0\\ n\equiv 0 \mod{2}}} \frac{1}{p^{nu}} &= \pth{1-\frac{1}{p^{2u}}}^{-1}, \\
\sum_{\substack{n\geq 0\\ n\equiv 1 \mod{2}}} \frac{1}{p^{nu}} &= \pth{1-\frac{1}{p^u}}^{-1} - \pth{1-\frac{1}{p^{2u}}}^{-1},
\end{align*}
the sums over $\boldn$ can be written as products of such terms, and so we have
\begin{align*}
B_{p,2k}(\boldu;0) &= \sum_{\substack{A\subset\set{1,\ldots,2k}\\ \abs{A} \equiv 0 \mod{2}}} \prod_{i\in \bar{A}} \pth{1-\frac{1}{p^{2u_i}}}^{-1} \prod_{i\in A} \pth{\pth{1-\frac{1}{p^{u_i}}}^{-1} - \pth{1-\frac{1}{p^{2u_i}}}^{-1} } \\
&= \prod_{i=1}^{2k} \pth{1-\frac{1}{p^{2u_i}}}^{-1} \sum_{\substack{A\subset\set{1,\ldots,2k}\\ \abs{A} \equiv 0 \mod{2}}}  \prod_{i\in A} \frac{1}{p^{u_i}} \\
&= \prod_{1\leq i \leq j \leq 2k}\pth{1-\frac{1}{p^{u_i+u_j}}}^{-1} \bigg[\prod_{1\leq i < j \leq 2k}\pth{1-\frac{1}{p^{u_i+u_j}}}\sum_{\substack{A\subset\set{1,\ldots,2k}\\ \abs{A} \equiv 0 \mod{2}}}  \prod_{i\in A} \frac{1}{p^{u_i}}\bigg]\\
&= \prod_{1\leq i \leq j \leq 2k}\pth{1-\frac{1}{p^{u_i+u_j}}}^{-1} C_{p,2k}(\boldu).
\end{align*}
Finally, note that 
\[
\sum_{\substack{A\subset\set{1,\ldots,2k}\\ \abs{A} \equiv 0 \mod{2}}}  \prod_{i\in A} \frac{1}{p^{u_i}} = 1+\sum_{1\leq i< j \leq 2k} \frac{1}{p^{u_i+u_j}} + O\fracp{1}{p^{4\theta}}
\]
upon isolating the terms with $\abs{A} = 0$ and $\abs{A} = 2$. 
\end{proof}

\subsection{Proof of Lemma \ref{lem:DirichletSeriesDiagonal}}

Recall $B_p(\boldu)=B_p(\boldu;0) = B_{p,4}(\boldu;0)$  in  \eqref{eq:Bdef} and 
 $C_p(\boldu) = C_{p,4}(\boldu)$ by \eqref{def:Cpu} and \eqref{def:Cpu-2}. By Lemma \ref{lem:BsumEval}, we have
\begin{align}
\cold_{\ell}(\boldu) &= \sum_{\substack{(n_1n_2n_3n_4,2)=1\\ n_1n_2n_3n_4=\square}} \frac{1}{n_1^{u_1}n_2^{u_2}n_3^{u_3}n_4^{u_4}} \prod_{p \mid n_1n_2n_3n_4} g_{\ell}(p) \nonumber\\
&= \prod_{p\neq 2} \sumpth{1+g_\ell(p)\sum_{\substack{n_1,n_2,n_3,n_4\geq0\\ n_1+n_2+n_3+n_4\equiv 0 \mod{2}\\ n_1+n_2+n_3+n_4>0}} \frac{1}{p^{n_1u_1+n_2u_2+n_3u_3+n_4u_4}}} \label{equ:D1-111}\\
&= \prod_{p\neq 2} \Big(1-g_\ell(p) + g_\ell(p) B_p(\boldu)\Big) \label{equ:D1-111++}\\
&= \sumpth{\prod_{1\leq i\leq j\leq {4}} \zeta_2(u_i+u_j)} \prod_{p\neq 2}  \sumpth{(1-g_\ell(p)) \prod_{1\leq i \leq j \leq 4} \pth{1-\frac{1}{p^{u_i+u_j}}}+ g_\ell(p)C_p(\boldu)}\nonumber \\
&= \sumpth{\prod_{1\leq i\leq j\leq {4}} \zeta_2(u_i+u_j)} \prod_{p\neq 2} \colh_{\ell,p}(\boldu),\label{equ:D1-112}
\end{align}
say. For $\Re(u_i) \geq \theta \geq \frac{1}{4}+\ep$, note that since $1-g_\ell(p) \ll \frac{1}{p}$ and $g_\ell(p) \asymp 1$, we have
\[
\colh_{\ell,p}(\boldu) = 1-g_\ell(p) + O\fracp{1}{p^{1+2\theta}} + g_\ell(p)\pth{1+O\fracp{1}{p^{4\theta}}} = 1 + O\fracp{1}{p^{1+4\ep}}.
\]
Thus $\colh_\ell(\boldu) = \prod_{p\neq 2} \colh_{\ell,p}(\boldu)$ is analytic and bounded uniformly for $\Re(u_i) \geq \frac{1}{4}+\ep$.

To prove the identities for $\colh_\ell(\half,\half,\half,\half)$, note that
\begin{equation}\label{eq:Chalf}
C_p\pth{\half,\half,\half,\half} = \pth{1-\frac{1}{p}}^{6}\pth{1+\frac{6}{p}+\frac{1}{p^2}},
\end{equation}
and so 
\begin{align*}
\colh_\ell\pth{\half,\half,\half,\half} &= \prod_{p\neq 2} \pth{(1-g_\ell(p))\pth{1-\frac{1}{p}}^{10} +g_\ell(p)\pth{1-\frac{1}{p}}^{6}\pth{1+\frac{6}{p}+\frac{1}{p^2}}}\\
&= \prod_{p\neq 2} \pth{1-\frac{1}{p}}^{10}\pth{1-g_\ell(p) +g_\ell(p)\pth{1-\frac{1}{p}}^{-4}\pth{1+\frac{6}{p}+\frac{1}{p^2}}}.
\end{align*}
In particular, a straightforward calculation gives
\begin{equation}\label{eq:Colh1Half}
\colh_1\pth{\half,\half,\half,\half} = \prod_{p\neq 2} \pth{1-\frac{1}{p}}^{7}\pth{1+\frac{7}{p}-\frac{2}{p^2}+\frac{3}{p^3}-\frac{1}{p^4}}
\end{equation}
and
\[
\colh_2\pth{\half,\half,\half,\half} = \prod_{p\neq 2} \frac{\pth{1-\frac{1}{p}}^{10}}{1+\frac{1}{p}}\sumpth{\frac{1}{p} +\pth{1-\frac{1}{p}}^{-4}\pth{1+\frac{6}{p}+\frac{1}{p^2}}}.
\]
The identity \eqref{equ:h2} now follows from
\[
\pth{1-\frac{1}{p}}^{-4}\pth{1+\frac{6}{p}+\frac{1}{p^2}} = \frac{1}{2} \sumpth{\pth{1+\frac{1}{\sqrt{p}}}^{-4}+ \pth{1-\frac{1}{\sqrt{p}}}^{-4}},
\]
which is easily seen by 
\begin{align*}
\pth{1+\frac{1}{\sqrt{p}}}^{-4}+ \pth{1-\frac{1}{\sqrt{p}}}^{-4} &= \pth{1-\frac{1}{p}}^{-4}\sumpth{\pth{1+\frac{1}{\sqrt{p}}}^{4}+ \pth{1-\frac{1}{\sqrt{p}}}^{4}} \\
&= \pth{1-\frac{1}{p}}^{-4} \pth{2+\frac{12}{p}+\frac{2}{p^2}}.
\end{align*}

\subsection{Proof of Lemma \ref{lem:DirichletSeriesOffDiagonal}}
Throughout the proof, bold symbols indicate the associated 4-tuple of variables. Thus $\boldu = (u_1,u_2,u_3,u_4)$ and $\boldn = (n_1,n_2,n_3,n_4)$, and the variables $n_i$ are non-negative integers. Write $ \eta = \underset{i=1,2,3,4}{\min }\{\Re(u_i),\tfrac{1}{2}\}$ and $\theta = \min \{\Re(s),\tfrac{1}{2}\}$. Note  that in the regions of $u_i$ and $s$ considered in the lemma, we have $\eta+\theta > \frac{1}{4}$ and $\Re(u_i) \geq -\frac{1}{20}$. Let us first show that it suffices to handle $Z$. Let
\[
J(k_2) := \frac{1}{k_2^{2s}} \summany_{(n_1n_2n_3n_4,2a)=1} \frac{1}{n_1^{u_1}n_2^{u_2}n_3^{u_3}n_4^{u_4}} \frac{G_{k_1k_2^2}(n_1n_2n_3n_4)}{n_1n_2n_3n_4}.
\]
Then
\[
Z^*(\boldu,s;k_1,a) = \sum_{\substack{k_2 \geq 1 \\ k_2\ \text{even}}} J(k_2) - \sum_{\substack{k_2 \geq 1 \\ k_2\ \text{odd}}} J(k_2) = 2 \sum_{\substack{k_2 \geq 1 \\ k_2\ \text{even}}} J(k_2) - \sum_{k_2 \geq 1} J(k_2).
\]
Since $G_{4k}(n) = G_k(n)$ for any odd $n$ and any $k\neq 0$, we have $J(2k_2) = 2^{-2s}J(k_2)$, and so
\[
Z^*(\boldu,s;k_1,a) = \pth{2^{1-2s} - 1} Z(\boldu,s;k_1,a).
\]

We now investigate the function $Z$. By multiplicativity,
\[
Z(\boldu,s;k_1,a) = \prod_p \colf_p(\boldu,s;k_1,a),
\]
where for $p\nmid 2a$
\[
\colf_p(\boldu,s;k_1,a) := \sum_{k_2 \geq 0} \frac{1}{p^{2k_2s}}\sum_{\boldn} \frac{1}{p^{n_1u_1+n_2u_2+n_3u_3+n_4u_4}} \frac{G_{k_1p^{2k_2}}(p^{n_1+n_2+n_3+n_4})}{p^{n_1+n_2+n_3+n_4}},
\]
and for $p\mid 2a$
\[
\colf_p(\boldu,s;k_1,a) = \pth{1-\frac{1}{p^{2s}}}^{-1}.
\]
When $p\mid a$, we write $\colf_p$ as
\begin{align*}
\colf_p(\boldu,s;k_1,a) &= \Bigg[\pth{1-\frac{1}{p^{2s}}} Z_{2,p}(\boldu,s;k_1,a)\Bigg]^{-1}Z_{2,p}(\boldu,s;k_1,a)
\end{align*}
where
\begin{align}
Z_{2,p}(\boldu,s;k_1,a) &= \prod_{i=1}^4 \sumpth{1-\frac{\chi_\mathfrak{m}(p)}{p^{\frac{1}{2}+u_i}}}\sumpth{1-\frac{\chi_\mathfrak{m}(p)}{p^{\frac{1}{2}+u_i+2s}}}^{-1} \nonumber\\
&\quad\times\prod_{1\leq i\leq j\leq 4} \pth{1-\frac{1}{p^{u_i+u_j+2s}}}\pth{1-\frac{1}{p^{u_i+u_j+1}}}^{-1},
\label{eq:Z2,p divides 2a}
\end{align}
For primes not dividing $2a$, we divide into two cases depending as $p \mid k_1$ or not. \\

\noindent\textbf{Case 1:} $p\nmid 2a$ and $p \mid k_1$. In this case
\begin{align}
\colf_p(\boldu,s;k_1,a) &= \sum_{k_2\geq 0} \frac{1}{p^{2k_2s}} \bigg(\sum_{h=0}^{k_2} \frac{\phi(p^{2h})}{p^{2h}} \sum_{\substack{\boldn\\ n_1+n_2+n_3+n_4=2h}} \frac{1}{p^{n_1u_1+n_2u_2+n_3u_3+n_4u_4}} \nonumber\\
&\quad- \frac{1}{p} \sum_{\substack{\boldn\\ n_1+n_2+n_3+n_4=2k_2+2}} \frac{1}{p^{n_1u_1+n_2u_2+n_3u_3+n_4u_4}}\bigg),\nonumber \\
&=: \colf_p^{(1)}(\boldu,s) - \colf_p^{(2)}(\boldu,s),
\label{eq:F,k1}
\end{align}
say. The first term is
\begin{align*}
\colf_p^{(1)}(\boldu,s)& = \sum_{k_2\geq 0} \frac{1}{p^{2k_2s}} \sum_{h=0}^{k_2} \frac{\phi(p^{2h})}{p^{2h}} \sum_{\substack{\boldn\\ n_1+n_2+n_3+n_4=2h}} \frac{1}{p^{n_1u_1+n_2u_2+n_3u_3+n_4u_4}}
\\
&= \sum_{h\geq 0}\frac{\phi(p^{2h})}{p^{2h}} \sum_{\substack{\boldn\\ n_1+n_2+n_3+n_4=2h}} \frac{1}{p^{n_1u_1+n_2u_2+n_3u_3+n_4u_4}} \sum_{k_2\geq h} \frac{1}{p^{2k_2s}}  \\ 
&=\pth{1-\frac{1}{p^{2s}}}^{-1} \sum_{h\geq 0}\frac{\phi(p^{2h})}{p^{2h+2hs}}\sum_{\substack{\boldn\\ n_1+n_2+n_3+n_4=2h}} \frac{1}{p^{n_1u_1+n_2u_2+n_3u_3+n_4u_4}}. 
\end{align*}
Isolating the $h=0$ term and inverting the order of summation in the remaining sum, we find that the above is
\begin{equation*}
\colf_p^{(1)}(\boldu,s) = \pth{1-\frac{1}{p^{2s}}}^{-1}\sumpth{\frac{1}{p}+ \pth{1-\frac{1}{p}} B_p(\boldu+s)},
\end{equation*}
where $B_p(\boldu)$ is defined by \eqref{eq:Bdef}. By a similar calculation, the second term of \eqref{eq:F,k1} is
\begin{align*}
\colf_p^{(2)}(\boldu,s) &=-\frac{1}{p^{1-2s}} \sum_{k_2\geq 0} \frac{1}{p^{2k_2s}} \sum_{\substack{\boldn\\ n_1+n_2+n_3+n_4=2k_2+2}} \frac{1}{p^{n_1u_1+n_2u_2+n_3u_3+n_4u_4}}  \\
&= \frac{1}{p^{1-2s}}\sumpth{1- B_p(\boldu+s)},
\end{align*}
and so
\[
\colf_p(\boldu,s;k_1,a) = \pth{1-\frac{1}{p^{2s}}}^{-1}  Y_p(\boldu,s;k_1,a),
\]
where
\begin{align*}
&Y_p(\boldu,s;k_1,a)\\
&:= \frac{1}{p} + \pth{1-\frac{1}{p}} B_p(\boldu+s) + \pth{1-\frac{1}{p^{2s}}} \frac{1}{p^{1-2s}}(1-B_p(\boldu+s)) \\
&= \frac{1}{p^{1-2s}} +\pth{1-\frac{1}{p^{1-2s}}}B_p(\boldu+s) \\
&=\prod_{1\leq i \leq j \leq 4} \pth{1-\frac{1}{p^{u_i+u_j+2s}}}^{-1} \Bigg[\frac{1}{p^{1-2s}}\prod_{1\leq i \leq j \leq 4} \pth{1-\frac{1}{p^{u_i+u_j+2s}}} + \pth{1-\frac{1}{p^{1-2s}}} C_p(\boldu+s)\Bigg]
\end{align*}
by Lemma \ref{lem:BsumEval}. We then write this as
\[
Y_p(\boldu,s;k_1,a) = \prod_{1\leq i \leq j \leq 4} \pth{1-\frac{1}{p^{u_i+u_j+2s}}}^{-1} \prod_{1\leq i \leq j \leq 4} \pth{1-\frac{1}{p^{u_i+u_j+1}}} Z_{2,p}(\boldu,s;k_1,a),
\]
where
\begin{align*}
Z_{2,p}(\boldu,s;k_1,a) &:= \prod_{1\leq i \leq j \leq 4} \pth{1-\frac{1}{p^{u_i+u_j+1}}}^{-1} \\
&\quad\times\Bigg[\frac{1}{p^{1-2s}}\prod_{1\leq i \leq j \leq 4} \pth{1-\frac{1}{p^{u_i+u_j+2s}}} + \pth{1-\frac{1}{p^{1-2s}}} C_p(\boldu+s)\Bigg]
\end{align*}
for $p\nmid 2a$ and $p|k_1$.
By Lemmas \ref{lem:EulerFactorBound} and \ref{lem:BsumEval}, we have
\begin{align}
Z_{2,p}(\boldu,s;k_1,a) &= \prod_{1\leq i \leq j \leq 4} \pth{1-\frac{1}{p^{u_i+u_j+1}}}^{-1}  \sumpth{1 - \sum_{1\leq i \leq j \leq 4} \frac{1}{p^{u_i+u_j+1}} + O\fracp{1}{p^{4\eta+2\Re(s) + 2\theta}}}\nonumber\\
&= 1 + O\pth{\frac{1}{p^{4\eta+2\Re(s) + 2\theta}} + \frac{1}{p^{4\eta + 2}}}
\label{eq:Z2bound,p divides k1}
\end{align}
 Combining the above estimates, we have that when $p\nmid 2a$, $p \mid k_1$, 
\[
\colf_p(\boldu,s;k_1,a) = \pth{1-\frac{1}{p^{2s}}}^{-1}\prod_{1\leq i \leq j \leq 4} \pth{1-\frac{1}{p^{u_i+u_j+2s}}}^{-1} \pth{1-\frac{1}{p^{u_i+u_j+1}}} Z_{2,p}(\boldu,s;k_1,a).
\]
Note that since $p \mid k_1$, we have $\chi_\mathfrak{m}(p) = 0$, and so the above is exactly the Euler factor at $p$ in the expression for $Z$ given in the lemma.\\

\noindent\textbf{Case 2:} $p\nmid 2ak_1$. In this case
\begin{align*}
\colf_p(\boldu,s;k_1,a) &= \sum_{k_2\geq 0} \frac{1}{p^{2k_2s}} \bigg(\sum_{h=0}^{k_2} \frac{\phi(p^{2h})}{p^{2h}} \sum_{\substack{\boldn\\ n_1+n_2+n_3+n_4=2h}} \frac{1}{p^{n_1u_1+n_2u_2+n_3u_3+n_4u_4}} \\
&\quad + \frac{\chi_\mathfrak{m}(p)}{\sqrt{p}} \sum_{\substack{\boldn\\ n_1+n_2+n_3+n_4=2k_2+1}} \frac{1}{p^{n_1u_1+n_2u_2+n_3u_3+n_4u_4}}\bigg), \\
&= \colf_p^{(1)}(\boldu,s) + \colf_p^{(2)}(\boldu,s),
\end{align*}
As before, the first sum is
\[
\colf_p^{(1)}(\boldu,s) = \pth{1-\frac{1}{p^{2s}}}^{-1}\sumpth{\frac{1}{p}+ \pth{1-\frac{1}{p}} B_p(\boldu+s)},
\]
and the second sum is now
\[
\colf_p^{(2)}(\boldu,s) = \frac{\chi_\mathfrak{m}(p)}{\sqrt{p}} \sum_{k_2\geq 0} \frac{1}{p^{2k_2s}} \sum_{\substack{\boldn\\ n_1+n_2+n_3+n_4=2k_2+1}} \frac{1}{p^{n_1u_1+n_2u_2+n_3u_3+n_4u_4}} =\frac{\chi_\mathfrak{m}(p)}{p^{\frac{1}{2}-s}} B_p(\boldu+s;1).
\]
Thus
\begin{equation}\label{eq:colfY-no-k1}
\colf_p(\boldu,s;k_1,a) = \pth{1-\frac{1}{p^{2s}}}^{-1} \prod_{i=1}^4  \sumpth{1-\frac{\chi_\mathfrak{m}(p)}{p^{\frac{1}{2}+u_i}}}^{-1}Y_p(\boldu,s;k_1,a),
\end{equation}
where now
\begin{align}
Y_p(\boldu,s;k_1,a) &= \prod_{i=1}^4  \sumpth{1-\frac{\chi_\mathfrak{m}(p)}{p^{\frac{1}{2}+u_i}}} \Bigg[\frac{1}{p} + \pth{1-\frac{1}{p}}B_p(\boldu+s;0) + \pth{1-\frac{1}{p^{2s}}}\frac{\chi_\mathfrak{m}(p)}{p^{\frac{1}{2}-s}}B_p(\boldu+s,1)\Bigg]\nonumber \\
&= \prod_{i=1}^4  \sumpth{1-\frac{\chi_\mathfrak{m}(p)}{p^{\frac{1}{2}+u_i}}} \Bigg[\frac{1}{p} + \pth{1-\frac{1}{p}}B_p(\boldu+s;0)\nonumber \\
&\quad+ \pth{1-\frac{1}{p^{2s}}}\frac{\chi_\mathfrak{m}(p)}{p^{\frac{1}{2}-s}} \sumpth{\prod_{i=1}^4  \pth{1-\frac{1}{p^{u_i+s}}}^{-1} - B_p(\boldu+s;0)}\Bigg].
\label{eq:Y-no-k1}
\end{align}
At this point, we note that
\[
Y_p\pth{\boldu,\half;1,1} = 1-\frac{1}{p} + \frac{1}{p}\prod_{i=1}^4  \sumpth{1-\frac{1}{p^{\frac{1}{2}+u_i}}},
\]
which proves \eqref{eq:resZathalf}, and also
\[
Y_p\pth{\boldu,0;1,1} =  \pth{\frac{1}{p} + \pth{1-\frac{1}{p}}B_p(\boldu;0)} \prod_{i=1}^4  \sumpth{1-\frac{1}{p^{\frac{1}{2}+u_i}}}
\]
which proves \eqref{eq:resZat0}. 

We now evaluate the factors $Y_p$ more explicitly for arbitrary $s$ in the case when $p\nmid 2ak_1$. The expression in brackets in \eqref{eq:Y-no-k1} is
\begin{align*}
\frac{1}{p} + B_p(\boldu+s;0) \pth{1-\frac{1}{p} - \frac{\chi_\mathfrak{m}(p)}{p^{\frac{1}{2}-s}} + \frac{\chi_\mathfrak{m}(p)}{p^{\frac{1}{2}+s}}} + \pth{1-\frac{1}{p^{2s}}}\frac{\chi_\mathfrak{m}(p)}{p^{\frac{1}{2}-s}} \prod_{i=1}^4  \pth{1-\frac{1}{p^{u_i+s}}}^{-1},
\end{align*}
and by Lemma \ref{lem:BsumEval}, this equals
\begin{align*}
&\prod_{1\leq i \leq j \leq 4} \pth{1-\frac{1}{p^{u_i+u_j+2s}}}^{-1}\Bigg[\frac{1}{p}\prod_{1\leq i \leq j \leq 4} \pth{1-\frac{1}{p^{u_i+u_j+2s}}} + C_p(\boldu+s) \pth{1-\frac{1}{p} - \frac{\chi_\mathfrak{m}(p)}{p^{\frac{1}{2}-s}} + \frac{\chi_\mathfrak{m}(p)}{p^{\frac{1}{2}+s}}} \\
&\quad+ \pth{1-\frac{1}{p^{2s}}}\frac{\chi_\mathfrak{m}(p)}{p^{\frac{1}{2}-s}} \prod_{i=1}^4  \pth{1+\frac{1}{p^{u_i+s}}}\prod_{1\leq i < j \leq 4}  \pth{1-\frac{1}{p^{u_i+u_j+2s}}}\Bigg].
\end{align*}
Inserting this into \eqref{eq:Y-no-k1} and rewriting, we have
\begin{align*}
Y_p(\boldu,s;k_1,a) &= \prod_{1\leq i \leq j \leq 4} \pth{1-\frac{1}{p^{u_i+u_j+2s}}}^{-1}  \prod_{1\leq i \leq j \leq 4} \pth{1-\frac{1}{p^{u_i+u_j+1}}} \prod_{i=1}^4 \sumpth{1-\frac{\chi_\mathfrak{m}(p)}{p^{\frac{1}{2}+u_i+2s}}} \\
&\quad\times Z_{2,p}(\boldu,s;k_1,a),
\end{align*}
where
\begin{align}
Z_{2,p}(\boldu,s;k_1,a) &= \prod_{i=1}^4 \sumpth{1-\frac{\chi_\mathfrak{m}(p)}{p^{\frac{1}{2}+u_i}}}\sumpth{1-\frac{\chi_\mathfrak{m}(p)}{p^{\frac{1}{2}+u_i+2s}}}^{-1} \prod_{1\leq i\leq j\leq 4}\pth{1-\frac{1}{p^{u_i+u_j+1}}}^{-1}\nonumber\\
&\quad\times\Bigg[\frac{1}{p}\prod_{1\leq i \leq j \leq 4} \pth{1-\frac{1}{p^{u_i+u_j+2s}}} + C_p(\boldu+s) \sumpth{1-\frac{1}{p} - \frac{\chi_\mathfrak{m}(p)}{p^{\frac{1}{2}-s}} + \frac{\chi_\mathfrak{m}(p)}{p^{\frac{1}{2}+s}}}\nonumber \\
&\quad+ \pth{1-\frac{1}{p^{2s}}}\frac{\chi_\mathfrak{m}(p)}{p^{\frac{1}{2}-s}} \prod_{i=1}^4  \pth{1+\frac{1}{p^{u_i+s}}}\prod_{1\leq i < j \leq 4}  \pth{1-\frac{1}{p^{u_i+u_j+2s}}}\Bigg].
\label{eq:Z2, p not dividing k1}
\end{align}
The expression in brackets can now be expanded into a finite sum. The first summand in brackets is
\[
\frac{1}{p}\prod_{1\leq i \leq j \leq 4} \pth{1-\frac{1}{p^{u_i+u_j+2s}}} = \frac{1}{p} + O\fracp{1}{p^{2\eta+2\Re(s)+1}},
\]
the second is
\begin{align*}
&C_p(\boldu+s) \pth{1-\frac{1}{p} - \frac{\chi_\mathfrak{m}(p)}{p^{\frac{1}{2}-s}} + \frac{\chi_\mathfrak{m}(p)}{p^{\frac{1}{2}+s}}} \\
&= 1-\frac{1}{p} - \frac{\chi_\mathfrak{m}(p)}{p^{\frac{1}{2}-s}} + \frac{\chi_\mathfrak{m}(p)}{p^{\frac{1}{2}+s}} + O\pth{\frac{1}{p^{4\eta+4\Re(s)}} + \frac{1}{p^{4\eta+3\Re(s) + \frac{1}{2}}}  +  \frac{1}{p^{4\eta + 5\Re(s) + \frac{1}{2}}}},
\end{align*}
and the third is
\begin{align*}
&\pth{1-\frac{1}{p^{2s}}}\frac{\chi_\mathfrak{m}(p)}{p^{\frac{1}{2}-s}}\prod_{i=1}^4  \pth{1+\frac{1}{p^{u_i+s}}}\prod_{1\leq i < j \leq 4}  \pth{1-\frac{1}{p^{u_i+u_j+2s}}} \\
&= \frac{\chi_\mathfrak{m}(p)}{p^{\frac{1}{2}-s}} - \frac{\chi_\mathfrak{m}(p)}{p^{\frac{1}{2}+s}} + \sum_{i=1}^4 \frac{\chi_\mathfrak{m}(p)}{p^{\frac{1}{2}+u_i}} - \sum_{i=1}^4 \frac{\chi_\mathfrak{m}(p)}{p^{\frac{1}{2}+u_i+2s}} + O\pth{\frac{1}{p^{3\eta + 2\Re(s)+\frac{1}{2}}} + \frac{1}{p^{3\eta + 4\Re(s)+\frac{1}{2}}}}.
\end{align*}
Thus the expression in brackets in \eqref{eq:Z2, p not dividing k1} is 
\begin{align*}
1 + \sum_{i=1}^4 \frac{\chi_\mathfrak{m}(p)}{p^{\frac{1}{2}+u_i}} - \sum_{i=1}^4 \frac{\chi_\mathfrak{m}(p)}{p^{\frac{1}{2}+u_i+2s}} + O\pth{\frac{1}{p^{2\eta + 2\Re(s) + 1}} + \frac{1}{p^{4\eta+4\Re(s)}} + \frac{1}{p^{3\eta + 2\Re(s)+\frac{1}{2}}} + \frac{1}{p^{3\eta + 4\Re(s)+\frac{1}{2}}}},
\end{align*}
where we have used the fact $\eta+ \Re(s) \geq \eta+\theta>1/4$.
As well, we have
\[
\prod_{i=1}^4 \pth{1-\frac{\chi_\mathfrak{m}(p)}{p^{\frac{1}{2}+u_i}}} = 1 - \sum_{i=1}^4 \frac{\chi_\mathfrak{m}(p)}{p^{\frac{1}{2}+u_i}} + \sum_{1\leq i < j \leq 4} \frac{1}{p^{1+u_i+u_j}} +  O\fracp{1}{p^{3\eta+\frac{3}{2}}}.
\]
Multiplying the expressions in the previous two displays, we find that
\begin{align*}
&Z_{2,p}(\boldu,s;k_1,a) \\
&= \prod_{i=1}^4 \sumpth{1-\frac{\chi_\mathfrak{m}(p)}{p^{\frac{1}{2}+u_i+2s}}}^{-1} \prod_{1\leq i\leq j\leq 4}\pth{1-\frac{1}{p^{u_i+u_j+1}}}^{-1} \\
&\quad \times \sumpth{1 - \sum_{1\leq i \leq j\leq 4} \frac{1}{p^{u_i+u_j+1}} - \sum_{i=1}^4 \frac{\chi_\mathfrak{m}(p)}{p^{\frac{1}{2}+u_i+2s}} + O\pth{\frac{1}{p^{3\eta+2\theta+\frac{1}{2}}} + \frac{1}{p^{4\eta+4\theta}} + \frac{1}{p^{3\eta + 4\Re(s)+\frac{1}{2}}}}},
\end{align*}
where we have used the fact $3\eta+\frac{3}{2}, 1+2\eta + 2\Re(s) \geq 1+ 2\theta +3\eta$.
By Lemma \ref{lem:EulerFactorBound} we find that
\begin{equation}\label{eq:Z2bound,p not dividing k1}
Z_{2,p}(\boldu,s;k_1,a) = 1 + O\pth{\frac{1}{p^{3\eta+2\theta+\frac{1}{2}}} + \frac{1}{p^{4\eta+4\theta}} + \frac{1}{p^{3\eta + 4\Re(s)+\frac{1}{2}}}}.
\end{equation}
Combining this with \eqref{eq:colfY-no-k1}, we have
\begin{align*}
\colf_p(\boldu,s;k_1,a) &= \pth{1-\frac{1}{p^{2s}}}^{-1} \prod_{i=1}^4  \pth{1-\frac{\chi_\mathfrak{m}(p)}{p^{\frac{1}{2}+u_i}}}^{-1} \prod_{1\leq i\leq j \leq 4} \pth{1-\frac{1}{p^{u_i+u_j+2s}}}^{-1} \pth{1-\frac{1}{p^{u_i+u_j+1}}} \\
&\quad \times \prod_{i=1}^4 \pth{1-\frac{\chi_\mathfrak{m}(p)}{p^{\frac{1}{2}+u_i+2s}}}^{-1} Z_{2,p}(\boldu,s;k_1,a).
\end{align*}
As in Case 1, this is exactly the Euler factor at $p$ in the expression for $Z$ given in the lemma.

We now prove that in the regions (i) --- (iii) of the lemma, we have the bound
\[
Z_2(\boldu,s;k_1,a) \ll \tau(a).
\]
If $p\nmid 2a$, $p \mid k_1$, then 
\[
Z_{2,p}(\boldu,s;k_1,a) = 1 + O\fracp{1}{p^{4\eta+2\Re(s) + 2\theta}}
\]
by \eqref{eq:Z2bound,p divides k1}, and if $p\nmid 2ak_1$, then
\[
Z_{2,p}(\boldu,s;k_1,a) = 1 + O\sumpth{\frac{1}{p^{3\eta+2\theta+\frac{1}{2}}} + \frac{1}{p^{4\eta+4\theta}} + \frac{1}{p^{3\eta + 4\Re(s)+\frac{1}{2}}}}
\]
by \eqref{eq:Z2bound,p not dividing k1}. We note that in any of the regions (i) -- (iii), we have that all of
\[
4\eta + 2\Re(s) + 2\theta,\ 3\eta+2\theta+\half,\ 4\eta + 4\theta,\ 3\eta+4\Re(s) + \half
\]
are bounded below by $1+\theta'$ for some absolute $\theta' > 0$. It follows that we need only consider the ``correction'' factors for $p \mid a$ given in \eqref{eq:Z2,p divides 2a}. These are given by
\begin{align*}
&Z_{2,p}(\boldu,s;k_1,a) \\
&= \prod_{i=1}^4 \pth{1-\frac{\chi_\mathfrak{m}(p)}{p^{\frac{1}{2}+u_i}}}\pth{1-\frac{\chi_\mathfrak{m}(p)}{p^{\frac{1}{2}+u_i+2s}}}^{-1} \prod_{1\leq i \leq j \leq 4} \pth{1-\frac{1}{p^{u_i+u_j+2s}}}\pth{1-\frac{1}{p^{u_i+u_j+1}}}^{-1}.
\end{align*}
and thus
\begin{align*}
&\abs{Z_{2,p}(\boldu,s;k_1,a)} \\
&\leq \prod_{i=1}^4 \sumpth{1+\frac{1}{p^{\frac{1}{2}+\eta}}}\sumpth{1-\frac{1}{p^{\frac{1}{2}+\eta+2\Re(s)}}}^{-1} \prod_{1\leq i \leq j \leq 4} \pth{1+\frac{1}{p^{2\eta+2\Re(s)}}}\pth{1-\frac{1}{p^{2\eta+1}}}^{-1} \\
&=\prod_{i=1}^4 \sumpth{1+\frac{1}{p^{\eta+\frac{1}{2}+2 \min(0,\Re(s))}} + O\fracp{1}{p^{2\eta+4\Re(s) + 1}}} \prod_{1\leq i \leq j \leq 4} \pth{1+ O\fracp{1}{p^{2\eta+2\theta}}}.
\end{align*}
The worst case occurs in the region given by (ii), and we have
\[
\abs{Z_{2,p}(\boldu,s;k_1,a)} \leq \prod_{i=1}^4 \sumpth{1+O\sumpth{\frac{1}{p^\frac{1}{2}}}}^{14} \leq \prod_{i=1}^4 \sumpth{1+\frac{C}{p^\frac{1}{2}}}
\]
for some absolute positive constant $C$. Thus
\begin{align*}
\prod_{p \mid a} \abs{Z_{2,p}(\boldu,s;k_1,a)} &= \prod_{\substack{p \mid a\\ p\leq C^2}} \abs{Z_{2,p}(\boldu,s;k_1,a)} \prod_{\substack{p \mid a\\ p>C^2}} \abs{Z_{2,p}(\boldu,s;k_1,a)} \\
&\leq \prod_{\substack{p \mid a\\ p\leq C^2}} \sumpth{1+\frac{C}{p^\frac{1}{2}}} \prod_{\substack{p \mid a\\ p>C^2}} 2 \\
&\leq 2^{\omega(a)}\prod_{p\leq C^2} \sumpth{1+\frac{C}{p^\frac{1}{2}}} \ll \tau(a).
\end{align*}
\subsection{Proof of Lemma \ref{lem:DirichletSeriesSuma}}
Since $k_1 = 1$, note that if $p \mid a$, then by \eqref{eq:Z2,p divides 2a},
\begin{align*}
Z_{2,p}(\boldu,s;1,a) &= \prod_{i=1}^4 \sumpth{1-\frac{1}{p^{\frac{1}{2}+u_i}}}\sumpth{1-\frac{1}{p^{\frac{1}{2}+u_i+2s}}}^{-1} \prod_{1\leq i \leq j \leq 4} \pth{1-\frac{1}{p^{u_i+u_j+2s}}}\pth{1-\frac{1}{p^{u_i+u_j+1}}}^{-1}, \\
&= Z^\flat(p),
\end{align*}
say. Likewise, if $p\nmid 2a$, then by \eqref{eq:Y-no-k1} and the definition of $Y$ given in the statement of Lemma \ref{lem:DirichletSeriesOffDiagonal}, we have
\begin{align*}
Z_{2,p}(\boldu,s;1,a) &= Z^\flat(p)\Bigg[\frac{1}{p} + \pth{1-\frac{1}{p}}B_p(\boldu+s;0) + \pth{1-\frac{1}{p^{2s}}}\frac{1}{p^{\frac{1}{2}-s}}B_p(\boldu+s,1)\Bigg]\\
&= Z^\flat(p)Z^\sharp(p).
\end{align*}
say. Thus
\[
Z_2(\boldu,s;1,a) = \prod_{p\neq 2} Z^\flat(p) \prod_{p\nmid 2a} Z^\sharp(p) = \prod_{p\neq 2} Z^\flat(p)Z^\sharp(p) \sumpth{\prod_{p\mid a} Z^\sharp(p)}^{-1}.
\]
Since $Z_2(\boldu,s,1,a)$ is multiplicative in $a$, we have
\begin{align*}
Z_3(\boldu,s) &= \sumpth{\prod_{p\neq 2} Z^\flat(p)Z^\sharp(p)} \sum_{(a,2)=1} \frac{\mu(a)}{a^{2-2s}} \prod_{p\mid a} Z^\sharp(p)^{-1} \\
&=\sumpth{\prod_{p\neq 2} Z^\flat(p)}\prod_{p\neq 2} \pth{Z^\sharp(p)-\frac{1}{p^{2-2s}}}.
\end{align*}
Since
\begin{equation}\label{eq:Z3flat}
\prod_{p\neq 2} Z^\flat(p) = \prod_{1\leq i\leq j\leq 4}\frac{\zeta_2(u_i+u_j+1)}{\zeta_2(u_i+u_j+2s)} \prod_{i=1}^4 \frac{\zeta_2(\frac{1}{2}+u_i+2s)}{\zeta_2(\frac{1}{2}+u_i)},
\end{equation}
this gives \eqref{eq:Z3}. It remains to see that $Z_3(\boldu,s)$ is analytic and bounded uniformly in the region given by \eqref{eq:Z3Region}. First recall that
\[
Z^\sharp(p)= \frac{1}{p} + \pth{1-\frac{1}{p}}B_p(\boldu+s;0) + \sumpth{\frac{1}{p^{\frac{1}{2}-s}}-\frac{1}{p^{\frac{1}{2}+s}}}B_p(\boldu+s,1),
\]
and also
\[
B_p(\boldu+s,1) = \prod_{i=1}^4 \pth{1-\frac{1}{p^{u_i+s}}}^{-1} - B_p(\boldu+s,0).
\]
Thus
\[
Z^\sharp(p) = \frac{1}{p} + \sumpth{1-\frac{1}{p}-\frac{1}{p^{\frac{1}{2}-s}}+\frac{1}{p^{\frac{1}{2}+s}}}B_q(\boldu+s;0) + \prod_{i=1}^4 \pth{1-\frac{1}{p^{u_i+s}}}^{-1} \sumpth{\frac{1}{p^{\frac{1}{2}-s}}-\frac{1}{p^{\frac{1}{2}+s}}}.
\]
Note that in the region \eqref{eq:Z3Region}, we have
\[
1-\frac{1}{p}-\frac{1}{p^{\frac{1}{2}-s}}+\frac{1}{p^{\frac{1}{2}+s}} \asymp 1
\]
and also the inequalities
\begin{equation}\label{eq:Z3ineqs}
\frac{1}{2} + \Re(s) \geq \frac{9}{20}, \quad \Re(u_i) + \Re(s) \geq \frac{17}{60}.
\end{equation}
Since
\[
\prod_{1\leq i\leq j\leq 4} \pth{1-\frac{1}{p^{u_i+u_j+2s}}} B_p(\boldu+s;0) = C_p(\boldu+s) = 1 + O\fracp{1}{p^{4\theta+4\Re(s)}}
\]
by Lemma \ref{lem:BsumEval}, we have
\begin{align*}
&\prod_{1\leq i\leq j\leq 4} \pth{1-\frac{1}{p^{u_i+u_j+2s}}} \sumpth{Z^\sharp(p) - \frac{1}{p^{2-2s}}} \\
&= \prod_{1\leq i\leq j\leq 4} \pth{1-\frac{1}{p^{u_i+u_j+2s}}}\sumpth{\frac{1}{p}+\prod_{i=1}^4 \pth{1-\frac{1}{p^{u_i+s}}}^{-1} \sumpth{\frac{1}{p^{\frac{1}{2}-s}}-\frac{1}{p^{\frac{1}{2}+s}}}} \\
&\quad  + 1-\frac{1}{p}-\frac{1}{p^{\frac{1}{2}-s}}+\frac{1}{p^{\frac{1}{2}+s}}+ O\sumpth{\frac{1}{p^{4\theta+4\Re(s)}} +\frac{1}{p^{2-2\Re(s)}}} \\
&= \prod_{1\leq i\leq j\leq 4} \pth{1-\frac{1}{p^{u_i+u_j+2s}}}\sumpth{\prod_{i=1}^4 \pth{1-\frac{1}{p^{u_i+s}}}^{-1} \sumpth{\frac{1}{p^{\frac{1}{2}-s}}-\frac{1}{p^{\frac{1}{2}+s}}}} \\
&\quad + 1-\frac{1}{p^{\frac{1}{2}-s}}+\frac{1}{p^{\frac{1}{2}+s}} + O\sumpth{\frac{1}{p^{4\theta+4\Re(s)}} + \frac{1}{p^{1+2\theta+2\Re(s)}} + \frac{1}{p^{2-2\Re(s)}}}.
\end{align*}
Now
\begin{align*}
&\prod_{i=1}^4 \pth{1-\frac{1}{p^{u_i+s}}}^{-1} \sumpth{\frac{1}{p^{\frac{1}{2}-s}}-\frac{1}{p^{\frac{1}{2}+s}}} \\
&= \sumpth{\frac{1}{p^{\frac{1}{2}-s}}-\frac{1}{p^{\frac{1}{2}+s}}}\sumpth{1+\sum_{i=1}^4 \frac{1}{p^{u_i+s}} + O\fracp{1}{p^{2\theta+2\Re(s)}}}\\
&=\frac{1}{p^{\frac{1}{2}-s}}-\frac{1}{p^{\frac{1}{2}+s}} + \sum_{i=1}^4 \sumpth{\frac{1}{p^{u_i+\frac{1}{2}}} - \frac{1}{p^{u_i+2s+\frac{1}{2}}}} + O\sumpth{\frac{1}{p^{2\theta+\frac{1}{2}+\Re(s)}} + \frac{1}{p^{2\theta+\frac{1}{2}+3\Re(s)}}},
\end{align*}
and also
\[
\prod_{1\leq i\leq j\leq 4} \pth{1-\frac{1}{p^{u_i+u_j+2s}}} = 1 + O\fracp{1}{p^{2\theta+2\Re(s)}}.
\]
Thus
\begin{align*}
&\prod_{1\leq i\leq j\leq 4} \pth{1-\frac{1}{p^{u_i+u_j+2s}}}\sumpth{\prod_{i=1}^4 \pth{1-\frac{1}{p^{u_i+s}}}^{-1} \sumpth{\frac{1}{p^{\frac{1}{2}-s}}-\frac{1}{p^{\frac{1}{2}+s}}}} \\
&= \frac{1}{p^{\frac{1}{2}-s}}-\frac{1}{p^{\frac{1}{2}+s}} + \sum_{i=1}^4 \sumpth{\frac{1}{p^{u_i+\frac{1}{2}}} - \frac{1}{p^{u_i+2s+\frac{1}{2}}}}+ O\sumpth{\frac{1}{p^{2\theta+\frac{1}{2}+\Re(s)}} + \frac{1}{p^{2\theta+\frac{1}{2}+3\Re(s)}}}
\end{align*}
Combining these last few estimates and using the inequalities \eqref{eq:Z3ineqs}, we have
\begin{align*}
&\prod_{1\leq i\leq j\leq 4} \pth{1-\frac{1}{p^{u_i+u_j+2s}}} \sumpth{Z^\sharp(p) - \frac{1}{p^{2-2s}}} \\
&= 1 + \sum_{i=1}^4 \sumpth{\frac{1}{p^{u_i+\frac{1}{2}}} - \frac{1}{p^{u_i+2s+\frac{1}{2}}}} + O\sumpth{\frac{1}{p^{1+\frac{1}{60}}}} \\
&= \prod_{i=1}^4 \sumpth{1-\frac{1}{p^{u_i+\frac{1}{2}}}}^{-1} \sumpth{1-\frac{1}{p^{u_i+2s+\frac{1}{2}}}}\sumpth{1 + O\sumpth{\frac{1}{p^{1+\frac{1}{60}}}}}.
\end{align*}
Combining this with \eqref{eq:Z3flat}, we find that in the region \ref{eq:Z3Region}, we have
\begin{align*}
Z_3(\boldu,s) &=  \sumpth{\prod_{1\leq i\leq j\leq 4}\zeta(u_i+u_j+1)} \prod_{p\neq 2}\sumpth{1 + O\sumpth{\frac{1}{p^{1+\frac{1}{60}}}}}.
\end{align*}
This is clearly analytic and uniformly bounded in the region \eqref{eq:Z3Region}, as claimed.

\subsection{Proof of Lemma \ref{lem:iden}}
In this section, we prove Lemma \ref{lem:iden}. Write 
\begin{align}   
E_0^{\#}:
& =\sum_{\substack{n_1, n_2,n_3,n_4 \geq 0 \\ n_1 + n_2 + n_3 + n_4 \equiv 0 \, (\operatorname{mod} 2)}} \frac{1}{p^{(\frac{1}{2}+z_1)n_1+(\frac{1}{2}+z_2)n_2 +(\frac{1}{2}+z_3)n_3 +(\frac{1}{2}+z_4)n_4 }}\nonumber\\
& =\sumpth{\sum_{\substack{n_2 + n_3 + n_4 \, \operatorname{even} \\ n_1 \, \operatorname{even}}}
+
\sum_{\substack{n_2 + n_3 + n_4 \, \operatorname{odd} \\ n_1 \, \operatorname{odd}}} }
\frac{1}{p^{(\frac{1}{2}+z_1)n_1+(\frac{1}{2}+z_2)n_2 +(\frac{1}{2}+z_3)n_3 +(\frac{1}{2}+z_4)n_4 }}\nonumber\\
& =\left(1-\frac{1}{p^{1+2z_1}} \right)^{-1}\sum_{n_2 + n_3 + n_4 \, \operatorname{even}}
\frac{1}{p^{(\frac{1}{2}+z_2)n_2 +(\frac{1}{2}+z_3)n_3 +(\frac{1}{2}+z_4)n_4 }}\nonumber\\
&\quad+
\left[ \left(1-\frac{1}{p^{\frac{1}{2}+ z_1}} \right)^{-1}-\left(1-\frac{1}{p^{1+2z_1}} \right)^{-1} \right] \sum_{n_2 + n_3 + n_4 \, \operatorname{odd}}
\frac{1}{p^{(\frac{1}{2}+z_2)n_2 +(\frac{1}{2}+z_3)n_3 +(\frac{1}{2}+z_4)n_4 }}  \nonumber\\
& =\left(1-\frac{1}{p^{1+2z_1}} \right)^{-1}\sum_{n_2 + n_3 + n_4 \, \operatorname{even}}
\frac{1}{p^{(\frac{1}{2}+z_2)n_2 +(\frac{1}{2}+z_3)n_3 +(\frac{1}{2}+z_4)n_4 }}r\nonumber\\
&\quad+
\frac{1}{p^{\frac{1}{2}+ z_1}}  \left(1-\frac{1}{p^{1+2z_1}} \right)^{-1} \sum_{n_2 + n_3 + n_4 \, \operatorname{odd}}
\frac{1}{p^{(\frac{1}{2}+z_2)n_2 +(\frac{1}{2}+z_3)n_3 +(\frac{1}{2}+z_4)n_4 }} .
\label{E0New}
\end{align}
Similarly,  write 
\begin{align}   
E_1^{\#}
& =\sum_{\substack{n_1, n_2,n_3,n_4 \geq 0 \\ n_1 + n_2 + n_3 + n_4 \equiv 1 \, (\operatorname{mod} 2)}} \frac{1}{p^{(\frac{1}{2}+z_1)n_1+(\frac{1}{2}+z_2)n_2 +(\frac{1}{2}+z_3)n_3 +(\frac{1}{2}+z_4)n_4 }}\nonumber\\
& =\sumpth{\sum_{\substack{n_2 + n_3 + n_4 \, \operatorname{odd} \\ n_1 \, \operatorname{even}}}
+
\sum_{\substack{n_2 + n_3 + n_4 \, \operatorname{even} \\ n_1 \, \operatorname{odd}}} }
\frac{1}{p^{(\frac{1}{2}+z_1)n_1+(\frac{1}{2}+z_2)n_2 +(\frac{1}{2}+z_3)n_3 +(\frac{1}{2}+z_4)n_4 }}\nonumber\\
& =\left(1-\frac{1}{p^{1+2z_1}} \right)^{-1}\sum_{n_2 + n_3 + n_4 \, \operatorname{odd}}
\frac{1}{p^{(\frac{1}{2}+z_2)n_2 +(\frac{1}{2}+z_3)n_3 +(\frac{1}{2}+z_4)n_4 }}\nonumber\\
&\quad+
\frac{1}{p^{\frac{1}{2}+ z_1}}  \left(1-\frac{1}{p^{1+2z_1}} \right)^{-1} \sum_{n_2 + n_3 + n_4 \, \operatorname{even}}
\frac{1}{p^{(\frac{1}{2}+z_2)n_2 +(\frac{1}{2}+z_3)n_3 +(\frac{1}{2}+z_4)n_4 }}  .
 \label{E1New}
 \end{align}
It follows from \eqref{E0New} and \eqref{E1New} that
\begin{align}
&\frac{1}{p} + \left(1 - \frac{1}{p} \right)  E_0^{\#} + \frac{1}{p^{\frac{1}{2}-z_1}} 
\left(1 - \frac{1}{p^{2z_1}} \right)E_1^{\#} -\frac{1}{p^{2-2z_1}}\nonumber \\
&=
K_1 +  K_2 \sum_{n_2 + n_3 + n_4 \, \operatorname{even}}
\frac{1}{p^{(\frac{1}{2}+z_2)n_2 +(\frac{1}{2}+z_3)n_3 +(\frac{1}{2}+z_4)n_4 }}\nonumber\\
&\quad+ K_3  \sum_{n_2 + n_3 + n_4 \, \operatorname{odd}}
\frac{1}{p^{(\frac{1}{2}+z_2)n_2 +(\frac{1}{2}+z_3)n_3 +(\frac{1}{2}+z_4)n_4 }} ,
\label{Iden1}
\end{align}
where 
\begin{align*}
K_1 &:= \frac{1}{p} \left(1-\frac{1}{p^{1-2z_1}}\right) , \nonumber\\
K_2 &:= \left(1-\frac{1}{p}\right)\left( 1 - \frac{1}{p^{1+ 2z_1}}\right)^{-1}+ \frac{1}{p} \left( 1- \frac{1}{p^{2z_1}}\right)\left(1- \frac{1}{p^{1+ 2z_1}} \right)^{-1},\nonumber\\
K_3&:=\frac{1}{p^{\frac{1}{2}+ z_1}}  \left(1-\frac{1}{p}\right)\left( 1- \frac{1}{p^{1+ 2z_1}}\right)^{-1} + \frac{1}{p^{\frac{1}{2}-z_1}}
\left(1-\frac{1}{p^{2z_1}} \right)\left(1- \frac{1}{p^{1+2z_1}} \right)^{-1}.
\end{align*}
We simplify $K_2, K_3$ by
\begin{align*}
K_2 &= \left(1- \frac{1}{p^{1+2z_1}} \right) ^{-1} \left[1-\frac{1}{p}+ \frac{1}{p}\left(1- \frac{1}{p^{2z_1}} \right) \right]= 1,\nonumber\\
K_3&= \left(1- \frac{1}{p^{1+2z_1}} \right) ^{-1} \left(-\frac{1}{p^{\frac{3}{2}+ z_1}} + \frac{1}{p^{\frac{1}{2}-z_1}}\right) = \frac{1}{p^{\frac{1}{2}-z_1}}.
\end{align*}
Together with \eqref{Iden1}, this gives 
\begin{align}  
&\frac{1}{p} + \left(1 - \frac{1}{p} \right)  E_0^{\#} + \frac{1}{p^{\frac{1}{2}-z_1}} 
\left(1 - \frac{1}{p^{2z_1}} \right)E_1^{\#} -\frac{1}{p^{2-2z_1}}\nonumber\\
&=\frac{1}{p} \left(1 - \frac{1}{p^{1-2z_1}} \right) + \sum_{n_2 + n_3 + n_4 \, \operatorname{even}}
\frac{1}{p^{(\frac{1}{2}+z_2)n_2 +(\frac{1}{2}+z_3)n_3 +(\frac{1}{2}+z_4)n_4 }}
\nonumber \\
&\quad+ \frac{1}{p^{\frac{1}{2}-z_1}}     \sum_{n_2 + n_3 + n_4 \, \operatorname{odd}}
\frac{1}{p^{(\frac{1}{2}+z_2)n_2 +(\frac{1}{2}+z_3)n_3 +(\frac{1}{2}+z_4)n_4 }} .
 \label{Iden2}
 \end{align}
By \eqref{E0New} with  replacing $z_1$ by $-z_1$, we see 
\begin{align}   
& \sum_{\substack{n_1, n_2,n_3,n_4 \geq 0 \\ n_1 + n_2 + n_3 + n_4 \equiv 0 \, (\operatorname{mod} 2)}} \frac{1}{p^{(\frac{1}{2}-z_1)n_1+(\frac{1}{2}+z_2)n_2 +(\frac{1}{2}+z_3)n_3 +(\frac{1}{2}+z_4)n_4 }}\nonumber\\
&= \left(1-\frac{1}{p^{1-2z_1}} \right)^{-1}  \Bigg( \sum_{n_2 + n_3 + n_4 \, \operatorname{even}}
\frac{1}{p^{(\frac{1}{2}+z_2)n_2 +(\frac{1}{2}+z_3)n_3 +(\frac{1}{2}+z_4)n_4 }}\nonumber \\
&\quad+ \frac{1}{p^{\frac{1}{2}-z_1}}\sum_{n_2 + n_3 + n_4 \, \operatorname{odd}}
\frac{1}{p^{(\frac{1}{2}+z_2)n_2 +(\frac{1}{2}+z_3)n_3 +(\frac{1}{2}+z_4)n_4 }}\Bigg).
 \label{Iden2.5}
 \end{align}
The last two terms of \eqref{Iden2} are exactly identical to the last two terms in the parentheses of \eqref{Iden2.5}. Inserting  \eqref{Iden2.5} into \eqref{Iden2} we get 
\begin{align*}
&\frac{1}{p} + \left(1 - \frac{1}{p} \right)  E_0^{\#} + \frac{1}{p^{\frac{1}{2}-z_1}} 
\left(1 - \frac{1}{p^{2z_1}} \right)E_1^{\#} -\frac{1}{p^{2-2z_1}}\nonumber \\
&=\frac{1}{p} \left(1 - \frac{1}{p^{1-2z_1}} \right) 
+ \left(1-\frac{1}{p^{1-2z_1}} \right)
\sum_{\substack{n_1, n_2,n_3,n_4 \geq 0 \\ n_1 + n_2 + n_3 + n_4 \equiv 0 \, (\operatorname{mod} 2)}} \frac{1}{p^{(\frac{1}{2}-z_1)n_1+(\frac{1}{2}+z_2)n_2 +(\frac{1}{2}+z_3)n_3 +(\frac{1}{2}+z_4)n_4 }}\nonumber \\
&=  \left(1-\frac{1}{p^{1-2z_1}} \right) \sumpth{\frac{1}{p}+   \sum_{\substack{n_1, n_2,n_3,n_4 \geq 0 \\ n_1 + n_2 + n_3 + n_4 \equiv 0 \, (\operatorname{mod} 2)}} \frac{1}{p^{(\frac{1}{2}-z_1)n_1+(\frac{1}{2}+z_2)n_2 +(\frac{1}{2}+z_3)n_3 +(\frac{1}{2}+z_4)n_4 }} }.
\end{align*}
This completes the proof of Lemma \ref{lem:iden}(i). 

By \eqref{eq:Z3} and Lemma \ref{lem:iden}(i), the Euler factors of $Z_4(\tfrac{1}{2},\tfrac{1}{2}+z_2-z_1,\tfrac{1}{2}+z_3-z_1,\tfrac{1}{2}+z_4-z_1,z_1)$ are 
\begin{align} 
&\left(1-\frac{1}{p} \right)\left( 1- \frac{1}{p^{1-2z_1}}\right)^{-1}  \prod_{\beta,\gamma \in \{-z_1,z_2,z_3,z_4\}} \left(1-\frac{1}{p^{1+ \beta+ \gamma}} \right)\nonumber \\
&\quad \times \left( \frac{1}{p}+\left(1-\frac{1}{p}\right)B_p(\boldsymbol{z}+\tfrac{1}{2};0) +\left(1-\frac{1}{p^{2z_1}}\right)\frac{1}{p^{\frac{1}{2}-z_1}}B_p(\boldsymbol{z}+\tfrac{1}{2};1)-\frac{1}{p^{2-2z_1}}\right)\nonumber\\
&=\left(1-\frac{1}{p} \right)\left( 1- \frac{1}{p^{1-2z_1}}\right)^{-1}  \prod_{\beta,\gamma \in \{-z_1,z_2,z_3,z_4\}} \left(1-\frac{1}{p^{1+ \beta+ \gamma}} \right)\nonumber\\
&\quad\times 
\left(1-\frac{1}{p^{1-2z_1}} \right)\sumpth{\frac{1}{p}+   \sum_{\substack{n_1, n_2,n_3,n_4 \geq 0 \\ n_1 + n_2 + n_3 + n_4 \equiv 0 \, (\operatorname{mod} 2)}} \frac{1}{p^{(\frac{1}{2}-z_1)n_1+(\frac{1}{2}+z_2)n_2 +(\frac{1}{2}+z_3)n_3 +(\frac{1}{2}+z_4)n_4 }} }\nonumber\\
&=\left(1- \frac{1}{p^2} \right) \frac{p}{p+1}  \prod_{\beta,\gamma  \in \{-z_1,z_2,z_3,z_4\}} \left(1-\frac{1}{p^{1+ \beta+ \gamma}} \right)\nonumber\\
&\quad\times 
\sumpth{\frac{1}{p}+   \sum_{\substack{n_1, n_2,n_3,n_4 \geq 0 \\ n_1 + n_2 + n_3 + n_4 \equiv 0 \, (\operatorname{mod} 2)}} \frac{1}{p^{(\frac{1}{2}-z_1)n_1+(\frac{1}{2}+z_2)n_2 +(\frac{1}{2}+z_3)n_3 +(\frac{1}{2}+z_4)n_4 }} }.
 \label{Iden4}
 \end{align}
 On the other hand, it follows from \eqref{equ:D1-111} and  $\eqref{equ:D1-112}$ that 
\begin{align*}
&\colh_{2,p}(u_1, u_2, u_3, u_4)\\
&= \prod_{1\leq i\leq j\leq {4}} \left(1- \frac{1}{p^{u_i +u_j}} \right)\sumpth{1+g_2(p)\sum_{\substack{n_1,n_2,n_3,n_4\geq0\\ n_1+n_2+n_3+n_4\equiv 0 \mod{2}\\ n_1+n_2+n_3+n_4>0}} \frac{1}{p^{n_1u_1+n_2u_2+n_3u_3+n_4u_4}}} .
\end{align*}
Then  
\begin{align*}
&\colh_{2,p}(\tfrac{1}{2}-z_1, \tfrac{1}{2}+ z_2, \tfrac{1}{2}+z_3, \tfrac{1}{2}+z_4)\\
&= \prod_{\beta,\gamma  \in \{-z_1,z_2,z_3,z_4\}} \left(1-\frac{1}{p^{1+ \beta+ \gamma}} \right)\\
&\quad \times \sumpth{1+\frac{p}{p+1}\sum_{\substack{n_1,n_2,n_3,n_4\geq0\\ n_1+n_2+n_3+n_4\equiv 0 \mod{2}\\ n_1+n_2+n_3+n_4>0}} \frac{1}{p^{(\frac{1}{2}-z_1)n_1+(\frac{1}{2}+ z_2)n_2+(\frac{1}{2}+ z_3)n_3+(\frac{1}{2}+ z_4)n_4}}} \\
&=\frac{p}{p+1} \prod_{\beta,\gamma  \in \{-z_1,z_2,z_3,z_4\}} \left(1-\frac{1}{p^{1+ \beta+ \gamma}} \right)\\
&\quad \times \sumpth{\frac{1}{p}+\sum_{\substack{n_1,n_2,n_3,n_4\geq0\\ n_1+n_2+n_3+n_4\equiv 0 \mod{2}}} \frac{1}{p^{(\frac{1}{2}-z_1)n_1+(\frac{1}{2}+ z_2)n_2+(\frac{1}{2}+ z_3)n_3+(\frac{1}{2}+ z_4)n_4}}},
\end{align*}
which together with \eqref{Iden4} completes the proof.

\bibliography{references}
\vspace{5em}

\noindent {\scshape SDU-ANU Joint Science College, Shandong University, Weihai 264209, China}\\
\textit{Email address:} {\UrlFont{qlshen@sdu.edu.cn}}\\

\noindent {\scshape Mathematics Department, University of Georgia, Athens, GA 30605}\\
\textit{Email address:} {\UrlFont{joshua.stucky@uga.edu}}

\end{document}